\documentclass[a4paper,10pt,reqno]{amsart}
\usepackage[english]{babel}
\usepackage{amsmath, amssymb, amsthm, amscd}
\usepackage{enumerate}
\usepackage{palatino}
\usepackage{mathpazo}
\usepackage{graphicx}
\usepackage{caption}
\usepackage{subcaption}
\usepackage{mathrsfs}
\usepackage{paralist}
\usepackage{bm}
\usepackage[a4paper]{geometry}
\usepackage{url}
\usepackage{hyperref}

\numberwithin{equation}{section}

\DeclareMathOperator{\ind}{ind}
\DeclareMathOperator{\im}{im}
\DeclareMathOperator{\id}{id}

\DeclareMathOperator{\sign}{sign}

\newcommand{\NN}{\mathbb{N}}

\newcommand{\RR}{\mathbb{R}}
\newcommand{\ZZ}{\mathbb{Z}}

\newcommand{\bRR}{\overline{\mathbb{R}}}

\DeclareMathOperator{\codim}{codim}

\newcommand{\LL}{\mathcal{L}}
\newcommand{\WW}{\mathcal{W}}

\newcommand{\XX}{\mathfrak{X}}

\newcommand{\XXb}{\mathfrak{X}^{\text{back}}}
\newcommand{\pib}{\pi^{\text{back}}}
\newcommand{\XXtilde}{\tilde{\mathfrak{X}}}

\newcommand{\algint}{\#_{\text{alg}}}
\newcommand{\orint}{\#_{\text{or}}}
\DeclareMathOperator{\Crit}{Crit }

\newcommand{\Eunder}{\underline{E}}

\newcommand{\fatX}{\mathbf{X}}
\newcommand{\fatY}{\mathbf{Y}}
\newcommand{\fatZ}{\mathbf{Z}}

\newcommand{\Wtilde}{\widetilde{W}}
\newcommand{\Mtilde}{\widetilde{\mathcal{M}}}

\newcommand{\gammaunder}{\underline{\gamma}}

\newcommand{\Cloc}{C^\infty_{\mathrm{loc}}}

\newcommand{\vout}{v_{\mathrm{out}}}
\newcommand{\vin}{v_{\mathrm{in}}}

\newcommand{\RTree}{\mathrm{RTree}}
\newcommand{\BinTree}{\mathrm{BinTree}}

\newcommand{\Split}{\mathrm{split}}

\newcommand{\pdd}[2]{\frac{\partial #1}{\partial #2}}

\newcommand{\MM}{\mathcal{M}}

 \newcommand{\PP}{\mathcal{P}}

\newcommand{\BB}{\mathcal{B}}

\newcommand{\GG}{\mathcal{G}}

\newcommand{\bWW}{\overline{\mathcal{W}}}

\newcommand{\Acal}{\mathcal{A}}
\newcommand{\bAcal}{\overline{\mathcal{A}}}

\newcommand{\morseind}{\mu_{\text{Morse}}}

\newcommand{\zero}{\mathbb{O}}

\newcommand{\tv}{\pitchfork}

\theoremstyle{plain}
\newtheorem{theorem}{Theorem}[section]
\newtheorem{prop}[theorem]{Proposition}
\newtheorem{lemma}[theorem]{Lemma}
\newtheorem{cor}[theorem]{Corollary}  
\newtheorem*{theorem*}{Theorem}

\theoremstyle{definition}
\newtheorem{definition}[theorem]{Definition}
\newtheorem*{definition*}{Definition}

\theoremstyle{remark}
\newtheorem{remark}[theorem]{Remark}
\newtheorem{example}[theorem]{Example}

\begin{document}
\setlength{\parindent}{0cm}
\setcounter{tocdepth}{1}

\title{Perturbed gradient flow trees and $A_\infty$-algebra structures on Morse cochain complexes}
\author{Stephan Mescher}
\address{School of Mathematical Sciences \\ Queen Mary University of London \\ Mile End Road \\ London E1 4NS}
\email{s.mescher@qmul.ac.uk}

\date{\today}

\begin{abstract}
We elaborate on an idea of M. Abouzaid of equipping the Morse cochain complex of a smooth Morse function on a closed oriented manifold with the structure of an $A_\infty$-algebra. This is a variation on K. Fukaya's definition of Morse-$A_\infty$-categories for closed oriented manifolds involving families of Morse functions. The purpose of this article is to provide a coherent and detailed treatment of Abouzaid's approach including a discussion of all relevant analytic notions and results.
\end{abstract} 

\maketitle

\tableofcontents

\section*{Introduction}  
  
\subsection*{Morse (co)homology}  
  
The upshot of the construction of Morse homology is that with every complete finite-dimensional manifold $M$ and every Morse function $f:M \to \RR$ which is bounded from below one can associate a chain complex of free abelian groups generated by the critical points of $f$. We denote this complex by
  \begin{equation*}
   C_*(f) := \bigoplus_{x \in \Crit f} \ZZ \cdot x \ , \quad \text{where} \ \ C_j(f) = \bigoplus_{\morseind(x,f)=j} \ZZ \cdot x \ .
  \end{equation*}
  Here, $\Crit f \subset M$ denotes the set of critical points of a Morse function and $\morseind$ denotes the Morse index of a critical point. Since the Morse index is bounded from below by zero and from above by the dimension of $M$, the group $C_j(f)$ is only non-zero if $j \in \{0,1,\dots,n\}$. 
  
  A differential on $C_*(f)$ is defined via counting elements of zero-dimensional moduli spaces of negative gradient flow trajectories. More precisely, for $x,y \in \Crit f$ we define 
  \begin{equation*} 
   \MM(x,y,g) := \Bigl\{\gamma:\RR \stackrel{C^\infty}{\longrightarrow} M \ \Big| \ \dot{\gamma} = - (\nabla^g f) \circ \gamma, \ \lim_{s \to -\infty} \gamma(s) = x , \ \lim_{s \to +\infty}\gamma(s) = y \Big\} \ , 
  \end{equation*} 
  where $g$ is a Riemannian metric on $M$ and $\nabla^g f$ denotes the gradient vector field of $f$ with respect to $g$. It is shown in \cite{Schwarz} that for generic choice of $g$, this space is a smooth manifold of dimension $\mu(x)-\mu(y)$ for any such $x$ and $y$. Moreover, the reparametrization $\RR \times \MM(x,y,g) \to \MM(x,y,g)$, $(s,\gamma) \mapsto \gamma(\cdot+s)$, is a smooth, free and proper $\RR$-action on $\MM(x,y,g)$ whenever $x \neq y$, implying that the quotient space by this group action, the space of unparametrized flow lines from $x$ to $y$ 
  \begin{equation*}
   \widehat{\MM}(x,y,g) := \MM(x,y,g) / \RR \ ,
  \end{equation*} 
  is a manifold of dimension $\mu(x)-\mu(y)-1$. Moreover, it is shown in \cite{Schwarz} that if $\mu(x) = \mu(y)+1$, the space $\widehat{\MM}(x,y,g)$ is a finite set and we can use certain orientation constructions to define an algebraic count of elements of $\widehat{\MM}(x,y,g)$. Define this number by $n(x,y,g) := \algint \widehat{\MM}(x,y,g)$.
  
  For a generically chosen metric $g$, the differential on $C_*(f)$ is then defined by 
  \begin{equation*}
   \partial^g: C_*(f) \to C_{*-1}(f) \ , \quad (x \in \Crit f) \mapsto \sum_{\stackrel{y \in \Crit f}{\mu(y)=\mu(x)-1}} n(x,y,g) \cdot y \ ,
  \end{equation*}
  and by extending this definition $\ZZ$-linearly to all of $C_*(f)$. \bigskip
    
  There is also the dual notion of Morse cohomology. The Morse cochain complex is defined by \index{Morse cochain complex} 
   \begin{align*}
   &C^*(f) := \prod_{x \in \Crit f} \ZZ \cdot x \ , \quad \text{where} \ \ C^j(f) = \prod_{\morseind(x,f)=j} \ZZ \cdot x \ , \\
   &\delta^g: C^*(f) \to C^{*+1}(f) \ , \quad (x \in \Crit f) \mapsto \sum_{\stackrel{z \in \Crit f}{\mu(z)=\mu(x)+1}} n(z,x,g) \cdot z \ ,
  \end{align*}
  and by extending the codifferential $\ZZ$-linearly to $C^*(f)$. The duality between the Morse chain and cochain complex is obvious. 
  
  By considering compactifications of one-dimensional spaces of unparametrized flow lines, one shows that indeed $\partial^g \circ \partial^g = 0$ and $\delta^g \circ \delta^g=0$. Moreover, the (co)ho\-mo\-lo\-gy of these complexes is isomorphic to singular (co)homology with integer coefficients: 
  \begin{equation*}
   H_*(f) \cong H^{\text{sing}}_*(M;\ZZ) \ , \quad H^*(f) \cong H^*_{\text{sing}}(M;\ZZ) \ ,
  \end{equation*} 
  where $H_*(f) := H_*(C_*(f),\partial^g)$ and $H^*(f) := H^*(C^*(f),\delta^g)$. In particular, the (co)ho\-mo\-lo\-	gy of the Morse (co)chain complex is independent of the choice of the Morse function and the Riemannian metric. The original proof of this statement is spread throughout the literature and usually attributed to Thom, Smale and Witten.
  
  \subsection*{$A_\infty$-categories of Morse cochain complexes}
  Having realized the singular cohomology of a manifold in terms of Morse theory, it is natural to ask whether the cup product of singular cohomology possesses a Morse-theoretic realization as well. The answer is yes and such a description of the cup product was discovered by Kenji Fukaya, as well as a family of higher order multiplications on Morse cochain complexes 
  \begin{equation*}
  C^*(f_1) \otimes C^*(f_2) \otimes \dots \otimes C^*(f_d) \to C^*(f_0) 
  \end{equation*}
  for every $d \geq 2$, where $f_0, f_1,\dots,f_d \in C^{\infty}(M)$ are Morse functions and $M$ is a closed oriented manifold. \index{Morse cup product}
  
  These higher order multiplications had their first appearance in \cite{FukayaAinfty} and were further elaborated on in \cite{FukayaQuant}. The analytic details were carried out by Fukaya and Yong-Geun Oh in \cite{FukayaOh}. See also \cite{AbouzaidMorseTrop} for a generalization to compact manifolds with boundary. \bigskip
  
  However, there is a huge difference between the singular and the Morse-theoretic cup product. The singular one has the property of being associative on the level of cochains, while the Morse-theoretic one is \emph{not} associative on the cochain level. Indeed, Fukaya has shown that the family of higher order multiplications satisfies the $A_\infty$-equations. 
  
  A technical difficulty that occurs is that the multiplications are maps from the product of the Morse cochain complexes of $d$ different Morse functions instead of the $d$-fold product of the Morse cochain complex of a single Morse function. Here, it is not possible to choose the same Morse function in every factor. To understand this, we have to delve deeper into the definitions of the higher order multiplications. \bigskip
  
  
  Similar to the Morse differential and codifferential, the higher order multiplications are defined via counting elements of zero-dimensional moduli spaces. For $d \geq 2$, these moduli spaces consist of continuous maps from a rooted tree (seen as a one-dimensional CW complex) to the manifold $M$ which map the root and the leaves of the tree to fixed critical points, while they edgewise fulfill negative gradient flow equations. For example, let $f_i \in C^\infty(M)$, $x_i \in \Crit f_i$ for $i \in \{0,1,\dots,d\}$ and $T$ be a $d$-leafed rooted binary tree. We consider maps $I:T \to M$ with $I(v_i) = x_i$ for every $i$, where $v_0$ denotes the root and $v_i$ denotes the $i$-th leaf of $T$ for every $i \in \{1,2,\dots,d\}$. Moreover, we want that for every edge $e$ there is a Morse function $f_e$ such that \index{Morse ribbon tree}
  \begin{equation*} 
  \dot{I}_e = - (\nabla^g f_e)\circ I_e
  \end{equation*}
  where $I_e$ denotes the restriction of $I$ to $e$. (The edges of $T$ will be parametrized by suitable intervals.) The property of the corresponding moduli spaces being smooth manifolds can be rephrased as a transverse intersection problem. Therefore, the Morse functions involved have to satisfy certain transversality conditions which make it impossible to choose $f_1=\dots=f_d$. 
  
  Fukaya's method is to start with arbitrary, but distinct, Morse functions associated with every edge leading to a leaf of $T$. For every other edge $e$ we define $f_e$ as the sum of the functions $f_{e'}$ for every $e'$ whose incoming vertex coincides with the outgoing vertex of $e$. For a generic choice of Riemannian metric $g$, the corresponding moduli space is then a smooth manifold. 
  
  For example, the multiplication for $d=2$, which is nothing but the cochain-level cup product, is defined as a map
  \begin{equation*}
   C^*(f_1)\otimes C^*(f_2) \to C^*(f_1+f_2) \ ,
  \end{equation*}
  if $f_1$, $f_2$ and $f_1+f_2$ are Morse functions on $M$. Using these properties, Fukaya was able to define an $A_\infty$-category whose objects are smooth functions on $M$ and whose morphism sets are Morse cochain complexes. 
  
\subsection*{Abouzaid's approach: perturbing gradient flow lines}
  
In several situations, for example if one is interested in the Hochschild homology of Morse cochain complexes, it is difficult to work with an $A_\infty$-category due to the presence of infinitely many cochain complexes as objects. Instead, it is preferrable to work with an $A_\infty$-algebra, i.e. an $A_\infty$-category that has precisely one object, see \cite{MescherHochschild}. Such a construction is provided by an alternative approach to Fukaya's $A_\infty$-structures in Morse theory, which was suggested by Mohammed Abouzaid in \cite{AbouzaidPlumbings}. Abouzaid defines products $C^*(f)^{\otimes d} \to C^*(f)$ for a closed oriented manifold $M$, a fixed Morse function $f:M \to \RR$ and every $d \geq 2$ by considering moduli spaces of continuous maps from a tree to the manifold, which are trajectories of \emph{perturbations} of the negative gradient flow of a Morse function on a closed manifold. 
  Using transversality theory, one shows that for generic choices of perturbations, the resulting moduli spaces will be manifolds. Then one can construct the higher order multiplications along the same lines as in Fukaya's works by counting elements of zero-dimensional moduli spaces. 
  
  Abouzaid has shown that for good choices of perturbations, the Morse cochain complex $C^*(f)$ with the resulting higher order multiplications becomes an $A_\infty$-algebra. Moreover, his results imply that the Hochschild homology of this $A_\infty$-algebra is independent of the choices of perturbations and isomorphic to the Hochschild homology of the singular cochains of the manifold. \bigskip
  
  In fact, Abouzaid's perturbation method is the Morse-theoretic analogue of the more sophisticated Floer-theoretic perturbations constructed by Paul Seidel in \cite[Part II]{SeidelBook}. Seidel defined perturbations for pseudo-holomorphic curves which are modelled on boundary-pointed disks, which require more subtle analytic methods, especially concerning compactness properties of the resulting moduli spaces. \bigskip
  
  The aim of this article is to give a detailed construction of the higher order multiplications on the Morse cochains of a fixed Morse function using the perturbation methods of Abouzaid. Moreover, we want to show that for convenient choices of perturbations the Morse cochain complex equipped with these multiplications becomes an $A_\infty$-algebra. We will provide the reader with many analytic details and a self-contained construction of higher order multiplications after Abouzaid's ideas. \bigskip
  
  
  \emph{The basic ideas of this article are implicitly contained in \cite{AbouzaidPlumbings}. Many of the results are stated or partially proven in Abouazid's work and we will provide the analytic details underlying Abouzaid's methods. For the sake of readability, we refrain from giving a reference to the corresponding result in Abouzaid's work at every result we are considering and see this paragraph as a general reference to \cite{AbouzaidPlumbings}.} \bigskip

  We assume that the reader is familiar with standard results of Morse theory as well as the construction of Morse homology. The classic reference for Morse theory is \cite{MilMorse}, while newer treatments include \cite{NicoMorse}, \cite{MatsuMorse} and Chapter 6 of \cite{JostRiemGeom}. 
  
  There are also certain references specializing on Morse homology. We will loosely orient our treatment on the textbook \cite{Schwarz}, which stresses the meaning of Morse trajectory spaces in the construction of Morse homology. Further references include \cite{BanyagaHurtubise}, which takes a classic differential-topological approach, \cite{HutchingsNotes}, \cite{WeberMorseWitten} and \cite{AbboMajer}. \bigskip
  
  We assume that the reader is familiar with certain constructions related to Morse homology. Throughout this article and especially in Section \ref{SectionPerturbationsGradFlow} we will use several results for spaces of (unperturbed) negative gradient flow trajectories and give references to those results whenever they are needed. Moreover, we will rely on certain gluing analysis results for Morse trajectories which transfer with only minor modifications to our situation. The necessary results are contained in \cite{Schwarz}, \cite{SchwarzEqui} and \cite{WehrheimMWC}.  
  
  Furthermore, Sections \ref{SectionModuliSpacesPerturbed} to \ref{CompactificationsOneDimAinfty} require certain results from graph theory which we will mostly state without providing proofs. The proofs of those results are elementary and therefore left to the reader. \bigskip
  
  In Section 1, we carefully introduce the necessary spaces of perturbations and the different kinds of perturbed negative gradient flow trajectories. We will rely on certain results on unperturbed negative gradient flow trajectories which we will briefly present. Thereupon, we discuss moduli spaces of these perturbed trajectories and certain evaluation maps defined on the moduli spaces.
  
  The results about the evaluation maps are further generalized in Section 2. After introducing certain more general perturbation spaces, we consider moduli spaces of families of perturbed Morse trajectories satisfying endpoint conditions, which are in a certain sense nonlocal. Finally, we state and prove a general regularity theorem about these moduli spaces.
  
  This nonlocal regularity theorem is applied to a special situation in Section 3. This section starts with the introduction of several notions from graph theory. Afterwards, we discuss how the results from Section 2 can be applied to consider moduli spaces of maps from a tree to a manifold which edgewise fulfill a perturbed negative gradient flow equation.
 
  In Section 4 we investigate sequential compactness properties of the moduli spaces from Section 3. We introduce a notion of geometric convergence with which we can describe the possible limiting behaviour of sequences in the moduli spaces of perturbed negative gradient flow trees. We further give explicit descriptions of the moduli spaces in which the respective limits are lying. 
  
  The objects of study in Section 5 are zero- and one-dimensional moduli spaces of perturbed negative gradient flow trees. By applying the results of Section 4 to zero-dimensional spaces we are able to define multiplications on Morse cochain complexes via counting elements of these moduli spaces. Applying the abovementioned gluing results from Morse theory, we will describe compactifications of one-dimensional moduli spaces and finally use the counts of their boundary components to prove that the Morse cochain complex equipped with the multiplications becomes an $A_\infty$-algebra. \bigskip

  A discussion of orientability and orientations of the moduli spaces of perturbed Morse ribbon trees is required in order to construct higher order multiplications with integer coefficients, since they will be constructed by using oriented intersection numbers. For the benefit of the flow of reading, all results about orientations are not contained in the main part of this article, but are found in Appendix \ref{AppendixOrient}. We will give references to the respective parts of this appendix whenever required. \bigskip

\textit{Throughout this article, let $M$ be a closed $n$-dimensional smooth oriented manifold, $n \in \NN_0$, and $f \in C^\infty (M)$ be a Morse function. Let $\Crit f := \{x \in M \ | \ df_x=0\}$ denote the set of critical points of $f$ and for any $x \in \Crit f$ let $\mu(x)$ denote its Morse index with respect to $f$. We will further use the convention that the empty set is a manifold of every dimension.} 
  
\section*{Acknowledgements}

The contents of this article are part of the author's Ph. D. thesis \cite{MescherThesis}, which was submitted to the University of Leipzig and supervised by Prof. Dr. Matthias Schwarz. The author thanks Prof. Schwarz for his guidance and support. Moreover, the author is grateful to Prof. Dr. Alberto Abbondandolo for many helpful discussions.

The author further thanks Dr. Matti Schneider and Dr. Roland Voigt for helpful comments on earlier versions of this article.  
  
 \section{Perturbations of gradient flow trajectories}
\label{SectionPerturbationsGradFlow}
 We start by repeating some necessary definitions and results from finite-dimensional Morse theory. A discussion of these notions and results can be found in many sources, among them \cite[Chapter 6]{JostRiemGeom}, \cite{SchwarzEqui} and \cite{BanyagaHurtubise}.
 
\begin{definition}
Let $x \in \Crit f$. For every Riemannian metric $g$ on $M$, the \emph{unstable manifold of $x$ with respect to $(f,g)$} is defined by 
\begin{equation*}
 W^u(x,f,g) := \Bigl\{ y \in M \ \Big| \ \lim_{t \to - \infty} \phi^{- \nabla^g f}_t (y) = x \  \Bigr\} \ ,
\end{equation*} 
where $\phi^{- \nabla^g f}$ denotes the flow of the vector field $- \nabla^g f$, i.e. the negative gradient flow of $f$ with respect to $g$. The \emph{stable manifold of $x$ with respect to $(f,g)$} is defined by
\begin{equation*}
 W^s(x,f,g) := \Bigl\{ y \in M \ \Big| \ \lim_{t \to + \infty} \phi^{- \nabla^g f}_t (y) = x \ \Bigr\} \ .
\end{equation*} 
\end{definition}

It is well-known from Morse theory, see \cite[Section 6.3]{JostRiemGeom} or \cite[Theorem 2.7]{WeberMorseWitten}, that for every $x \in \Crit f$ the unstable and stable manifolds of $x$ with respect to $(f,g)$ are smooth manifolds of dimension
\begin{equation*}
 \dim  W^u(x,f,g) = \mu(x) \ , \quad \dim W^s(x,f,g) = n - \mu(x) \ , 
\end{equation*}
respectively, and that both $W^u(x,f,g)$ and $W^s(x,f,g)$ are embedded submanifolds of $M$. Moreover, $M$ can be decomposed into the unstable and stable manifolds with respect to $(f,g)$: 
\begin{equation*}
 M =  \bigcup_{x \in \Crit f}W^u(x,f,g) = \bigcup_{x \in \Crit f} W^s(x,f,g) \ . 
\end{equation*}

\begin{definition} \index{Morse-Smale pair}
 We call a pair $(f,g)$ consisting of a smooth function $f \in C^\infty(M)$ and a smooth Riemannian metric $g$ on $M$ a \emph{Morse-Smale pair} if $f$ is a Morse function and if $W^u(x,f,g)$ and $W^s(y,f,g)$ intersect transversely for all $x,y \in \Crit f$. 
\end{definition}

Equipping the set of smooth Riemannian metrics on $M$ with a suitable topology, one shows that for any given Morse function $f \in C^\infty(M)$ the set of Riemannian metrics $g$ for which $(f,g)$ is a Morse-Smale pair is a generic subset of the space of Riemannian metrics, see \cite[Section 2.3]{Schwarz} or \cite[Section 6.1]{BanyagaHurtubise} for details. \bigskip

We next want to rephrase the definition of unstable and stable manifolds in terms of spaces of negative half-trajectories as it is done in \cite{SchwarzEqui}. This reformulation will be helpful for introducing the notion of perturbed gradient flow half-trajectories. Define
\begin{equation*}
 \bRR := \RR \cup \{- \infty,+\infty\} \ , \quad \bRR_{\geq 0} := [0,+\infty) \cup \{+\infty\} \ , \quad \bRR_{\leq 0} := (-\infty,0] \cup \{-\infty\} \ .
\end{equation*}
In \cite[Section 2.1]{Schwarz}, $\bRR$ is equipped with a smooth structure by choosing a diffeomorphism $h: \RR \to (-1,1)$ with
\begin{equation*}
 h(0)=0 \ , \quad \lim_{s \to - \infty} h(s) = -1 \ , \quad  \lim_{s \to + \infty} h(s) = 1 \ ,
\end{equation*}
by extending $h$ in the obvious way to a bijection $h: \bRR \to [-1,1]$ and requiring that this extended map is a diffeomorphism of manifolds with boundary. In the following we always assume such a map $h$ to be chosen and fixed. We equip $\bRR_{\leq 0}$ and $\bRR_{\geq 0}$ with the unique smooth structures which make them smooth submanifolds with boundary of $\bRR$.  \bigskip

For a Riemannian metric $g$ on $M$ and $x \in \Crit f$ define spaces of curves
\begin{align*}
&\PP_-(x) := \Bigl\{ \gamma \in H^1\left(\bRR_{\leq 0},M\right) \ \Big| \ \lim_{s \to - \infty} \gamma(s) = x  \Bigr\} \ , \\
&\PP_+(x) := \Bigl\{ \gamma \in H^1\left(\bRR_{\geq 0},M\right) \ \Big| \ \lim_{s \to + \infty} \gamma(s) = x  \Bigr\} \ ,
\end{align*}
where $H^1$ always denotes spaces of maps of Sobolev class $W^{1,2}$ with respect to the Riemannian metric $g$. These spaces can be equipped with structures of Hilbert manifolds, see \cite[Appendix A]{Schwarz}. The following spaces can be equipped with structures of Banach bundles over $\PP_-(x)$, $\PP_+(x)$, respectively, by
$$\LL_-(x) := \bigcup_{\gamma \in \PP_-(x)} L^2 (\gamma^*TM) \ , \qquad \LL_+(x) := \bigcup_{\gamma \in \PP_+(x)} L^2(\gamma^*TM) \ , $$
together with the obvious bundle projections, where $L^2(\gamma^*TM)$ denotes the space of sections of the vector bundle $\gamma^*TM$ which are $L^2$-integrable with respect to $g$. It can be shown by the methods of \cite[Appendix A]{Schwarz} that $\LL_-(x) \to \PP_-(x)$ and $\LL_+(x) \to \PP_+(x)$ are indeed smooth Banach bundles over Hilbert manifolds. 

The following notion is implicitly used in \cite{Schwarz} and \cite{SchwarzEqui}, though its name seems to have its first appearance in the literature in \cite{HWZ1} in a much more general context. We will use it in our Morse-theoretic framework in the upcoming theorems.

\begin{definition}
\label{DefFredholmSection} \index{Fredholm section}
 Let $M$ be a Banach manifold and $E \to M$ be a Banach bundle over $M$ with typical fiber $F$. A section of the bundle $s:M \to E$ is called a \emph{Fredholm section of index $r \in \ZZ$}, if there is a local trivialization 
 \begin{equation*}
  \left\{\left(U_\alpha, \varphi_\alpha: E|_{U_\alpha} \stackrel{\cong}{\to} U_\alpha \times F\right) \right\}_{\alpha \in I} \ ,
 \end{equation*}
 where $I$ is an index set, such that for every $\alpha \in I$ the map
 \begin{equation*}
  pr_2 \circ \varphi_\alpha \circ s|_{U_\alpha} : U_\alpha \to F \ ,
 \end{equation*}
where $pr_2:U_\alpha \times F \to F$ denotes the projection onto the second factor, is a Fredholm map of index $r$. 
 \end{definition}

The following theorem is well-known in Morse theory and can be proven by the methods of \cite[Section 2.2]{Schwarz}:

\begin{theorem}
\label{unperturbedFredholm} \label{unstable} Let $g$ be a Riemannian metric on $g$ and $\nabla^g f$ be the associated gradient vector field. The maps $\partial^g_- : \PP_-(x) \to \LL_-(x)$ and $\partial^g_+:\PP_+(x) \to \LL_+(x)$, both given by $\gamma\mapsto \dot{\gamma} + \nabla^g f \circ \gamma$, are smooth sections of $\LL_-(x)$ and $\LL_+(x)$, resp., and Fredholm sections of index
\begin{equation*}
\ind \partial^g_- = \mu(x) \ , \qquad \ind \partial^g_+ = n - \mu(x) \ . 
\end{equation*}
Furthermore, the spaces 
\begin{align*}
\WW^u(x,g) := \{ \gamma \in \PP_-(x) \ | \ \dot{\gamma} + \nabla^g f \circ \gamma = 0 \} = (\partial^g_-)^{-1}(\zero_{\LL_-(x)}) \ , \\
\WW^s(x,g) := \{ \gamma \in \PP_+(x) \ | \ \dot{\gamma} + \nabla^g f \circ \gamma = 0 \} = (\partial^g_+)^{-1}(\zero_{\LL_+(x)}) \ , 
\end{align*}
 are smooth manifolds with $\dim \WW^u(x,g) = \mu(x)$ and $\dim \WW^s(x,g)=n -\mu(x)$, where $\zero_{\LL_\pm(x)}$ denotes the image of the zero section of $\LL_\pm(x)$.
\end{theorem}

\begin{remark}
\label{EvalTrajUnstable} \index{Morse half-trajectories!relation to (un)stable manifolds}
\begin{enumerate}
\item It is well-known that the evaluation map $\WW^u(x,f,g) \to M$, $\gamma \mapsto \gamma(0)$, is an embedding onto $W^u(x,f,g)$ and the map $\WW^s(x,f,g) \to M$, $\gamma \mapsto \gamma(0)$, is an embedding onto $W^s(x,f,g)$.

\item Using standard regularity results for solutions of ordinary differential equations, one can show that $\WW^u(x,f,g)$ and $\WW^s(x,f,g)$ consist of smooth curves only, see \cite[Proposition 2.9]{Schwarz}. More precisely, we can identify as sets:
\begin{align*}
 \WW^u(x,f,g) &= \Bigl\{ \gamma \in C^\infty\left((-\infty,0],M \right) \ \Big| \ \lim_{s \to - \infty} \gamma(s) = x \ , \ \ \dot{\gamma} + \nabla^g f \circ \gamma = 0  \Bigr\} \ , \\
 \WW^s(x,f,g) &= \Bigl\{ \gamma \in C^\infty\left([0,+\infty),M \right) \ \Big| \ \lim_{s \to + \infty} \gamma(s) = x \ , \ \ \dot{\gamma} + \nabla^g f \circ \gamma = 0  \Bigr\} \ . 
\end{align*}
\end{enumerate}
\end{remark}
In the course of this article, we will consider moduli spaces which are similar to those of the form 
\begin{equation*}
\{ \gamma \in \WW^u(x,f,g) \ | \ \gamma(0) \in N \} \ ,
\end{equation*}
where $N$ is an arbitrary smooth submanifold of $M$, see Figure \ref{UnstableUnter}. More precisely, we would like to use transversality theorems to show that for a generic choice of Riemannian metric $g$, the former space is a smooth manifold. 

\begin{figure}[h]
 \centering
 \includegraphics[scale=0.5]{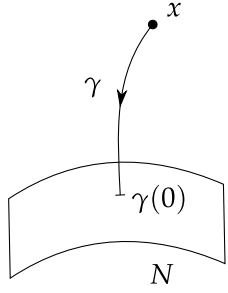}
 \caption{An element of $\{ \gamma \in \WW^u(x,f,g) \ | \ \gamma(0) \in N \}$.}
 \label{UnstableUnter}
\end{figure}

The problem is that we will not be able to apply transversality theory and in particular the Sard-Smale transversality theorem if $N$ contains critical points of $f$ without an additional condition on the Morse index of $x$. (For statements in this direction see \cite[Lemma 4.10]{SchwarzEqui} and \cite[Appendix A.2]{AbboSchwarzGT}.) We overcome this difficulty by using time-dependent vector fields to perturb the negative gradient flow equation following the approach of \cite{AbouzaidPlumbings}.  \bigskip

Define 
\begin{align*} 
\XX_-(M) := &\left\{ X \in C^{n+1} \left( (-\infty,0] \times M,TM\right) \ \left| \ X(s,\cdot) \in C^{n+1}(TM) \quad \forall s \in (-\infty,0] , \right. \right.\\
  &\left. \phantom{C^{n+1}} \qquad \qquad \qquad \qquad \qquad \qquad \qquad \qquad X(s,x) = 0 \quad \forall s \leq -1, \ x \in M \ \right\} \ ,
\end{align*} 
where $C^{n+1}(TM)$ denotes the space of sections of the bundle $TM$ of class $C^{n+1}$, i.e. the space of vector fields of class $C^{n+1}$. Then $\XX_-(M)$ is a closed linear subspace of the Banach space of $C^{n+1}$-sections of $\pi^* TM$, where $\pi: (-\infty,0] \times M \to M$
denotes the projection and where we equip $C^{n+1}(\pi^*TM)$ with the $C^{n+1}$-norm induced by the chosen metric $g$ on $M$ and the standard metric on $(-\infty,0]$.

\begin{definition}
 Let $x \in \Crit f$ and $X \in \XX_-(M)$. A curve $\gamma \in \PP_-(x)$ which satisfies the equation
 \begin{equation*}
  \dot{\gamma}(s) + \nabla^g f(\gamma(s)) + X(s,\gamma(s)) = 0 
 \end{equation*}
for every $s \in (-\infty,0]$ is called a \emph{perturbed negative gradient flow half-trajectory} (or simply a \emph{perturbed negative half-trajectory}) with respect to $(f,g,X)$.  
\end{definition}
\begin{remark}
 We consider vector fields of class $C^{n+1}$ as perturbations instead of smooth vector fields, since we want to momentarily apply the Sard-Smale transversality theorem. This theorem requires the space of parameters to be a Banach manifold modelled on a separable Banach space. 
 While this property obviously holds true for $\XX_-(M)$, the subspace of smooth vector fields in $\XX_-(M)$ fails to be complete, as it is well-known from functional analysis.
 
 There are two ways to overcome this problem in the literature, see \cite{FloerUnregularized} and \cite[Section 2.3]{Schwarz} for the one approach and \cite{FHSTransversality} for the other. Since for every important result of this article it suffices to consider time-dependent vector fields which are finitely many times differentiable, we will not further elaborate on these approaches. 
 
 Moreover, we choose the perturbations to be $(n+1)$ times differentiable, since this is the minimal amount of regularity that will be required in our situation for the application of the Sard-Smale transversality theorem in the proof of Theorem \ref{NegEvalSubmersion}.
\end{remark}

\textit{Throughout the rest of this article, we fix a Riemannian metric $g$ on $M$ and will mostly leave it out of the notation.} \bigskip

In the following, we write $\XX_- := \XX_-(M)$. 

\begin{theorem}
\label{PerturbedFredholm}
Let $x \in \Crit f$. The map 
\begin{equation*}
F^- : \XX_- \times \PP_-(x) \to \LL_-(x) \ , \qquad (Y,\gamma) \mapsto \left( s \mapsto \dot{\gamma}(s) + \nabla f(\gamma(s)) + Y(s,\gamma(s)) \right) \ ,	
\end{equation*}
is an $(n+1)$ times differentiable map of between Banach manifolds and for every $Y \in \XX_-$, the map $F^-_Y := F^-(Y,\cdot)$
is a $C^{n+1}$-section of $\LL_-(x)$. Furthermore, every $F^-_Y$ is a Fredholm section of index
\begin{equation*}
 \ind F^-_Y =\mu(x) \ .
\end{equation*}
\end{theorem}

\begin{remark}
 Note that if $0 \in \XX_-$ denotes the zero vector field, then 
 \begin{equation}
 \label{F-0}
  F^-_0 = \partial^g \ . 
 \end{equation}
Thus, Theorem \ref{PerturbedFredholm} generalizes Theorem \ref{unperturbedFredholm}.
\end{remark}

\begin{proof}[Proof of Theorem \ref{PerturbedFredholm}]
For a fixed $Y \in \XX_-$, one shows in in strict analogy with the proof of \cite[Theorem 12]{Schwarz} that the map $F^-(Y, \cdot): \PP_-(x) \to \LL_-(x)$ is $(n+1)$ times differentiable for every $Y \in \XX_-$. Moreover, $F^-$ is continuous and affine, hence smooth, in the $\XX_-$-component. So $F^-$ is $(n+1)$ times differentiable by the total differential theorem.

The Fredholm property of the $F^-_Y$ follows in precisely the same way as in the unperturbed case in part \ref{unstable} of Theorem \ref{unperturbedFredholm}, i.e. by similar methods as in the proof of \cite[Proposition 2.2]{Schwarz}. 

By part \ref{unstable} of Theorem \ref{unperturbedFredholm}, the operator $F^-_0 = \partial^g$ has index $\mu(x)$. Since $\XX_-$ is connected and the Fredholm index is locally constant, it follows that $\ind F^-_Y = \mu(x)$ for every $Y \in \XX_-$.
\end{proof}

\begin{theorem}
\label{PerturbedUnstable}
The map $F^-_Y: \PP_-(x) \to \LL_-(x)$ is transverse to $\zero_{\LL_-(x)}$ for all $Y \in \XX_-$ and $x \in \Crit f$.
\end{theorem}

\begin{proof}
 This can be shown by transferring the methods used in \cite[Section 6.3]{JostRiemGeom} and in \cite[Section 2]{AbboMajer} to the formalism of \cite{Schwarz}. Compared to \cite[Section 6.3]{JostRiemGeom}, we need to replace the flow of the vector field $-\nabla f$ by the time-dependent flow of the time-dependent vector field $- \nabla f - X$ and repeat the line of arguments used to prove \cite[Corollary 6.3.1]{JostRiemGeom} (see the preparations of Proposition \ref{FiniteLengthManifold} below for a more detailed discussion of flows of time-dependent vector fields).
\end{proof}

\begin{cor}
\label{PerturbedUnstMfld}
 For every $Y \in \XX_-$, the space
 \begin{equation*}
 W^-(x,Y) := \left\{\gamma \in \PP_-(x) \ | \  \dot{\gamma}(s) + \nabla f(\gamma(s)) + Y(s,\gamma(s)) = 0 \ \ \forall s \in (-\infty,0] \right\}
 \end{equation*}
 is a submanifold of $\PP_-(x)$ of class $C^{n+1}$ with $\dim W^-(x,Y) = \mu(x)$.
\end{cor} 
\begin{proof}
 This follows from Theorem \ref{PerturbedUnstable} and the observation that $W^-(x,Y) = (F_Y^-)^{-1}(\zero_{\LL_-(x)})$.
\end{proof}

\begin{remark} 
 This corollary generalizes the submanifold statement of part \ref{unstable} of Theorem \ref{unperturbedFredholm}, since it obviously holds that
 \begin{equation*}
  W^-(x,0) = \WW^u(x,f,g) \ .
 \end{equation*}
\end{remark}

So far, we have shown that the space of perturbed negative half-trajectories emanating from a given critical point can be described as a differentiable manifold. The introduction of time-dependent perturbations gives us a much higher flexibility considering evaluation maps and couplings with other trajectory spaces. We will see this in the next theorems and in Section \ref{NonlocalGeneralizations}.

\begin{remark}
The following differences occur between the space of unperturbed negative-half-trajectories $\WW^u(x,f,g)$ and the perturbed trajectory spaces $W^-(x,Y)$:
\begin{enumerate}
 \item While the constant half-trajectory $(t \mapsto x)$ is always an element of $\WW^u(x,f,g)$, it is not an element of $W^-(x,Y)$ for $Y \in \XX_-$ if there is some $s \in (-1,0]$ such that $Y(s,x) \neq 0$.
 \item While $\WW^u(x,f,g)$ will be a smooth submanifold of $\PP_-(x)$ if the Morse function $f$ is smooth, the differentiability of $W^-(x,Y)$ is always given by the differentiability of $Y$. For our later considerations, it will only be of importance for $W^-(x,Y)$ to be of class $C^1$. This is guaranteed for our choices of $Y$.
\end{enumerate}
\end{remark}

Define
\begin{equation*}
 \Wtilde^-(x,\XX_-) := \left(F^-\right)^{-1}\left( \zero_{\LL_-(x)}\right) \subset \XX_- \times \PP_-(x) \ .
\end{equation*}
Since $F^-_Y$ is transverse to $\zero_{\LL_-(x)}$ for every $Y \in U$, the total map $F^-$ is transverse to $ \zero_{\LL_-(x)}$, so $\Wtilde^-(x,\XX_-)$ is a Banach submanifold of $\XX_- \times \PP_-(x)$. 

The following statement is a straightforward analogue of \cite[Lemma 4.10]{SchwarzEqui} for perturbed half-trajectories and we will prove it by the same methods.

\begin{theorem}
\label{transverTotal} 
 The map $E^- : \Wtilde^-(x,\XX_-) \to M$, $(Y,\gamma) \mapsto \gamma(0)$, is a submersion of class $C^{n+1}$.
\end{theorem}

\begin{proof}
 We start by determining the tangent bundle of $\Wtilde^-(x,\XX_-)$. Let $\nabla$ be a Riemannian connection on $TM$. (To distinguish the notation from $\nabla f$, we will write $\nabla^g f$ for the gradient vector field with respect to $g$ in this proof.) 
 
 One uses the methods of \cite[Appendix A]{Schwarz} and the proof of \cite[Lemma 4.10]{SchwarzEqui} as well as the simple fact that $F^-(\cdot,\gamma)$ is an affine map for every $\gamma \in \PP_-(x)$ to show that the differential of $F^-$ is given by
 \begin{equation}
 \label{diffofF}
  \left(DF^-_{(Y,\gamma)}(Z,\xi)\right)(t) = (\nabla_t \xi)(t) + \left(\nabla_\xi \nabla^g f\right)( \gamma(t)) + \left(\nabla_\xi Y\right)(\gamma(t)) + Z(t,\gamma(t)) \ ,
 \end{equation}
 at every $(Y,\gamma) \in U$ and for all $Z \in T_Y \XX_- \cong \XX_-$, $\xi \in T_\gamma \PP_-(x)$, where $\nabla_t \xi$ denotes the covariant derivative of $\xi$ along the curve $\gamma$. By definition of $\Wtilde^-(x,\XX_-)$ as a submanifold of the product, we know that
\begin{align*}
 &T_{(Y,\gamma)} \Wtilde^-(x,\XX_-) = \ker (DF^-)_{(Y,\gamma)} \\
    &= \left\{(Z,\xi) \in T_{(Y,\gamma)}(\XX_- \times \PP_-(x)) \ \left| \ \nabla_t \xi + \left(\nabla_\xi \nabla^g f\right)\circ \gamma + \left(\nabla_\xi Y\right) \circ \gamma + Z \circ \gamma = 0 \right. \right\} \ .
\end{align*}

Moreover, we can extend the map $E^-$ in the obvious way to a map defined on the whole product $E^-: \XX_- \times \PP_-(x) \to M$ and for all $Z \in T_Y \XX_-$ and $\xi \in T_\gamma \PP_-(x)$, the differential of $E^-$ is given by
\begin{equation}
\label{diffofE}
 DE^-_{(Y,\gamma)}[(Z,\xi)] = \xi(0) \ .
\end{equation}
Let $v \in T_{\gamma(0)}M$ be given and choose $\xi_0 \in T_\gamma \PP_-(x)$ with the following properties:
\begin{align}
 &\xi_0 \in C^{n+1}(\gamma^* TM) \label{xibed1} \ , \\
 &\left(\nabla_t \xi_0 + \left(\nabla_{\xi_0} \nabla^g f\right)\circ \gamma + \left(\nabla_{\xi_0} Y\right) \circ \gamma \right)(s) = 0\quad \forall s \in (-\infty,-1] \label{xibed2} \ , \\
 &\xi_0(0) = v \label{xibed3} \ .
\end{align}

The existence of a $\xi_0 \in T_{\gamma} \PP_-(x)$ satisfying (\ref{xibed1}), (\ref{xibed2}) and (\ref{xibed3}) is easy to see: 

We define $\xi_0$ on $(-\infty,-1]$ as an arbitrary solution of the first order ordinary differential equation (\ref{xibed2}) which exists by the standard results for uniqueness and existence of solutions. One can then show iteratively (similar as in the proof of \cite[Proposition 2.9]{Schwarz}), using that $f$ is smooth and $Y$ is of class $C^{n+1}$, that $\xi_0$ can be chosen to be $(n+1)$ times differentiable on $(-\infty,-1]$. Moreover, one applies \cite[Lemma 2.10]{Schwarz} to solutions of (\ref{xibed2}) and derives that $\left\|\xi_0\right\|_{C^{n+1}} < + \infty$. This implies (\ref{xibed1}).

We continue this $\xi_0$ by an arbitrary vector field along $\gamma|_{[-1,0]}$ such that (\ref{xibed1}) and (\ref{xibed3}) are satisfied. By (\ref{diffofE}) and (\ref{xibed3}), we know that 
\begin{equation*}
 DE^-_{(Y,\gamma)}[(0,\xi_0)] = v \ .
\end{equation*}
We then pick $Z_0 \in \XX_-$ with the following property:
\begin{equation*}
 Z_0(t,\gamma(t)) = - (DF^-_{(Y,\gamma)}[(0,\xi_0)])(t) \quad \forall t \in (-\infty,0] .
\end{equation*}
This property does not contradict the definition of $\XX_-$ since (\ref{xibed2}) is equivalent to 
\begin{equation*}
(DF^-_{(Y,\gamma)}[(0,\xi_0)])(t) = 0 \ \ \text{for every} \ \  t \in (-\infty,-1] \ .
\end{equation*}
The existence of such a $Z_0$ is obvious. By the definitions of $\xi_0$ and $Z_0$ we obtain using (\ref{diffofF}) and (\ref{diffofE}):
$$ DE^-_{(Y,\gamma)}[(Z_0,\xi_0)] = v \ , \qquad DF^-_{(Y,\gamma)}[(Z_0,\xi_0)] = 0 \ . $$
The last line implies that $(Z_0,\xi_0) \in T_{(Y,\gamma)}\Wtilde^-(x,U)$. So for an arbitrary $v \in T_{\gamma(0)}M$ we have found such a $(Z_0,\xi_0)$ which maps to $v$ under $DE^-$ and therefore proven the claim.
\end{proof}

The following theorem is derived from Theorem \ref{transverTotal} by standard methods of proving transversality results in Morse and Floer homology (see e.g. \cite[Section 2.3]{Schwarz}, \cite{FHSTransversality} or \cite[Section 5]{HutchingsNotes}).

\begin{theorem}
\label{NegEvalSubmersion}
 Let $N \subset M$ be a closed submanifold. There is a generic set $\GG \subset \XX_-$ such that for every $Y \in \GG$, the map
\begin{equation*}
 E^-_Y: W^-(x,Y) \to M \ , \quad   \gamma \mapsto E^-(Y,\gamma) \ , 
\end{equation*} \index{endpoint evaluation map!for negative half-trajectories}
is transverse to $N$. Consequently, for $Y \in \GG$, the space
\begin{equation*}
 W^-(x,Y,N) := \{ \gamma \in W^-(x,Y) \ | \ \gamma(0) \in N \} = \left(E^-_Y\right)^{-1}(N)
\end{equation*}
is a manifold of class $C^{n+1}$ with 
\begin{equation*}
 \dim W^-(x,Y,N) = \mu(x) - \codim_M N \ .
\end{equation*}
\end{theorem}

\begin{proof}
 We want to apply the Sard-Smale transversality theorem in the form of \cite[Proposition 2.24]{Schwarz} or \cite[Theorem 5.4]{HutchingsNotes} using Theorem \ref{transverTotal}. For this purpose, we have to express $W^-(x,Y,N)$ in a slightly different way. Define
\begin{equation*}
  \zero_{\LL_-(x),N} :=  \{(\gamma,0) \in \LL_-(x) \ | \ \gamma(0) \in N \} \ .
\end{equation*}
Since the map $\PP_-(x) \to M$, $\gamma \mapsto \gamma(0)$, is a submersion, $\zero_{\LL_-(x),N}$ is the image of the restriction of the zero-section to a Banach submanifold of $\PP_-(x)$ whose codimension is given by $\codim N$.

Comparing the definitions of both sides of the following equation, one checks that for every $Y \in \XX_-$ we have
\begin{equation}
  E_Y^{-1}(N) = \left(F^-_Y\right)^{-1}(\zero_{\LL_-(x),N} ) \ .
\end{equation}
But it is easy to see that Theorems \ref{PerturbedUnstable} and \ref{transverTotal} together imply that the total map $F^-$ is transverse to $\zero_{\LL_-(x),N} $. The map $F^-$ is defined on the product $\XX_- \times \PP_-(x)$. Moreover, $F^-$ is a Fredholm map of index
\begin{equation*}
 0 \leq \mu(x) \leq n
\end{equation*}
by Theorem \ref{PerturbedFredholm}. The same theorem implies that $F^-$ is of class $C^{n+1}$ which shows that the differentiability requirements of the Sard-Smale transversality theorem are satisfied by $F^-$.
Viewing $\XX_-$ as a parameter space, one sees that Theorem \ref{transverTotal} implies all the remaining requirements for the Sard-Smale theorem. We derive:

There is a generic subset $\GG_x \subset \XX_-$ such that for $Y \in \GG_x$, the map $F^-_Y$ is transverse to $\zero_{\LL_-(x),N}$, so it follows that for $Y \in \GG_x$, the space $\left(E^-_Y\right)^{-1}(N) = W^-(x,Y,N)$ is a manifold of class $C^{n+1}$ with
\begin{align*}
 \dim W^-(x,Y,N) &= \ind F^-_Y - \codim_{\zero_{\LL_-(x)}} \zero_{\LL_-(x),N}  = \mu(x) - \codim_M N \ .
\end{align*}
We have so far constructed an individual generic set $\GG_x$ for each $x \in \Crit f$. Define
\begin{equation*}
 \GG := \bigcap_{x \in \Crit f} \GG_x \subset \XX_-(M) \ .
\end{equation*}
Since $\Crit f$ is a finite set, $\GG$ is a finite intersection of generic sets, hence itself generic. Moreover, for every $Y \in \GG$, the map $E^-_{Y}$ is transverse to $N$.
\end{proof}

We obtain analogous results for positive half-trajectories, i.e. for curves 
\begin{equation*}
[0,+\infty) \to M \ .
\end{equation*} 
We next give a brief account of these results and omit the proofs, since they can be done along the same lines as for perturbed negative half-trajectories. \\

The analogous space of perturbations for curves $[0,\infty) \to M$ is given by:
\begin{align*}
\XX_+(M) := &\left\{ Y \in C^{n+1} \left( [0,+\infty) \times M,TM\right) \ \left| \ Y(s,\cdot) \in C^{n+1}(TM) \quad \forall s \in [0,+\infty) , \right. \right. \\
  &\left. \phantom{C^{n+1} booooooooooooooooooooooooooooooooo} Y(s,x) = 0 \quad \forall s \geq 1, \ x \in M \ \right\} \ .
\end{align*}
We summarize the corresponding results in the following theorem. It can be derived from part 2 of Theorem \ref{unperturbedFredholm} in the same way as Corollary \ref{PerturbedUnstMfld} and Theorems \ref{transverTotal} and \ref{NegEvalSubmersion} are derived from part 1 of Theorem \ref{unperturbedFredholm}.
\begin{theorem}
\label{PerturbedStable}
\begin{enumerate}
 \item For all $Y \in \XX_+(M)$ and $x \in \Crit f$ the space
\begin{equation*}
W^+(x,Y) := \{\gamma \in \PP_+(x) \ | \ \dot{\gamma}(s) + \nabla f (\gamma(s)) + Y(s,\gamma(s)) = 0 \}
\end{equation*} 
is a submanifold of $\PP_+(x)$ of class $C^{n+1}$ with $\dim W^+(x,Y) = n - \mu(x)$.
\item For $x \in \Crit f$ define $\Wtilde^+(x,\XX_+) := \{(Y,\gamma) \in \XX_+(M) \times \PP_+(x) \ | \ \gamma \in W^+(x,Y) \}$. The space $\Wtilde^+(x,\XX_+)$ is a Banach submanifold of $\XX_+(M) \times \PP_+(x)$, and the map
\begin{equation*}
 E^+: \Wtilde^+(x,\XX_+) \to M \ , \quad   (Y,\gamma) \mapsto \gamma(0) \ , 
\end{equation*}
is a submersion of class $C^{n+1}$. \index{endpoint evaluation map!for positive half-trajectories}
\item Let $N \subset M$ be a closed submanifold. There is a generic subset $\GG \subset \XX_+$ such that for all $Y \in \GG$ and $x \in \Crit f$ the space
\begin{equation*}
 W^+(x,Y,N) := \{ \gamma \in W^+(x,Y) \ | \ \gamma(0)=x  \}
\end{equation*}
is a manifold of class $C^{n+1}$ of dimension $\dim W^+(x,Y,N) = \dim N - \mu(x)$.
\end{enumerate}
\end{theorem}

Next we consider a third type of perturbed trajectories of the negative gradient flow of $f$, namely curves $\gamma: [0,l] \to M$, where $l \in [0,+\infty)$ is allowed to vary, satisfying a perturbed negative gradient flow equation. Once again we start by considering the unperturbed case and derive our results for perturbed trajectories from this case.  \\

Consider the set
\begin{equation*}
 \MM(f,g) := \left\{ (l,\gamma) \ | \ l \in [0,+\infty) \ , \ \ \gamma: [0,l] \to M \ , \ \ \dot{\gamma}+\nabla^g f \circ \gamma = 0 \right\} \ .
\end{equation*} 
By the unique existence of solutions of ordinary differential equations for given initial values, the following map is easily identified as a bijection:
\begin{equation}
\label{DiffeoFiniteLength}
\varphi: \MM(f,g)  \to [0,+\infty) \times M \ , \quad  (l,\gamma) \mapsto (l, \gamma(0)) \ ,
\end{equation}
whose inverse is given by the map 
\begin{equation*}
\varphi^{-1}: [0,+\infty) \times M \to \MM(f,g) \ , \quad (l,x) \mapsto \left(l, \left(t \mapsto \phi^{-\nabla^g f}_t(x), \ t \in [0,l]\right) \right) \ , 
\end{equation*}
where $\phi^{-\nabla^g f}$ again denotes the negative gradient flow of $f$ with respect to $g$. 

We equip $[0,+\infty)$ with the canonical structure of a smooth manifold with boundary. Then the product manifold $[0,+\infty) \times M$ is a smooth manifold with boundary as well. \\

\textit{We can therefore equip $\MM(f,g)$ with a topology and a smooth structure such that $\varphi$ becomes a diffeomorphism of manifolds with boundary. From now on, let $\MM(f,g)$ always be equipped with this smooth structure.} 

\begin{definition}
 We call $\MM(f,g)$ the space of \emph{finite-length trajectories} of the negative gradient flow of $f$. 
\end{definition}

Spaces of finite-length trajectories as well as their smooth structures are also considered in \cite[Section 2]{WehrheimMWC} in greater detail. To study perturbed finite-length gradient flow trajectories, we momentarily introduce a convenient space of perturbations. The construction of the perturbation spaces requires more effort since we want the interval length of the trajectory to be a parameter of the perturbing vector fields as well. This becomes necessary in Section \ref{ConvergenceBehaviour} when we consider the limiting behaviour of sequences of perturbed finite-length trajectories if the length tends to $+ \infty$. \\

Consider a smooth function $\chi: [0,+\infty)\to [0,1]$ with the following properties:
\begin{itemize}
 \item $\chi(l) = \tfrac{l}{3}$ if $l \leq 3- \delta$ for some small $\delta>0$ \ , 
 \item $\chi(l) = 1$ if $l \geq 3$ \ ,
 \item $\dot{\chi}(l) > 0$ if $l < 3$.
\end{itemize}
See Figure \ref{GraphChi} for a picture of the graph of a possible choice of $\chi$. 
\begin{figure}[h]
 \centering
 \includegraphics[scale=0.8]{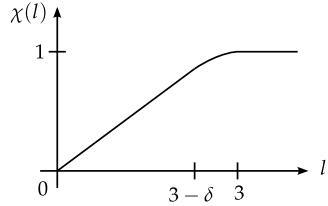}
 \caption{The graph of $\chi$}
 \label{GraphChi}
\end{figure}

\textit{We choose such a function $\chi$ once and for all and keep it fixed throughout the rest of this article.} \\

Moreover, for any map 
\begin{equation*} 
Z: [0,+\infty) \times [0,+\infty) \times M \to TM \ , \quad \left(l,t,x \right) \mapsto Z\left(l,t,x \right) \ , 
\end{equation*}
and any $l \geq 0$ define maps
\begin{align*}
 s_0(l,Z): [0,+\infty) \times M &\to TM \ , \qquad   \left(t,x\right) \mapsto \begin{cases}
                                      Z\left(l,t,x\right) & \text{if} \ \ t \in [0,1] \ , \\
                                      0 \in T_xM & \text{if} \ \ t > 1 \ ,
                                     \end{cases} \ , \\
 e_0(l,Z): [0,+\infty) \times M &\to TM \ , \qquad \left(t,x\right) \mapsto \begin{cases}
                                    Z\left(l,t+l,x\right) & \text{if} \ \ t \in [-1,0] \ , \\
                                    0 \in T_x M & \text{if} \ \ t < -1 \ .
                                   \end{cases}
\end{align*}
One checks from these definitions that if $Z$ is chosen such that $Z(l,t,\cdot)$ is a vector field on $M$ for all $l,t \in [0,+\infty)$, then both $s_0(l,Z)$ and $e_0(l,Z)$ are time-dependent sections of $TM$, although they might not be continuous. 

We further put for all $l \in [0,+\infty)$ and $Z: [0,+\infty) \times [0,+\infty) \times M \to TM$ with the above properties:
\begin{equation}
\label{split0}
\Split_{0}(l,Z) := \left(s_0(l,Z),e_0(l,Z) \right) \ .
\end{equation}
Define 
\begin{align}
&\XX_0(M) := \left\{ X \in C^{n+1}\left([0,+\infty)^2\times M,TM \right)  \ \left| \ X(l,t,\cdot) \in C^{n+1}(TM) \ \ \forall \ l,t \in [0,+\infty) \ , \right. \right. \notag \\ 
	      &\qquad \qquad \qquad X(l,t,x) = 0 \  \text{ if } \ t \in [\chi(l),l-\chi(l)]  \ , \label{XXcond1} \\
	      &\qquad \qquad \qquad (D^kX)_{(0,t,x)} = 0 \quad  \forall k \in \{0,1,\dots,n+1\}, \ (t,x) \in [0,+\infty)\times M \ , \label{XXcond3} \\
	      &\qquad \qquad \Bigl. \phantom{C^{n+1}} \lim_{l \to +\infty} \Split_0(l,Y) \ \text{ exists in } \ \XX_+(M) \times \XX_-(M) \ \Bigr\} \ .  \label{XXcond2}
\end{align}
 Concerning condition (\ref{XXcond1}), it suffices for the rest of this section to note that this condition especially implies that for each $l \in \RR_{> 0}$ there is a $\delta > 0$ such that
\begin{equation*}
X(l,t,x) = 0 \quad \text{for all } \ \ t \in \left(\tfrac{l}{2}-\delta, \tfrac{l}{2}+\delta\right)  , \ \ x \in M \ \text{and} \	 X \in \XX_0(M) \ .
\end{equation*}
Moreover, if $l \geq 3$, then
\begin{equation}
\label{conseqgeq3}
X(l,t,x) = 0 \quad \text{for every } \ \ t \in [1,l-1] \ .
\end{equation}
See Figure \ref{GraphsXX0} for an illustration of these observations. Condition (\ref{XXcond3}) will not be required until Section \ref{CompactificationsOneDimAinfty}, see Remark \ref{RemXXcond3}. 

From (\ref{conseqgeq3}) and the definition of the maps $s_0$ and $e_0$ we derive the following observation.

\begin{lemma}
 $\XX_0(M)$ is a Banach space with the $C^{n+1}$-norm induced by the given metrics. 
\end{lemma}
\begin{proof}
Apparently, $\XX_0(M)$ is a subset of $C^{n+1}(\pi^*TM)$, where $\pi: [0,+\infty)^2 \times M \to M$ denotes the projection onto the second factor. One checks without difficulties that the space 
$$\XXtilde_0(M) := \left\{X \in C^{n+1}(\pi^*TM) \ \middle| \ \text{$X$ satisfies \eqref{XXcond1} and \eqref{XXcond3}} \right\} $$
is a closed linear subspace of $C^{n+1}(\pi^*TM)$. One shows that for $l \geq 3$ and $X \in \XXtilde_0(M)$, the vector fields $s_0(l,X)$ and $e_0(l,x)$ from above are of class $C^{n+1}$, implying that the map 
$$\XX_0(M) \to \XX_+(M) \times \XX_-(M), \ \ X \mapsto \Split_0(l,X) \ , $$
is well-defined and continuously linear for every $l \geq 3$. It follows from this observation and the definition of $\XX_0(M)$ that the latter is a closed linear subspace of $\XXtilde_0(M)$, hence of $C^{n+1}(\pi^*TM)$, inheriting a Banach space structure. 
\end{proof}

\begin{figure}[h!]
 \begin{subfigure}{0.3\textwidth}
 \centering
 \includegraphics[scale=0.8]{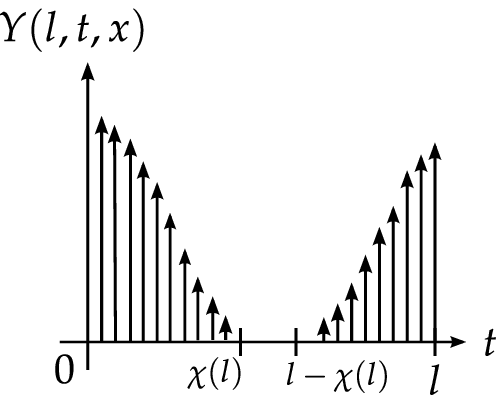}
 \caption{for small $l < 3$}
 \end{subfigure} \qquad  
 \begin{subfigure}{0.3\textwidth}
 \centering
 \includegraphics[scale=0.8]{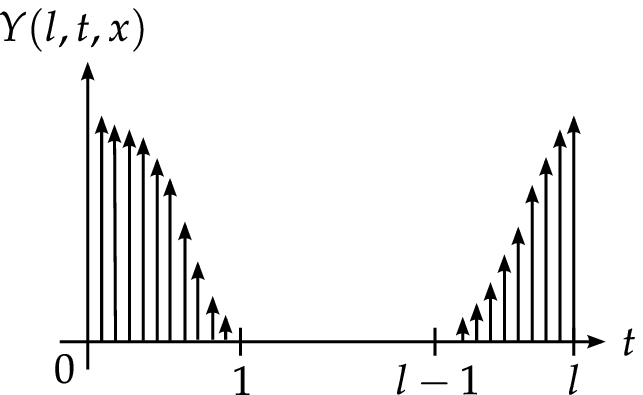}
 \caption{for large $l \geq 3$}
 \end{subfigure}  \phantom{3} \\
 \begin{subfigure}{\textwidth}
 \hspace*{3.3cm}
 \includegraphics[scale=0.8]{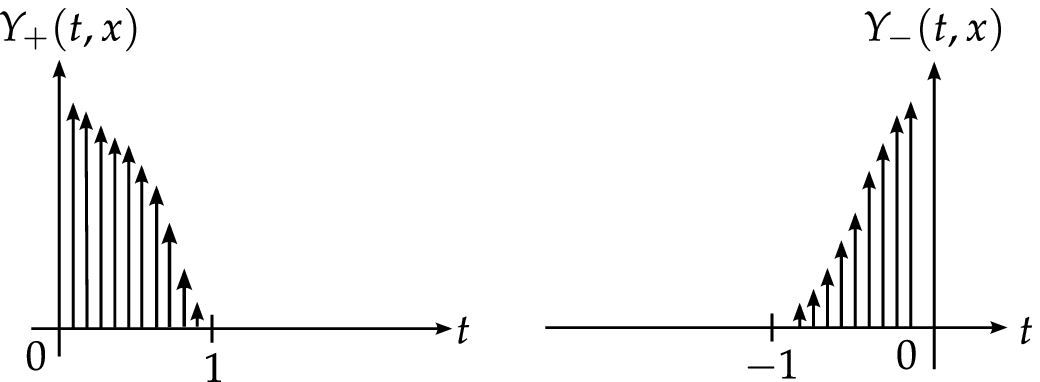}
 \caption{The graph of $\displaystyle\lim_{l \to + \infty} \Split_0(l,Y) = (Y_+,Y_-)$}
 \end{subfigure}
 \caption{The values of $t \mapsto Y(l,t,x)$ for $Y \in \XX_0(M)$ and fixed $l \geq 0$, $x \in M$ and an illustration of the map $\Split_0$.}
 \label{GraphsXX0}
\end{figure}

\begin{remark}
\label{reasonforXXcond1}
One derives from (\ref{conseqgeq3}) that for any $Y\in C^{n+1}\left([0,+\infty)\times [0,+\infty) \times M, TM \right)$ which obeys condition (\ref{XXcond1}) and for which the limit $Y_0 :=  \lim_{l \to + \infty} \Split_0(l,Y)$ exists pointwise, the pair of time-dependent vector fields $Y_0$ is of class $C^{n+1}$ and lies in $\XX_+(M) \times \XX_-(M)$. This observation is the reason for introducing condition (\ref{XXcond1}), since it will be needed for the compactifications of certain moduli spaces in the upcoming sections.
\end{remark}

For any $l$ we can view $X_l$ as a time-dependent vector field on $M$, where 
\begin{equation*}
 X_l(t,x) := X(l,t,x) \quad \forall t \in [0,+\infty) \ , \ \ x \in M \ .
\end{equation*}
In the following, we will consider flows of time-dependent vector fields, see \cite[Chapter 17]{LeeSmooth} for details. The \emph{time-dependent flow} of a time-dependent vector field $Y: \RR \times M \to TM$, is the map $\phi^Y: \RR \times \RR \times M \to M$ that is uniquely defined by the following two properties:
\begin{itemize}
\item $\phi^Y_{t,t}(x)= x$ for every $t \in \RR$ and $x \in M$, 
\item For every $t \in \RR$ and $x \in M$ the map $\RR \to M$, $s \mapsto \phi^Y_{s,t}(x)$ is the unique solution of 
\begin{equation*}
\begin{cases}
\gamma(t) &=x \ , \\
\dot{\gamma}(s)&=Y(s,\gamma(s)) \quad \forall \ s \in \RR \ , 
\end{cases}
\end{equation*}
\end{itemize}
where $\phi^Y_{s,t}(x) := \phi^Y(s,t,x)$ for every $s,t \in [0,+\infty)$ and $x \in M$. Consider the space \index{perturbed finite-length trajectories}
\begin{align}
\Mtilde := \{ (Y,l,\gamma) \ | \ &Y \in \XX_0(M), \ l \in [0,+\infty) , \ \gamma:[0,l] \to M , \notag \\
	      &\qquad \qquad \dot{\gamma}(s) + (\nabla f)(\gamma(s)) + Y_l(s,\gamma(s)) = 0 \} \ . \label{pertgradflowfinite}
\end{align}
The following map is then easily identified as a bijection: 
\begin{equation}
\label{Defofpsi}
\begin{aligned}
 \psi: \XX_0(M) \times \MM(f,g) &\to \Mtilde \ , \\
   (Y,l,\gamma) &\mapsto \left(Y, \left(l, \left(s \mapsto \phi^{Y,l}_{s,\frac{l}{2}}\left(\gamma\left(\tfrac{l}{2}\right) \right) \right), \ s \in [0,l] \right) \right) \ ,
\end{aligned}
\end{equation}
where $\phi^{Y,l}_{s,t}$ denotes the flow of the time-dependent vector field 
\begin{equation*}
(s,x) \mapsto - Y_l(s,x) - (\nabla f)(x) \ .
\end{equation*}
We want to describe $\psi$ more intuitively. First of all, for any $x \in M$, by definition of the flow the curve $\alpha_x: [0,l] \to M$, 
\begin{equation*}
\alpha_x(s) := \phi^{Y,l}_{s,\frac{l}{2}}(x) \ ,
\end{equation*}
denotes the unique solution of \eqref{pertgradflowfinite} which is defined on $[0,l]$ and which fulfills $\alpha_x\left(\frac{l}{2}\right)=x$.
Therefore, if $(Y,l,\gamma) \in \XX_0(M) \times \MM(f,g)$ and if we define $\tilde{\gamma}:[0,l] \to M$ by putting
\begin{equation*}
 \left(Y,l,\tilde{\gamma} \right) := \psi(Y,l,\gamma) \ , 
\end{equation*}
then $\tilde{\gamma}$ will be the unique solution of (\ref{pertgradflowfinite}) with $\tilde{\gamma}\left(\frac{l}{2} \right) = \gamma \left(\frac{l}{2} \right)$. Loosely speaking, $\psi$ maps a solution of the negative gradient flow equation to the solution of (\ref{pertgradflowfinite}) defined on the same interval and coinciding with $\gamma$ at time $\frac{l}{2}$. 

Moreover, the inverse of $\psi$ is given by the map $\Mtilde \to \XX_0(M) \times \MM(f,g)$, 
\begin{equation*}
  (Y, (l,\gamma)) \to \Bigl(Y, \Bigl(l, \Bigl(s \mapsto \phi^{-\nabla f}_{s-\frac{l}{2}}\left(\gamma\left(\tfrac{l}{2}\right) \right) \Bigr) , \ s \in [0,l] \Bigr) \Bigr) \ . 
\end{equation*}
Since $\psi$ is a bijection and its domain is the product of a Banach space and a smooth Banach manifold with boundary, we can equip $\Mtilde$ with the unique structure of a Banach manifold with boundary such that $\psi$ becomes a smooth diffeomorphism of Banach manifolds with boundary. Let $\Mtilde$ always be equipped with this smooth structure. 

\begin{prop} \index{perturbed finite-length trajectories}
\label{FiniteLengthManifold}
 For every $Y \in \XX_0(M)$, the space
 \begin{equation*}
  \MM(Y) := \{(l,\gamma) \ | \ l \in [0,+\infty), \ \gamma: [0,l] \to M, \ \dot{\gamma}(s) + (\nabla f)(\gamma(s)) + Y_l(s,\gamma(s)) = 0 \}
 \end{equation*} 
has the structure of a smooth manifold with boundary such that it is diffeomorphic to $\MM(f,g)$. In particular:
\begin{equation*}
 \dim \MM(Y) = n+1 \ .
\end{equation*}
\end{prop}

\begin{proof}
 Let $Y \in \XX_0(M)$. The space $\MM(Y)$ can be reformulated as
 \begin{equation*}
  \MM(Y) = \{(l,\gamma) \ | \ (Y,l,\gamma) \in \Mtilde  \} \ .
 \end{equation*}
This implies that  $\psi \left(\{Y\} \times \MM(f,g) \right) = \{Y\} \times \MM(Y) \ .$

Since $\{Y\} \times \MM(f,g)$ is a smooth submanifold with boundary of $\XX_0(M) \times \MM(f,g)$, the space $\{Y\} \times \MM(Y)$ is a smooth submanifold of $\Mtilde$ diffeomorphic to $\{Y\} \times \MM(f,g)$ by definition of the smooth structure on $\Mtilde$. Forgetting about the factor $\{Y\}$ shows the claim.
\end{proof}

For $l \geq 0$, let $\Mtilde_l := \left\{(Y,\lambda,\gamma) \in \Mtilde \ \middle| \ \lambda = l  \right\}$.

\begin{theorem}
\label{FiniteLengthSubmersion}
 \begin{enumerate}[a)]
  \item The map $E: \Mtilde \to M^2$, \ $(Y,(l,\gamma)) \mapsto (\gamma(0),\gamma(l))$, \ is of class $C^{n+1}$. \index{endpoint evaluation map!for finite-length trajectories}
\item Its restriction to $\Mtilde_l$ is a submersion for every $l > 0$. 
\end{enumerate}
\end{theorem}

\begin{proof}
\begin{enumerate}[a)]
 \item By definition of the smooth structure on $\Mtilde$, we need to show that the map
 \begin{equation*}
  E \circ \psi: \XX_0(M) \times \MM(f,g) \to M^2 
 \end{equation*}
is of class $C^{n+1}$ and that its restriction to the interior of its domain is a submersion. By definition of the smooth structure on $\MM(f,g)$, this is in turn equivalent to showing that the map
\begin{equation*}
 E_0: \XX_0(M) \times [0,+\infty) \times M \to M^2 \ , \quad E_0 := E \circ \psi \circ \left(\id_{\XX_0(M)} \times \varphi^{-1} \right) \ , 
\end{equation*}
is of class $C^{n+1}$ and that its restriction to $\XX_0(M) \times [0,+\infty)\times M$ is a submersion of class $C^{n+1}$. 

For this purpose, we want to write down the map $E_0$ more explicitly. Let $\phi^f$ denote the negative gradient flow of $(f,g)$. For $(Y,l,x) \in \XX_0(M) \times [0,+\infty) \times M$, we compute
\begin{align*}
 \left(\psi \circ \varphi^{-1} \right)(Y,l,x) &= \psi \left(Y,\left(l,\left(s \mapsto \phi^f_s(x), \ s \in [0,l]\right) \right)\right) \\
    &= \left(Y,\left(l, \left(s \mapsto \phi^{Y,l}_{s,\frac{l}{2}}\left(\phi^f_{\frac{l}{2}}(x) \right), \ s \in [0,l] \right) \right) \right) \ .
\end{align*}
Using this result, we further derive
\begin{align}
 E_0(Y,l,x) &= E\Bigl(Y,\Bigl(l, \Bigl(s \mapsto \phi^{Y,l}_{s,\frac{l}{2}}\Bigl(\phi^f_{\frac{l}{2}}(x) \Bigr), \ s \in [0,l] \Bigr) \Bigr) \Bigr) \notag \\
	    &= \Bigl( \Bigl(\phi^{Y,l}_{0,\frac{l}{2}}\circ \phi^f_{\frac{l}{2}}\Bigr)(x) ,\Bigl(\phi^{Y,l}_{l,\frac{l}{2}}\circ\phi^f_{\frac{l}{2}}\Bigr)(x) \Bigr) \ . \label{formulaforE0}
\end{align}
Consider the map
\begin{align*}
 \sigma: \XX_0(M) \times [0,+\infty) \times \RR \times M &\to T\XX_0(M) \times T [0,+\infty) \times T \RR \times TM \ , \\
   (Y,l,t,x) &\mapsto (0,0,0,Y(l,t,x)) \ .
\end{align*}
The evaluation map defined on a mapping space consisting of maps of class $C^{n+1}$ is itself of class $C^{n+1}$. Therefore, $\sigma$ is a vector field of class $C^{n+1}$ on $\XX_0(M) \times [0,+\infty)\times \RR \times M$ and has a flow of class $C^{n+1}$ (defined for non-negative time values):
\begin{align*}
 \phi^\sigma: [0,+\infty) \times \XX_0(M) \times [0,+\infty) \times \RR \times M &\to \XX_0(M) \times [0,+\infty) \times \RR \times M \ ,  \\
  (s,Y,l,t,x) &\mapsto \phi^\sigma_s(Y,l,t,x) \ .
\end{align*}
Let $\pi: \XX_0(M) \times [0,+\infty) \times \RR \times M \to M$ denote the projection onto the last factor and put
\begin{equation*}
 \tilde{\phi}^\sigma: [0,+\infty) \times \XX_0(M) \times [0,+\infty) \times \RR \times M \to  M \ ,  \quad \tilde{\phi}^\sigma := \pi \circ \phi^\sigma \ .
\end{equation*}
By definition of $\sigma$, the uniqueness of integral curves of vector fields implies that
\begin{equation*}
 \tilde{\phi}^\sigma(s,Y,l,t,x) = \phi^{Y,l}_s(t,x) \ . 
\end{equation*}
Therefore, the map $(s,Y,l,t,x) \mapsto \phi^{Y,l}_s(t,x)$ is of class $C^{n+1}$. Applying this result and the fact that $\phi^f$ is a smooth map to (\ref{formulaforE0}), one sees that $E_0$ is a composition of maps of class $C^{n+1}$ and therefore itself of class $C^{n+1}$. 

\item The claim is equivalent to $E_0|_{\XX_0(M) \times \{l\} \times M}$ being a submersion for every $l > 0$. 

We can w.l.o.g. assume that $l=2$ and put $\phi^Y_{s,t}:=\phi^{Y,2}_{s,t}$ for all $s$ and $t$. This simplifies the notation since by (\ref{formulaforE0}):
\begin{equation*}
 E_{0,2}(Y,x) := E_0(Y,2,x) = \left(\left(\phi^{Y}_{0,1}\circ\phi^f_1\right)(x), \left(\phi^{Y}_{2,1}\circ\phi^f_1\right)(x) \right) \ .
\end{equation*}
Moreover, for every $v \in T_xM$ the following holds:
\begin{equation}
 \label{DiffofE0on0v}
 \left(DE_{0,2}\right)_{(Y,x)}[(0,v)]= \left( \left(D \phi^{Y}_{0,1} \right)_{\phi^f_1(x)}\left[\left(D \phi^f_1 \right)_x[v]\right],\left(D \phi^{Y}_{2,1} \right)_{\phi^f_1(x)}\left[\left(D \phi^f_1 \right)_x[v]\right] \right) \ .
\end{equation}
Let $(v_1,v_2) \in T_{E_{0,2}(Y,x)}M^2$. We need to find an element of $T_{(Y,x)} (\XX_0(M) \times M)$ which maps to $(v_1,v_2)$ under $(DE_{0,2})_{(Y,x)}$. Since flow maps are diffeomorphisms, there exists a unique $w_0 \in T_{\phi^f_1(x)} M$ with $\left(D\phi^{Y}_{0,1}\right)_{\phi^f_1(x)}[w_0] = v_1$ and a unique $v_0 \in T_xM$ with $\left(D\phi^f_1\right)_x[v_0] = w_0$. For this choice of tangent vectors, equation \eqref{DiffofE0on0v} yields:
\begin{equation*}
 \left(DE_{0,2}\right)_{(Y,x)}[(0,v_0)] = \left(v_1, \left(D\phi^{Y}_{2,1}\right)_{\phi^f_1(x)}[w_0] \right) \ .
\end{equation*}
To show the claim, we therefore need to find $Z \in T_Y \XX_0(M)$ with
\begin{equation*}
 \left(DE_{0,2}\right)_{(Y,x)}[(Z,0)]= \left(0,v_2- \left(D\phi^{Y}_{2,1}\right)_{\phi^f_1(x)}[w_0] \right) \ ,
\end{equation*}
since this would imply $\left(DE_{0,2}\right)_{(Y,x)}(Z,v_0)=(v_1,v_2)$. We will construct such a $Z$ using isotopy theory. Put $y_0 := \left(\phi^{Y}_{2,1} \circ \phi^f_1 \right)(x)$ and pick a smooth curve $\alpha:[0,1] \to M$ with
\begin{equation}
\alpha(0)  = y_0 \ , \qquad  \dot{\alpha}(0) = v_2 - \left(D\phi^{Y}_{2,1}\right)_{\phi^f_1(x)}[w_0] \in T_{y_0} M \ . \label{alphadotbeginning}
\end{equation}
Let $y_1 :=\alpha(1)$. We can view $\alpha$ as a smooth isotopy from $y_0$ to $y_1$, where we see $y_0$ and $y_1$ as zero-dimensional submanifolds of $M$. By the isotopy extension theorem (see \cite[Section 8.1]{Hirsch}), we can extend $\alpha$ to a smooth diffeotopy 
\begin{equation*}
 F: [0,1] \times M \to M \ .
\end{equation*}
Moreover, we can choose $F$ to have its support in a small neighborhood of $\alpha([0,1])$ and especially such that
\begin{equation}
\label{vanishesbelow43}
 F\left(t,\phi^{Y}_{s,1}\left(\phi^f_1(x)\right)\right) = \phi^{Y}_{s,1}\left(\phi^f_1(x) \right) \quad \forall s \in \left[0,\tfrac{4}{3}\right] \ .
\end{equation}
For every $s \in [0,2]$ we then define a time-dependent tangent vector at $\phi^{Y_2}_{s,1}\left(\phi^f_1(x) \right)$ by
\begin{equation*}
 Z_2\left(s,\phi^{Y_2}_{s,1}\left(\phi^f_1(x) \right)\right) := \pdd{}{t} F\left(t,\phi^{Y_2}_{s,1}\left(\phi^f_1(x) \right)\right)\Big|_{t=0} \ ,
\end{equation*}
where $\pdd{}{t}$ denotes the derivative of $F$ in the $[0,1]$-direction. Using (\ref{alphadotbeginning}) and (\ref{vanishesbelow43}), this yields:
\begin{equation}
 \label{conditionsonZ2}
 \begin{aligned}
  &Z_2\left(s,\phi_{s,1}\left(\phi^f_1(x)\right)\right) = 0 \quad \forall s \in \left[0,\tfrac{4}{3} \right] \ , \\
  &Z_2\left(2,\phi_{2,1}\left(\phi^f_1(x)\right)\right) = v_2 - \left(D\phi^{Y}_{2,1}\right)_{\phi^f_1(x)}[w_0] \ .
 \end{aligned}
\end{equation}
We can extend $Z_2$ to a smooth, time-dependent vector field fulfilling
\begin{equation}
\label{extensionofZ2}
 Z_2(t,x) = 0 \quad \forall t \in \left[0,\tfrac{4}{3}\right] \ , \ \ x \in M \ .
\end{equation}
Such a vector field satisfying (\ref{extensionofZ2}) can in turn be extended to a parametrized vector field $Z \in \XX_0(M)$ satisfying
\begin{equation*}
 Z(2,t,x) = Z_2(t,x)\qquad \forall t \in \RR, \ x \in M \ .
\end{equation*}
(Note that by definition of $\XX_0(M)$, condition (\ref{extensionofZ2}) is required for this extendability of $Z_2$.) For such a particular choice of $Z$ we obtain
\begin{align*}
 \left(DE_{2,0}\right)_{(Y,x)}(Z,0) &= \left(\left(D \tilde{\phi}^\sigma_{0,1}\right)_{(Y,\phi_f^1(x))}[(Z,0)],\left(D \tilde{\phi}^\sigma_{2,1}\right)_{(Y,\phi_f^1(x))}[(Z,0)] \right) \\
    &= \left(Z_2\left(0,\phi^{Y}_{0,1}\left(\phi^f_1(x)\right)\right), Z_2\left(2,\phi^{Y}_{2,1}\left(\phi^f_1(x)\right)\right) \right) \\
    &\stackrel{(\ref{conditionsonZ2})}{=} \left(0,v_2 - \left(D\phi^{Y}_{2,1}\right)_{\phi^f_1(x)}[w_0]\right) \ .
\end{align*}
Hence we have shown that for any choice of $v_1$ and $v_2$ there iss $(Z,v_0) \in T_{(Y,x)}\left(\XX_0(M) \times M\right)$ with $\left(DE_{0,2}\right)_{(Y,x)}[(Z,v_0)]=(v_1,v_2)$, which shows the claim.
\end{enumerate}
\end{proof}

\begin{remark}
 Note that for every $(Y,(0,\gamma)) \in \Mtilde$, we obtain $E(Y,0,\gamma) = (\gamma(0),\gamma(0))$. So the restriction of $E$ to $\Mtilde_0$ is \emph{not} submersive if $n > 0$.
\end{remark}

In the same way as Theorem \ref{NegEvalSubmersion} is derived from Theorem \ref{transverTotal}, we can deduce the following statement from Theorem \ref{FiniteLengthSubmersion} and the Sard-Smale theorem. The attentive reader will have no problem providing a detailed proof.

\begin{theorem}
\label{FiniteLengthSardSmale}
 Let $N \subset M^2$ be a submanifold. There is a generic set $\GG \subset \XX_0(M)$, such that for every $Y \in \GG$ the space
 \begin{equation*}
  \MM(Y,N) := \left\{(l,\gamma) \in \MM(Y) \ | \ (\gamma(0),\gamma(l)) \in N \right\}
 \end{equation*}
is a submanifold with boundary of $\MM(Y)$ of class $C^{n+1}$ with
$$\dim \MM(Y,N) = n+1 - \codim_{M^2} N = \dim N + 1-n \ . $$
\end{theorem}
See Figure \ref{FiniteLengthUnter} for an illustration of a space of the form $\MM(Y,N)$. 
\begin{figure}[h]
 \centering
 \includegraphics[scale=0.5]{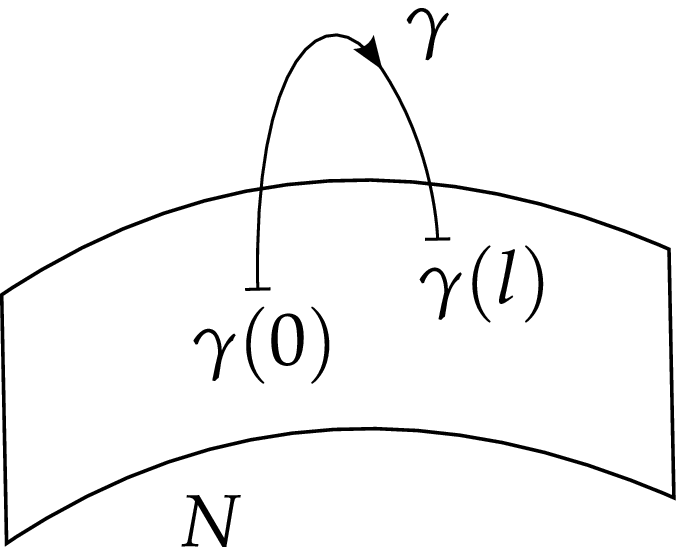}
 \caption{An element of $\MM(Y,N)$}
 \label{FiniteLengthUnter}
\end{figure}

\section{Nonlocal generalizations}
\label{NonlocalGeneralizations}

In this section we derive a more general, and in a certain sense nonlocal, version of the transversality results from the previous section. More precisely, we want to derive a transversality theorem, our Theorem \ref{NonlocalTransversality}, in which:
\begin{itemize}
 \item all three types of trajectories from the previous section are considered at once, 
 \item the perturbations may depend on the interval length parameters of all of the finite-length trajectories involved,
 \item the limiting behaviour of the perturbations for length parameters becoming infinitely large can be controlled a priori.
\end{itemize}

The first bullet point makes it possible to consider more sophisticated constructions of moduli spaces of trajectories starting and ending in submanifolds of the manifolds we are considering. The second and third bullet point give us the possibility to control the compactifications of these moduli spaces in a convenient way, as we will see for the special case of Morse ribbon trees in Sections \ref{ConvergenceBehaviour} and \ref{CompactificationsOneDimAinfty}.

We start by defining generalizations of the perturbation spaces $\XX_{\pm}(M)$ and $\XX_0(M)$ from the previous section. These will be needed to construct perturbations fulfilling the condition in the second bullet. \bigskip

For $k >0$, $j \in \{1,2,\dots,k\}$, $\lambda \geq 0$, $Y \in C^{n+1} \left([0,+\infty)^k \times (-\infty,0] \times M, TM\right)$ we define
\begin{equation}
\label{DefOfcj}
\begin{aligned}
 &c_j(\lambda,Y) \in C^{n+1} \left([0,+\infty)^{k-1} \times (-\infty,0] \times M, TM\right) \ , \\
 &c_j(\lambda,Y)\left(l_1,\dots,l_{k-1},t,x \right) := Y\left(l_1,\dots,l_{j-1},\lambda,l_j,\dots,l_{k-1},t,x\right) \ .
\end{aligned}
\end{equation}
In other words, $c_j(\lambda,Y)$ is the contraction of $Y$ obtained by inserting $\lambda$ into the $j$-th component. 

\begin{definition}
For $k \in \NN_0$ we recursively define a space $\XX_-(M,k)$ by putting
 \begin{align*}
  &\XX_-(M,0) := \XX_-(M) \ , \\
  &\XX_-(M,k) := \left\{Y \in C^{n+1}\left( [0,+\infty)^k \times (-\infty,0]\times M, TM \right) \left| \ Y\left( \vec{l},\cdot, \cdot\right) \in \XX_-(M) \ \ \forall \ \vec{l} \in [0,+\infty)^k, \right. \right. \\
  &\qquad \qquad \qquad  \qquad \qquad \qquad \Bigl. \lim_{\lambda \to +\infty} c_j(\lambda,Y)\text{ exists in } \XX_-(M,k-1) \ \ \forall j \in \{1,2,\dots,k\} \Bigr\} \ \ \text{for $k>0$.}
 \end{align*} 
In strict analogy with the spaces $\XX_-(M,k)$ we define
 \begin{equation*}	
  \XX_+(M,k) \subset C^{n+1}\left([0,+\infty)^{k} \times [0,+\infty) \times M, TM\right)
 \end{equation*} 
for every $k \in \NN_0$ with $\XX_+(M,0) = \XX_+(M)$.
\end{definition}

We next want to define analogous perturbation spaces $\XX_0(M,k)$ for finite-length trajectories for every $k \in \NN_0$. This requires some additional preparations, namely we need to introduce straightforward generalizations of the map $\Split_0$ from (\ref{split0}).  \bigskip

For $k \in \NN_0$, $l \geq 3$ and a map 
$$ Z: [0,+\infty)^k \times[0,+\infty) \times [0,+\infty) \times M \to TM \ , \quad \left(\vec{l},l,t,x \right) \mapsto Z\left(\vec{l},l,t,x \right) \ ,$$
we define maps
\begin{align*}
 s_k(l,Z): [0,+\infty)^k \times [0,+\infty) \times M &\to TM , \ \ 
   \left(\vec{l},t,x\right) \mapsto \begin{cases}
                                      Z\left(\vec{l},l,t,x\right) & \text{if} \ \  t \in [0,1] \ , \\
                                      0 \in T_xM & \text{if} \ \ t > 1 \ ,
                                     \end{cases} \\
 e_k(l,Z): [0,+\infty)^k \times (-\infty,0] \times M &\to TM , \  \ \left(\vec{l},t,x\right) \mapsto \begin{cases}
                                    Z\left(\vec{l},l,t+l,x\right) & \text{if} \ \ t \in [-1,0] \ , \\
                                    0 \in T_x M & \text{if} \ \ t < -1 \ .
                                   \end{cases}
\end{align*}
We then put
\begin{equation}
\label{EqDefSplitk}
\Split_{k}(l,Z) := \left(s_k(l,Z),e_k(l,Z) \right) \ .
\end{equation}

\begin{definition}
\label{pmPerturbationSpaces}
 For every $k \in \NN_0$ we define a space $\XX_0(M,k)$ recursively by putting
 \begin{align}
  &\XX_0(M,0) := \XX_0(M) \ , \notag \\
   &\XX_0(M,k) := \left\{Y \in C^{n+1}\left( [0,+\infty)^k \times [0,+\infty) \times [0,+\infty) \times M, TM \right) \right. \notag \\
  &\qquad \qquad \qquad \qquad \Big| \ Y\left( \vec{l},\cdot,\cdot,\cdot\right) \in \XX_0(M) \ \ \forall \  \vec{l} \in [0,+\infty)^k \ , \label{nonlocalfinitecond} \\
  &\qquad \qquad \qquad \qquad  \lim_{\lambda \to +\infty} c_j(\lambda,Y)\text{ exists in } \XX_0(M,k-1) \ \ \forall j \in \{1,2,\dots,k\} \ , \notag \\
  &\qquad \qquad \qquad \qquad \Bigl. \lim_{l \to + \infty} \Split_k(l,Y) \text{ exists in } \XX_+(M,k) \times \XX_-(M,k) \ \Bigr\}  \ \ \  \text{for $k>0$}, \notag
 \end{align} 
 where the maps $c_1,\dots,c_k$ are defined in strict analogy with (\ref{DefOfcj}).
\end{definition}

\begin{remark}
 Arguing as in Remark \ref{reasonforXXcond1}, one observes that for any 
 \begin{equation*}
  Y\in C^{n+1}\left( [0,+\infty)^k \times [0,+\infty) \times [0,+\infty)\times M, TM \right)
 \end{equation*}
which fulfills condition (\ref{nonlocalfinitecond}) and for which the limit $Y_0 :=  \lim_{l \to + \infty} \Split_k(l,Y)$ exists pointwise, $Y_0$ is a pair of parametrized vector fields of class $C^{n+1}$ and already lies in $\XX_+(M,k) \times \XX_-(M,k)$.
\end{remark}

One checks that all the spaces $\XX_\pm(M,k)$ and $\XX_0(M,k)$ are closed linear subspaces of Banach spaces, hence themselves Banach spaces. \bigskip

Before we provide the nonlocal transversality theorem we are longing for, we still need to define another analytic notion, namely the notion of background perturbations. These perturbations are introduced with regard to the limiting behaviour of perturbations. Technically, this method will result in the consideration of affine subspaces of the spaces $\XX_\pm(M,k)$ and $\XX_0(M,k)$. 

\begin{definition}
 For $k \in \NN$ the spaces $\XXb_-(M,k)$ and $\XXb_+(M,k)$ are defined by 
 \begin{align}
  \XXb_\pm(M,k) := &\left\{ \fatX = \left(X_1,\dots,X_k\right) \in  \XX_\pm (M,k-1)^k \ \right. \notag \\
   &\quad \Bigl| \Bigl. \lim_{\lambda \to \infty} c_i \left(\lambda,X_j\right) =  \lim_{\lambda \to \infty} c_{j-1}\left(\lambda,X_i \right) \ \ \forall \ i < j \in \{1,2,\dots,k\} \Bigr. \Bigr\} \ . \label{CommutingCondition}
 \end{align} 
 $\XXb_-(M,k)$ (resp. $\XXb_+(M,k)$) is called the space of \emph{negative (resp. positive) $k$-parametrized background perturbations on $M$}. For $\fatX=(X_1,X_2,\dots,X_k) \in \XXb_\pm(M,k)$ define 
\begin{align*}
 \XX_\pm(M,k,\fatX) := &\Bigl\{ Y \in \XX_\pm(M,k) \ \Big| \ \lim_{\lambda \to +\infty} c_j(\lambda,Y) = X_j \in \XX_\pm(M,k-1) \ \ \forall j \in \{1,2,\dots,k\} \ \Bigr\} \ .
\end{align*}
\end{definition}

\begin{lemma}
\label{SemiInfAffineSubspace}
 For all $k \in \NN$ and $\fatX \in \XXb_\pm(M,k)$, the space $\XX_\pm(M,k,\fatX)$ is a Banach submanifold of $C^{n+1}\left( [0,+\infty)^k \times \RR_{\pm} \times M, TM \right)$, where $\RR_+ := [0,+\infty)$, $\RR_- := (-\infty,0]$.
\end{lemma}

\begin{proof}
 Since every $\fatX \in \XXb_\pm(M,k)$ satisfies condition (\ref{CommutingCondition}) by definition, one can easily find $X \in \XX_{\pm}(M,k)$ satisfying $\lim_{\lambda \to + \infty} c_j(\lambda,X) = X_j$. Given such an $X$, we can reformulate $\XX_\pm(M,k,\fatX)$ as an affine subspace
\begin{equation}
\label{affineequation}
 \XX_\pm(M,k,\fatX) = X + \XX_\pm(M,k,0) \ .
\end{equation}
 where $0 \in \XXb_\pm(M,k)$ denotes the background perturbation consisting only of vector fields which are constant to zero. It is easy to check that the space $\XX_{\pm}(M,k,0)$ is a closed linear subspace of $\XX_{\pm}(M,k)$.
 
 Hence by (\ref{affineequation}), the space $\XX_\pm(M,k,\fatX)$ is a closed affine subspace and therefore a Banach submanifold of $C^{n+1}\left( [0,+\infty)^k \times \RR_{\pm} \times M, TM \right)$.
\end{proof}

For the definition of background perturbations for finite-length trajectories, we additionally need an analogous condition to (\ref{CommutingCondition}) involving the maps $\Split_k$. 

\begin{definition} \index{background perturbations!for finite-length trajectories}
\begin{enumerate}
 \item For every $k \in \NN_0$ \emph{the space of $k$-parametrized finite-length background perturbations} is defined as 
 \begin{align}
  \XXb_0(M,k) := \Big\{ (X_1,\dots,&X_k,X^+,X^-) \in \left(\XX_0(M,k-1)\right)^{k} \times \XX_+(M,k) \times \XX_-(M,k) \phantom{\lim_{\lambda \to \infty}}  \notag \\
    &\Big| \  \lim_{\lambda \to +\infty} c_i \left(\lambda,X_j\right) =  \lim_{\lambda \to +\infty} c_{j-1}\left(\lambda,X_i \right) \ \ \forall \ i < j \in \{1,2,\dots,k\} \ , \label{finitecommute} \\
    &\Bigl. \ \lim_{\lambda \to +\infty} \Split_{k-1}(\lambda,X_j) = \lim_{\lambda \to +\infty} \left(c_j(\lambda,X^+), c_j(\lambda,X^-)\right) \ \Bigr\} \ . \label{finitecommutesplit}
 \end{align} 
 \item For $\fatX = (X_1,\dots,X_k,X^+, X^-) \in \XXb_0(M,k)$, we define 
 \begin{align*}
&\XX_0(M,k,\fatX)\\  &= \Big\{Y \in \XX_0(M,k) \ \Big| \  \lim_{\lambda \to \infty} c_j(\lambda,Y) = X_j \ \ \forall \ j \in \{1,2,\dots,k\} \ , \ \ \lim_{l \to \infty} \Split_k(l,Y) = \left(X^+,X^-\right) \ \Big\} \ .
 \end{align*}
 \end{enumerate} 
\end{definition}

\begin{lemma}
 For all $k \in \NN_0$ and $\fatX \in \XXb_0(M,k)$, the space $\XX_0(M,k,\fatX)$ is a Banach submanifold of $\XX_0(M,k)$.
\end{lemma}

\begin{proof}
 We show the claim in analogy with the proof of Lemma \ref{SemiInfAffineSubspace}. Since $\fatX \in \XXb_0(M,k)$ is defined to satisfy conditions (\ref{finitecommute}) and (\ref{finitecommutesplit}), we can find a parametrized vector field $X \in \XX_0(M,k)$ satisfying
 \begin{equation*}
   \lim_{\lambda \to +\infty} c_j \left(\lambda,X\right) = X_j  \ \ \forall \ j \in \{1,2,\dots,k\} \quad \text{and} \quad  \lim_{\lambda \to + \infty} \Split_k(\lambda,X) = (X^+,X^-) \ .
 \end{equation*}
 Given such an $X$, we can reformulate $\XX_0(M,k,\fatX)$ as a closed affine subspace
 \begin{equation*}
  \XX_0(M,k,\fatX) = X + \XX_0(M,k,0)
 \end{equation*}
 where $0 \in \XXb_0(M,k)$ denotes the background perturbations consisting only of vector field which are constantly zero. The rest of the proof is carried out along the same lines as the one of Lemma \ref{SemiInfAffineSubspace}. 
\end{proof}

\begin{lemma}
\label{pertparamevalsub}
 For $k \geq 1$, $\fatX_- \in \XXb_-(M,k)$, $\fatX_+ \in \XXb_+(M,k)$ and $\fatX_0 \in \XXb_0(M,k)$, consider the maps:
 \begin{align*}
  p_k^-: &[0,+\infty)^k \times \XX_-(M,k,\fatX_-) \to \XX_-(M) \ , \\
  p_k^+: &[0,+\infty)^k \times \XX_+(M,k,\fatX_+) \to \XX_+(M) \ , \\
  p_k^0: &[0,+\infty)^k \times \XX_0(M,k,\fatX_0) \to \XX_0(M) \ ,
 \end{align*}
which are all three defined by $\left(\vec{l},Y\right) \mapsto Y\left(\vec{l},\cdot\right)$. All three maps are of class $C^{n+1}$. Moreover, for every $\vec{l} \in [0,+\infty)^k$, the maps 
\begin{align*}
p_k^\pm \left(\vec{l},\cdot\right): \XX_{\pm}(M,k,\fatX_{\pm}) &\to \XX_{\pm}(M) \ , \qquad
p^0_k\left(\vec{l},\cdot \right): \XX_0(M,k,\fatX_0) \to \XX_0(M) \ , 
\end{align*}
are surjective submersions.
\end{lemma}

\begin{proof}
 The $(n+1)$-fold differentiability of the maps follows again since evaluation maps considered on spaces of maps of class $C^{n+1}$ are themselves of class $C^{n+1}$. Moreover, one checks that for any $\vec{l} \in [0,+\infty)^k$ the maps $p_k^\pm\left(\vec{l},\cdot\right)$ and $p_k^0\left(\vec{l},\cdot\right)$ are surjective and continuous linear maps between Banach spaces, hence they are submersions.
\end{proof}

In the notation of Lemma \ref{pertparamevalsub}, we will occasionally write:
\begin{equation*}
 Y_\pm \left(\vec{l}\right) := p^{\pm}_k\left(\vec{l},Y_{\pm}\right) \ , \quad Y_0 \left(\vec{l}\right) := p^{0}_k\left(\vec{l},Y_{0}\right) \ ,
\end{equation*}
for $\vec{l} \in [0,+\infty)^k$, $Y_\pm \in \XX_\pm(M,k)$ and $Y_0 \in \XX_0(M,k)$. \\

The ultimate aim of this section is to derive a nonlocal transversality theorem which considers perturbed negative semi-infinite, positive semi-infinite and finite-length Morse trajectories all at once. To improve the clarity of the exposition, we introduce some additional notation:

\begin{definition}
\label{DefNonlocalBackgroundPert}
 For $k_1,k_2,k_3 \in \NN_0$ with $k_2>0$ define \index{perturbation space!of type $(k_1,k_2,k_3)$}
 \begin{equation*}
  \XX(k_1,k_2,k_3) := \left(\XX_-(M,k_2)\right)^{k_1} \times \left(\XX_0(M,k_2-1)\right)^{k_2} \times \left(\XX_+(M,k_2) \right)^{k_3} \ .
 \end{equation*} 
We call $\XX(k_1,k_2,k_3)$ \emph{the space of perturbations of type $(k_1,k_2,k_3)$}. We further define 
\begin{equation*}
\XXb(k_1,k_2,k_3) := \left(\XXb_-(M,k_2)\right)^{k_1} \times \left(\XXb_0(M,k_2-1)\right)^{k_2} \times \left(\XXb_+(M,k_2) \right)^{k_3}
\end{equation*} 
and call it \emph{the space of background perturbations of type $(k_1,k_2,k_3)$}.

For $\fatX \in \XX(k_1,k_2,k_3)$, where we write
\begin{equation*}
\fatX:= \left(\fatX^-_1,\fatX^-_2,\dots,\fatX^-_{k_1}, \fatX^0_{1},\fatX^0_{2},\dots,\fatX^0_{k_2}, \fatX^+_{1},\fatX^+_{2},\dots,\fatX^+_{k_3}\right) 
\end{equation*}
with $\fatX^-_i \in \XXb_-(M,k_2)$ for every $i \in \{1,2,\dots,k_1\}$, $\fatX^0_j \in \XXb_0(M,k_2-1)$ for every $j \in \{1,2,\dots,k_2\}$ and $\fatX^+_k \in \XXb_+(M,k_2)$ for every $k \in \{1,2,\dots,k_3\}$, we define
\begin{equation*}
 \XX(k_1,k_2,k_3,\fatX) := \prod_{i=1}^{k_1} \XX_-\left(M,k_2,\fatX^-_i\right) \times \prod_{j=1}^{k_2} \XX_0\left(M,k_2-1,\fatX^0_j\right) \times \prod_{k=1}^{k_3} \XX_+\left(M,k_2,\fatX^+_k \right)
\end{equation*}
\end{definition}

Having collected all necessary ingredients we are finally able to state and prove the aforementioned nonlocal transversality theorem. 

\begin{theorem} \index{nonlocal transversality theorem}
\label{NonlocalTransversality}
 Let $k_1,k_2,k_3 \in \NN_0$. Let $V$ be a smooth submanifold of $(0,+\infty)^{k_2}$ and $N$ be a smooth submanifold of $M^{k_1+2 k_2+k_3}$. If $k_2>0$, then let additionally $\fatX \in \XXb(k_1,k_2,k_3)$. There is a generic subset $\GG \subset \XX(k_1,k_2,k_3,\fatX)$ if $k_2>0$ and $\GG\subset\XX(k_1,0,k_3)$ if $k_2=0$, such that for every $\fatY = \left(Y^-_1,\dots,Y^-_{k_1},Y_1,\dots,Y_{k_2},Y^+_1,\dots,Y^+_{k_3}\right) \in \GG$ the space
\begin{align*} 
 &\MM_{\fatY}(x_1,x_2,\dots,x_{k_1},y_1,y_2,\dots,y_{k_3},(k_1,k_2,k_3),V,N)  \\
 &:= \left\{\left. \left(\gamma^-_1,\dots,\gamma^-_{k_1}, (l_1,\gamma_1),\dots,(l_{k_2},\gamma_{k_2}),\gamma^+_1,\dots,\gamma^+_{k_3} \right) \ \right| \right. (l_1,l_2,\dots,l_{k_2}) \in V \ , \\	
 &\qquad \qquad \ \ \gamma^-_i \in W^-\left(x_i, Y^-_i(l_1,l_2,\dots,l_{k_2})\right) \ \forall i \in \{1,2,\dots,k_1\} \ , \\
 &\qquad \qquad \ \ \gamma^+_j \in W^-\left(y_j, Y^+_j(l_1,l_2,\dots,l_{k_2})\right) \ \forall j \in \{1,2,\dots,k_3\} \ , \\
 &\qquad \qquad \ \ (l_j,\gamma_j) \in \MM\left(Y_j(l_1,\dots,l_{j-1},l_{j+1},\dots,l_{k_2}) \right) \ \forall j \in \{1,2,\dots,k_2\} \ , \\
 &\qquad \qquad \ \  \Big((E^-_{Y^-_i(l_1,\dots,l_{k_2})}(\gamma^-_i))_{i=1,2,\dots,k_1},(E_{Y_j(l_1,\dots,l_{j-1},l_{j+1},\dots,l_{k_2})}(l_j,\gamma_j))_{j=1,2,\dots,k_2},  \\ &\left.\phantom{\gamma_{k_3}^+ boooooooooooooooooooooooooooooooooooooooooo}(E^+_{Y^+_k(l_1,\dots,l_{k_2})}(\gamma^+_k))_{k=1,2,\dots,k_3}\Big) \in N \ \right\}
\end{align*}
is a manifold of class $C^{n+1}$ for all $x_1,x_2,\dots,x_{k_1},y_1,y_2\dots,y_{k_3} \in \Crit f$. For every $\fatY \in \GG$ its dimension is given by 
\begin{align*}
 &\dim \MM_{\fatY}(x_1,x_2,\dots,x_{k_1},y_1,y_2,\dots,y_{k_3},(k_1,k_2,k_3),V,N) \\
 &\qquad \qquad = \sum_{i=1}^{k_1} \mu(x_i) - \sum_{k=1}^{k_3} \mu(y_k) + (k_2+k_3)n + \dim V - \codim N \ .
\end{align*}
\end{theorem}

\begin{remark}
 This transversality theorem is nonlocal with respect to two different viewpoints. Firstly, observe that the perturbations of negative and positive semi-infinite and finite-length curves depend \emph{on all the lengths $l_1,\dots,l_{k_2}$ of the perturbed finite-length trajectories involved}. So the perturbations of the semi-infinite trajectories change with the interval lengths of the perturbed finite-length trajectories. Secondly, if $V= \left(0,+\infty\right)^{k_2}$ and if $N$ is of the form
 \begin{equation*}
  N = \prod_{i=1}^{k_1} N^-_i \times \prod_{j=1}^{k_2} N^0_j \times \prod_{k=1}^{k_3} N^+_k \ , 
 \end{equation*}
 where $N^-_1,\dots,N^-_{k_1},N^+_1,\dots,N^+_{k_3} \subset M$ and $N^0_1,\dots,N^0_{k_2} \subset M^2$ are smooth submanifolds, then the statement of Theorem \ref{NonlocalTransversality} can be derived in the spirit of the analogous results of the previous section (up to the introduction of parameters). 
 
 Since $V$ can be an arbitrary submanifold of $\left(0,+\infty\right)^{k_2}$ and $N$ can be an arbitrary smooth submanifold of $M^{k_1+2 k_2+k_3}$ and \emph{does not need to be a product of submanifolds of the factors}, Theorem \ref{NonlocalTransversality} can be regarded as a nonlocal generalization of the corresponding results in the previous section.
\end{remark}

\begin{proof}[Proof of Theorem \ref{NonlocalTransversality}]
 For $k_2=0$, the statement already follows from Theorem \ref{transverTotal}, part 2 of Theorem \ref{PerturbedStable} and the Sard-Smale transversality theorem. In the following we therefore assume that $k_2 > 0$. Consider the map
 \begin{align*}
  &p: V \times \XX(k_1,k_2,k_3,\fatX) \to \XX_-(M)^{k_1} \times \XX_0(M)^{k_2} \times \XX_+(M)^{k_3} \\
  &\left(\vec{l},\left(Y^-_i\right)_{i=1,\dots,k_1},\left(Y^0_j\right)_{j=1,\dots,k_2},\left(Y^+_k\right)_{k=1,\dots,k_3} \right) \mapsto \\ 
  &\left(\left(p^-_{k_2}\left(\vec{l},Y^-_i\right)\right)_{i},\left(p^0_{k_2-1}\left((l_1,\dots,l_{j-1},l_{j+1},\dots,l_{k_2}),Y^0_j\right)\right)_{j},\left(p^+_{k_2}\left(\vec{l},Y^+_k\right)\right)_{k} \right) \ .
 \end{align*}
 Since the maps $p^-_{k_2}\left(\vec{l},\cdot\right)$, $p^0_{k_2-1}\left((l_1,\dots,l_{j-1},l_{j+1},\dots,l_{k_2}),\cdot\right)$ and $p^+_{k_2}\left(\vec{l},\cdot\right)$ are surjective submersions of class $C^{n+1}$ for every $\vec{l}$ by Lemma \ref{pertparamevalsub}, the map $p$ is a surjective submersion of class $C^{n+1}$.  
 
 Consider the space
 \begin{equation*}
  \widetilde{\MM}(V) := \Bigl\{\left((Y_1,(l_1,\gamma_1)),\dots,(Y_{k_2},(l_{k_2},\gamma_{k_2})) \right) \in \left(\widetilde{\MM}\right)^{k_2} \ \Big| \ (l_1,\dots,l_{k_2}) \in V \Bigr\} \ .
 \end{equation*}
 One checks that $\widetilde{\MM}(V)$ is a Banach submanifold of $\widetilde{\MM}^{k_2}$. Moreover, it follows from part 2 of Theorem \ref{FiniteLengthSubmersion} that the restriction of $(E^0)^{k_2}: \widetilde{\MM}^{k_2} \to M^{2k_2}$ to $\widetilde{\MM}(V)$ is a surjective submersion of class $C^{n+1}$. 
 
 Let $x_1,x_2,\dots,x_{k_1},y_1,y_2,\dots,y_{k_3} \in \Crit f$. Combining the above considerations of $\widetilde{\MM}(V)$ with Theorem \ref{transverTotal} and part 2 of Theorem \ref{PerturbedStable}, we obtain that the map
 \begin{equation*}
  \left(E^-\right)^{k_1} \times \left(E^0\right)^{k_2} \times \left(E^+\right)^{k_3}: \prod_{i=1}^{k_1} \Wtilde^-(x_i,\XX_-) \times \Mtilde(V) \times \prod_{k=1}^{k_3} \Wtilde^+(y_k,\XX_+) \to M^{k_1+2k_2+k_3}
 \end{equation*}
 is a surjective submersion of class $C^{n+1}$, where $\Mtilde_{>0} := \left\{(l,\gamma) \in \Mtilde \ \middle| \ l > 0\right\}$. We further consider the map
 \begin{align*}
  &\bar{p}: \XX(k_1,k_2,k_3,\fatX) \times \prod_{i=1}^{k_1} \PP_-(x_i) \times \Mtilde(V) \times \prod_{k=1}^{k_3} \PP_+(y_k) \to \\ &\qquad \prod_{i=1}^{k_1} \left(\XX_-(M)\times \PP_-(x_i)\right) \times \prod_{j=1}^{k_2}\left(\XX_0(M)\times \Mtilde_{>0}\right) \times \prod_{k=1}^{k_3} \left( \XX_+(M) \times \PP_+(y_k) \right) \ , \\
  &\left((Y_i^-)_{i=1,\dots,k_1},(Y_j^0)_{j=1,\dots,k_2},(Y^+_k)_{k=1,\dots,k_3}, (\gamma^-_i)_{i=1,\dots,k_1},(Y_j,l_j,\gamma_j)_{j=1,\dots,k_2},(\gamma^+_k)_{k=1,\dots,k_3}\right) \mapsto \\
  &\left(\left(p^-_{k_2}\left(\vec{l},Y^-_i\right),\gamma^-_i\right)_{i=1,\dots,k_1} \left(p^-_{k_2}\left((l_1,\dots,l_{j-1},l_{j+1},\dots,l_{k_2}),Y^0_j\right),(Y_j,l_j,\gamma_j)\right)_{j=1,\dots,k_2}, \right. \\ 
  &\phantom{booooooooooooooooooooooooooooooooooooooooooooooooooo}\left. \left(p^+_{k_2}\left(\vec{l},Y^+_k\right),\gamma^+_k\right)_{k=1,\dots,k_3} \right) \ .
 \end{align*}
 $\bar{p}$ can be written as a composition of the surjective submersion $p$ and several permutations of the different components involved, so $\bar{p}$ is itself a surjective submersion. Let
 \begin{equation*}
 \overline{\WW} := \left(\bar{p}\right)^{-1}\left(\prod_{i=1}^{k_1} \Wtilde^-(x_i,\XX_-) \times \prod_{j=1}^{k_2} \Mtilde_{>0} \times \prod_{k=1}^{k_3} \Wtilde^+(y_k,\XX_+) \right) \ ,
 \end{equation*}
 where for convenience we identify $\Mtilde_{>0} \cong \left\{(Y,(Y,l,\gamma)) \left| (Y,l,\gamma) \in \Mtilde_{>0} \right. \right\}$. Then
 \begin{equation*}
  \Eunder: \overline{\WW} \to M^{k_1+2k_2+k_3} \ , \quad \Eunder:= \left(\left(E^-\right)^{k_1} \times \left(E^0\right)^{k_2} \times \left(E^+\right)^{k_3}\right) \circ \bar{p}|_{\overline{\WW}} \ ,
 \end{equation*}
 is well-defined, where $E^-$, $E^0$ and $E^+$ are defined as in the previous section. Moreover, $\Eunder$ is a composition of surjective submersions of class $C^{n+1}$ and therefore has the very same property. For any $\fatY \in \XX(k_1,k_2,k_3,\fatX)$ put 
 \begin{align*}
 \overline{\WW}_{\fatY} :=  &\left\{ \left(\left(\gamma^-_1,\dots,\gamma^-_{k_1}\right), \left((l_1,\gamma_1),\dots,(l_{k_2},\gamma_{k_2})\right),\left(\gamma^+_1,\dots,\gamma^+_{k_3}\right) \right) \ \right. \\
 &\quad \left. \Big| \ \left(\fatY, \left(\left(\gamma^-_1,\dots,\gamma^-_{k_1}\right), \left((l_1,\gamma_1),\dots,(l_{k_2},\gamma_{k_2})\right),\left(\gamma^+_1,\dots,\gamma^+_{k_3}\right) \right) \right) \in \overline{\WW} \right\} \ .
 \end{align*}
 Comparing the dimensions of the respective spaces, one checks that for every choice of $\fatY \in \XX(k_1,k_2,k_3,\fatX)$ it holds that
 \begin{equation}
 \label{decisiveidentity}
  \MM_{\fatY}(x_1,\dots,x_{k_1},y_1,\dots,y_{k_3},(k_1,k_2,k_3),N) = \Eunder_{\fatY}^{-1}(N) \ ,
 \end{equation}
 where $\Eunder_{\fatY} := \Eunder|_{\bar{\WW}_{\fatY}}$. We then proceed as in the proofs of Theorem \ref{NegEvalSubmersion} and \ref{FiniteLengthSardSmale} to write $\Eunder$ as a map defined on a product manifold in which $\XX(k_1,k_2,k_3,\fatX)$ is among the factors. Applying the Sard-Smale transversality theorem yields that for generic choice of $\fatY \in \XX(k_1,k_2,k_3,\fatX)$, the map $\Eunder_{\fatY}$ is transverse to $N$ which by (\ref{decisiveidentity}) shows the claim for this particular choice of $x_1,\dots,x_{k_1}$ and $y_1,\dots,y_{k_3}$. The simultaneous transversality for all choices of critical points for a generic $\fatY$ can be shown in precisely the same way as in the proof of Theorem \ref{NegEvalSubmersion}. We omit the details. \\
 
 It remains to compute the dimension of $\MM_{\fatY}(x_1,\dots,x_{k_1},y_1,\dots,y_{k_3},(k_1,k_2,k_3),N)$. One checks that for an arbitrary fixed $\vec{l} \in \left(0,+\infty \right)^{k_2}$ the following equation holds:
 \begin{align*}
  \dim \overline{\WW}_{\fatY} = \dim \left( \prod_{i=1}^{k_1} \Wtilde^- \left(x_i,Y^-_i \left(\vec{l}\right)\right) \times \right. &\prod_{j=1}^{k_2} \Mtilde_{>0} \left(Y^0_j \left(l_1,\dots,l_{j-1},l_{j+1},\dots,l_{k_2}\right)\right) \\ &\left. \times \prod_{k=1}^{k_3} \Wtilde^+\left(y_k,Y^+_k\left(\vec{l}\right) \right) \right) - \codim V \ .
 \end{align*}
 Combining this identity with the dimension formula from the Sard-Smale theorem and the results of the previous section, we can conclude:
 \begin{align*}
  &\dim \MM_{\fatY}(x_1,\dots,x_{k_1},y_1,\dots,y_{k_3},(k_1,k_2,k_3),N) =\dim \overline{\WW} - \codim N \\
 &=\sum_{i=1}^{k_1} \mu(x_i) + k_2 (n+1) + \sum_{k=1}^{k_3} (n-\mu(y_k)) - \codim V - \codim N \\
 &=\sum_{i=1}^{k_1} \mu(x_i) - \sum_{k=1}^{k_3} \mu(y_k) + (k_2+k_3)n + \dim V - \codim N \ .
 \end{align*}
\end{proof}

\begin{remark}
 Compared to Abouzaid's results, Theorem \ref{NonlocalTransversality} generalizes Lemma 7.3 from \cite{AbouzaidPlumbings}.
\end{remark}

In the rest of this article we will focus our attention on moduli spaces of perturbed Morse ribbon trees which will be shown to occur as a special case of Theorem \ref{NonlocalTransversality}. 

\section{Moduli spaces of perturbed Morse ribbon trees}
\label{SectionModuliSpacesPerturbed}

Throughout the rest of this article, we will discuss perturbed Morse ribbon trees, which can be interpreted as continuous maps from a tree to the manifold $M$ which edgewise fulfill a perturbed negative gradient flow equation. In this section, we will make this notion precise in terms of the constructions of Sections \ref{SectionPerturbationsGradFlow} and \ref{NonlocalGeneralizations}. Moreover, we will apply Theorem \ref{NonlocalTransversality} to equip moduli spaces of perturbed Morse ribbon trees with the structures of finite-dimensional manifolds of class $C^{n+1}$. \bigskip

We begin by giving a brief account of trees and their terminology, which we will use in the remainder of this article. We refrain from giving an elaborate discussion of trees and related notions from graph theory. Instead, we give a topological definition in terms of CW complexes. We will state and use several facts from graph theory without giving proofs, since this would lead us afar from the Morse-theoretic constructions we want to consider. 

\begin{definition}
 A \emph{tree complex} is a connected and simply connected one-dimensional CW complex. A 0-cell of a tree complex is called \emph{external} if it is connected to precisely one 1-cell. A 0-cell which is not external is called \emph{internal}. \bigskip
 
 We call a tree complex \emph{ordered} if it is equipped with an ordering of its external 0-cells, i.e. a bijection 
 \begin{equation*}
 \{\text{external 0-cells}\} \to \{0,1,\dots,d\}
 \end{equation*}
 for suitable $d \in \NN_0$.
\end{definition}

We want to define ordered trees as equivalence classes of ordered tree complexes. For this purpose, we define a notion of isomorphisms of ordered CW complexes.  

\begin{definition}
A map between two ordered tree complexes $T_1 \to T_2$ is called an \emph{isomorphism of ordered tree complexes} if it is a cellular homeomorphism which preserves the ordering of the external 0-cells, i.e. it maps the $i$-th external vertex of $T_1$ to the $i$-th external vertex of $T_2$ for every $i \in \{0,1,\dots,d\}$.
\end{definition}

One checks without further difficulties that this notion of isomorphism defines an equivalence relation on the class of ordered tree complexes.

Moreover, since isomorphisms of ordered tree complexes are by definition cellular maps which preserve the ordering of the external vertices, an elementary line of argument shows that these isomorphisms induce equivalence relations on the sets of 0-cells and 1-cells as well. 

\begin{definition} \index{tree}
An isomorphism class $T$ of ordered tree complexes is called \emph{an ordered tree}. Isomorphism classes of 0-cells of $T$ are called \emph{vertices of $T$} and isomorphism classes of 1-cells of $T$ are called \emph{edges of $T$}. Denote the set of all vertices of $T$ by $V(T)$ and the set of all edges of $T$ by $E(T)$.

A vertex is called \emph{external (internal)} if it is represented by an external (internal) 0-cell. Denote the set of all external (internal) vertices of $T$ by $V_{ext}(T)$ (resp. $V_{int}(T))$. 

An edge is called \emph{external} if it is represented by a 1-cell which is connected to an external 0-cell. An edge is called \emph{internal} if it is not external. Denote the set of all external (internal) edges by $E_{ext}(T)$ (resp. $E_{int}(T)$).
\end{definition}

Since by definition isomorphisms of ordered tree complexes preserve the orderings of the external 0-cells, these orderings in turn induce orderings on the external vertices of ordered trees.  

\begin{definition} 
Let $T$ be an ordered tree.
\begin{enumerate}
 \item Denote the set of external vertices of $T$ according to their ordering by
 \begin{equation*}
 V_{ext}(T) = \left\{v_0(T),v_1(T),\dots ,v_d(T) \right\}
 \end{equation*}
for suitable $d \in \NN_0$. Furthermore, for $d \geq 2$ denote for $i \in \{0,1,\dots,d\}$ the unique external edge connected to $v_i(T)$ by $e_i(T)$.  We will occasionally just write $v_0,\dots,v_d$ if it is clear which tree we are referring to. 
\item The vertex $v_0(T)$ is called the \emph{root of $T$}, $v_1(T),\dots, v_d(T)$ are called the \emph{leaves of $T$} and an ordered tree is called \emph{$d$-leafed} if it has precisely $d$ leaves. 
 \item A vertex of $T$ is called \emph{$n$-valent}, $n \in \NN$, (univalent, bivalent, trivalent, ...) if it is connected to precisely $n$ different edges of $T$.
 \item For $d\geq 2$, an ordered $d$-leafed tree is called \emph{ribbon tree} if every internal vertex is at least trivalent, i.e. $n$-valent for some $n \geq 3$. The set of all $d$-leafed ribbon trees for a fixed $d \geq 2$ is denoted by $\RTree_d$. 
 \item If $e$ denotes the unique ordered tree having one single edge and two (external) vertices, then we will put $\RTree_1 := \{e\}$.
 \item A ribbon tree is called \emph{binary tree} if every internal vertex is trivalent. 
 \end{enumerate}
\end{definition}

\begin{example}
 \begin{figure}[h]
\centering
\includegraphics[scale=1.3]{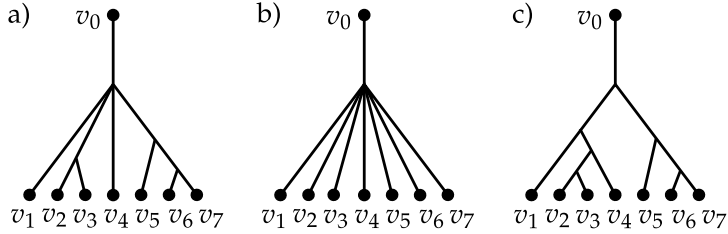}
\caption{Elements of $\RTree_7$}
 \label{RibbonTreeExamples}
\end{figure}

Figure \ref{RibbonTreeExamples} depicts three different examples of $7$-leafed ribbon trees. Note that tree b) has no internal edges and a single internal vertex and that tree c) is a binary tree. 
\end{example}


Let $T \in \RTree_d$ for some $d \in \NN$. For any $v \in V(T)$ let $P_v$ denote the minimal subtree of $T$ which connects the root with $v$, i.e. the unique subtree of $T$ with $V_{ext}(P_v) = \left\{v_0,v \right\}$.
The following statement can be shown by elementary methods of graph theory:

\begin{lemma}
\label{LemmaDirected} 
 Let $T \in \RTree_d$ and $e \in E(T)$. Let $v, w \in V(T)$ be the two distinct vertices which are connected to $e$. Then either $v \in V\left(P_w\right)$ or  $w \in V \left(P_v\right)$.
\end{lemma}

\begin{definition}
 In the situation of Lemma \ref{LemmaDirected}, assume that $v \in V \left(P_w \right)$. We then call $v$ the \emph{incoming vertex of $e$} and $w$ the \emph{outgoing vertex of $e$}. For any given edge $e$ we further denote its incoming and outgoing vertex by
 \begin{equation*}
  \vin(e) \quad \text{and} \quad \vout(e) \ . 
 \end{equation*}
\end{definition} 

The proof of Lemma \ref{LemmaDirected} is elementary. Instead of providing the proof, we illustrate the situation with Figure \ref{ExampleDirected} which depicts an example with $d=5$. In the picture, the tree $P_v$ is the subtree whose edges are the dashed edges in the picture. Here, $w \in V(P_v)$.

\begin{figure}[h]
\centering
\includegraphics[scale=1.3]{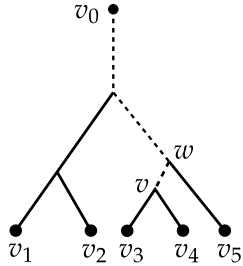}
 \caption{An illustration of Lemma \ref{LemmaDirected}.}
 \label{ExampleDirected}
\end{figure}

\begin{remark}
 In this terminology, the root of a ribbon tree is always the incoming vertex of the edge it is connected to while the leaves of a ribbon tree are always the outgoing vertices of the respective edges they are connected to, i.e.
 \begin{equation*}
  v_0(T) = \vin\left(e_0(T) \right) \ , \quad v_i(T) = \vout \left(e_i(T)\right) \ \ \forall i \in \{1,\dots,d\} \ .
 \end{equation*}
\end{remark}

\begin{lemma}
\label{OutgoingUniqueness}
 Let $T \in \RTree_d$ and $v \in V(T) \setminus \{v_0(T)\}$. There is a unique edge $e_v \in E(T)$ with $\vout(e_v) = v$.
\end{lemma}

\begin{proof}
 Since $v \neq v_0(T)$, the subtree $P_v$ has at least one edge. By definition of $P_v$ there is an edge $e \in E(P_v) \subset E(T)$ with $\vout(e) = v$. 

 Suppose there is an $f \in E(T)$ with $e \neq f$ and $\vout(f) = v$. By definition of an outgoing vertex, this would imply that there is a minimal subtree $P$ of $T$ with 
 \begin{equation*}
 V_{ext}(P) = \{v_0(T),v\}
 \end{equation*}
 and with $P \neq P_v$. Given $P_v$ and $P$, one could easily construct two continuous paths $\alpha_1,\alpha_2: [0,1] \to T$ with $\alpha_i(0) = v_0(T)$ and $\alpha_i(1) = v$ for $i \in \{1,2\}$ and such that 
 \begin{equation*}
  \im \alpha_1 \subset P_v \ , \quad \im \alpha_2 \subset P \ .
 \end{equation*}
Via reparametrization and concatenation, one could construct a loop in $T$ out of $\alpha_1$ and $\alpha_2$ having nontrivial homotopy class. Such a loop can not exist, since $T$ is by definition simply connected. Therefore, the claim follows.
\end{proof}

After this graph-theoretic digression, we will focus our attention on maps $T \to M$. 

For any $T \in \RTree_d$ and $e \in E(T)$ let 
\begin{equation*}
a_e: [0,1] \to T
\end{equation*}
denote its attaching map. (Actually, it is an isomorphism class of attaching maps, but the meaning of this term is obvious.) We always assume that 
\begin{equation*}
 a_e(0) = \vin(e) \ , \quad a_e(1) = \vout(e) \ .
\end{equation*}
Via precomposing with the attaching maps, we can express every continuous map $I: T \to M$ as a collection of maps 
\begin{equation*}
 \left\{I \circ a_e: [0,1] \to M \ | \ e \in E(T) \right\} \ .
\end{equation*}
Conversely, a family of maps
\begin{equation*}
 \left\{J_e: [0,1] \to M \ | \ e \in E(T) \right\}
\end{equation*}
induces a continuous map $J: T \to M$ if and only if it satisfies the compatibility condition
\begin{equation}
\label{ContinuityCondition}
  J_e(1) = J_f(0)  \quad \text{if} \quad \vout(e) = \vin(f) \qquad \forall e,f \in E(T) \ . 
\end{equation}
To apply Theorem \ref{NonlocalTransversality} in this context, we want to express condition (\ref{ContinuityCondition}) as a certain submanifold condition for a submanifold of $M^N$ for some $N \in \NN$, which will lead us to the definition of $T$-diagonals after some minor preparations. \\

For the rest of this article, we introduce the following notation for a given $d \in \NN$: 
\begin{equation*}
 k: \RTree_d  \to \NN_0 \ , \quad   T \mapsto | E_{int}(T)| \ . 
\end{equation*} 
Let $T \in \RTree_d$ and consider the product manifold $M^{1+ 2k(T)+d}$. For convenience, we label its factors by the edges of $T$, and write a point in $M^{1+2k(T)+d}$ as
\begin{equation*}
(q_0, (q^e_{in},q^e_{out})_{e \in E_{int}(T)},q_1,\dots, q_d ) \in M^{1+2k(T)+d}
\end{equation*}
We actually need an ordering on $E_{int}(T)$ to make this well-defined, but since the ordering is irrelevant for what follows, we silently \emph{assume that we have chosen an arbitrary ordering on $E_{int}(T)$ throughout the following discussion and leave it out of the notation.}  \\

Let $w_0 \in V(T)$ be given by $w_0 = \vout(e_0(T))$. (Recall that $\vin(e_0(T)) = v_0(T)$.) Define
\begin{align*}
\Delta_{w_0} := &\left\{ (q_0, (q^e_{in},q^e_{out})_{e \in E_{int}(T)},q_1,\dots, q_d ) \in M^{1+2k(T)+d} \right. \\ 
&\qquad \Big| \  q_0 = q_{in}^e \ \text{ for every } \ e \in E_{int}(T) \ \text{ with } \ \vin(e) = w_0 \ ,  \\
&\qquad \qquad  q_0 = q_j \ \text{ for every } \ j \in \{1,2,\dots,d\} \ \text{ with } \ \vin(e_j(T)) = w_0 \ \Big\} \ .
\end{align*}
Let $v \in V_{int}(T) \setminus \{w_0\}$. By Lemma \ref{OutgoingUniqueness} there is a \emph{unique} internal edge $e_v \in E_{int}(T)$ with $v_{out}(e_v) = v$. For any such $v$ and associated $e_v$ we define
\begin{align}
\Delta_v := &\left\{ (q_0, (q^e_{in},q^e_{out})_{e \in E_{int}(T)},q_1,\dots, q_d ) \in M^{1+2k(T)+d} \right. \label{DefDeltav} \\ 
&\qquad \Big| \  q_{out}^{e_v} = q_{in}^e \ \text{ for every } \ e \in E_{int}(T) \ \text{ with } \  \vin(e) = v \ , \notag \\
&\qquad \qquad q_{out}^{e_v} = q_j \ \text{ for every } \ j \in \{1,\dots,d\} \ \text{ with } \ \vin(e_j(T)) = v \ \Big\} \ . \notag
\end{align}

\begin{definition} 
\label{DefDeltaT}
 The \emph{$T$-diagonal} $\Delta_T \subset M^{1+2k(T)+d}$ is the space given by
\begin{equation*}
 \Delta_T := \bigcap_{v \in V_{int}(T)} \Delta_v \ .
\end{equation*}
\end{definition}

The definition of $\Delta_T$ will be motivated in Remark \ref{DeltaTexplicitly}.

\begin{prop}
\label{PropcodimDeltaT}
 For every $T \in \RTree_d$ with $d \geq 2$, the $T$-diagonal is a submanifold of $M^{1+2k(T)+d}$ with
\begin{equation}
\label{codimDeltaT}
 \codim \Delta_T = (k(T)+d) \cdot n \ .
\end{equation}
\end{prop}

\begin{proof}
For every $v \in V_{int}(T)$, the space $\Delta_v$ is easily seen to be a submanifold of $M^{1+2k(T)+d}$ with
\begin{equation*}
 \codim \Delta_v = \left|  \{e \in E(T) \ | \ \vin(e) = v \} \right|\cdot n \ .
\end{equation*}
Furthermore, for every $V \subset V_{int}(T)$ and $v_0 \in V_{int}(T)$ with $v_0 \notin V$, the spaces $\Delta_{v_0}$ and $\bigcap_{v \in V} \Delta_{v}$ are smooth submanifolds and transverse to each other. Then one can build up the space $\Delta_T$ as a finite sequence of transverse intersections of submanifolds and derive that it is itself a smooth submanifold, whose codimension is given by
\begin{align*}
 \codim \Delta_T &= \sum_{v \in V_{int}(T)} \codim \Delta_v = \sum_{v \in V_{int}(T)} \left| \{e \in E(T) \ | \ \vin(e) = v \} \right| \cdot n \\
   &= \left| \{e \in E(T) \ | \ \vin(e) \in V_{int}(T) \} \right| \cdot n \ .
\end{align*}
All that remains to do is to determine the cardinality of $\{e \in E(T) \ | \ \vin(e) \in V_{int}(T)\}$. By definition of $e_0(T)$ we know that $\vin(e_0(T)) = v_0 \notin V_{int}(T)$, such that $e_0(T)$ does not lie in this set. But $e_0(T)$ is the \emph{only} edge not lying in this set. The incoming vertex of an internal edge is by definition internal. Moreover, for all external edges $e$ with $e \neq e_0$ it holds that
\begin{equation*}
 \vout(e) \in \{v_1(T),\dots,v_d(T)\} \subset V_{ext}(T) \ ,
\end{equation*}
so the incoming vertex of $e$ has to be internal since $d \geq 2$. Consequently, we obtain
\begin{equation*}
 \codim \Delta_T = \left|(E(T) \setminus \{e_0(T)\}) \right| \cdot n = (k(T) + d)\cdot n \ .
\end{equation*}
\end{proof}


\begin{remark} \index{tree diagonal}
\label{DeltaTexplicitly}
 Combining the conditions in the definitions of the different $\Delta_v$, $v \in V_{int}(T)$, one derives that the $T$-diagonal is explicitly described by 
 \begin{align*}
  \Delta_T := &\left\{ (q_0, (q^e_{\mathrm{in}},q^e_{\mathrm{out}})_{e \in E_{int}(T)},q_1,\dots, q_d ) \in M^{1+2k(T)+d}\right. \\ 
&\quad \left| \  q_{\mathrm{out}}^{e} = q_{\mathrm{in}}^f \ \text{ for all } \ e,f \in E_{int}(T) \ \text{ with } \  \vout(e) = \vin(f) \ , \right. \\
&\quad \ \ q_{\mathrm{out}}^{e} = q_j \ \text{ for every } \ e \in E_{int}(T),  \ j \in \{1,2,\dots,d\} \ \text{ with } \ \vout(e) =\vin(e_j(T)) \ , \\
  &\quad \left. \ q_0 = q_{\mathrm{in}}^e \ \text{ for every } \ e \in E_{int}(T) \ \text{ with } \ \vout(e_0(T)) = \vin(e) \ \right\} \ .
 \end{align*}
 Let $\{J_e: [0,1] \to M \ | e \in E(T) \}$ be a family of maps. If one writes down condition (\ref{ContinuityCondition}) in greater detail for the different types of edges, one will obtain that the family induces a continuous map $T \to M$ if and only if
 \begin{equation*}
  \left(J_{e_0}(1), \left(J_e(0),J_e(1) \right)_{e \in E_{int}(T)},J_{e_1}(0),\dots,J_{e_d}(0) \right) \in \Delta_T \ ,
 \end{equation*}
 where $e_i := e_i(T)$ for every $i \in \{0,1,\dots,d\}$. In other words, we have expressed condition \eqref{ContinuityCondition} in terms of the $T$-diagonal.
\end{remark}

After all these preparations, we build the bridge to the Morse-theoretic constructions from Sections \ref{SectionPerturbationsGradFlow}  and \ref{NonlocalGeneralizations}. We want to consider maps from ribbon trees to $M$ which are not only continuous, but whose attaching maps fulfill (up to reparametrization) a perturbed negative gradient flow equation. Similar as in Remark \ref{DeltaTexplicitly}, we will piece these maps together out of a family of maps indexed by the set of edges of the ribbon tree. 

We start with describing suitable perturbation spaces. Let $d \geq 2$ and $T \in \RTree_d$ be fixed. 

\begin{definition} \index{perturbation space!for ribbon trees}
We introduce \emph{the space of perturbations for $T$} as 
\begin{equation*}
 \XX(T) := \XX(1,k(T),d) \ ,
\end{equation*}
In the case that $k(T)>0$, we further consider \emph{the space of background perturbations for $T$} \
 \begin{equation*}
  \XXb(T) := \XXb(1,k(T),d) \ ,
 \end{equation*} 
 using the notation of Section \ref{NonlocalGeneralizations}. For $\fatX \in \XXb(T)$, \emph{the space of perturbations for $T$ with background perturbation $\fatX$} is defined by
 \begin{equation*}
  \XX(T,\fatX) := \XX(1,k(T),d,\fatX) \ ,
 \end{equation*}
 once more using the notation of Section \ref{NonlocalGeneralizations}.
\end{definition}

We are again silently assuming that an ordering of $E_{int}(T)$ has been chosen and keep it out of the notation. We therefore write perturbations $\fatY \in \XX(T)$ as
$$\fatY=(Y_0,(Y_e)_{e \in E_{int}(T)},Y_1,\dots,Y_d) \ , $$
where $Y_0 \in \XX_-(M,k(T))$, $Y_e \in \XX_0(M,k(T)-1)$ for all $e \in E_{int}(T)$ and $Y_i \in \XX_+(M,k(T))$ for all $i \in \{1,2,\dots,d\}$. In strict analogy with the notation from Definition \ref{DefNonlocalBackgroundPert}, we further write $\fatX \in \XXb(T)$ as 
\begin{equation*}
 \fatX = \Bigl(\fatX^-, \left(\fatX^0_e\right)_{e \in E_{int}(T)}, \fatX^+_1,\dots,\fatX^+_d \Big) \ , 
\end{equation*}
where
\begin{itemize}
 \item $\fatX^- \in \XXb_-(M,k(T))$ is written as $\fatX^- = \left(X^-_e\right)_{e \in E_{int}(T)}$\ , 
\item $\fatX^0_e \in \XXb_0(M,k(T)-1)$ is denoted by
\begin{equation*}
\fatX^0_e = \Big(\left(X^0_{ef}\right)_{f \in E_{int}(T)\setminus\{e\}},X_{e+},X_{e-} \Big) \quad \text{for every} \  e \in E_{int}(T) \ ,
\end{equation*}
\item $\fatX^+_i \in \XXb_+(M,k(T))$ is given by 
\begin{equation*}
\fatX^+_i = \left(X^+_e \right)_{e \in E_{int}(T)} \quad \text{for every} \  i \in \{1,2,\dots,d\} \ .
\end{equation*}
\end{itemize}
Note that if $\fatY \in \XX(T,\fatX)$, then the above notation implies that
\begin{equation*}
Y_0 \in \XX_-(M,k(T),\fatX^-) \ , \quad  Y_e \in \XX_0(M,k(T),\fatX^0_e)  \ \ \text{and} \ \  Y_i \in \XX_+(M,k(T),\fatX^+_i)
\end{equation*}
for every $e \in E_{int}(T)$ and $i \in \{1,\dots,d\}$. 

\begin{definition} \index{perturbed Morse ribbon trees}
\label{DefPertMorseRT} 
 Let $\fatX \in \XXb(T)$ and $\fatY \in \XX(T,\fatX)$ and $x_0,x_1,\dots,x_d \in \Crit f$. \emph{The moduli space of $\fatY$-perturbed Morse ribbon trees modelled on $T$} starting in $x_0$ and ending in $x_1,\dots,x_d$ is defined as 
 \begin{align*}
  &\Acal^d_{\fatY}(x_0,x_1,\dots,x_d,T) \\ &:= \left\{ \left. \left(\gamma_0,\left(l_e,\gamma_e\right)_{e \in E_{int}(T)},\gamma_1,\dots,\gamma_d \right) \ \right| \ \gamma_0 \in \PP_-(x_0), \ \gamma_i \in \PP_+(x_i) \ \forall i \in \{1,2,\dots,d\} \ , \right. \\
  &\quad l_e>0, \ \gamma_e \in H^1([0,l_e],M) \ \forall e \in E_{int}(T) \ ,    \\
  &\quad \dot{\gamma}_0(s) + (\nabla f)(\gamma_0(s)) + Y_0\left(\left(l_e\right)_{e \in E_{int}(T)}, s,\gamma_0(s)\right) = 0 \quad \forall s \in (-\infty,0] \ , \\
  &\quad \dot{\gamma}_i(s) + (\nabla f)(\gamma_i(s)) + Y_i\left(\left(l_e\right)_{e \in E_{int}(T)}, s,\gamma_i(s)\right) = 0 \quad \forall s \in [0,+\infty) \ \forall i \in \{1,\dots,d\} \ , \\
  &\quad \dot{\gamma}_e(s) + (\nabla f)(\gamma_e(s)) + Y_e\left(\left(l_f\right)_{f \in E_{int}(T)\setminus\{e\}},l_e, s,\gamma_e(s)\right) = 0 \quad \forall s \in [0,l_e] \ \forall e \in E_{int}(T)\ , \\
  &\quad \gamma_e(l_e) = \gamma_f(0) \ \text{ for all } \ e,f \in E_{int}(T) \ \text{ with } \  \vout(e) = \vin(f) \ ,  \\
  &\quad \gamma_e(l_e) = \gamma_j(0) \ \text{ for all } \ e \in E_{int}(T),  \ j \in \{1,2,\dots,d\} \ \text{ with } \ \vout(e) =\vin(e_j(T)) \ , \\
  &\quad  \gamma_0(0) = \gamma_e(0) \ \text{ for all } \ e \in E_{int}(T) \ \text{ with } \ \vout(e_0(T)) = \vin(e) \ \Big\} \ .
 \end{align*}
\end{definition}

\begin{figure}[h]
 \centering
 \includegraphics[scale=0.8]{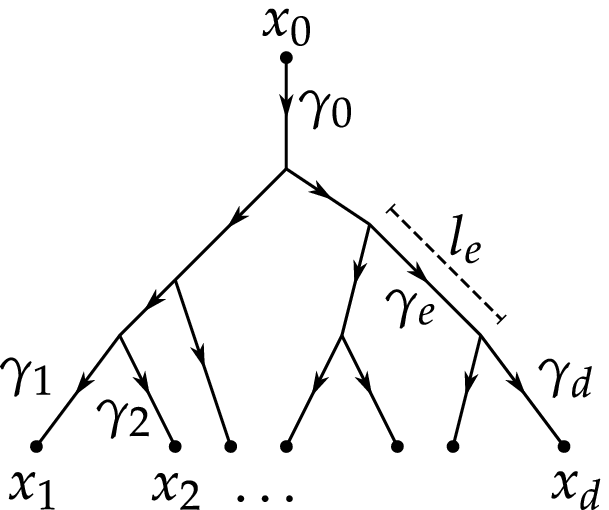}
 \caption{An element of $\Acal^d_{\fatY}(x_0,x_1,\dots,x_d,T)$.}
 \label{PMRT}
\end{figure}

See Figure \ref{PMRT} for a picture of a perturbed Morse ribbon tree. The conditions in the last three lines of the definition of $\Acal^d_{\fatY}(x_0,x_1,\dots,x_d,T)$ can be loosely rephrased by saying that all the trajectories associated to the edges of $T$ are coupled at the common associated vertices. They ensure that the family of trajectories induces a continuous map $T \to M$. 

The proof of the following theorem is a major application of the results of Section \ref{NonlocalGeneralizations} and especially of Theorem \ref{NonlocalTransversality}.

\begin{theorem} \index{perturbed Morse ribbon trees!dimension formula}
\label{MainTheoremMorseTrees}
 Let $\fatX \in \XXb(T)$. For generic choice of $\fatY \in \XX(T,\fatX)$, the moduli space $\Acal^d_{\fatY}(x_0,x_1,\dots,x_d,T)$ can be equipped with the structure of a manifold of class $C^{n+1}$ of dimension
\begin{equation*}
 \dim \Acal^d_{\fatY}(x_0,x_1,\dots,x_d,T) = \mu(x_0) - \sum_{j=1}^d \mu(x_j) + k(T)
\end{equation*}
for all $x_0,x_1,\dots,x_d \in \Crit f$.
\end{theorem}

\begin{proof}
 In strict analogy with Remark \ref{DeltaTexplicitly}, the endpoint conditions, i.e. the last three lines in the definition of $\Acal^d_{\fatY}(x_0,x_1,\dots,x_d,T)$, are equivalent to
 \begin{equation*}
    \left(\gamma_0(0), \left(\gamma_e(0),\gamma_e(l_e) \right)_{e \in E_{int}(T)},\gamma_1(0),\dots,\gamma_d(0) \right) \in \Delta_T \ .
 \end{equation*}
 Using the notation of Section \ref{NonlocalGeneralizations}, this condition can be reformulated as 
 \begin{equation*}
  \Eunder_{\fatY} \left(\gamma_0,\left(l_e,\gamma_e\right)_{e \in E_{int}(T)},\gamma_1,\dots,\gamma_d\right) \in \Delta_T \ ,
 \end{equation*}
 where $\Eunder_{\fatY}$ is defined as the corresponding product of evaluation maps as in the definition of the moduli space in Theorem \ref{NonlocalTransversality}.

 This reformulation implies that in the notation of Theorem \ref{NonlocalTransversality} the following equality holds true:
 \begin{equation*}
   \Acal^d_{\fatY}(x_0,x_1,\dots,x_d,T) =\MM_{\fatY}\left(x_0,(x_1,\dots,x_d),(1,k(T),d),\left(0,+\infty\right)^{k(T)},\Delta_T\right) \ .
 \end{equation*}
 Thus, the statement is nothing but a straightforward consequence of Theorem \ref{NonlocalTransversality}. It remains to show that the claimed dimension formula coincides with the one from Theorem \ref{NonlocalTransversality}. The latter one yields:
 \begin{align*}
  \dim \Acal^d_{\fatY}(x_0,x_1,\dots,x_d,T,\fatX) &= \mu(x_0)- \sum_{i=1}^d \mu(x_i) + (k(T)+d)n + k(T) - \codim \Delta_T \\
      &= \mu(x_0)- \sum_{i=1}^d \mu(x_i) + k(T) \ ,
 \end{align*}
where the second equality is a consequence of Proposition \ref{PropcodimDeltaT}. 
\end{proof}

\begin{definition} \index{regular!perturbations for Morse ribbon trees}
A perturbation $\fatY \in \XX(T)$ for which $\Acal^d_{\fatY}(x_0,x_1,\dots,x_d,T,\fatX)$ can be equipped with the structure of a manifold of class $C^{n+1}$ for all $x_0,x_1,\dots,x_d \in \Crit f$ is called \emph{regular}.
\end{definition}

\section{The convergence behaviour of sequences of perturbed Morse ribbon trees}
\label{ConvergenceBehaviour}

After having constructed moduli spaces of perturbed Morse ribbon trees in the previous section, we want to investigate sequential compactness properties of these moduli spaces. Our starting point is the consideration of certain sequential compactness results for spaces of perturbed Morse trajectories of the three different types we introduced in Section \ref{SectionPerturbationsGradFlow}. 
We will show that in all three cases, every sequence in the respective moduli space without a convergent subsequence has a subsequence that converges geometrically against a family of trajectories. The notion of geometric convergence will be made precise in the course of this section. 

\bigskip

Afterwards, we will draw conclustions from these results for spaces of perturbed Morse ribbon trees. The main result of this section is Theorem \ref{CompactnessMorseTrees}, following immediately from the results for perturbed Morse trajectories. It roughly states that every sequence of perturbed Morse ribbon trees without a subsequence which converges in its respective domain, has a geometrically convergent subsequence. This notion of geometric convergence will be defined later. 

Having established the main compactness result, we will focus our attention on certain special cases. Theorems \ref{LengthToZero}, \ref{ConvergenceHalfTraj} and \ref{Fconvergence} describe possibilites for the limiting behaviour of sequences of perturbed Morse ribbon trees in greater detail. 

Finally, we formulate a theorem for describing the situation of simultaneous occurence of the convergence phenomena described in the aforementioned three theorems. This theorem will enable us to investigate the spaces of geometric limits of sequences of perturbed Morse ribbon trees. Eventually, we will use the structure of the limit spaces to construct higher order multiplications on Morse cochain complexes in Section \ref{CompactificationsOneDimAinfty}. \bigskip

The sequential compactness theorems for moduli spaces of perturbed Morse trajectories can be shown in strict analogy with the compactness results in the unperturbed case. The compactness result for semi-infinite Morse trajectories is shown in \cite[Section 4]{SchwarzEqui}, while the compactness result for finite-length Morse trajectories is proven in \cite{WehrheimMWC}, along with a more general statement covering the semi-infinite case as well. 

The line of argument of the respective reference can be transferred to the case of perturbed Morse trajectories with only minor changes. Therefore, we will omit parts of the proofs of the decisive compactness theorems in this section and apply the respective trajectory compactness theorem instead. \bigskip

\textit{Throughout this section, we assume that every ribbon tree is equipped with an ordering of its internal edges. All results of this section are independent of the chosen orderings. We will come back to these choices in Section \ref{CompactificationsOneDimAinfty} and Appendix \ref{AppendixOrient}.} \bigskip

We begin this section by stating the crucial sequential compactness theorem for perturbed semi-infinite Morse trajectories. For every $x,y \in \Crit f$ let 
\begin{equation*}
 \widehat{\MM}(x,y) := \widehat{\MM}(x,y,g)
\end{equation*}
denote the space of unparametrized negative gradient flow lines of $f$ with respect to $g$, as defined in the introduction of this article.

\begin{theorem} \index{compactness of moduli spaces!of negative half-trajectories}
\label{CompactnessHalfTrajectories}
  \begin{enumerate}[(i)]
   \item Let $x \in \Crit f$, $\{\gamma_n\}_{n \in \NN} \subset \PP_-(x)$, $Y_\infty \in \XX_-(M)$ and $\{Y_n\}_{n\in \NN}$ be a sequence in $\XX_-(M)$ which converges against $Y_\infty$ and such that
   \begin{equation*}
   \gamma_n \in W^-\left(x,Y_n \right) \quad \text{for every} \ \  n \in \NN \ .
   \end{equation*}
    If $\{\gamma_n\}_{n \in \NN}$ does not have a subsequence which converges in $\PP_-(x)$, then there will be $m \in \NN$, $x_1,\dots,x_m \in \Crit f$ with $\mu(x) > \mu(x_1) > \dots > \mu(x_m)$ and curves 
 \begin{equation*}
 (\hat{g}_1,\dots,\hat{g}_m, \gamma_-) \in \widehat{\MM}(x,x_1) \times \widehat{\MM}(x_1,x_2) \times \dots \times \widehat{\MM}(x_{m-1},x_m) \times W^-\left(x_m,Y_\infty\right) \ ,
 \end{equation*}
 as well as a sequence $\{\tau^j_q \}_{q \in \NN} \subset \RR$ diverging to $-\infty$ for every $j \in \{1,2,\dots,m\}$ and a subsequence $\{\gamma_{n_q}\}_{q \in \NN}$ such that
\begin{itemize}
 \item $\{\gamma_{n_q}\}_{q \in \NN}$ converges against $\gamma_-$ in the $C^\infty_{loc}$-topology,
 \item $\Bigl\{\left. \gamma_{n_q} \left(\cdot + \tau^j_q\right)\right|_{(-\infty,T]}\Bigr\}_{q \geq q_{T,j}}$ converges against $g_j|_{(-\infty,T]}$ in the $C^\infty_{loc}$-topology for all $T>0$ and $q_{T,j} \in \NN$ for which the restrictions are well-defined, where $g_j$ is a representative of $\hat{g}_j$ for every $j \in \{1,2,\dots,m\}$. 	
\end{itemize}
\item Let $y \in \Crit f$, $\{\gamma_n\}_{n \in \NN} \subset \PP^+(y)$, $\{Y_n\}_{n\in\NN} \subset \XX_+(M)$ and $Y_\infty \in \XX_+(M)$ such that  $\{Y_n\}_{n\in\NN}$ converges against $Y_\infty$ and
\begin{equation*}
   \gamma_n \in W^+\left(x,Y_n\right) \quad \text{for every} \ \  n \in \NN \ .
   \end{equation*}
If $\{\gamma_n\}_{n\in\NN}$ does not have a subsequence which converges in $\PP_+(y)$, then there will be $m \in \NN$, $y_1,\dots,y_m \in \Crit f$ with $\mu(y_1) > \dots > \mu(y_m) > \mu(y)$ and curves 
 \begin{equation*}
 (\gamma_+, \hat{g}_1,\dots,\hat{g}_m) \in W^+\left(y_1, Y_\infty\right) \times \widehat{\MM}(y_1,y_2) \times \dots \times \widehat{\MM}(y_{m-1},y_m) \times \widehat{\MM}(y_m,y) \ ,
 \end{equation*}
as well a sequence $\{\tau^j_q\}_{q \in \NN} \subset \RR$ diverging to $+\infty$ for every $j \in \{1,\dots,m\}$ and a subsequence $\{\gamma_{n_q}\}_{q \in \NN}$ such that
\begin{itemize}
 \item $\{\gamma_{n_q}\}_{q \in \NN}$ converges against $\gamma_+$ in the $C^\infty_{loc}$-topology,
 \item $\Bigl\{\left. \gamma_{n_q}\left(\cdot + \tau^j_q\right)\right|_{[-T,+\infty)}\Bigr\}_{q \geq q_{T,j}}$ converges against $g_j|_{[-T,+\infty)}$ in the $C^\infty_{loc}$-topology for every $T > 0$ and every $q_{n,T} \in \NN$ such that the restrictions are well-defined, for a representative $g_j$ of $\hat{g}_j$ for every $j \in \{1,2,\dots,m\}$. 
\end{itemize}
\index{compactness of moduli spaces!of positive half-trajectories}
  \end{enumerate}
\end{theorem}

Theorem \ref{CompactnessHalfTrajectories} is shown by applying the methods used to prove the corresponding result for unperturbed half-trajectories. For compactness theorems for spaces of unperturbed Morse trajectories, see \cite[Theorem 2.6]{WehrheimMWC} or \cite[Section 2.4]{Schwarz} for analogous results for trajectories defined on $\RR$.
\begin{figure}[h]
 \centering
 \includegraphics[scale=1.2]{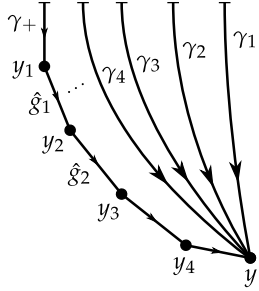}
 \caption{A geometrically convergent sequence of perturbed semi-infinite half-\-tra\-jec\-to\-ries ending in $y \in \Crit f$.}
 \label{PosGeomConv}
\end{figure}

Figure \ref{PosGeomConv} illustrates part (ii) of Theorem \ref{CompactnessHalfTrajectories}. For an illustration of part (i), turn Figure \ref{PosGeomConv} upside down and invert the directions of all arrows. 

\begin{remark}
   Concerning Theorem \ref{CompactnessHalfTrajectories}, it is of great importance that the spaces $\XX_\pm(M)$ are defined such that for every $Y \in \XX_{\pm}(M)$, the value $Y\left(t,x\right)$ vanishes whenever $t$ lies below the fixed value $-1$ and above the fixed value $1$ in the positive case. 
  
  If the perturbations were allowed to be non-vanishing for arbitrary time parameters, the sequences of reparametrization times $\{\tau^j_n\}_{n \in \NN}$ would in general not exist in the above setting. 
\end{remark}

\begin{definition} 
\begin{enumerate}[(i)]
\item In the situation of Theorem \ref{CompactnessHalfTrajectories} (i), we say that $\{\gamma_n\}_{n \in \NN}$ \emph{converges geometrically against $(\hat{g}_1,\dots,\hat{g}_m,\gamma_-)$} and we call $(\hat{g}_1,\dots,\hat{g}_m,\gamma_-)$ the \emph{geometric limit of $\left\{\gamma_n\right\}_{n \in \NN}$}. 
\item In the situation of Theorem \ref{CompactnessHalfTrajectories} (ii), we say that $\{\gamma_n\}_{n \in \NN}$ \emph{converges geometrically against $(\gamma_+, \hat{g}_1,\dots,\hat{g}_m)$} and we call $(\gamma_+, \hat{g}_1,\dots,\hat{g}_m)$ the \emph{geometric limit of $\left\{\gamma_n\right\}_{n \in \NN}$}. 
\end{enumerate}
\end{definition}

\begin{remark} 
 The following special case of Theorem \ref{CompactnessHalfTrajectories} will be relevant in the discussion of the convergence of perturbed Morse ribbon trees.  
 Let $Y \in \XX_-(M,k,\fatX)$ for a given background perturbation $\fatX=(X_1,\dots,X_k) \in \XXb_-(M,k)$. Let $\{(l_{1n},\dots,l_{kn})\}_{n \in \NN}$ be a sequence in $(0,+\infty)^k$, such that $\{l_{in}\}_{n\in \NN}$ diverges against $+\infty$ for a unique $i \in \{1,2,\dots,k\}$ and that $\{l_{jn}\}_{n\in \NN}$ converges against $l_{j\infty}\in [0,+\infty)$ if $j \neq i$. Define $Y_n \in \XX_-(M)$ for every $n \in \NN$ by 
 $$Y_n(t,x) := Y\left((l_{1n},\dots,l_{kn}),t,x\right) \quad \forall t \in (-\infty,0] \ , \ \ x \in M \ , $$
 let $x \in \Crit f$ and let $\{\gamma_n\}_{n\in \NN}$ be a sequence with $\gamma_n \in W^-(x,Y_n)$ for every $n \in \NN$. If $\{\gamma_n\}_{n\in\NN}$ converges geometrically in this situation, then it will hold for the geometric limit $(\hat{g}_1,\dots,\hat{g}_{m},\gamma_-)$ by definition of a background perturbation that
 \begin{equation*}
  \gamma_- \in W^-(x_m,X_i(l_{1\infty},\dots,l_{(i-1)\infty},l_{(i+1)\infty},\dots,l_{k\infty})) \ .
 \end{equation*}
 The analogous remark is true for positive semi-infinite half-trajectories. 
\end{remark}

The convergence theorem for perturbed finite-length Morse trajectories takes a similar form, but we need to introduce the right notion of convergence in advance.

\begin{definition}
\label{DefFiniteLengthConvergent}
Let $\{Y_n\}_{n \in \NN} \subset \XX_0(M)$ be a sequence converging against $Y_\infty\in \XX_0(M)$ and let $\left\{ (l_n,\gamma_n)\right\}_{n \in \NN}$ be a sequence with $l_n \geq 0$,  $\gamma_n: [0,l_n] \to M$ and $(l_n,\gamma_n) \in \MM \left(Y_n \right)$ for all $n \in \NN$. The sequence  $\left\{ (l_n,\gamma_n)\right\}_{n \in \NN}$ is called \emph{convergent} iff the sequence $\left\{ \psi_{Y_n}(l_n,\gamma_n) \right\}_{n \in \NN}$ converges in $\MM(f,g)$. Here, we use the map $\psi$ from \eqref{Defofpsi} and put $\psi_{Y_n}:= \psi\left( Y_n,\cdot\right)$ for every $n \in \NN$. Every such  
 \begin{equation*}
  \psi_{Y_n}: \MM\left(Y_n\right) \stackrel{\cong}{\to} \MM(f,g)
 \end{equation*}
 is a diffeomorphism of class $C^{n+1}$.
\end{definition}

\begin{theorem}
\label{FiniteLengthCompactness} \index{compactness of moduli spaces!of finite-length trajectories}
Let $k \in \NN_0$ and let $\{Y_n\}_{n\in\NN} \subset \XX_0(M)$ be a sequence converging against some $Y_\infty \in \XX_0(M)$. Consider a sequence $\{(l_n,\gamma_n)\}_{n \in \NN}$ with $(l_n,\gamma_n) \in \MM\left(Y_n\right)$ for every $n \in \NN$. If the sequence $\{(l_n,\gamma_n)\}_{n \in \NN}$ does not have a convergent subsequence, then there are a subsequence $\{(l_{n_q},\gamma_{n_q})\}_{q \in \NN}$ with 
\begin{equation*}
\lim_{q\to \infty} l_{n_q} = +\infty \ ,
\end{equation*}
 $m \in \NN$, $x_1,\dots,x_m \in \Crit f$, $\{s^j_q\}_{q \in \NN} \subset \RR$ for every $j \in \{1,\dots,m-1\}$ as well as curves
\begin{equation*}
\left(\gamma_+,\hat{g}_1,\dots,\hat{g}_{m-1},\gamma_-\right) \in W^+\left(x_1,Y_+\right) \times \widehat{\MM}(x_1,x_2) \times \dots \times \widehat{\MM}(x_{m-1},x_{,m}) \times W^-\left(x_m,Y_-\right)
\end{equation*}
where $Y_+ \in \XX_+(M)$ and $Y_- \in \XX_-(M,k)$ are defined by 
\begin{equation*}
 (Y_+,Y_-) = \lim_{q  \to \infty} \Split_0\left(l_{n_q},Y_{n_q}\right) \ .
\end{equation*}
with $\Split_0$ is given as in \eqref{split0}, such that:
\begin{itemize}
 \item $\Bigl\{\left.\gamma_{n_q}\right|_{[0,T]}\Bigr\}_{q\geq q_T}$ converges against $\gamma_+|_{[0,T]}$ in the $C^\infty$-topology for every $T > 0$ and $q_T \in \NN$ such that the restrictions are well-defined,
 \item $\Bigl\{\left.\gamma_{n_q}\left(\cdot + s^j_q\right)\right|_{[-T,T]} \Bigr\}_{q \geq q_{T,j}}$ converges against $g_j|_{[-T,T]}$ in the $C^\infty$-topology for all $T > 0$ and $q_{T,j} \in \NN$ for which the restrictions are well-defined, where $g_j$ is a representative of $\hat{g}_j$ for every $j \in \{1,\dots,m-1\}$,
 \item $\Bigl\{\left.\gamma_{n_q}\left(\cdot + l_{n_q}\right)\right|_{[-T,0]}\Bigr\}_{q \geq q_T}$ converges against $\gamma_-|_{[-T,0]}$ in the $C^\infty$-topology for all $T > 0$ and $q_T \in \NN$ for which the restrictions are well-defined.
\end{itemize}
\end{theorem}

\begin{remark}
 For the validity of Theorem \ref{FiniteLengthCompactness}, the compactness of $M$ is required. If $M$ was non-compact, there might be sequences of finite-length trajectories whose interval lengths tends to infinity, but which are not geometrically convergent. 
 
 Speaking in terms of geometric intuition this corresponds to the starting points or end points of the trajectory sequence ``escaping to infinity``. This phenomenon can obviously not occur in compact manifolds. 
\end{remark}

\begin{definition}
\label{DefFiniteLenthGeomConv} 
 In the situation of Theorem \ref{FiniteLengthCompactness}, we say that the sequence $\{(l_{n_q},\gamma_{n_q})\}_{q \in \NN}$ \emph{converges geometrically against $(\gamma_+,\hat{g}_1,\dots,\hat{g}_{m-1},\gamma_-)$} and we call $(\gamma_+,\hat{g}_1,\dots,\hat{g}_{m-1},\gamma_-)$ the \emph{geometric limit of $\{(l_{n_q},\gamma_{n_q})\}_{q \in \NN}$}. 
\end{definition}

See Figure \ref{FinGeomConv} for an illustration of the geometric convergence of Theorem \ref{FiniteLengthCompactness}. 

\begin{figure}[h]
 \centering
 \includegraphics[scale=0.9]{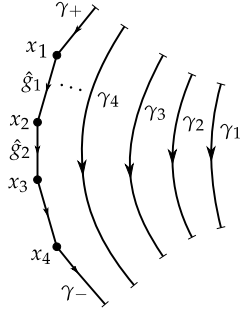}
 \caption{A geometrically convergent sequence of perturbed finite-length Morse trajectories.}
 \label{FinGeomConv}
\end{figure}

We want to derive a convergence theorem for sequences of perturbed Morse ribbon trees from Theorems \ref{CompactnessHalfTrajectories} and \ref{FiniteLengthCompactness}. Before stating the theorem, we will make some general observations and then motivate it by drawing immediate conclusions from Theorems \ref{CompactnessHalfTrajectories} and \ref{FiniteLengthCompactness}. \bigskip

We consider a sequence of perturbed Morse ribbon trees 
\begin{equation*}
 \left\{ \gammaunder_n \right\}_{n\in\NN}=\left\{\left(\gamma_{0n},\left(l_{en},\gamma_{en} \right)_{n \in \NN},\gamma_{1n},\dots,\gamma_{dn} \right)  \right\}_{n\in \NN} \subset \Acal^d_{\fatY}(x_0,x_1,\dots,x_d,T)
\end{equation*}
for some $d\geq 2$, $T \in \RTree_d$, $x_0,x_1,\dots,x_d \in \Crit f$ and 
\begin{equation*}
 \fatY=(Y_0,(Y_e)_{e \in E_{int}(T)},Y_1,\dots,Y_d) \in \XX(T) \ .
\end{equation*}
By definition of $\Acal^d_{\fatY}(x_0,x_1,\dots,x_d,T)$, we can regard each component sequence of $\left\{ \gammaunder_n\right\}_n$ as a sequence of perturbed Morse trajectories. Before applying the compactness results from above to the component sequences, we make an observation on the convergence of the edge length sequences. \bigskip 

The following lemma is a simple consequence of the Bolzano-Weierstra\ss{} theorem.

\begin{lemma}
 \label{BolWeier}
 Let $k \in \NN$ and $\left\{\vec{l}_n=\left(l_{1n},\dots,l_{kn}\right)\right\}_{n \in \NN} \subset [0,+\infty)^k$ be a sequence. Then there are a subsequence $\left\{ \vec{l}_{n_q}\right\}_{q \in \NN}$ and two disjoint subsets 
$I_1,I_2 \subset \{1,2,\dots,k\}$ with $I_1 \cup I_2 = \{1,2,\dots,k\}$, such that
\begin{itemize}
\item for every $i \in I_1$, the sequence $\{l_{in_q}\}_{q \in \NN}\subset [0,+\infty)$ diverges to $+\infty$,
\item for every $i \in I_2$, the sequence $\{l_{in_q}\}_{q \in \NN}$ converges with $l_{i\infty}:= \lim_{q \to \infty} l_{in_q}$.
\end{itemize}
\end{lemma}

In the situation of Lemma \ref{BolWeier}, define $\vec{l}_\infty \in [0,+\infty)^{|I_2|}$ by 
\begin{equation}
\label{veclinfty}
 \vec{l}_\infty :=\left(l_{i\infty}\right)_{i \in I_2} \ .
\end{equation}

We want to apply Lemma \ref{BolWeier} to the sequence $\left\{\gammaunder_n\right\}_{n\in\NN}$ from above. For all $n \in \NN$ we let $\vec{l}_n \in [0,+\infty)^{k(T)}$ denote the vector whose entries are given by the family $(l_{en})_{e \in E_{int}(T)}$, ordered by the given ordering of $E_{int}(T)$. We assume w.l.o.g. that every component sequence of $\left\{\vec{l}_n\right\}_{n\in\NN}$ as a sequence in $(0,+\infty)^{k(T)}$ either converges or diverges to $+\infty$, since by Lemma \ref{BolWeier} this property holds true up to passing to a subsequence. 

Then the following will hold for every $n \in \NN$:
\begin{align*}
 &\gamma_{0n} \in W^-\left(x_0,Y_0\left(\vec{l}_n\right)\right) \ , \quad \gamma_{in} \in W^+\left(x_i,Y_i\left(\vec{l}_n\right)\right) \  \forall i \in \{1,2,\dots,d\} \ , \\
 &(l_{en},\gamma_{en}) \in \MM\left(Y_e\left((l_{fn})_{f \in E_{int}(T)\setminus\{e\}}\right)\right) \quad \forall e \in E_{int}(T) \ .
\end{align*}
By applying Theorem \ref{CompactnessHalfTrajectories} with $Y_n = Y_i\left(\vec{l}_n\right)$ for every $n \in \NN$ we derive that at least one of the following is true for each $i \in \{0,1,\dots,d\}$, i.e. for each of the component sequences associated with external edges of $T$:
\begin{itemize}
 \item $\{\gamma_{in}\}_{n\in\NN}$ has a subsequence that converges in $\PP_\pm(x_i)$,  
 \item $\{\gamma_{in}\}_{n\in\NN}$ has a subsequence that converges geometrically against a family of trajectories.
\end{itemize}
Applying Theorem \ref{FiniteLengthCompactness} with $Y_n = Y_e((l_{fn})_{f \in E_{int}(T)\setminus\{e\}})$ to the component sequences associated with elements of $E_{int}(T)$ yields that for every $e \in E_{int}(T)$ at least one of the following holds:
\begin{itemize}
 \item $\{(l_{en},\gamma_{en})\}_{n\in\NN}$ has a subsequence that converges in the sense of Definition \ref{DefFiniteLengthConvergent} and the length of its limit curve is positive,
 \item $\{(l_{en},\gamma_{en})\}_{n\in\NN}$ has a subsequence that converges in the sense of Definition \ref{DefFiniteLengthConvergent} and the length of its limit curve is zero,
 \item $\{(l_{en},\gamma_{en})\}_{n\in\NN}$ has a subsequence that converges geometrically against a family of trajectories.
\end{itemize}
In both situations, we have used that by definition of the spaces $\XX_{\pm}(M,k)$ and $\XX_0(M,k)$, the respective sequence $\{Y_n\}_{n\in\NN}$ from above is indeed convergent in the respective space.

Moreover, since $\left\{\gammaunder_n\right\}_{n\in\NN}$ has finitely many components, we can find a ''common`` subsequence for all edges of $T$ such that one of the above holds. More precisely, there is a subsequence $\left\{\gammaunder_{n_q}\right\}_{q\in\NN}$ such that all the sequences $\{\gamma_{in_q}\}_{q\in\NN}$, $i \in \{0,1,\dots,d\}$, and $\{(l_{en_q},\gamma_{en_q})\}_{q\in\NN}$, $e \in E_{int}(T)$, either converge or converge geometrically. \bigskip 

While we have not distinguished betweed the first two bullet points for finite-length components in Theorem \ref{FiniteLengthCompactness}, we will do so in the following considerations. The reason for this distinction lies in the definition of the moduli spaces of perturbed Morse ribbon trees. 

Suppose that every component sequence of $\left\{\gammaunder_n\right\}_{n\in\NN}$ converges. If $\lim_{n\to\infty} l_{en} > 0$ for every $e \in E_{int}(T)$, then the family of the limits will again be identified with an element of $\Acal^d_{\fatY}(x_0,x_1,\dots,x_d,T)$. 

If there is an $e \in E_{int}(T)$ for with $\lim_{n \to \infty} l_{en}=0$, then there will be no such identification, since by definition of $\Acal^d_{\fatY}(x_0,x_1,\dots,x_d,T)$, every finite-length component of an element of this moduli space must have positive length. Instead, in the course of this section we will identify such limits with perturbed Morse ribbon trees \emph{modelled on different trees than $T$.}

\bigskip

These considerations motivate the distinction between the sets $E_1$ and $E_3$ in the following convergence theorem for sequences of perturbed Morse ribbon trees, which summarizes our elaborations on components of sequences of perturbed Morse ribbon trees:

\begin{theorem} \index{compactness of moduli spaces!of Morse ribbon trees}
\label{CompactnessMorseTrees}
 Let $d \geq 2$, $T \in \RTree_d$, $\fatY \in \XX(T)$ and $x_0,x_1,\dots,x_d \in \Crit f$.  Let
 \begin{equation*}
  \left\{\left(\gamma_{0n},\left(l_{en},\gamma_{en} \right)_{n \in \NN},\gamma_{1n},\dots,\gamma_{dn} \right) \right\}_{n \in \NN}
 \end{equation*}
 be a sequence in $\Acal^d_{\fatY}(x_0,x_1,\dots,x_d,T)$. For every $e \in E(T)$ define a sequence $\left\{\bar{\gamma}_{en} \right\}_{n \in \NN}$ by putting
 \begin{equation*}
 \bar{\gamma}_{en} := \begin{cases}
                       \gamma_{in} & \text{if} \quad  e = e_i(T) \quad \text{for some} \ i \in \{0,1,\dots,d\} \ , \\
                       \left(l_{en},\gamma_{en} \right) & \text{if} \quad  e \in E_{int}(T) \ ,
                      \end{cases}
 \end{equation*}
 for every $n \in \NN$. There are sets $E_1,E_2 \subset E(T)$, $E_3 \subset E_{int}(T)$ with
 \begin{equation*}
  E(T) = E_1 \sqcup E_2 \sqcup E_3 \ , 
 \end{equation*}
 such that
 \begin{itemize}
  \item for every $e \in E_1$, the sequence $\left\{\bar{\gamma}_{en} \right\}_{n \in \NN}$ has a convergent subsequence and 
	  \begin{equation*}
	  \liminf_{n\to \infty} l_{en} > 0 \quad \text{if} \ e \in E_{int}(T) \ ,
	  \end{equation*}
  \item for every $e \in E_2$, the sequence $\left\{\bar{\gamma}_{en} \right\}_{n \in \NN}$ has a geometrically convergent subsequence,
  \item for every $e \in E_3$, the sequence $\left\{(l_{en},\gamma_{en}) \right\}_{n\in \NN}$ has a convergent subsequence 
  \begin{equation*}
  \left\{(l_{en_q},\gamma_{en_q}) \right\}_{q\in \NN} \ \ \text{with} \ \ \lim_{q \to \infty} l_{en_q} = 0 \ .
  \end{equation*}
 \end{itemize}
\end{theorem}

Theorem \ref{CompactnessMorseTrees} immediately leads to the following notion of convergence for sequences of perturbed Morse ribbon trees: 

\begin{definition} 
Let $\left\{\left(\gamma_{0n},\left(l_{en},\gamma_{en}\right)_{e \in E_{int}(T)},\gamma_{1n},\dots,\gamma_{dn} \right) \right\}_{n \in \NN}$ be a sequence of perturbed Morse ribbon trees. 
\begin{enumerate} 
 \item Let $i \in \{0,1,\dots,d\}$. Whenever the sequence $\{\gamma_{in}\}_{n \in \NN}$ converges, we denote its limit by $\gamma_{i\infty}$.
 \item Let $e \in E_{int}(T)$. Whenever the sequence $\{(l_{en},\gamma_{en})\}_{n \in \NN}$ converges, we denote its limit by $(l_{e\infty},\gamma_{e\infty})$.  
 \item We say that the sequence is  \emph{convergent} if every component sequence $\left\{\bar{\gamma}_{en}\right\}_{n \in \NN}$, $e \in E(T)$, converges (where we have used the notation from Theorem \ref{CompactnessMorseTrees}).
 \item The \emph{limit} of a convergent sequence of perturbed Morse ribbon trees is defined as the product of the limits of the component sequences.
 \end{enumerate}
\end{definition}

Note that a sequence of perturbed Morse ribbon trees has a convergent subsequence in the sense of the previous definition if and only if we can choose $E_2 = \emptyset$ in Theorem \ref{CompactnessMorseTrees}. \bigskip 

In the following we will describe the convergence behaviour of sequences of perturbed Morse ribbon trees in several special cases in greater detail. 

\begin{definition}
\label{QuotientTree} 
Let $d \in \NN$, $d \geq 2$. For $T \in \RTree_d$ and $e \in E_{int}(T)$, we define
 \begin{equation*}
  T/e \in \RTree_d
 \end{equation*}
 as the unique tree we obtain from $T$ after collapsing the edge $e$. More precisely, $T/e$ is the unique tree with $E(T/e) = E(T) \setminus \{e\}$ and such that $\vout(f)=\vin(f')$ for all $f,f' \in E(T)\setminus\{e\}$ with 
 \begin{equation*}
 \vout(f)=\vin(e) \quad \text{and} \quad \vin(f')=\vout(e) \ .
 \end{equation*}
 Let $d \in \NN$, $d \geq 2$. For $T \in \RTree_d$ and $F \subset E_{int}(T)$, we define 
 \begin{equation*}
  T/F \in \RTree_d
 \end{equation*}
 as the unique tree we obtain from $T$ after collapsing every edge contained in $F$. More precisely, if $F=\{f_1,f_2,\dots,f_m\}$  for suitable $m \in \NN$, then we define $T/F$ inductively by 
   $$ T/F := \left( \ldots \left( \left( \left( T/f_1 \right) / f_2 \right) / f_3 \right) \ldots \right) / f_m \ .$$
 One checks that the tree $T/F$ is independent of the choice of ordering of $F$.  
 \end{definition}
 
\begin{figure}[h]
 \centering
 \includegraphics[scale=0.6]{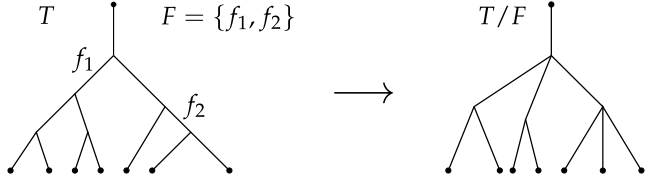}
 \caption{An example of the collapse of internal edges of a ribbon tree.}
 \label{FigQuotientTree}
\end{figure}

Figure \ref{FigQuotientTree} shows an example of such a ``quotient tree''. Note that especially 
\begin{equation*}
 k(T/F) = k(T) - |F| \ 
\end{equation*}
for all $T \in \bigcup_{d\geq 2}\RTree_d$ and $F\subset E_{int}(T)$. 

In the following we will relate moduli spaces of perturbed Morse ribbon trees modelled on $T$ with those modelled on $T/F$ for some $F \subset E_{int}(T)$. For this purpose, we need to find a way of considering perturbations in $\XX(T)$ and $\XX(T/F)$ at the same time. Before we investigate these relations, we therefore introduce a technique for relating perturbations associated with different ribbon trees to each other. 

\begin{definition}
 Let $d \in \NN$, $T \in \RTree_d$ and $F \subset E_{int}(T)$. 
\begin{enumerate}
 \item For $Y \in \XX_{\pm}(M,k(T))$ we define 
 \begin{align*}
&Y/F \in \XX_{\pm}(M,k(T/F)), \ \ (Y/F)\left((l_e)_{e \in E_{int}(T/F)},t,x\right) := Y\left((\tilde{l}_e)_{e \in E_{int}(T)},t,x \right) \ , \\
&\text{where} \quad \tilde{l}_e := \begin{cases}
                               l_e & \text{if} \ \ e \in E_{int}(T)\setminus F \ , \\
                               0 & \text{if} \ \ e \in F \ . 
                             \end{cases}
\end{align*}
Analogously, for $Z \in \XX_0(M,k(T)-1)$ and $f \in E_{int}(T)\setminus F$ we define 
\begin{align*}
&Z/F \in \XX_0(M,k(T/F)-1), \\ 
&(Z/F)\left((l_e)_{e \in E_{int}(T/F)\setminus \{f\}},l_f,t,x\right):= Z\left((\tilde{l}_e)_{e \in E_{int}(T)\setminus \{f\}},l_f,t,x \right) \ , 
\end{align*}
where $\tilde{l}_e$ is given as above. 
\item We define a map 
\begin{equation}
\label{piF} 
\begin{aligned}
 \pi_F: \XX(T) &\to \XX(T/F) \ ,  \\
\left(Y_0,(Y_e)_{e \in E_{int}(T)},Y_1,\dots,Y_d\right) &\mapsto \left(Y_0/F,(Y_e/F)_{e \in E_{int}(T)\setminus F},Y_1/F,\dots,Y_d/F\right) \ , 
\end{aligned}
\end{equation}
$\pi_F$ is a composition of an evaluation map and a projection. Therefore, $\pi_F$ is easily seen to be continuous and surjective. 
\item Assume that $k:= k(T)>0$. We define maps
$$\pib_{\pm,F}: \XXb_{\pm}(M,k) \to \XXb_{\pm}(M,k(T/F)) \ , \quad (X_1,\dots,X_k) \mapsto (X_1/F,\dots,X_k/F)\ .$$
If $k(T/F)>0$, then  we further let $e \in E_{int}(T) \setminus F$ and define 
\begin{align*}
&\pib_{0,F}: \XXb_0(M,k(T)-1) \to \XXb_0(M,k(T/F)-1) \ , \\
&((X_f)_{f \in E_{int}(T)\setminus \{e\}},X_+,X_-) \mapsto ((X_f/F)_{f \in E_{int}(T/F)\setminus \{e\}},X_+/F,X_-/F) \ .
\end{align*}
In terms of these maps, we define for $k(T/F)>0$ a map
\begin{align*}
\pib_F: \XXb(T) &\to \XXb(T/F) \ , \\
\left(\fatX_0,(\fatX_e)_{e},\fatX_1,\dots,\fatX_d\right) &\mapsto \left(\pib_{-,F}(\fatX_0),(\pib_{0,F}(\fatX_e))_{e},\pib_{+,F}(\fatX_1),\dots,\pib_{+,F}(\fatX_d)\right) \ ,
\end{align*}
and for $k(T/F)=0$ a map 
\begin{align*}
\pib_F: \XXb(T) &\to \XX_-(M)\times \XX_+(M)^d \ , \\
\left(\fatX_0,(\fatX_e)_{e},\fatX_1,\dots,\fatX_d\right) &\mapsto \left(\pib_{-,F}(\fatX_0),\pib_{+,F}(\fatX_1),\dots,\pib_{+,F}(\fatX_d)\right) \ ,
\end{align*}
In analogy with $\pi_F$, the map $\pib_F$ is continuous and surjective. 
\end{enumerate}
\end{definition}
To relate perturbations in $\XX(T)$ with those in $\XX(T/F)$, we further need to develop a formalism for simultaneous choices of perturbations for families of ribbon trees. At the same time, we introduce analogous notions for background perturbations, which will not be required in this section, but in the following one. 

\begin{definition}
\label{PertDatum} 
 \begin{enumerate}
  \item A \emph{$d$-perturbation datum} is a family $\fatY = \left(\fatY_T\right)_{T \in \RTree_d}$ with $\fatY_T \in \XX(T)$ for each $T \in \RTree_d$.
 \item A $d$-perturbation datum $\fatY = \left(\fatY_T\right)_{T \in \RTree_d}$ is called \emph{universal} if for all $T \in \RTree_d$ and $F \subset E_{int}(T)$ it holds that $\pi_F\left(\fatY_T\right) = \fatY_{T/F}$.
 \item Given a $d$-perturbation datum $\fatY = \left(\fatY_T\right)_{T \in \RTree_d}$ and for $T \in \RTree_d$ we write
 \begin{equation*}
 \Acal^d_{\fatY} (x_0,x_1,\dots,x_d,T) := \Acal^d_{\fatY_T} (x_0,x_1,\dots,x_d,T)
 \end{equation*}
 for all $x_0,x_1,\dots,x_d \in \Crit f$. 
  \item Let $\RTree_d^*:= \{T \in \RTree_d \ | \ k(T)>0\}$.  A \emph{background $d$-perturbation datum} is a family $\fatX = \left(\fatX_T\right)_{T \in \RTree^*_d}$, such that $\fatX_T \in \XXb(T)$ for each $T \in \RTree^*_d$. 
 \item We call a background $d$-perturbation datum $\fatX = \left(\fatX_T\right)_{T \in \RTree^*_d}$  \emph{universal} if is satisfies both of the following conditions: 
 \begin{itemize}
 \item for all $T \in \RTree^*_d$ and $F \subset E_{int}(T)$ with $k(T/F)>0$ it holds that 
 $\pib_F\left(\fatX_T\right) = \fatX_{T/F}$,
 \item for all $T,T' \in \RTree^*_d$ it holds with $F:=E_{int}(T)$ and $F':=E_{int}(T')$ that
 $$\pib_F(\fatX_T) = \pib_{F'}(\fatX_{T'}) \ . $$
 \end{itemize}
 \end{enumerate}
\end{definition}

The next theorem describes the case that all component sequences of a sequence of perturbed Morse ribbon trees converge and that there are internal edges whose associated sequences of edge lengths tend to zero. This corresponds to the case 
\begin{equation*}
E_1 = E(T) \setminus F \ , \quad E_2 = \emptyset \ , \quad E_3=F
\end{equation*}
for some $F \subset E_{int}(T)$ in Theorem \ref{CompactnessMorseTrees}.

\begin{theorem} 
\label{LengthToZero}
 Let $d \in \NN$, $d \geq 2$, $T \in \RTree_d$ and let $\fatY \in \prod_{T \in \RTree_d} \XX(T)$ be a universal $d$-perturbation datum.
 Let $x_0,x_1,\dots,x_d \in \Crit f$ and let 
\begin{equation*}
  \left\{\left(\gamma_{0n},\left(l_{en},\gamma_{en}\right)_{e \in E_{int}(T)},\gamma_{1n},\dots,\gamma_{dn} \right) \right\}_{n \in \NN} \subset \Acal^d_{\fatY}(x_0,x_1,\dots,x_d,T)
\end{equation*}
be a sequence, for which all of the sequences $\left\{\gamma_{in}\right\}_{n \in \NN}$, $i \in \{0,1,\dots,d\}$, and $\left\{(l_{en},\gamma_{en})\right\}_{n \in \NN}$, $e \in E_{int}(T)$, converge and for which there is an $F \subset E_{int}(T)$ such that
\begin{align*}
&l_{f\infty} = 0 \  \text{for every} \ f \in F \  , \\
&l_{f\infty} > 0 \  \text{for every} \ f \in E_{int}(T) \setminus F \  .
\end{align*}
Then $\left(\gamma_{0\infty},\left(l_{e\infty},\gamma_{e\infty}\right)_{e \in E_{int}(T)\setminus F},\gamma_{1\infty},\dots,\gamma_{d\infty} \right) \in \Acal^d_{\fatY}(x_0,x_1,\dots,x_d,T/F) \ .$
\end{theorem}

\begin{proof}
By definition of $\Acal^d_{\fatY}(x_0,x_1,\dots,x_d,T)$, it holds for every $n \in \NN$ that
\begin{equation*}
  \left(\gamma_{0n}(0),\left(\gamma_{en}(0),\gamma_{en}(l_{en})\right)_{e \in E_{int}(T)},\gamma_{1n}(0),\dots,\gamma_{dn}(0)\right) \in \Delta_T \ .
\end{equation*}
Since $\Delta_T$ is closed, the convergence of the sequence implies
\begin{equation*}
  \lim_{n \to \infty} \left(\gamma_{0n}(0), (\gamma_{en}(0),\gamma_{en}(l_{en}))_{e \in E_{int}(T)},\gamma_{1n}(0),\dots,\gamma_{dn}(0) \right) \in \Delta_T \  .
\end{equation*}
Moreover, since $\lim_{n \to \infty} l_{fn} = 0$ for every $f \in F$, we conclude
\begin{equation*}
 \lim_{n \to \infty} \gamma_{fn}(0) = \lim_{n \to \infty} \gamma_{fn}(l_{fn}) \quad \forall f \in F \ , 
\end{equation*}
which yields
\begin{equation}
\label{limitcondDeltaT}
\begin{aligned}
 &\lim_{n \to \infty} \left(\gamma_{0n}(0), (\gamma_{en}(0),\gamma_{en}(l_{en}))_{e \in E_{int}(T)},\gamma_{1n}(0),\dots,\gamma_{dn}(0) \right) \\ &\qquad \qquad \in \left\{(q_0,(q_{in}^e,q_{out}^e)_{e \in E_{int}(T)},q_1,\dots,q_d) \in \Delta_T \ \left| \ q_{in}^f = q_{out}^f \quad \forall f \in F \right. \right\} \ .
\end{aligned}
\end{equation}
From the definition of the $T$-diagonal (see Remark \ref{DeltaTexplicitly}) we derive
\begin{align*}
 &\left\{(q_0,(q_{in}^e,q_{out}^e)_{e \in E_{int}(T)},q_1,\dots,q_d) \in \Delta_T \ \left| \ q_{in}^f = q_{out}^f \quad \forall f \in F \right. \right\} \\
&=\left\{(q_0,(q_{in}^e,q_{out}^e)_{e \in E_{int}(T)},q_1,\dots,q_d) \in \Delta_T \left| \ q_{in}^e = q_{out}^{e'} \ \forall e,e' \in E_{int}(T) \ \text{ s.t. }\right. \right.  \\ 
&\left.\qquad \qquad \exists f_1,\dots,f_m \in F \  \text{ with } \ q_{out}^e=q_{in}^{f_1}, \ q_{out}^{f_1}=q_{in}^{f_2},\ \dots, q_{out}^{f_{n-1}}=q_{in}^{f_n}, \  q_{out}^{f_n} = q_{in}^{e'} \ \right\} \ .
\end{align*}
By definition of $\Delta_T$ and $\Delta_{T/F}$, the projection $M^{1+2k(T)+d} \to M^{1+2k(T/F)+d}$ which projects away from the components associated with elements of $F$ maps this space diffeomorphically onto 
\begin{align*}
 &\Delta_{T/F}=\left\{(q_0,(q_{in}^e,q_{out}^e)_{e \in E_{int}(T/F)},q_1,\dots,q_d) \in M^{1+2k(T/F)+d} \ \right. \\
  &\quad \ \left| \ q_{out}^{e} = q_{in}^f \ \text{ for every } \ e,f \in E_{int}(T/F)\ \text{ with } \  \vout(e) = \vin(f) \ , \right. \\
&\quad \ \ \ \ q_{out}^{e} = q_j \ \text{ for every } \ e \in E_{int}(T/F),  \ j \in \{1,\dots,d\} \ \text{ with } \ \vout(e) =\vin(e_j(T/F))  \ , \\
  &\quad \left. \phantom{q_{in}^f} q_0 = q_{in}^e \ \text{ for every } \ e \in E_{int}(T/F)\ \text{ with } \ \vout(e_0(T/F)) = \vin(e)  \ \right\} \ .
\end{align*}
Thus, condition (\ref{limitcondDeltaT}) implies
\begin{equation}
\label{inDeltaTF}
\begin{aligned}
&\left(\gamma_{0\infty}(0), (\gamma_{e\infty}(0),\gamma_{e\infty}(l_{e\infty}))_{e \in E_{int}(T)\setminus F},\gamma_{1\infty}(0),\dots,\gamma_{d\infty}(0) \right) \\
&=\lim_{n \to \infty} \left(\gamma_{0n}(0), (\gamma_{en}(0),\gamma_{en}(l_{en}))_{e \in E_{int}(T)\setminus F},\gamma_{1n}(0),\dots,\gamma_{dn}(0) \right) \in \Delta_{T/F} \ . 
\end{aligned}
\end{equation}
Furthermore, since $\displaystyle\lim_{n \to \infty} l_{fn}=0$ for every $f \in F$, we obtain
\begin{align*}
 &\gamma_{0\infty} \in W^-\left(x_0,(Y_0/F)\left((l_e)_{e \in E_{int}(T/F)} \right) \right) \ , \\
 &(l_{e\infty},\gamma_{e\infty}) \in \MM\left((Y_e/F)\left((l_f)_{f \in E_{int}(T/F)\setminus\{e\}} \right) \right) \ \forall e \in E_{int}(T/F) \ , \\
 &\gamma_{i\infty} \in W^+\left(x_i,(Y_i/F)\left((l_e)_{e \in E_{int}(T/F)}\right) \right) \ .
\end{align*}
Together with \eqref{inDeltaTF}, these observations yield:
\begin{equation*}
 \Rightarrow \left(\gamma_{0\infty},\left(l_{e\infty},\gamma_{e\infty}\right)_{e \in E_{int}(T)\setminus F},\gamma_{1\infty},\dots,\gamma_{d\infty} \right) \in \Acal^d_{\pi_F\left(\fatY_T\right)}(x_0,x_1,\dots,x_d,T/F) \ . 
\end{equation*}
Finally, we make use of the universality property. Up to this point, we only know that the limit is a perturbed Morse ribbon tree modelled on $T/F$, but with respect to a perturbation which is induced by $\fatY_T$. 

For a general perturbation datum, this perturbation is not related to $\fatY_{T/F}$. But since the $d$-perturbation datum $\fatY$ is universal, it holds that $\pi_F(\fatY_T)=\fatY_{T/F}$, which shows the claim.
\end{proof}

The following theorem describes the case that all sequences of curves associated with internal edges of the tree converge and that all but one of the sequences associated with external edges converge while the remaining one is geometrically convergent. Formally speaking, it gives a precise description of the cases
\begin{equation*}
 E_1 = E(T) \setminus \{e_i(T)\}, \quad E_2 = \{e_i(T)\}, \quad E_3=\emptyset , \quad \text{for some} \ i \in \{0,1,\dots,d\},
\end{equation*}
in Theorem \ref{CompactnessMorseTrees}.

\begin{theorem} \index{perturbed Morse ribbon trees!convergence of sequences}
\label{ConvergenceHalfTraj} 
 Let $d \geq 2$, $T \in \RTree_d$ and $\fatY \in \XX(T)$. Let $x_0,x_1,\dots,x_d \in \Crit f$ and let 
\begin{equation*}
  \left\{\left(\gamma_{0n},\left(l_{en},\gamma_{en}\right)_{e \in E_{int}(T)},\gamma_{1n},\dots,\gamma_{dn} \right) \right\}_{n \in \NN} \subset \Acal^d_{\fatY}(x_0,x_1,\dots,x_d,T) 
\end{equation*}
such that
\begin{equation*}
 \liminf_{n \to \infty} l_{en} > 0 \quad \forall e \in E_{int}(T) \ .
\end{equation*}
For $e \in E(T)$ and $n \in \NN$ we put 
\begin{equation*}
\bar{\gamma}_{en} := \begin{cases}  
                      (l_{en},\gamma_{en}) & \text{if} \ \ e \in E_{int}(T) \ , \\
                      \gamma_{in} & \text{if} \ \ e = e_i(T), \ i \in \{0,1,\dots,d\} \ . 
                     \end{cases}
\end{equation*}
For any such $e$ and $i$ we further put $(l_{e\infty},\gamma_{e\infty}) := \lim_{n\to\infty} (l_{en},\gamma_{en})$ if $\{(l_{en},\gamma_{en})\}_{n\in\NN}$ converges and $\gamma_{i\infty} := \lim_{n\to \infty} \gamma_{in}$ if $\{\gamma_{in}\}_{n \in \NN}$ converges.
 \begin{enumerate}
  \item Assume that the sequence $\left\{\bar{\gamma}_{en}\right\}_{n \in \NN}$ converges for every $e \in E(T) \setminus \{e_0(T)\}$ and that $\left\{\gamma_{0n} \right\}_{n \in \NN}$ converges geometrically against some 
  \begin{equation*}
  (\hat{g}_1,\dots,\hat{g}_{m-1},\gamma_-) \in \widehat{\MM}(x_0,y_1) \times \prod_{j=1}^{m-1} \widehat{\MM}(y_j,y_{j+1}) \times W^-\left(y_m,Y_0\left(\vec{l}_\infty\right)\right) \ ,
  \end{equation*}
  where $\vec{l}_\infty = (l_{e\infty})_{e \in E_{int}(T)} \in (0,+\infty)^{E_{int}(T)}$. Then
  \begin{equation*}
   \left(\gamma_{-},\left(l_{e\infty},\gamma_{e\infty}\right)_{e \in E_{int}(T)},\gamma_{1\infty},\dots,\gamma_{d\infty} \right) \in \Acal^d_{\fatY}(y_m,x_1,\dots,x_d,T) \ .
  \end{equation*} 
  \item Assume that there is an $i \in \{1,\dots,d\}$ such that the sequence $\left\{\bar{\gamma}_{en}\right\}_{n \in \NN}$ converges for every $e \in E(T) \setminus \{e_i(T)\}$ and that $\left\{\gamma_{in} \right\}_{n \in \NN}$ converges geometrically against some 
  \begin{equation*}
  (\gamma_+,\hat{g}_1,\dots,\hat{g}_{m}) \in W^+\left(y_1,Y_i\left(\vec{l}_\infty\right)\right)  \times \prod_{j=1}^{m-1} \widehat{\MM}(y_j,y_{j+1}) \times \widehat{\MM}(y_m,x_i) \ ,
  \end{equation*}
    where $\vec{l}_\infty = (l_{e\infty})_{e \in E_{int}(T)} \in (0,+\infty)^{E_{int}(T)}$. Then
  \begin{align*}
   &\left(\gamma_{0\infty},\left(l_{e\infty},\gamma_{e\infty}\right)_{e \in E_{int}(T)},\gamma_{1\infty},\dots,\gamma_{(i-1)\infty},\gamma_+,\gamma_{(i+1)\infty},\dots,\gamma_{d\infty} \right) \\ 
   &\qquad \qquad \qquad \in \Acal^d_{\fatY}(x_0,x_1,\dots,x_{i-1},y_1,x_{i+1},\dots,x_d,T) \ .
  \end{align*}
 \end{enumerate}
\end{theorem}

See Figure \ref{FigSemiInfConv} for an illustration of the geometric limits in both parts of Theorem \ref{ConvergenceHalfTraj}. The left-hand picture illustrates the first part, while the right-hand picture illustrates the second part.

\begin{figure}[h]
 \centering
 \includegraphics[scale=0.6]{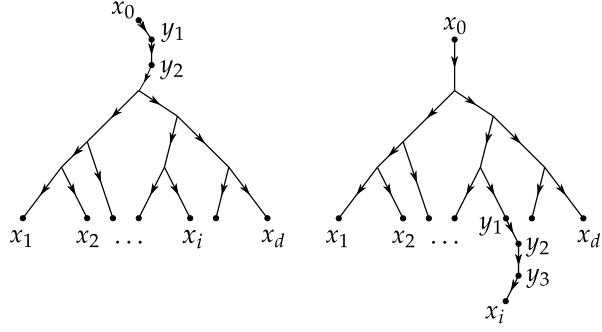}
 \caption{The geometric convergence in Theorem \ref{ConvergenceHalfTraj}.}
 \label{FigSemiInfConv}
\end{figure}

\begin{proof} 
We first note that in both parts of the theorem, it obviously holds that 
\begin{equation}
\label{EqPertLimit}
\lim_{n \to \infty} Y_i\left((l_{en})_{e\in E_{int}(T)}\right) = Y_i\left(\vec{l}_\infty\right) \quad 	\text{in} \ \ \XX(T) \quad \forall i \in \{0,1,\dots,d\} \ . 
\end{equation}
\begin{enumerate}
\item  By \eqref{EqPertLimit} and part 1 of Theorem \ref{CompactnessHalfTrajectories}, $\{\gamma_{0n}\}_{n \in \NN}$ converges to $\gamma_-$ in the $\Cloc$-topology, which especially implies that $\lim_{n \to \infty} \gamma_{0n}(0)= \gamma_-(0)$. This has the following consequence:
\begin{align*}
 &\left(\gamma_-(0), \left(\gamma_{e\infty}(0),\gamma_{e\infty}(l_{e \infty})\right)_{e \in E_{int}(T)},\gamma_{1\infty}(0),\dots,\gamma_{d\infty}(0)\right) \\
&= \lim_{n \to \infty} \left(\gamma_{0n}(0), \left(\gamma_{en}(0),\gamma_{en}(l_{en})\right)_{e \in E_{int}(T)},\gamma_{1n}(0),\dots,\gamma_{dn}(0)\right) \in \Delta_T \ ,
\end{align*}
which yields that
\begin{equation*}
 \left( \gamma_-, \left(l_{e \infty},\gamma_{e \infty}\right)_{e \in E_{int}(T)},\gamma_{1\infty},\dots,\gamma_{d\infty}\right) \in \Acal^d_{\fatY}(y_m,x_1,\dots,x_d,T) \ .
\end{equation*}
\item This is shown along the lines of the first part by applying \eqref{EqPertLimit} and part 2 of Theorem \ref{CompactnessHalfTrajectories}.
\end{enumerate}
\end{proof}

The next theorem covers the last of the special cases of convergence that we are considering. It discusses the case of geometric convergence of a family of sequences associated to internal edges of a ribbon tree while all other sequences are converging. This corresponds to the case
\begin{equation*}	
 E_1 = E(T)\setminus E_2, \quad E_2 \subset E_{int}(T), \quad E_3 = \emptyset \ , 
\end{equation*}
in Theorem \ref{CompactnessMorseTrees}. \bigskip

Remember that we have equipped $E_{int}(T)$ with a fixed ordering for every $T \in \RTree_d$, which is necessary to make sense of the identification
\begin{equation*}
 \XX(T) = \XX(1,k(T),d) \ . 
\end{equation*}
So for any $F \subset E_{int}(T)$ we can view $F$ as a subset of $\{1,2,\dots,k\}$, with $k:=k(T)$.

In view of this identification we consider for $Y \in \XX_{*}(M,k)$ and $F= \{i_1,\dots,i_d\}$ the vector field 
\begin{align}
&Y_F \in \XX_\pm(M,k-|F|) \ , \notag \\
 &Y_F := \lim_{\lambda_1\to+\infty} \lim_{\lambda_2\to+\infty} \dots \lim_{\lambda_{i_d}\to+\infty} c_{i_1}\left(\lambda_1,c_{i_2}\left(\lambda_2,\dots, c_{i_d}\left(\lambda_d,Y\right)\dots \right) \right) \ , \label{Ylimlimlim}
\end{align} 
where $\XX_*(M,k)$ denotes one of the three spaces $\XX_-(M,k)$, $\XX_+(M,k)$ and $\XX_0(M,k)$. Here, $c_{i}(\lambda,Z)$ again denotes the contraction map given by inserting $\lambda$ into the $i$-th parameter component of $Z$. 

Since $Y$ is of class $C^{n+1}$ and the limits in \eqref{Ylimlimlim} exist by definition of $\XX_\pm(M,k)$, the parametrized vector field $Y_F$ is well defined and actually independent of the order of the limits in \eqref{Ylimlimlim}.

\begin{theorem} \index{perturbed Morse ribbon trees!convergence of sequences}
\label{Fconvergence}
 Let $d \geq 2$, $T \in \RTree_d$ and $\fatY=(Y_0,(Y_e)_{e\in E_{int}(T)},Y_1,\dots,Y_d) \in \XX(T)$. Let $x_0,x_1,\dots,x_d \in \Crit f$ and 
\begin{equation*}
  \left\{\left(\gamma_{0n},\left(l_{en},\gamma_{en}\right)_{e \in E_{int}(T)},\gamma_{1n},\dots,\gamma_{dn} \right) \right\}_{n \in \NN} \subset \Acal^d_{\fatY}(x_0,x_1,\dots,x_d,T)
\end{equation*}
be a sequence with $\liminf_{n\to \infty} l_{en}> 0 \quad \forall e \in E_{int}(T)$. Assume that there is $F \subset E_{int}(T)$, such that for every $f \in F$ the sequence $\left\{(l_{fn},\gamma_{fn}) \right\}_{n \in \NN}$ converges geometrically against some 
\begin{align*}
 &\left(\gamma_{f+},\hat{g}_1,\dots,\hat{g}_m,\gamma_{f-}\right) \in W^+\left(y_{f+},(Y_{f+})_F\left(\vec{l}_{\infty}\right)\right) \times \widehat{\MM}(y_{f+},y_{f1}) \times \prod_{j=1}^{m-2} \widehat{\MM}(y_{fj},y_{f(j+1)}) \\ &\qquad \qquad \qquad \qquad \qquad \qquad \qquad \times \widehat{\MM}(y_{f(m-1)},y_{f-}) \times W^-\left(y_{f-},(Y_{f-})_F\left(\vec{l}_\infty \right) \right) \ ,
\end{align*}
where $(Y_{f+},Y_{f-}) = \lim_{\lambda \to \infty} \Split_{k(T)-1}(\lambda,Y_f)$, and that $\left\{\bar{\gamma}_{en} \right\}_{n \in \NN}$ converges for every $e \in E(T) \setminus F$, where $\bar{\gamma}_{en}$ is defined as in Theorem \ref{ConvergenceHalfTraj}. Then, using the notation from Theorem \ref{NonlocalTransversality}):
\begin{align*}
 &\left(\gamma_{0\infty},\left(\gamma_{f-} \right)_{f \in F}, \left(l_{e\infty},\gamma_{e\infty}\right)_{e \in E_{int}(T) \setminus F},\gamma_{1\infty},\dots,\gamma_{d\infty},\left(\gamma_{f+} \right)_{f \in F} \right) \\
 & \in \MM_{\fatY_{F}}\left(\left(x_0,\left(y_{f-}\right)_{f \in F} \right),\left(x_1,\dots,x_d,\left(y_{f+}\right)_{f \in F} \right),(1+|F|,k(T)-|F|,d+|F|) , \right. \\ &\qquad \qquad \left. \qquad \qquad \qquad \qquad \qquad \qquad\qquad \qquad \qquad \qquad \left(\RR_{>0}\right)^{k(T)-|F|},\sigma_F(\Delta_T) \right) \ ,
\end{align*}
 for some diffeomorphism $\sigma_F: M^{1+2k(T)+d} \stackrel{\cong}{\to} M^{1+2k(T)+d}$,
which is a permutation of the different copies of $M$, and where we put
\begin{equation}
\label{YF}
\begin{aligned}
&\fatY_F \in \XX(1+|F|,k(T)-|F|,d+|F|) \ , \\
&\fatY_F := \left((Y_0)_F,\left((Y_{f-})_F \right)_{f \in F},\left((Y_e)_F \right)_{e \in E_{int}(T) \setminus F} ,(Y_1)_F,\dots,(Y_d)_F,\left((Y_{f+})_F \right)_{f \in F} \right) \ .
\end{aligned}
\end{equation}
\end{theorem}

See Figure \ref{FigureIntGeomConv} for an illustration of the geometric convergence in Theorem \ref{Fconvergence}.

\begin{figure}[h]
 \centering
 \includegraphics[scale=0.7]{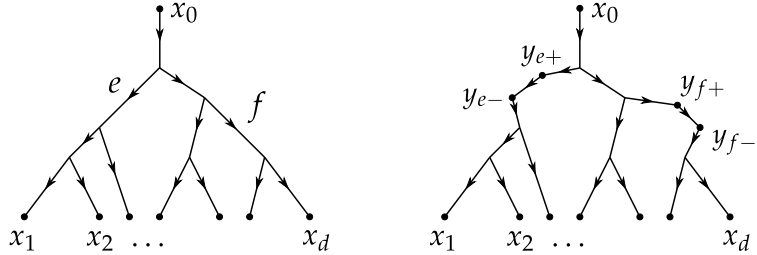}
 \caption{Geometric convergence along $F= \{e,f\}$.}
 \label{FigureIntGeomConv}
\end{figure}

\begin{proof}
 Let $f \in F$. By Theorem \ref{FiniteLengthCompactness}, we know that $\left\{\gamma_{fn}|_{[0,T]}\right\}_{n \geq n_T}$ converges against $\gamma_{f+}|_{[0,T]}$ in the $C^\infty$-topology for every $T \geq 0$ and sufficiently big $n_T \in \NN$, such that the restrictions are well-defined. This especially implies that
 \begin{equation}
  \label{fpluslimit}
  \lim_{n \to \infty} \gamma_{fn}(0) = \gamma_{f+}(0) \ .
 \end{equation}
 Moreover, we know by Theorem \ref{FiniteLengthCompactness} that $\left\{\gamma_{fn}(\cdot +l_{fn})|_{[-T,0]}\right\}_{n \in \NN}$ converges against $\gamma_{f-}|_{[-T,0]}$ in the $C^\infty$-topology for every $T \geq 0$ and sufficiently big $n_T \in \NN$, such that the restrictions are well-defined, which yields
 \begin{equation}
 \label{fminuslimit}
  \lim_{n \to \infty} \gamma_{fn}(l_n) = \gamma_{f-}(0) \ .
 \end{equation}
 By definition of $\Acal^d_{\fatY}(x_0,x_1,\dots,x_d,T)$, the following holds for every $n \in \NN$:
 \begin{equation*}
  \left(\gamma_{0n}(0),\left(\gamma_{en}(0),\gamma_{en}(l_{en})\right)_{e \in E_{int}(T)},\gamma_{1n}(0),\dots,\gamma_{dn}(0) \right) \in \Delta_T \ .
 \end{equation*}
Thus, equations (\ref{fpluslimit}) and (\ref{fminuslimit}) imply
\begin{align*}
 &\left(\gamma_{0 \infty}(0),\left(\gamma_{f-}(0)\right)_{f \in F},\left(\gamma_{en}(0),\gamma_{en}(l_{en})\right)_{e \in E_{int}(T) \setminus F},\gamma_{1\infty}(0),\dots,\gamma_{d\infty}(0), \left(\gamma_{f+}(0)\right)_{f \in F} \right)  \\
 &\in \Bigl\{\Bigl(q_0,\left(q^f_{in} \right)_{f \in F},\left(q^e_{in},q^e_{out}\right)_{e \in E_{int}(T)\setminus F},q_1,\dots,q_d,\left(q^f_{out}\right)_{f \in F} \Bigr) \in M^{1+2k(T)+d} \Bigr. \\
&\left. \qquad \qquad \qquad \qquad \qquad \qquad \left| \ \left(q_0,\left(q^e_{in},q^e_{out}\right)_{e \in E_{int}(T)},q_1,\dots,q_d \right) \in \Delta_T \right. \right\} \ .
\end{align*}
Obviously, there is a permutation of the factors $\sigma_F: M^{1 + 2k(T)+d} \to M^{1+2k(T)+d}$ which maps $\Delta_T$ diffeomorphically onto this set. The claim immediately follows. 
\end{proof}

\begin{definition} \index{boundary spaces of moduli spaces!of perturbed Morse ribbon trees}
\label{DefMYF}
 In the situation of Theorem \ref{Fconvergence}, we define:
 \begin{align*}
  &\BB_{\fatY}\left(\left(x_0,\left(y_{e-}\right)_{e \in F} \right),\left(x_1,\dots,x_d,\left(y_{e+}\right)_{e \in F} \right),T,F\right) \\ &:=\MM_{\fatY_{F}}\Bigl(\left(x_0,\left(y_{e-}\right)_{e \in F} \right),\left(x_1,\dots,x_d,\left(y_{e+}\right)_{e \in F} \right),(1+|F|,k(T)-|F|,d+|F|), \Bigr. \\
  &\qquad \qquad \qquad \qquad \qquad \qquad \qquad \qquad \qquad \qquad \qquad \qquad \left. \left( 0,+\infty \right)^{k(T)-|F|},\sigma_F(\Delta_T) \right) \ .
 \end{align*} 
\end{definition}

Before we conclude  this section by a general theorem describing limit spaces of sequences of perturbed Morse ribbon trees and a final regularity proposition, we consider the following regularity result for the spaces from Definition \ref{DefMYF}.

\begin{prop} 
\label{RegularityIntBreak}
 Let $d \geq 2$, $T \in \RTree_d$. For generic choice of $\fatY \in \XX(T)$, the following holds: For all $F \subset E_{int}(T)$ and $x_0,x_1,\dots,x_d,y_0,y_1,\dots,y_d \in \Crit f$  with 
 \begin{equation}
 \label{Assumxy}
 \mu(x_0) \geq \mu(y_0) \quad \text{and} \quad \mu(y_i) \geq \mu(x_i)\quad  \text{for every} \quad i \in \{1,2,\dots,d\} \ ,
 \end{equation}
 and for all $e \in F$ and $y_{e+},y_{e-} \in \Crit f$ with
 \begin{equation}
 \label{AssumF}
 \mu(y_{e+}) \geq \mu(y_{e-}) \ ,
 \end{equation}
  the spaces $\Acal^d_{\fatY}(x_0,x_1,\dots,x_d,T)$ and $\BB_{\fatY}\left(\left(y_0,\left(y_{e-}\right)_{e \in F} \right),\left(y_1,\dots,y_d,\left(y_{e+}\right)_{e \in F} \right),T,F\right)$ are manifolds of class $C^{n+1}$. The following inequality is true for any of these choices:
 \begin{equation*}
  \dim \BB_{\fatY}\left(\left(y_0,\left(y_{e-}\right)_{e \in F} \right),\left(y_1,\dots,y_d,\left(y_{e+}\right)_{e \in F} \right),T,F\right) \leq \dim \Acal^d_{\fatY}(x_0,x_1,\dots,x_d,T) - |F| \ .
 \end{equation*}
\end{prop}

\begin{proof}
 Let $F \subset E_{int}(T)$. It follows immediately from Theorem \ref{NonlocalTransversality} that there is a generic subset 
 \begin{equation*}
 \GG_F \subset \XX(1+|F|,k(T)-|F|,d+|F|)
 \end{equation*}
 such that for every $\fatZ \in \GG_F$, the space
 \begin{align*}
  &\MM_{\fatZ}\left(\left(y_0,\left(y_{e-}\right)_{e \in F} \right),\left(y_1,\dots,y_d,\left(y_{e+}\right)_{e \in F} \right),(1+|F|,k(T)-|F|,d+|F|), \right. \\
  &\qquad \qquad \qquad \qquad \qquad \qquad \qquad \qquad \qquad \qquad \qquad \qquad \left. \left(0,+\infty\right)^{k(T)-|F|},\sigma_F(\Delta_T) \right)
 \end{align*}
 is a manifold of class $C^{n+1}$. Furthermore, Theorem \ref{NonlocalTransversality} implies the existence of a generic subset $\GG \subset \XX(T)$ such that for every $\fatY \in \GG$, the space $\Acal^d_{\fatY}(x_0,\dots,x_d,T)$ is a manifold of class $C^{n+1}$. 
 
 To compare these perturbations, note that the map $p_F: \XX(T) \to \XX(1+|F|,k(T)-|F|,d+|F|)$, $\fatY \mapsto \fatY_F$, where $\fatY_F$ is defined as in (\ref{YF}), is continuous and surjective. Therefore, the set $p_F^{-1}(\GG_F)$ is generic in $\XX(T)$. It follows that the regularity statement holds for every 
 \begin{equation*}
 \fatY \in \bigcap_{F \subset E_{int}(T)} p_F^{-1}(\GG_F) \cap \GG \ , 
 \end{equation*}
 and since $\bigcap_{F \subset E_{int}(T)} p_F^{-1}(\GG_F) \cap \GG$ is a finite  intersection of generic sets, it is itself generic in $\XX(T)$. 
 
 It remains to show the dimension inequality. By Theorem \ref{NonlocalTransversality}, the dimension is computed as follows:
 \begin{align*}
  &\dim \BB_{\fatY}\left(\left(y_0,\left(y_{e-}\right)_{e \in F} \right),\left(y_1,\dots,y_d,\left(y_{e+}\right)_{e \in F} \right),T,F\right) \\
  &= \dim \MM_{\fatY_{F}}\Big(\left(y_0,\left(y_{e-}\right)_{e \in F} \right),\left(y_1,\dots,y_d,\left(y_{e+}\right)_{e \in F} \right),(1+|F|,k(T)-|F|,d+|F|),  \\
  &\phantom{booooooooooooooooooooooooooooooooooooooooooooooooooooooo}\left. \left(0,+\infty \right)^{k(T)-|F|},\sigma_F(\Delta_T) \right) \\
  &= \mu(y_0) - \sum_{i=1}^d \mu(y_i) + \sum_{e \in F} \left(\mu(y_{e-})-\mu(y_{e+}) \right) +(k(T)+d)n + k(T)-|F| - \codim \Delta_{T} \ . 
 \end{align*}
 Applying formula (\ref{codimDeltaT}) for the codimension of $\Delta_T$, we obtain:
\begin{align*}
 &\dim \BB_{\fatY}\left(\left(y_0,\left(y_{e-}\right)_{e \in F} \right),\left(y_1,\dots,y_d,\left(y_{e+}\right)_{e \in F} \right),T,F\right) \\
 &= \mu(y_0) - \sum_{i=1}^d \mu(y_i) + \sum_{e \in F} \left(\mu(y_{e-})-\mu(y_{e+}) \right) + k(T)-|F| \ .
\end{align*}
 Using assumptions (\ref{Assumxy}) and (\ref{AssumF}), we derive:
 \begin{align*}
   \dim \BB_{\fatY}\left(\left(y_0,\left(y_{e-}\right)_{e \in F} \right),\left(y_1,\dots,y_d,\left(y_{e+}\right)_{e \in F} \right),T,F\right) &\leq \mu(x_0) - \sum_{i=1}^d \mu(x_i) + k(T) - |F| \\
   &= \dim \Acal_{\fatY}(x_0,x_1,\dots,x_d) - |F| \ .
 \end{align*}
\end{proof}

Note that by Theorem \ref{CompactnessMorseTrees}, the different convergence phenomena described in Theorems \ref{LengthToZero}, \ref{ConvergenceHalfTraj} and \ref{Fconvergence} can occur simultaneously. 

The upcoming Theorem \ref{MostGeneralCompactness} subsumizes all convergence phenomena for sequences of perturbed Morse ribbon trees by stating a general compactness property of the moduli spaces $\Acal^d_{\fatY}(x_0,x_1,\dots,x_d,T)$.  It completes our discussion of geometric convergence phenomena for these sequences and is shown by applying the arguments used to prove the aforementioned theorems simultaneously. We omit the details.  

\begin{definition}
 Let $T \in \RTree_d$, $d \geq 2$, $Y \in \XX(T)$, $x_0,x_1,\dots,x_d \in \Crit f$ and let
 \begin{equation*}
  \left\{\left(\gamma_{0n},\left(l_{en},\gamma_{en}\right)_{e \in E_{int}(T)},\gamma_{1n},\dots,\gamma_{dn}\right) \right\}_{n \in \NN} \subset \Acal^d_{\fatY}(x_0,x_1,\dots,x_d,T)
 \end{equation*}
 be a sequence of perturbed Morse ribbon trees. 
 
 
 The sequence \emph{converges in $\Acal^d_{\fatY}(x_0,x_1,\dots,x_d,T)$} if all of its component sequences $\left\{\gamma_{in}\right\}_n$, $i \in \{0,1,\dots,d\}$, and $\left\{(l_{en},\gamma_{en})\right\}_n$, $e \in E_{int}(T)$, converge and if the product of the limits of the component sequences lies in $\Acal^d_{\fatY}(x_0,x_1,\dots,x_d,T)$. \index{perturbed Morse ribbon trees!convergence of sequences}
 
 The sequence \emph{converges geometrically} if every component sequence is either convergent or geometrically convergent and at least one of the following holds true:  \index{geometric convergence of sequences!of perturbed Morse ribbon trees}
 \begin{itemize}
  \item there exists a component sequence that converges geometrically,
  \item there exists an $e \in E_{int}(T)$ with $\lim_{n \to \infty} l_{en}=0$. 
 \end{itemize}

\end{definition}

\begin{remark}
 Since by definition of $\Acal^d_{\fatY}(x_0,x_1,\dots,x_d,T)$, the components of its elements include only finite-length trajectories $(l,\gamma)$ with $l > 0$, it follows that a sequence 
 \begin{equation*}
  \left\{\left(\gamma_{0n},\left(l_{en},\gamma_{en}\right)_{e \in E_{int}(T)},\gamma_{1n},\dots,\gamma_{dn} \right) \right\}_{n \in \NN} \subset \Acal^d_{\fatY}(x_0,x_1,\dots,x_d,T) \ , 
 \end{equation*}
 whose component sequences are all convergent, converges in $\Acal^d_{\fatY}(x_0,x_1,\dots,x_d,T)$ if and only if 
 \begin{equation*}
 \lim_{n \to \infty} l_{en} > 0 \  \quad \forall e \in E_{int}(T) \ .
 \end{equation*}
\end{remark}

\begin{theorem}
\label{MostGeneralCompactness} \index{compactness of moduli spaces!of perturbed Morse ribbon trees}
 Let $d \geq 2$, $T \in \RTree_d$, $\fatY \in \XX(T)$ and $x_0,x_1,\dots,x_d \in \Crit f$.  Let
 \begin{equation*}
  \left\{\gammaunder_n\right\}_{n \in \NN} = \left\{\left(\gamma_{0n},\left(l_{en},\gamma_{en} \right)_{e \in E_{int}(T)},\gamma_{1n},\dots,\gamma_{dn} \right) \right\}_{n \in \NN}
 \end{equation*}
 be a sequence in $\Acal^d_{\fatY}(x_0,x_1,\dots,x_d,T)$ which does not have a convergent subsequence. Then there are sets $F_1,F_2 \subset E_{int}(T) \ , \quad F_1 \cap F_2 = \emptyset$ and a subsequence $\left\{\gammaunder_{n_k}\right\}_{k \in \NN}$, such that the following holds:

 Up to a permutation of its components, $\left\{\gammaunder_{n_k}\right\}_{k \in \NN}$ converges geometrically against a product of unparametrized Morse trajectories and an element of
 \begin{equation*}
  \BB_{\fatY}\left(\left(y_0,\left(y_{e-}\right)_{e \in F_1} \right),\left(y_1,\dots,y_d,\left(y_{e+}\right)_{e \in F_1} \right),T/F_2,F_1 \right) \ , 
 \end{equation*}
  where $y_0,y_1,\dots,y_d \in \Crit f$, $y_{e-},y_{e+} \in \Crit f$ for $e \in F_1$ satisfy
 \begin{equation}
 \label{ConditionsMostGeneral}
  \mu(y_0) \leq \mu(x_0) \ , \quad \mu(y_i) \geq \mu(x_i) \ \forall i \in \{1,2,\dots,d\} \ , \quad \mu(y_{e+}) \geq \mu(y_{e-}) \ \forall e \in F_1 \ .
 \end{equation}
 Moreover, for every inequality in (\ref{ConditionsMostGeneral}) equality holds if and only if the respective critical points are identical.
\end{theorem}

We conclude this section by providing a regularity result for the moduli spaces occuring in Theorem \ref{MostGeneralCompactness}.

\begin{prop}
\label{RegularityMostGeneral}
For generic choice of $\fatY \in \XX(T)$, it holds for all $F_1,F_2 \subset E_{int}(T)$ with $F_1 \cap F_2 = \emptyset$, $y_0,y_1,\dots,y_d \in \Crit f$ and $y_{e+},y_{e-} \in \Crit f$ for $e \in F_1$ satisfying (\ref{ConditionsMostGeneral}) that the space  
 \begin{equation}
 \label{thebreakingspace}
  \BB_{\fatY}\left(\left(y_0,\left(y_{e-}\right)_{e \in F_1} \right),\left(y_1,\dots,y_d,\left(y_{e+}\right)_{e \in F_1} \right),T/F_2,F_1 \right)
 \end{equation}
 is a manifold of class $C^{n+1}$. Moreover, the following inequality holds for any of these choices:
 \begin{equation}
  \label{GeneralRegIneq}
  \begin{aligned}
  &\dim   \BB_{\fatY}\left(\left(y_0,\left(y_{e-}\right)_{e \in F_1} \right),\left(y_1,\dots,y_d,\left(y_{e+}\right)_{e \in F_1} \right),T/F_2,F_1 \right) \\
      &\qquad \qquad \leq \dim \Acal^d_{\fatY}(x_0,x_1,\dots,x_d,T) - |F_1| - |F_2| \ .
 \end{aligned}
 \end{equation}
\end{prop}

\begin{proof}
 The statement follows almost immediately from Proposition \ref{RegularityIntBreak}. 
 
 By Proposition \ref{RegularityIntBreak}, we can find a generic subset $\GG_{F_2}$ for any $F_2 \subset E_{int}(T)$, such that the space in (\ref{thebreakingspace}) is a manifold of class $C^{n+1}$ for any $F_1 \subset E(T/F_2) = E(T) \setminus F_2$ and any choice of critical points. 
 
 Consider the maps $\pi_{F_2}: \XX(T) \to \XX(T/F_2)$ from (\ref{piF}). Since these maps are surjective and continuous, the set $\pi_{F_2}^{-1}(\GG_{F_2})$ is a generic subset of $\XX(T)$ for every $F_2$. Therefore, the set 
 \begin{equation*}
  \bigcap_{F_2 \subset E_{int}(T)} \pi_{F_2}^{-1}(\GG_{F_2})
 \end{equation*}
is a generic subset of $\XX(T)$ having the desired properties.

Considering the inequality \eqref{GeneralRegIneq}, note that by Proposition \ref{RegularityIntBreak} the following holds for every $F_1,F_2 \subset E_{int}(T)$ as in the statement:
 \begin{align*}
  &\dim \BB_{\fatY}\left(\left(y_0,\left(y_{e-}\right)_{e \in F_1} \right),\left(y_1,\dots,y_d,\left(y_{e+}\right)_{e \in F_1} \right),T/F_2,F_1 \right) \\
      &\leq \dim \Acal^d_{\fatY}(x_0,x_1,\dots,x_d,T/F_2) - |F_1|   = \dim \Acal^d_{\fatY}(x_0,x_1,\dots,x_d,T) - |F_1|-|F_2| \ ,
 \end{align*} 
 where we use the dimension formula from Theorem \ref{MainTheoremMorseTrees}.
\end{proof}

\begin{definition} \index{regular!perturbation data}
 Let $d \geq 2$. A $d$-perturbation datum $\fatY = \left(\fatY_T\right)_{T\in \RTree_d}$ is called \emph{regular} if for every $T \in \RTree_d$ and every $F_1,F_2 \subset E_{int}(T)$ with $F_1 \cap F_2 = \emptyset$ the spaces 
 \begin{equation*}
 \BB_{\fatY}\left(\left(y_0,\left(y_{e-}\right)_{e \in F_1} \right),\left(y_1,\dots,y_d,\left(y_{e+}\right)_{e \in F_1} \right),T/F_2,F_1 \right) \quad \text{and} \quad \Acal^d_{\fatY}(x_0,x_1,\dots,x_d,T)
 \end{equation*}
  are manifolds of class $C^{n+1}$ for all $x_0,\dots,x_d,y_0,\dots,y_d \in \Crit f$ and $y_{e-},y_{e+} \in \Crit f$, $e \in F_1$.
\end{definition}

Proposition \ref{RegularityMostGeneral} and Theorem \ref{MainTheoremMorseTrees} together imply that for any $d \geq 2$, the set of regular $d$-perturbation data is generic in the space of all $d$-perturbation data $\displaystyle\prod_{T \in \RTree_d} \XX(T)$.

\section{\texorpdfstring{Higher order multiplications and the $A_\infty$-relations}{Higher order multiplications and the A-infinity-relations}}
\label{CompactificationsOneDimAinfty}

We continue by taking a closer look at zero- and one-dimensional moduli spaces of perturbed Morse ribbon trees. The results from Section \ref{ConvergenceBehaviour} enable us to show that zero-dimensional moduli spaces are in fact finite sets. This basic observation allows us to define homomorphisms $C^*(f)^{\otimes d} \to C^*(f)$ for every $d \geq 2$ via counting elements of these zero-dimensional moduli spaces. 

After constructing these higher order multiplications explicitly, we will study the compactifications of one-dimensional moduli spaces of perturbed Morse ribbon trees. The results from Section \ref{ConvergenceBehaviour} imply that one-dimensional moduli spaces can be compactified to one-dimensional manifolds with boundary, and we will explicitly describe their boundaries.

We will be able to show via counting elements of these boundaries that the higher order multiplications will satisfy the defining equations of an $A_\infty$-algebra, if we impose an additional consistency condition on the chosen perturbations. This consistency condition will be formulated in terms of the background perturbations that we introduced in Sections \ref{NonlocalGeneralizations} and \ref{SectionModuliSpacesPerturbed}. \bigskip

\textit{Throughout this section, we assume again that every ribbon tree is equipped with an ordering of its internal edges.} \bigskip 

We have introduced the notions of regular and universal perturbation data in Section \ref{ConvergenceBehaviour}. It will turn out that the perturbation data which allow us to draw the desired consequences are those which are \emph{both} regular \emph{and} universal. It is easy to see that there are perturbation data with either of these properties, but it is not obvious that there are indeed perturbation data with both properties. Before focussing on zero- and one-dimensional moduli spaces of perturbed Morse ribbon trees, we therefore begin this section with a nontrivial existence result. 

\begin{lemma} 
\label{ExistRegUniv}
 Let $d \geq 2$ and let $\fatX=(\fatX_T)_{T \in \RTree_d^*}$ be a universal background $d$-perturbation datum. Then there exists a regular and universal $d$-perturbation datum $\fatY=(\fatY_T)_{T \in \RTree_d}$ with $\fatY_T \in \XX(T,\fatX_T)$ for every $T \in \RTree_d$. 
\end{lemma}

\begin{proof}
One computes that since $\fatX$ is universal, the map$\pi_F:\XX(T) \to \XX(T/F)$ restricts to a map 
$$\pi_F|_{\XX(T,\fatX_T)}: \XX(T,\fatX_T) \to \XX(T/F,\fatX_{T/F}) \ , $$
for all $T \in \RTree_d$ and $F \subset E_{int}(T)$ with $k(T/F)>0$, which is a necessary condition for such a universal $d$-perturbation datum to exist.
 We will prove the claim by inductively constructing regular and universal perturbation data over the number of internal edges of the trees.
 
 Fix $d \geq 2$, let $T_0 \in \RTree_d$ be the unique $d$-leafed ribbon tree with $k(T_0)=0$ and put $\XX(T,\fatX_T):=\XX(T)$. (See picture b) in Figure \ref{RibbonTreeExamples}.) For any $x_0,x_1,\dots,x_d \in \Crit f$, the space $\Acal^d_{\fatY}(x_0,x_1,\dots,x_d,T_0)$ is then a subset of a product of spaces of semi-infinite perturbed Morse trajectories only, so the conditions defining universality are irrelevant for perturbed Morse ribbon trees modelled on $T_0$. Moreover, the geometric limits of all sequences $\Acal^d_{\fatY}(x_0,x_1,\dots,x_d,T_0)$ lie in products of unparametrized trajectory spaces and a space of perturbed Morse ribbon trees which are again modelled on $T_0$. Therefore, if $\fatY_{T_0} \in \XX(T_0)$ is regular, all possible boundary spaces of $\Acal^d_{\fatY}(x_0,x_1,\dots,x_d,T_0)$ are again smooth manifolds. \bigskip

 As induction hypothesis, we assume that we have found a regular perturbation $\fatY_T \in \XX(T,\fatX_T)$ for every $T \in \RTree_d$ \emph{with $k(T) \leq k$ for some fixed $k \in \{0,1,\dots,d-3\}$}, such that
 \begin{equation*}
  \pi_F(\fatY_T) =\fatY_{T/F}
 \end{equation*}
 for every $F \subset E_{int}(T)$ and such that every boundary space of the form
 \begin{equation*}
    \BB_{\fatY}\left(\left(y_0,\left(y_{e-}\right)_{e \in F_1} \right),\left(y_1,\dots,y_d,\left(y_{e+}\right)_{e \in F_1} \right),T/F_2,F_1 \right)
 \end{equation*}
 appearing in Theorem \ref{MostGeneralCompactness} is a manifold of class $C^{n+1}$. (These assumptions are well-defined, since $k(T) \leq k$ implies $k(T/F) \leq k$ for every $F \subset E_{int}(T)$.)
 
 Loosely speaking, this means that we assume that we have chosen a family of perturbations which is regular and universal for trees with \emph{up to $k$ internal edges}.
 
 Let now $T_1 \in \RTree_d$ with $k(T_1) = k+1$ and consider the perturbation space
 \begin{equation*}
 \tilde{\XX}:= \left\{ \fatY_1 \in \XX(T_1, \fatX_{T_1}) \ | \ \pi_F\left(\fatY_1\right) = \fatY_{T_1/F} \ \ \forall F \subset E_{int}(T_1) \right\} \ .
 \end{equation*}
 This is again well-defined since $k(T_1/F) \leq k$ for every non-empty $F \subset E_{int}(T_1)$, so $\fatY_{T_1/F}$ has already been chosen. $\tilde{\XX}$ as a closed affine subspace of $\XX(T_1)$, whose underlying linear subspace is given by 
 $$ \left\{ \fatY_1 \in \XX(T_1,0) \ | \ \pi_F\left(\fatY_1\right)=0 \ \ \forall F \subset E_{int}(T_1) \right\} \  ,  $$
 where $0$ denotes the family consisting of vanishing vector fields in $\XXb(T_1)$ and $\XX(T/F)$, respectively. Hence, $\tilde{\XX}$ is a Banach submanifold of $\XX(T)$. Remember that the proof of the regularity statement in Theorem \ref{MainTheoremMorseTrees} is essentially an application of Theorem \ref{NonlocalTransversality}. By taking a closer look at the proof of Theorem \ref{NonlocalTransversality} one checks without difficulty that the whole line of argument will still hold if we restrict to the (obviously non-empty) perturbation space $\tilde{\XX}$. More precisely, there is a generic subset $\GG\subset \tilde{\XX}$, such that for $\fatY_{T_1} \in \tilde{\XX}$ the space $\Acal^d_{\fatY_{T_1}}(x_0,x_1,\dots,x_d,T_1)$ is a manifold of class $C^{n+1}$ for all $x_0,x_1,\dots,x_d \in \Crit f$. 
 
 Moreover, the same is true in proof of Proposition \ref{RegularityIntBreak}, i.e. we can again restrict to elements of $\tilde{\XX}$ in the situation of this proposition. It follows that for generic choice of $\fatY_{T_1} \in \tilde{\XX}$, the space $\Acal^d_{\fatY_{T_1}}(x_0,x_1,\dots,x_d,T_1)$ as well as all boundary spaces of the form $\BB_{\fatY_{T_1}}\left(\dots\right)$ are manifolds of class $C^{n+1}$. 
 
 For every $T_1 \in \RTree_1$ we choose such a generic $\fatY_{T_1} \in \XX(T_1,\fatX_{T_1})$. Then the family
 \begin{equation*}
   \left(\fatY_T\right)_{T \in \RTree_d, \ k(T) \leq k+1}
 \end{equation*}
satisfies the regularity and universality condition for every tree with up to $k+1$ internal edges. Proceeding inductively, we can therefore construct a regular and universal $d$-perturbation datum. 
\end{proof}

The following theorem will allow us to define higher order multiplications on the Morse cochain complex of $f$. 

\begin{theorem} 
\label{ZerodimFinite}
 Let $\fatY = \left(\fatY_T\right)_{T \in \RTree_d}$ be a regular and universal $d$-perturbation datum, $d \geq 2$. Then for all $T \in \RTree_d$ and $x_0,x_1,\dots,x_d \in \Crit f$ with
 \begin{equation}
  \label{equzerodim}
   \mu(x_0) = \sum_{i=1}^d \mu(x_i)-k(T) \ ,
 \end{equation}
 the space $\Acal^d_{\fatY}(x_0,x_1,\dots,x_d,T)$ is a finite set.
\end{theorem}

\begin{proof}
 Since $\fatY$ is a regular $d$-perturbation datum, Theorem \ref{MainTheoremMorseTrees} implies that $\Acal^d_{\fatY}(x_0,x_1,\dots,x_d,T)$ is a zero-dimensional manifold, hence a discrete set. To show that it is finite, it therefore suffices to show that $\Acal^d_{\fatY}(x_0,x_1,\dots,x_d,T)$ is sequentially compact, which we will do by using the results from Section \ref{ConvergenceBehaviour}.
 
 Assume that there is a sequence $\left\{\left(\gamma_{0n},\left(l_{en},\gamma_{en}\right)_{e \in E_{int}(T)},\gamma_{1n},\dots,\gamma_{dn}\right) \right\}_{n \in \NN}$, which does not have a subsequence converging in $\Acal^d_{\fatY}(x_0,x_1,\dots,x_d,T)$. By Theorem \ref{MostGeneralCompactness}, there are only two cases that can occur:
 \begin{enumerate}
  \item The sequence has a subsequence for which all component sequences converge and there is $F \subset E_{int}(T)$, $F \neq \emptyset$, with $\displaystyle\lim_{n \to \infty} l_{en} = 0$ for every $e \in F$. 
	
	By Theorem \ref{LengthToZero} and the universality of $\fatY$, we can identify the limit with an element of $\Acal^d_{\fatY}(x_0,x_1,\dots,x_d,T/F)$. Since $\fatY$ is regular, $\Acal^d_{\fatY}(x_0,x_1,\dots,x_d,T/F)$ is a manifold of class $C^{n+1}$ of expected dimension	
$$	 \dim \Acal^d_{\fatY}(x_0,x_1,\dots,x_d,T/F) = \mu(x_0) - \sum_{i=1}^d \mu(x_i) + k(T) - F \stackrel{(\ref{equzerodim})}{=} -F < 0 \ ,$$
	since $F$ is non-empty. This yields $\Acal^d_{\fatY}(x_0,x_1,\dots,x_d,T/F) = \emptyset$. Therefore, since $\fatY$ is regular, the limit can not exist and the collapse of internal edges can not occur.
	
  \item The sequence has a geometrically convergent subsequence, such that at least one of the component sequences converges geometrically. 
  
    By Theorem \ref{MostGeneralCompactness}, there exist $F_1,F_2 \subset E(T)$ with $F_1 \neq \emptyset$, such that the subsequence converges geometrically against an element of 
     \begin{equation*}
  \BB_{\fatY}\left(\left(y_0,\left(y_{e-}\right)_{e \in F_1} \right),\left(y_1,\dots,y_d,\left(y_{e+}\right)_{e \in F_1} \right),T/F_2,F_1 \right)
 \end{equation*} 
 for some critical points having the same properties as in Theorem \ref{MostGeneralCompactness}. But by Proposition \ref{RegularityMostGeneral}, this space is a manifold of class $C^{n+1}$ whose expected dimension fulfils:
  \begin{align*}
  &\dim \BB_{\fatY}\left(\left(y_0,\left(y_{e-}\right)_{e \in F_1} \right),\left(y_1,\dots,y_d,\left(y_{e+}\right)_{e \in F_1} \right),T/F_2,F_1 \right) \\
  &\leq \dim \Acal^d_{\fatY}(x_0,x_1,\dots,x_d,T/F_2) - |F_1| \stackrel{(\ref{equzerodim})}{=} - |F_1|-|F_2|  <0 \ ,
  \end{align*}
since $F_1$ is nonempty. Thus, the limit can not exist and this case of geometric convergence does not occur in our situation. 
 \end{enumerate}

 By Theorem \ref{MostGeneralCompactness}, these two are the only possible cases for any such sequence. 
 
 Therefore, every sequence in $\Acal^d_{\fatY}(x_0,x_1,\dots,x_d,T)$ has a convergent subsequence. We conclude that for any choice of $x_0,x_1,\dots,x_d \in \Crit f$ satisfying (\ref{equzerodim}) the space $\Acal^d_{\fatY}(x_0,x_1,\dots,x_d,T)$ is compact and discrete, hence finite. 
\end{proof}

Theorem \ref{ZerodimFinite} is of great importance, since it allows us to count elements of  zero-dimensional moduli spaces and thereby construct maps, similar to the counting of elements of zero-\-di\-men\-sio\-nal spaces of unparametrized Morse trajectories to define the differentials of Morse (co)chain complexes, see \cite{Schwarz}. \bigskip

Before we do so, we need to consider the orientability of moduli spaces of perturbed Morse ribbon trees. The consideration of orientations on these moduli spaces will be necessary in order to construct higher order multiplications which satisfy the defining equations of an $A_\infty$-algebra. Theorem \ref{ZerodimFinite} implies that an oriented zero-dimensional moduli space consists of a finite number of points which are each equipped with an orientation, i.e. a sign.

For this purpose, we want rephrase the question of regularity of moduli spaces of Morse ribbon trees explicitly as a transverse intersection problem. Constructing orientations on the spaces involved and using oriented intersection theory, we will therefore be able to define algebraic intersection numbers. Up to a minor modification, these algebraic intersection numbers will be the coefficients of the higher order multiplications. 

For the necessary results from oriented intersection theory, we refer to the textbooks \cite[Section 5.2]{Hirsch}, \cite[Section 3.3]{GuilleminPollack} and \cite[Section 5.6]{BanyagaHurtubise}. \bigskip

Let $d \geq 2$, $T \in \RTree_d$ and $\fatY=(Y_0,(Y_e)_{e \in E_{int}(T)},Y_1,\dots,Y_d ) \in \XX(T)$. For given critical points $x_0,x_1,\dots,x_d \in \Crit f$, consider the space 
\begin{align*}
 \MM^d_{\fatY}(x_0,x_1,\dots,x_d,T):= &\left\{ \left(\gamma_0,(l_e,\gamma_e)_{e \in E_{int}(T)},\gamma_1,\dots,\gamma_d\right) \ \left| \ \gamma_0 \in W^-\left(x_0,Y_0\left((l_e)_{e \in E_{int}(T)}\right)\right)  ,   \right. \right. \\
      &\qquad(l_e,\gamma_e) \in \MM \left(Y_e\left((l_f)_{f \in E_{int}(T) \setminus \{e\}} \right) \right) \ \text{and} \ l_e > 0 \quad \forall e \in E_{int}(T)  , \\
      &\qquad \qquad   \left. \gamma_i \in W^+\left(x_i,Y_i\left((l_e)_{e \in E_{int}(T)} \right) \right) \ \forall i \in \{1,2,\dots,d\} \right\} \ .
\end{align*} 
Note that we introduced this space in the proof of Theorem \ref{NonlocalTransversality} under the name $\bWW_{\fatY}$.

As discussed in the proof of Theorem \ref{NonlocalTransversality}, $\MM^d_{\fatY}(x_0,x_1,\dots,x_d,T)$ is a smooth manifold of dimension
\begin{equation*}
 \dim \MM^d_{\fatY}(x_0,x_1,\dots,x_d,T) = \mu(x_0) - \sum_{i=1}^d \mu(x_i) + (k(T)+d)n + k(T) \ .
\end{equation*}
Consider the map 
\begin{align}
\Eunder_{\fatY}: \MM^d_{\fatY}(x_0,x_1,\dots,x_d,T) &\to M^{1+2k(T)+d} \ , \label{DefofEunderY}\\
\left(\gamma_0,(l_e,\gamma_e)_{e \in E_{int}(T)},\gamma_1,\dots,\gamma_d\right) &\mapsto \left(\gamma_0(0),(\gamma_e(0),\gamma_e(l_e))_{e \in E_{int}(T)},\gamma_1(0),\dots,\gamma_d(0)\right) \ , \notag
\end{align}
which we also introduced in the proof of Theorem \ref{NonlocalTransversality}, see especially (\ref{decisiveidentity}). By definition of regularity, the perturbation $\fatY$ is regular if and only if the map $\Eunder_{\fatY}$ intersects the $T$-diagonal $\Delta_T \subset M^{1+2k(T)+d}$ transversely, and we can write
\begin{equation*}
 \Acal^d_{\fatY}(x_0,x_1,\dots,x_d,T) = \Eunder_{\fatY}^{-1}(\Delta_T) \ .
\end{equation*}

The manifolds $\MM^d_{\fatY}(x_0,x_1,\dots,x_d,T)$ are orientable, but the definition of an orientations on them requires further considerations of orientations on the three types of moduli spaces of (unperturbed) negative gradient flow trajectories we introduced in Section \ref{SectionPerturbationsGradFlow}. We postpone these considerations to Appendix \ref{AppendixOrientMorseTraj}. The details about  technical and the interested reader can find the details about orienting the spaces $\MM^d_{\fatY}(x_0,x_1,\dots,x_d,T)$ in Appendix \ref{SectionOrientPerturbedMRT}. \bigskip

The orientation of $\MM^d_{\fatY}(x_0,x_1,\dots,x_d,T)$ especially depends on the choice of ordering of $E_{int}(T)$.  \medskip 

\emph{We will make special choices of these orderings in the appendix and assume that from here on all ribbon trees are equipped with the canonical orderings of their internal edges from Definition \ref{DefCanonicalOrdering}.} \bigskip

For any $T \in \RTree_d$, the $T$-diagonal is constructed in Definition \ref{DefDeltaT} as an intersection of the spaces $\Delta_v$, which are oriented manifolds, since $M$ is oriented. Hence, $\Delta_T$ is a transverse intersection of oriented manifolds and therefore itself oriented. A delicate issue is that the orientation of $\Delta_T$ depends on the order of the intersection of the $\Delta_v$. In Appendix \ref{SectionOrientPerturbedMRT} with an orientation that is independent of this order. \index{orientations!of $T$-diagonals} \bigskip 

Throughout the rest of this article, we will view the $T$-diagonals as oriented manifolds with the orientations from Appendix \ref{SectionOrientPerturbedMRT}.  \bigskip 

Since the space $\MM^d_{\fatY}(x_0,x_1,\dots,x_d,T)$ is oriented and the $T$-diagonal is an oriented submanifold of the oriented manifold $M^{1+2k(T)+d}$ (with the product orientation), $\Acal^d_{\fatY}(x_0,x_1,\dots,x_d,T)$ is an \emph{oriented} manifold in the case of transverse intersection. \bigskip \index{orientations!of $\Acal^d_{\fatY}(x_0,x_1,\dots,x_d)$}

\emph{Throughout the rest of this article, we will equip the spaces $\Acal^d_{\fatY}(x_0,x_1,\dots,x_d,T)$ with the orientations from Appendix \ref{SectionOrientPerturbedMRT}.} \bigskip

In the course of this section, we will further use the following fact from graph theory without giving a proof:

\begin{center}
 \centering
 \emph{A $d$-leafed ribbon tree has at most $(d-2)$ internal edges for every $d \geq 2$. Moreover, a $d$-leafed ribbon tree has $(d-2)$ internal edges if and only if it is a binary tree. }
\end{center} \bigskip

Let $x_0,x_1,\dots,x_d \in \Crit f$ with $\mu(x_0) =\sum_{i=1}^d \mu(x_i) +2-d$. With the above fact, Theorem \ref{MainTheoremMorseTrees} implies that for a regular and universal $d$-perturbation datum $\fatY$, the space $\Acal^d_{\fatY}(x_0,x_1,\dots,x_d,T)$ will be
\begin{itemize}
 \item a finite set if $T$ is a binary tree,
 \item empty if $T$ is non-binary. 
\end{itemize}
Thus, the coefficients we are going to introduce will be well-defined. 

\begin{definition}
\label{AinftyMorseCoefficients}
 Let $d \in \NN$, $d \geq 2$, $\fatY= \left(\fatY_T\right)_{T \in \RTree_d}$ be a regular and universal $d$-perturbation datum and let $x_0,x_1,\dots,x_d \in \Crit f$ with
 \begin{equation*}
  \mu(x_0) = \sum_{i=1}^d \mu(x_i)+2-d \ .
 \end{equation*}
We define
\begin{equation*}
 a^d_{\fatY}(x_0,x_1,\dots,x_d) := (-1)^{\sigma(x_0,x_1,\dots,x_d)} \sum_{T \in \RTree_d} \algint \Acal^d_{\fatX}(x_0,x_1,\dots,x_d,T) \in \ZZ \ ,	
\end{equation*} 
where
\begin{equation*}
 \sigma(x_0,x_1,\dots,x_d) := (n+1)\Bigl(\mu(x_0) + \sum_{i=1}^d (d+1-i)\mu(x_i) \Bigr) \ .
\end{equation*} 
Here, $\algint \Acal^d_{\fatX}(x_0,x_1,\dots,x_d,T)$ denotes a twisted oriented intersection number of 
\begin{equation*}
\Eunder_{\fatY_T}:\MM^d_{\fatY}(x_0,x_1,\dots,x_d,T) \to M^{1+2k(T)+d}
\end{equation*}
with $\Delta_T$ for $T \in \RTree_d$, as defined in Appendix \ref{SectionOrientPerturbedMRT}. The number $\algint \Acal^d_{\fatX}(x_0,x_1,\dots,x_d,T)$ differs from the oriented intersection number by a sign that only depends on the topological type of $T$. 
\end{definition}

\begin{remark}
 The sign twist by the parity of $\sigma(x_0,x_1,\dots,x_d)$ is chosen for technical reasons. More precisely, the sign twists will turn out to make the higher order multiplications we are going to define fulfil the defining equations of an $A_\infty$-algebra. If this correction was omitted, the defining equations would only be satisfied up to signs.
 Our choice of $\sigma(x_0,x_1,\dots,x_d)$ is in accordance with Abouzaid's works \cite{AbouzaidMorseTrop} and \cite{AbouzaidPlumbings} up to notational changes.
\end{remark}

We have completed all the necessary preparations for the definition of the higher order multiplications on Morse cochain complexes. Remember that the underlying groups of the Morse chain complex of $f$ are the free abelian groups generated by the critical points of $f$ of corresponding Morse index.

\begin{definition}
\label{AinftyMorseMaps} 
 Let $d \geq 2$ and $\fatY = \left(\fatY_T\right)_{T\in\RTree_d}$ be a regular and universal $d$-perturbation datum. Define a graded homomorphism
 \begin{equation*}
  \mu_{d,\fatY}: C^*(f)^{\otimes d} \to C^*(f)
 \end{equation*}
 of degree  $\deg \mu_{d,\fatY} = 2-d$ as the $\ZZ$-linear extension of 
 \begin{equation*}
  x_1 \otimes x_2 \otimes \dots \otimes x_d \mapsto \sum_{\stackrel{x_0 \in \Crit f}{\mu(x_0) = \sum_{i=1}^d \mu(x_i)+2-d}} a^d_{\fatY}(x_0,x_1,\dots,x_d) \cdot x_0 
 \end{equation*}
 for all $x_1,x_2,\dots,x_d \in \Crit f$. 
\end{definition}

Choosing a regular and universal $d$-perturbation datum $\fatY_d$ for every $d \geq 2$, we thus construct a family of operations $\left\{\mu_{d,\fatY_d}:C^*(f)^{\otimes d} \to C^*(f) \right\}_{d \geq 2}$. \bigskip

In the remainder of this section, we will show that the Morse cochain complex of $f$ with these multiplications (and with an appropriate choice of an operation $C^*(f) \to C^*(f)$)  will be an $A_\infty$-algebra, if we impose a further condition on the family $\left(\fatY_d \right)_{d \geq 2}$.

In particular, it will turn out that the different perturbation data can not be chosen independently of each other, but have to be related in some way. We will make these relations precise in the language of background perturbations. \bigskip

{\it Until further mention, we consider a fixed $d \geq 2$, a fixed family of critical points $x_0,x_1,\dots,x_d$ satisfying
\begin{equation}
\label{CritPointOneDim}
 \mu(x_0) = \sum_{i=1}^d \mu(x_i) +3-d 
\end{equation}
and a fixed regular and universal $d$-perturbation datum $\fatY = \left(\fatY_T \right)_{T \in \RTree_d}$.} \bigskip 

We derive from Theorem \ref{MainTheoremMorseTrees} that for $T \in \RTree_d$ the space $\Acal^d_{\fatY}(x_0,x_1,\dots,x_d,T)$ is
\begin{itemize}
 \item a one-dimensional manifold, if $T$ is a binary tree,
 \item a finite set, if $T$ has precisely $(d-3)$ internal edges,
 \item empty in all other cases, i.e. if $T$ has less than $(d-3)$ internal edges.
\end{itemize}
We investigate the one-dimensional case in greater detail using the methods of Section \ref{ConvergenceBehaviour}.

\begin{theorem} \index{compactness of moduli spaces!of Morse ribbon trees!one-dimensional} \index{tree!collapsing edges}
\label{CompactnessOneDimModuli}
 Let $\fatY$ be a regular and universal perturbation datum and $T \in \RTree_d$ be a binary tree. Assume that a sequence in $\Acal_{\fatY}^d(x_0,x_1,\dots,x_d,T)$ does not have a convergent subsequence. Then the sequence will have a geometrically convergent subsequence whose geometric limit lies (up to permutation of its components) in
 \begin{align}
  &\bigcup_{e \in E_{int}(T)} \Acal^d_{\fatY}(x_0,x_1,\dots,x_d,T/e) \cup \bigcup_{\stackrel{y_0 \in \Crit f}{\mu(y_0) = \mu(x_0)-1}}\widehat{\MM}(x_0,y_0) \times \Acal^d_{\fatY}(y_0,x_1,\dots,x_d,T) \label{onedim1} \\ 
  &\cup \bigcup_{i=1}^d \bigcup_{\stackrel{y_i \in \Crit f}{\mu(y_i) = \mu(x_i)+1}} \Acal^d_{\fatY}(x_0,x_1,\dots,x_{i-1},y_i,x_{i+1},\dots,x_d,T) \times \widehat{\MM}(y_i,x_i) \label{onedim2} \\
  &\cup \bigcup_{e \in E_{int}(T)} \bigcup_{x_e \in \Crit f} \BB_{\fatY}\left((x_0,x_{e}),(x_1,\dots,x_d,x_{e}),T,\{e\} \right) \ . \label{onedim3}
 \end{align}
\end{theorem}

\begin{figure}[h]
 \centering
 \includegraphics[scale=0.6]{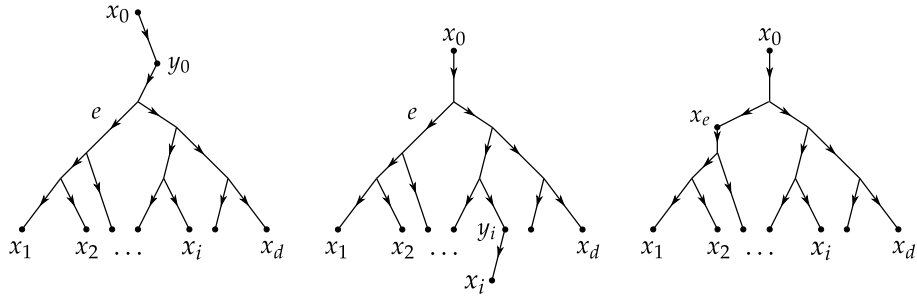}
 \caption{The limit spaces from Theorem \ref{CompactnessOneDimModuli}.}
 \label{FigureOneDimConv}
\end{figure}

Apart from the one involving a quotient tree, the different types of limit spaces in Theorem \ref{CompactnessOneDimModuli} are illustrated in Figure \ref{FigureOneDimConv}. The left-hand picture shows a space of type \eqref{onedim1}, while the  picture in the middle shows a space of type \eqref{onedim2} and the right-hand one depicts a space of type \eqref{onedim3}.

\begin{proof}[Proof of Theorem \ref{CompactnessOneDimModuli}]
 From Theorem \ref{MostGeneralCompactness} we deduce that any sequence without convergent subsequence has a subsequence which converges geometrically against an element of a product of a space of the form
 \begin{equation*}
   \BB_{\fatY}\left(\left(y_0,\left(y_{e-}\right)_{e \in F_1} \right),\left(y_1,\dots,y_d,\left(y_{e+}\right)_{e \in F_1} \right),T/F_2,F_1 \right)
 \end{equation*}
 with certain spaces of unparameterized Morse trajectories. By Proposition \ref{RegularityMostGeneral}, this space is a smooth manifold whose dimension satisfies
 \begin{align*}
  &\dim \BB_{\fatY}\left(\left(y_0,\left(y_{e-}\right)_{e \in F_1} \right),\left(y_1,\dots,y_d,\left(y_{e+}\right)_{e \in F_1} \right),T/F_2,F_1 \right) \\
  &\leq \dim \Acal^d_{\fatY}(x_0,x_1,\dots,x_d,T) - |F_1|-|F_2|  \stackrel{\eqref{CritPointOneDim}}{\leq} 1 - |F_1|-|F_2| \ .
 \end{align*}
Therefore, the space can only be nonempty if $\#(F_1 \cup F_2) \leq 1$, and we have to study the different cases that can occur. \bigskip

Consider the case of a geometrically convergent subsequence for which $F_1 \neq \emptyset$ in the above notation. This implies that $F_2 = \emptyset$ and that $F_1 = \{e\}$ for some $e \in E_{int}(T)$. One checks without difficulty (using once again Theorem \ref{NonlocalTransversality}) that the dimension of any space of the form 
\begin{equation*}
 \BB_{\fatY}\left(\left(y_0,y_{e-} \right),\left(y_1,\dots,y_d,y_{e+}\right),T,\{e\} \right)
\end{equation*}
is given by 
\begin{equation*}
 \dim \BB_{\fatY}\left(\left(y_0,y_{e-} \right),\left(y_1,\dots,y_d,y_{e+}\right),T,\{e\} \right) = \mu(y_0) - \sum_{q=1}^d \mu(y_q)+\mu(e_-) - \mu(e_+) +d-3 \ .
\end{equation*}
Inequality (\ref{GeneralRegIneq}) from Proposition \ref{RegularityMostGeneral} implies:
\begin{align}
 \mu(y_0) - \sum_{q=1}^d \mu(y_q) +\mu(y_{e-}) - \mu(y_{e+})+d-3&\leq (\mu(x_0)-\sum_{q=1}^d \mu(x_q) +d-2)-1 \notag \\
 \Leftrightarrow \mu(y_0) + \sum_{q=1}^d \mu(x_q) + \mu(y_{e-}) &\leq \mu(x_0) + \sum_{q=1}^d \mu(y_q) + \mu(y_{e+}) \ . \label{CompareDimBoundary}
\end{align}
But by (\ref{ConditionsMostGeneral}), this inequality is valid if and only if equality holds and 
\begin{equation*}
 x_0= y_0 \ , \quad x_i = y_i \ \ \forall i \in \{1,2,\dots,d\} \ , \quad y_{e+}=y_{e-} =: x_e \ .
\end{equation*}
This means that if a sequence in $\Acal^{d}_{\fatY}(x_0,x_1,\dots,x_d,T)$ has a geometrically convergent component sequence associated with an internal edge, then the subsequence will converge against an element of 
\begin{equation*}
  \bigcup_{x_e \in \Crit f} \BB_{\fatY}\left(\left(x_0,x_{e} \right),\left(x_1,\dots,x_d,x_{e}\right),T,\{e\} \right) \ .
\end{equation*}
Consider the case of a geometrically convergent subsequence with $F_1 =F_2= \emptyset$ in the above notation. Then all of the component sequences associated with internal edges have subsequences converging against a trajectories with positive interval length. Since we have assumed that the whole sequence does not have a convergent subsequence, there exist component sequences associated with external edges that have geometrically convergent subsequences.

Consequently, the subsequence converges geometrically against a product of un\-pa\-ra\-me\-trized Morse trajectories and an element of a space of the form 
\begin{equation*}
 \Acal^d_{\fatY}(y_0,y_1,\dots,y_d,T)
\end{equation*}
for some $y_0,\dots,y_d \in \Crit f$ satisfying the conditions from (\ref{ConditionsMostGeneral}). Define $I \subset \{0,1,\dots,d\}$ by demanding that $y_i\neq x_i$ if and only if $i \in I$. In particular, this means that for every $i\in I $, there is a strict inequality of the form 
\begin{equation}
\label{MorseIndIneq}
\mu(y_i) > \mu(x_i) \ \ \text{if} \ i \in \{1,2,\dots,d\} \quad \text{and} \quad \mu(y_i) < \mu(x_i) \ \ \text{if} \ i=0 \ .
\end{equation}
The dimension of $\Acal^d_{\fatY}(y_0,y_1,\dots,y_d,T)$ is then computed as follows:
\begin{align*}
 \dim  \Acal^d_{\fatY}(y_0,y_1,\dots,y_d,T) &= \mu(y_0) - \sum_{q=1}^d \mu(x_q)+d-2 \\
      &\stackrel{(\ref{MorseIndIneq})}{\leq} \mu(x_0) - \sum_{q=1}^d \mu(x_q) +d-2 - |I| =1-|I| \ ,
\end{align*}
by assumption on $x_0,x_1,\dots,x_d$. Thus, the space $\Acal^d_{\fatY}(y_0,y_1,\dots,y_d,T)$ is only non-empty if $|I| = 1$, i.e. $I = \{i\}$ for a unique $i \in \{0,1,\dots,d\}$. For $i \in \{1,2,\dots,d\}$ this means that the sequence converges geometrically against an element of 
\begin{equation*}
 \Acal^d_{\fatY}(x_0,x_1,\dots,x_{i-1},y_{i1},x_{i+1},\dots,x_d,T) \times \prod_{j=1}^{m-1} \widehat{\MM}(y_{ij},y_{i(j+1)}) \times \widehat{\MM}(y_{im},x_i)
\end{equation*}
for some $m \in \NN$ and $y_{i1},\dots,y_{im} \in \Crit f$ with $\mu(y_{i1}) > \mu(y_{i2}) > \dots > \mu(y_{im}) > \mu(x_i)$. But this obviously implies $\mu(y_{i1}) \leq \mu(x_i)+m$, and therefore 
$$\dim \Acal^d_{\fatY}(x_0,x_1,\dots,x_{i-1},y_{i1},x_{i+1},\dots,x_d,T) \leq \mu(x_0)-\sum_{q=1}^d \mu(x_q) +2-d-m = 1 - m \ ,$$
which immediately implies $m = 1$. Hence, the geometric limit we are considering lies in a space of the form
\begin{equation*}
 \Acal^d_{\fatY}(x_0,x_1,\dots,x_{i-1},y_{i},x_{i+1},\dots,x_d,T) \times \widehat{\MM}(y_i,x_i) \ ,
\end{equation*}
for some $y_i \in \Crit f$ with $\mu(y_i) = \mu(x_i)+1$. Analogously, the case $i=0$ leads to geometric convergence against elements of $\widehat{\MM}(x_0,y_0) \times \Acal^d_{\fatY}(y_0,x_1,\dots,x_d,T)$ for some $y_0 \in \Crit f$ with $\mu(y_0) = \mu(x_0)-1$. \bigskip 

It remains to discuss the case $F_1=\emptyset$ and $F_2\neq \emptyset$. Consider a sequence which has a geometrically convergent subsequence of that type and assume for the moment that all component sequences associated with external edges have convergent subsequences. Then the sequence of Morse ribbon trees has a subsequence for which all component subsequences converge. By Theorem \ref{LengthToZero}, the projection of that subsequence away from the finite-length trajectories associated with elements of $F_2$ converges against an element of
 \begin{equation*}
  \Acal^{d}_{\fatY}(x_0,x_1,\dots,x_d,T/F_2) \ .
 \end{equation*}
 This especially requires the space $\Acal^{d}_{\fatY}(x_0,x_1,\dots,x_d,T/F_2)$ to be non-empty. Consequently its expected dimension has to be non-negative. But from Theorem \ref{MainTheoremMorseTrees} and the regularity of $\fatY$ we know that
 \begin{align*}
  \dim \Acal^{d}_{\fatY}(x_0,x_1,\dots,x_d,T/F_2) &= \mu(x_0) - \sum_{i=1}^d \mu(x_i) + k(T) - |F_2| \\
      &= \mu(x_0) - \sum_{i=1}^d \mu(x_i) + d-2 - |F| \stackrel{(\ref{CritPointOneDim})}{=} 1- |F_2| \stackrel{!}{\geq} 0 \ .    
 \end{align*}
 This implies that $F$ contains at most one element. By assumption $F$ is non-empty, so it contains precisely one element, i.e. $F = \{e\}$ for some $e \in E_{int}(T)$. By similar arguments, one argues that for dimensional reasons, the case $F_2\neq \emptyset$ can not coincide with the geometric convergence of sequences associated with external edges. Hence, there are no further cases to consider. 

The claim follows by putting the different cases together.
\end{proof}

\begin{definition} \index{compactification}
 Let $A$ be a manifold and $B$ be a compact topological space with $A \subset B$. We say that \emph{$A$ can be compactified to $B$} if $B$ is homeomorphic to the topological closure of $A$. 
\end{definition}

The following theorem requires the use of so-called gluing methods in Morse theory. Analytically, this is a very delicate issue, and we refrain from giving a detailed discussion. \index{gluing in Morse theory}

Every situation occuring in the proof requiring gluing analysis is a straightforward application of the standard gluing results, either for perturbed negative or positive semi-infinite or for perturbed finite-length trajectories. 

The results for the unperturbed case are stated and proven in \cite{WehrheimMWC} and \cite{SchwarzEqui}. See also  \cite[Chapter 18]{KronheimerMrowka} for the case of finite-length trajectories. The results extend to the perturbed case, since the line of argument used to prove these theorems essentially relies on a local analysis of the moduli spaces of trajectories and locally the perturbed and the unperturbed case are described in the same way.

\begin{remark} 
\label{RemXXcond3}
For the first time in this article, we will need condition (\ref{XXcond3}) from the definition of the perturbation space $\XX_0(M)$ in the proof of the following statement. It will guarantee the differentiability of a certain map we will need for the boundary description in Theorem \ref{PreCompactficationOneDim}.  
\end{remark}

\begin{theorem}  \index{compactness of moduli spaces!of Morse ribbon trees!one-dimensional}
\label{PreCompactficationOneDim}
 Let $T \in \RTree_d$ be a binary tree. The space $\Acal^d_{\fatY}(x_0,x_1,\dots,x_d,T)$ can be compactified to a compact one-dimensional manifold $\bAcal^d_{\fatY}(x_0,x_1,\dots,x_d,T)$ of class $C^{n+1}$ whose boundary is given by 
 \begin{align*}
  &\partial \bAcal^d_{\fatY}(x_0,x_1,\dots,x_d,T) = \bigcup_{\stackrel{y_0 \in \Crit f}{\mu(y_0) = \mu(x_0)-1}}\widehat{\MM}(x_0,y_0) \times\Acal^d_{\fatY}(y_0,x_1,\dots,x_d,T) \\ 
  &\cup \bigcup_{i=1}^d \bigcup_{\stackrel{y_i \in \Crit f}{\mu(y_i) = \mu(x_i)+1}} \Acal^d_{\fatY}(x_0,x_1,\dots,x_{i-1},y_i,x_{i+1},\dots,x_d,T) \times \widehat{\MM}(y_i,x_i) \\
  &\cup \bigcup_{e \in E_{int}(T)} \bigcup_{x_e \in \Crit f} \BB_{\fatY}\left((x_0,x_{e}),(x_1,\dots,x_d,x_{e}),T,\{e\} \right) \cup \bigcup_{e \in E_{int}(T)} \Acal^{d}_{\fatY}(x_0,x_1,\dots,x_d,T/e) \ .
 \end{align*}
\end{theorem}

\begin{proof}
 The gluing results from \cite{WehrheimMWC}, \cite{SchwarzEqui} and \cite[Chapter 18]{KronheimerMrowka} cover three of the four different types of moduli spaces in the boundary. Gluing analysis for perturbed negative semi-infinite trajectories shows that every element of 
 \begin{equation*}
  \bigcup_{\mu(y_0) = \mu(x_0)-1}\widehat{\MM}(x_0,y_0) \times\Acal^d_{\fatY}(y_0,x_1,\dots,x_d,T)
 \end{equation*}
bounds a unique component of $\Acal^d_{\fatY}(x_0,x_1,\dots,x_d,T)$, while gluing analysis for perturbed positive semi-infinite trajectories shows that every element of 
\begin{equation*}
\bigcup_{i=1}^d \bigcup_{\mu(y_i) = \mu(x_i)+1} \Acal^d_{\fatY}(x_0,x_1,\dots,x_{i-1},y_i,x_{i+1},\dots,x_d,T) \times \widehat{\MM}(y_i,x_i)
\end{equation*}
bounds a unique component of $\Acal^d_{\fatY}(x_0,x_1,\dots,x_d,T)$. Furthermore, gluing analysis for perturbed finite-length trajectories shows that indeed every element of 
\begin{equation*}
 \bigcup_{e \in E_{int}(T)} \bigcup_{x_e \in \Crit f} \BB_{\fatY}\left((x_0,x_{e}),(x_1,\dots,x_d,x_{e}),T,\{e\} \right)
\end{equation*}
bounds a unique component of $\Acal^d_{\fatY}(x_0,x_1,\dots,x_d,T)$. \bigskip	

The proof that every element of $\bigcup_{e \in E_{int}(T)} \Acal^{d}_{\fatY}(x_0,x_1,\dots,x_d,T/e)$ is indeed contained in the boundary of $\Acal^d_{\fatY}(x_0,x_1,\dots,x_d,T)$ requires the use of a slightly different method, and we will perform this proof in greater detail. Its main idea and structure will be the following for a given $e \in E_{int}(T)$:
\begin{itemize}
 \item Instead of considering families of trajectories associated with \emph{all} edges of $T$ and intersecting it transversely with $\Delta_T$, we consider certain families of trajectories associated with \emph{all edges but $e$}.
 \item In reverse, we will not use the map $\Eunder_T$ to intersect the family with $\Delta_T$, but another map to be defined, which ``includes'' the intersection conditions involving $e$. (We will make this precise in the course of the proof.)
 \item While all edge lengths of elements in $\Acal^d_{\fatY}(x_0,x_1,\dots,x_d,T)$ are by definition positive, the reformulated transversality problem will naturally extend to the case of the trajectory associated with $e$ having length zero. 
 \item We show that in the case of a regular perturbation datum, the evaluation map on the families of trajectories under discussion will intersect both the interior and the boundary of this submanifold transversely. 
 \item We use the universality of the perturbation datum to identify the preimage of $\Delta_T$ under the abovementioned map with $\Acal^d_{\fatY}(x_0,x_1,\dots,x_d,T)\cup \Acal^d_{\fatY}(x_0,x_1,\dots,x_d,T/e)$ to show that this space is a one-dimensional manifold with boundary.
\end{itemize}
Let $e \in E_{int}(T)$ be fixed throughout the discussion and put $k := k(T)$. Consider the following map whose domain is a smooth manifold with boundary:
\begin{equation*}
 g: [0,+\infty) \times \left(0,+\infty \right)^{k-1} \times M \to M \ , \quad \left(l, \vec{l},x\right) \mapsto \phi^{Y_e\left(\vec{l}\right)}_{0,l}(x) \ .
\end{equation*}
Here, $\phi^{Y_e\left(\vec{l}\right)}_{0,l}$ denotes the time-$l$ map of the time-dependent vector field $Y_e\left(\vec{l}\right)$ with respect to the initial time zero. In other words, it is the time-$l$-map of the vector field $Y_e\left(\vec{l},0,\cdot\right): M \to TM$. In particular, 
\begin{equation}
\label{grestric0}
 g\left(0,\vec{l},x \right) = x
\end{equation}
for all $\vec{l} \in \left(0,+\infty \right)^{k-1}$, $x \in M$. Hence, the restriction of $g$ to $\{0\} \times \left(0,+\infty\right)^{k-1} \times M$ is independent of the $\left(0,+\infty \right)^{k-1}$-component. Condition (\ref{XXcond3}) in the definition of $\XX_0(M)$ implies that the vector field $Y_e$ and all existing derivatives tend to zero for $l \to 0$, so $g$ is of class $C^{n+1}$ at every $(0,\vec{l},x)$. The Parametrized Flow Theorem, see \cite[Theorem 21.4]{AbrahamRobbin}, implies that map $g$ is of class $C^{n+1}$ away from $l=0$, hence $g$ is of class $C^{n+1}$. 

\bigskip

Throughout the rest of the proof, we assume w.l.o.g. that $e$ is the \emph{first} edge in $E_{int}(T)$ according to the chosen ordering of $E_{int}(T)$. To prepare the following, we introduce another space whose meaning will become apparent momentarily:
\begin{align*}
 &\widetilde{\MM}^d_{\fatY}(T) := \left\{\left. (l_e, \gamma_0,(l_f,\gamma_f)_{f\neq e},\gamma_1,\dots,\gamma_d) \ \right| \ l_e \in [0,+\infty), \ \  \gamma_0 \in W^-\left(x_0,Y_0((l_f)_{f\neq e},l_e)\right) ,   \right. \\
 &\qquad \qquad \qquad (l_h,\gamma_h) \in \MM\left(Y_h\left((l_f)_{f\notin\{e,h\}},l_e \right)\right) \ \forall \ h \in E_{int}(T)\setminus \{e\} \ , \\ 
 &\left. \qquad \qquad \qquad \gamma_i \in W^+\left(x_i,Y_i\left((l_f)_{f\neq e},l_e \right)\right) \ \ \forall\  i \in \{1,2,\dots,d\}  \right\} \ .
 \end{align*}
One checks without difficulties that $\widetilde{\MM}^d_{\fatY}(T)$ is a manifold with boundary of class $C^{n+1}$ which is diffeomorphic to 
\begin{equation*}
 [0,+\infty)\times \WW^u(x_0) \times \prod_{f \in E_{int}(T)\setminus\{e\}} \MM \times \WW^s(x_1) \times \dots \times \WW^s(x_d) 
\end{equation*}
and that the following map is a submersion of class $C^{n+1}$: 
\begin{equation}
\label{MMtildeSubm}
\begin{aligned}
 \MM^d_{\fatY}(x_0,x_1,\dots,x_d,T) &\to \widetilde{\MM}^d_{\fatY}(T) \ , \\
 \left(\gamma_0,(l_f,\gamma_f)_{f\in E_{int}(T)},\gamma_1,\dots,\gamma_d\right) &\mapsto \left(l_e,\gamma_0,(l_f,\gamma_f)_{f\neq e},\gamma_1,\dots,\gamma_d\right) \ .
\end{aligned}
\end{equation}
Let $\left(\gamma_0,(l_e,\gamma_e),(l_f,\gamma_f)_{f\neq e},\gamma_1,\dots,\gamma_d \right) \in \MM^d_{\fatY}(x_0,x_1,\dots,x_d,T)$. This particularly means that 
\begin{equation*}
(l_e,\gamma_e) \in \MM\left(Y_e\left((l_f)_{f\neq e} \right) \right) \ .
\end{equation*}
In terms of the map $g$, this is equivalent to
\begin{equation}
\label{paramfloweq}
 \gamma_e(s) = \phi^{Y_e\left((l_f)_{f\neq e}\right)}_{0,s}(\gamma_e(0))=g\left(s,(l_f)_{f\neq e},\gamma_e(0) \right) \quad \forall s \in [0,l_e] \ .
\end{equation}
By definition, $\left(\gamma_0,(l_e,\gamma_e),(l_f,\gamma_f)_{f\neq e},\gamma_1,\dots,\gamma_d \right) \in \Acal^d_{\fatY}(x_0,x_1,\dots,x_d,T)$ if and only if 
\begin{align*}
 &\Eunder_T\left(\gamma_0,(l_e,\gamma_e),(l_f,\gamma_f)_{f\neq e},\gamma_1,\dots,\gamma_d \right) \in \Delta_T \\
 \Leftrightarrow \ &\left(\gamma_0(0),(\gamma_e(0),\gamma_e(l_e)),(\gamma_f(0),\gamma_f(l_f))_{f\neq e},\gamma_1(0),\dots,\gamma_d(0) \right) \in \Delta_T \\
 \stackrel{\eqref{paramfloweq}}{\Leftrightarrow} \ &\left(\gamma_0(0),\left(\gamma_e(0),g\left(l_e,(l_f)_{f\neq e}, \gamma_e(0) \right)\right),(\gamma_f(0),\gamma_f(l_f))_{f\neq e},\gamma_1(0),\dots,\gamma_d(0) \right) \in \Delta_T \ .
\end{align*}
Let $e' \in E(T)$ be the edge with $\vout(e')=\vin(e)$. We assume w.l.o.g. that $e' \in E_{int}(T)$, the remaining case $e'=e_0(T)$ is discussed along the same lines. 

Using the above equivalence, one checks that the map from \eqref{MMtildeSubm} maps the space $\Acal^d_{\fatY}(x_0,x_1,\dots,x_d,T)$ homeomorphically onto the following space:
\begin{align*}
 &\widetilde{\Acal}^d_{\fatY}(T) := \left\{ (l_e, \gamma_0,(l_f,\gamma_f)_{f\neq e},\gamma_1,\dots,\gamma_d) \in \widetilde{\MM}^d_{\fatY}(T) \ \Big| \ l_e > 0 \ , \right. \\
 &\left. \phantom{\widetilde{I}} \left(\gamma_{e'}(l_{e'}),g\left(l_e,(l_f)_{f\neq e},\gamma_{e'}(l_{e'}) \right)\right),\left(\gamma_0(0),(\gamma_f(0),\gamma_f(l_f))_{f\neq e},\gamma_1(0),\dots,\gamma_d(0) \right) \in \Delta_T \  \right\} \ .
\end{align*}
We will reformulate $\Acal^d_{\fatY}(x_0,x_1,\dots,x_d,T/e)$ in a very similar way. Without further mentioning, we will identify $E_{int}(T/e)$ with $E_{int}(T)\setminus\{e\}$. 

Let $\left(\gamma_0,(l_f,\gamma_f)_{f\neq e},\gamma_1,\dots,\gamma_d \right) \in \MM^d_{\fatY}(x_0,x_1,\dots,x_d,T/e)$. Then 
\begin{align*}
 &\left(\gamma_0,(l_f,\gamma_f)_{f\neq e},\gamma_1,\dots,\gamma_d \right) \in \Acal^d_{\fatY}(x_0,x_1,\dots,x_d,T/e) \\
 \Leftrightarrow \quad &\Eunder_{T/e}\left(\gamma_0,(l_f,\gamma_f)_{f\neq e},\gamma_1,\dots,\gamma_d \right)  \in \Delta_{T/e} \\
 \Leftrightarrow \quad &\left(\gamma_0(0),(\gamma_f(0),\gamma_f(l_f))_{f\neq e},\gamma_1(0),\dots,\gamma_d(0) \right) \in \Delta_{T/e} \ .
\end{align*}
 Comparing the definitions of $\Delta_T$ and $\Delta_{T/e}$ and letting $e'$ be given as above, this condition is equivalent to 
\begin{align*}
 &\left(\gamma_0(0), (\gamma_{e'}(0),\gamma_{e'}(0)),(\gamma_f(0),\gamma_f(l_f))_{f\neq e},\gamma_1(0),\dots,\gamma_d(0) \right) \in \Delta_{T} \\
 \stackrel{\eqref{grestric0}}{\Leftrightarrow} \ &\left(\gamma_0(0), \left(\gamma_{e'}(0),g\left(0,(l_f)_{f\neq e},\gamma_{e'}(0)\right)\right),(\gamma_f(0),\gamma_f(l_f))_{f\neq e},\gamma_1(0),\dots,\gamma_d(0) \right) \in \Delta_{T} \ .
\end{align*}
Using this condition and the universality of the perturbation datum, one checks that the space $\Acal^d_{\fatY}(x_0,x_1,\dots,x_d,T/e)$ is homeomorphic to 
\begin{align*}
 &\widetilde{\Acal}^d_{\fatY}(T/e) := \left\{ (0, \gamma_0,(l_f,\gamma_f)_{f\neq e},\gamma_1,\dots,\gamma_d) \in \widetilde{\MM}^d_{\fatY}(T) \  \right. \\
 &\quad \Bigl.\Big|  \left(\gamma_0(0),\left(\gamma_{e'}(0),g\left(0,(l_f)_{f\neq e},\gamma_{e'}(0)\right)\right),(\gamma_f(0),\gamma_f(l_f))_{f\neq e}, \gamma_1(0),\dots,\gamma_d(0) \right) \in \Delta_{T}  \Bigr\} \ .
\end{align*}
Note the strong similarity between the spaces $\widetilde{\Acal}^d_{\fatY}(T)$ and $\widetilde{\Acal}^d_{\fatY}(T/e)$. To state a transversality problem including both spaces, we consider the map 
\begin{align*}
 &\widetilde{E}: \widetilde{\MM}^d_{\fatY}(T) \to M^{1+2k+d} \ , \quad \left(l_e,\gamma_0,(l_f,\gamma_f)_{f\neq e},\gamma_1,\dots,\gamma_d\right) \mapsto \\ 
 &\qquad \left(\gamma_0(0),(\gamma_f(0),\gamma_f(l_f))_{f\neq e},\left(\gamma_{e'}(l_{e'}),g\left(l_e,(l_f)_{f\neq e},\gamma_{e'}(l_{e'}) \right)\right),\gamma_1(0),\dots,\gamma_d(0) \right) \ .
\end{align*}
Taking a close look at the definitions of $\widetilde{\Acal}^d_{\fatY}(T)$ and $\widetilde{\Acal}^d_{\fatY}(T/e)$, one deduces that 
\begin{equation*}
 \widetilde{E}^{-1}(\Delta_T) = \widetilde{\Acal}^d_{\fatY}(T) \cup \widetilde{\Acal}^d_{\fatY}(T/e) \ .
\end{equation*}
The regularity of $\fatY$ implies that both $\widetilde{E}$ and its restriction to the boundary of $\widetilde{\MM}^d_{\fatY}(T)$ are transverse to $\Delta_T$. Hence, $\widetilde{E}^{-1}(\Delta_T)$ is a one-dimensional manifold with boundary of class $C^{n+1}$. Its boundary is given by $\widetilde{\Acal}^d_{\fatY}(x_0,x_1,\dots,x_d,T/e)$ while its interior is given by $\widetilde{\Acal}^d_{\fatY}(x_0,x_1,\dots,x_d,T)$. In the light of the above identifications, this shows that every element of $\Acal^d_{\fatY}(x_0,x_1,\dots,x_d,T/e)$ bounds a unique component of $\Acal^d_{\fatY}(x_0,x_1,\dots,x_d,T)$, which we had to show. 
\end{proof}

The significance of Theorem \ref{PreCompactficationOneDim} will become apparent after a deeper investigation of the spaces $\BB_{\fatY}\left((x_0,x_{e}),(x_1,\dots,x_d,x_{e}),T,\{e\} \right)$. It will turn out that, if the $d$-perturbation datum $\fatY$ fulfills some additional condition, then for every $e \in E_{int}(T)$ and every choice of $x_e \in \Crit f$ the space $\BB_{\fatY}\left((x_0,x_{e}),(x_1,\dots,x_d,x_{e}),T,\{e\} \right)$ will be diffeomorphic to a product of two moduli spaces of perturbed Morse ribbon trees. 

This will be the key observation for proving that the $A_\infty$-equations are satisfied. \bigskip

We will use the following graph-theoretic lemma without proof:

\begin{lemma} 
\label{TheLayOfe}
 Let $T \in \RTree_d$, $e \in E_{int}(T)$. There exist $i \in \{1,2,\dots,d\}$ and $l \in \{0,1,\dots,d-i\}$ such that
\begin{equation*}
 e \in  \bigcap_{j=1}^{i-1} E(P_j(T))^c \cap  \bigcap_{k=i}^{i+l} E(P_k(T)) \cap  \bigcap_{m=i+l+1}^{d} E(P_m(T))^c 
\end{equation*}
where $\cdot^c$ denotes the complement in $E(T)$ and where $P_j(T)$ denotes the path from the root of $T$ to its $j$-th leaf for every $j \in \{1,2,\dots,d\}$. 
\end{lemma}

\begin{figure}[h]
 \centering
 \includegraphics[scale=0.8]{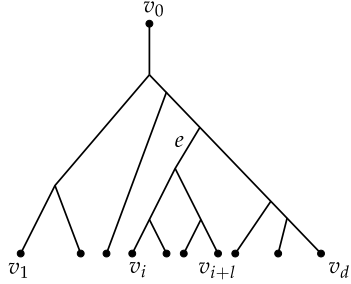}
 \caption{$e$ is of type $(i,l)$ in $T \in \RTree_d$.}
 \label{FigTypeil}
\end{figure}

See Figure \ref{FigTypeil} for an illustration of the situation in Lemma \ref{TheLayOfe}. 

\begin{definition} 
\label{Typeil} 
 Let $T \in \RTree_d$, $e \in E_{int}(T)$. If $i$ and $l$ are the numbers associated to $e$ by Lemma \ref{TheLayOfe}, then we say that \emph{$e$ is of type $(i,l)$ in $T$} and put $type(e):=(i,l)$.
\end{definition}

Consider a fixed $e \in E_{int}(T)$. We can construct a tree $T^e$ out of $T$ by removing $e$ from the edges of $T$, inserting a new vertex $v_e$ and two new edges $f_1$, $f_2$, such that:
\begin{align*}
 &V(T^e) = V(T) \cup \{v_e\} \ , \qquad E(T^e) = (E(T) \setminus \{e\}) \cup \{f_1,f_2\} \ , \\
 &\vout(f_1) = v_e \ , \quad \vin(f_1) = \vout(f) \quad \forall f \in E_{int}(T) \ \text{ with } \ \vout(f) = \vin(e) \ , \\
 &\vin(f_2) = v_e \ , \quad \vout(f_2) = \vin(f) \quad \forall f \in E_{int}(T) \ \text{ with } \ \vin(f) = \vout(e) \ . 
\end{align*}

\begin{figure}[h]
 \centering
 \includegraphics[scale=0.7]{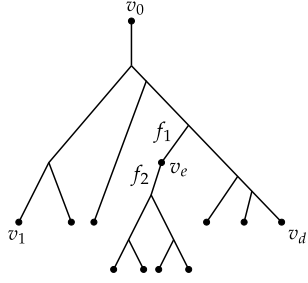}
 \caption{The tree $T^e$ for $T$ and $e$ from Figure \ref{FigTypeil}.}
 \label{FigureTe}
\end{figure}

For the tree $T$ and the edge $e$ from Figure \ref{FigTypeil}, the tree $T^e$ is depicted in Figure \ref{FigureTe}. 

Note that $T^e$ is not a ribbon tree, since the new vertex $v_e$ is by definition binary. Nevertheless, we can always decompose the tree $T^e$ into two ribbon trees as follows. 

\begin{definition}
\label{DefSplitting} 
 Let $e \in E_{int}(T)$ be of type $(i,l)$ for a convenient choice of $i \in \{1,\dots,d\}$ and $l \in \{1,\dots,d-i\}$. We define two subtrees $T^e_1,T^e_2 \subset T^e$ with
 \begin{equation*}
  \left(T^e_1,T^e_2 \right) \in \RTree_{d-l} \times \RTree_{l+1}
 \end{equation*}
by demanding that
\begin{itemize}
\item $v_0\left(T^e_1\right) = v_0(T)$, $v_0 \left(T^e_2\right) = v_e$, 
\item the leaves of $T^e_1$ are given by $\{v_1(T),\dots,v_{i-1}(T),v_e,v_{i+l+1}(T),\dots,v_d(T) \}$,
\item the leaves of $T^e_2$ are given by $\{v_i(T),v_{i+1}(T),\dots,v_{i+l}(T) \}$.
\end{itemize}

If we demand all of these conditions, the pair $(T^e_1,T^e_2)$ will be well-defined. We call $(T^e_1,T^e_2)$ \emph{the splitting of $T$ along $e$.} \index{splitting of a tree along an internal edge}
\end{definition}
 
 \begin{figure}[h]
 \centering
 \includegraphics[scale=0.7]{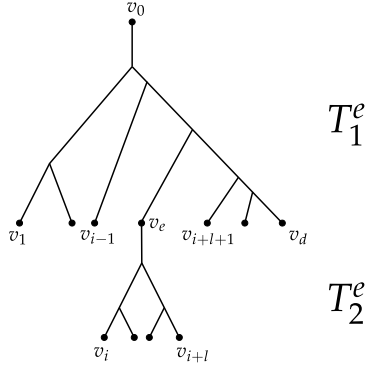}
 \caption{The splitting of $T$ into $(T^e_1,T^e_2)$ for the example from Figure \ref{FigTypeil}.}
 \label{FigTe12}
\end{figure}
 
 The situation is depicted in Figure \ref{FigTe12}. Note that especially:
\begin{equation*}
 E_{int}(T^e_1) \cup E_{int}(T^e_2) = E_{int}(T) \setminus \{e\} \ .
\end{equation*}
We state an important property of the respective diagonals.

\begin{lemma} \index{tree diagonal}
\label{PermutProductDiag}
 Let $T \in \RTree_d$ and $e \in E_{int}(T)$. There is a diffeomorphism 
 \begin{equation*}
  \Delta_T \stackrel{\cong}{\to} \Delta_{T^e_1} \times \Delta_{T^e_2} \ .
 \end{equation*}
\end{lemma}
\begin{proof}
 Consider the following permutation:
 \begin{align*}
  &M^{1+2k(T)+d} \stackrel{\cong}{\to} M^{1+2k(T^e_1)+d-l}\times M^{1+2k(T^e_2)+l+1} \ , \qquad \Bigl(q_0,\left(q^f_{in},q^f_{out}\right)_{f \in E_{int}(T)},q_1,\dots,q_d\Bigr) \mapsto \\ 
  &\Bigl( q_0,\left(q^f_{in},q^f_{out}\right)_{f \in E_{int}(T^e_1)},q_1,\dots q_{i-1},q^e_{in},q_{i+l+1},\dots,q_d,  q^e_{out},\left(q^f_{in},q^f_{out}\right)_{f \in E_{int}(T^e_2)},q_i,\dots,q_{i+l} \Bigr) \ ,
 \end{align*}
 where the components associated with internal edges are reordered according to the chosen orderings of $E_{int}(T^e_1)$ and $E_{int}(T^e_2)$. Comparing the definitions of the respective diagonals, one checks that this permutation induces the desired diffeomorphism.
\end{proof}

Let $T \in \RTree_d$ be a binary tree and $e \in E_{int}(T)$. Before we continue, we are going to make a few remarks on the space $\BB_{\fatY}\left((x_0,x_{e}),(x_1,\dots,x_d,x_{e}),T,\{e\} \right)$. In Section \ref{SectionModuliSpacesPerturbed} we have written background perturbations $\fatX \in \XXb(T)$ as
\begin{align*}
 &\fatX = \left(\fatX^-, \left(\fatX^0_e\right)_{e \in E_{int}(T)}, \fatX^+_1,\dots,\fatX^+_d \right) \ , \quad \text{where} \\
&\fatX^- = \left(X^-_e\right)_{e \in E_{int}(T)} \ , \quad \fatX^0_e = \left(\left(X^0_{ef}\right)_{f \in E_{int}(T)\setminus\{e\}},X_{e+},X_{e-} \right) \ \forall e \in E_{int}(T) \ , \\ 
&\fatX^+_i = \left(X^+_{ei} \right)_{e \in E_{int}(T)} \ \forall i \in \{1,2,\dots,d\} \ .
\end{align*}
In this section, we prefer to reorder the components of $\fatX$ and to write:
\begin{equation}
\label{BackPertReordered}
 \fatX = \left(X^-_e, (X^0_{ef})_{f \in E_{int}(T)\setminus \{e\}},X_{e+},X_{e-},X^+_{e1},\dots,X^+_{ed} \right)_{e \in E_{int}(T)} =: \left(\fatX_e\right)_{e \in E_{int}(T)}.
\end{equation}
The reason for this change of notation is the following proposition. \index{background perturbations!for ribbon trees}

\begin{prop} 
\label{BackPertConv}
 Let $\fatX \in \XXb(T)$. If $\fatY \in \XX(T,\fatX)$, then
 \begin{equation*}
  \fatY_{\{e\}} = \fatX_e
 \end{equation*}
for every $e \in E_{int}(T)$ in the notation of (\ref{YF}).
\end{prop}
\begin{proof}
 This is nothing but a slight change of notation and follows immediately from writing down the definitions.
\end{proof}

By definition of $\BB_{\fatY}\left((x_0,x_{e}),(x_1,\dots,x_d,x_{e}),T,\{e\} \right)$, we know that for every element
\begin{equation*}
 \left(\gamma_0,\left(l_f,\gamma_f\right)_{f \in E_{int}(T) \setminus e},\gamma_+,\gamma_-,\gamma_1,\dots,\gamma_d\right) \in \BB_{\fatY}\left((x_0,x_{e}),(x_1,\dots,x_d,x_{e}),T,\{e\} \right)
\end{equation*}
the following holds:
\begin{equation}
\label{BoundaryinDeltaT}
 \left(\gamma_0(0),\left(\gamma_f(0),\gamma_f(l_f)\right)_{f \in E_{int}(T) \setminus \{e\}},\gamma_+(0),\gamma_-(0),\gamma_1(0),\dots,\gamma_d(0)\right) \in \sigma_{\{e\}}(\Delta_T) \ ,
\end{equation}
where $\sigma_{\{e\}}$ is given by the map $\sigma_F$ from Theorem \ref{Fconvergence} with $F=\{e\}$. By Lemma \ref{PermutProductDiag} and the permutation constructed in its proof, (\ref{BoundaryinDeltaT}) is equivalent to
\begin{align}
 &\left(\gamma_0(0),\left(\gamma_f(0),\gamma_f(l_f)\right)_{f \in E_{int}(T^e_1)},\gamma_1(0),\dots,\gamma_{i-1}(0),\gamma_+(0),\dots,\gamma_d(0)\right) \in \Delta_{T^e_1} \ ,  \label{refactorDeltaT1} \\
 &\left(\gamma_-(0),\left(\gamma_f(0),\gamma_f(l_f)\right)_{f \in E_{int}(T^e_2)},\gamma_i(0),\dots,\gamma_{i+l}(0) \right) \in \Delta_{T^e_2} \ . \label{refactorDeltaT2}
\end{align}
Moreover, if $\fatY \in \XX(T,\fatX)$ for some $\fatX \in \XXb(T)$ as in (\ref{BackPertReordered}), Proposition \ref{BackPertConv} implies
\begin{align}
 &\gamma_0 \in W^-\left(x_0,X_{e0}\left((l_f)_{f \in E_{int}(T) \setminus \{e\}} \right) \right) \ , \label{BackPertConseq1} \\ 
 &(l_f,\gamma_f) \in \MM\left(X_{ef}\left((l_g)_{g \in E_{int}(T) \setminus \{e,f\}} \right) \right) \quad \forall f \in E_{int}(T)\setminus \{e\} , \\
 &\gamma_i \in W^+\left(x_i,X_{ei}\left((l_f)_{f \in E_{int}(T) \setminus\{e\}} \right) \right) \quad \forall i \in \{1,2,\dots,d\} \ , \\
 &\gamma_+ \in W^+\left(x_e,X_{e+}\left((l_f)_{f \in E_{int}(T) \setminus\{e\}} \right) \right), \ \ \gamma_- \in W^-\left(x_e,X_{e-}\left((l_f)_{f \in E_{int}(T) \setminus\{e\}} \right) \right) . \label{BackPertConseq5}
\end{align}
Comparing these properties with (\ref{refactorDeltaT1}) and (\ref{refactorDeltaT2}), it seems plausible to hope that for a convenient choice of $\fatX$, we can reorder the components of 
\begin{equation*}
\gammaunder := \left(\gamma_0,\left(l_f,\gamma_f\right)_{f \in E_{int}(T) \setminus \{e\}},\gamma_+,\gamma_-,\gamma_1,\dots,\gamma_d\right)
\end{equation*}
to obtain an element of the product space
\begin{equation}
\label{ProductTreeSpaces}
 \Acal^{d-l}_{\fatY_1}(x_0,x_1,\dots,x_{i-1},x_e,x_{i+l+1},\dots,x_d,T^e_1) \times \Acal^{l+1}_{\fatY_2}(x_i,\dots,x_{i+l},T^e_2)
\end{equation}
for certain perturbation data $\fatY_1 \in \XX(T^e_1)$ and $\fatY_2 \in \XX(T^e_2)$. This does not hold true for arbitrary choices of $\fatX$, since the perturbing vector field associated with a component of $\gammaunder$ depends on the parameters $l_f$ for \emph{all} $f \in E_{int}(T)\setminus\{e\}$. 

A necessary and sufficient condition for an identification of $\BB_{\fatY}\left((x_0,x_{e}),(x_1,\dots,x_d,x_{e}),T,\{e\} \right)$ with a product of moduli spaces as in (\ref{ProductTreeSpaces}) is that the parametrized vector fields in $\fatX_e$ associated with an edge of $T^e_i$, $i \in \{1,2\}$, depend \emph{only on those parameters $l_f$ with $f \in E(T^e_i)$}. \bigskip

In the following, we will introduce a method of constructing background perturbations. Given a $k$-perturbation datum for every $k < d$, we will inductively construct a background perturbation having the desired property for every $d$-leafed tree. \bigskip

Assume we have chosen a family $\fatY = \left(\fatY^k\right)_{2 \leq k < d}$ of perturbations, where $\fatY^k$ is a $k$-perturbation datum for every $2 \leq k < d$. We can write $\fatY$ as 
\begin{equation*}
 \fatY= \left(\fatY_T\right)_{T \in \bigcup_{k=2}^{d-1} \RTree_k} \ . 
\end{equation*}
For every $T \in \RTree_d$ we will construct an $\fatX^{\fatY} \in \XXb(T)$ out of the family $\fatY$. Fix $T \in \RTree_d$ and let $e \in E_{int}(T)$. We have seen above that
\begin{equation*}
 k(T^e_1) = d-l < d \ , \quad k(T^e_2) = l+1 < d \ .
\end{equation*}
Therefore, the family $\fatY$ especially contains perturbation data for the trees $T^e_1$ and $T^e_2$ which we denote by
$$ \fatY_{T^e_1} = \left(Y_{10},\left(Y_{1f} \right)_{f \in E_{int}(T^e_1)}, Y_{11},\dots,Y_{1(d-l)}\right) , \quad
 \fatY_{T^e_2} = \left(Y_{20},\left(Y_{2f} \right)_{f \in E_{int}(T^e_2)}, Y_{21},\dots,Y_{2(l+1)}\right) .$$
If $e$ is of type $(i,l)$ we define $\fatX^{\fatY} = \left(\fatX^{\fatY}_e\right)_{e \in E_{int}(T)} \in \XXb(T)$ by 
\begin{equation*}
 \fatX^{\fatY}_e := \left(Y_{10},\left(Z_f\right)_{f \in E_{int}(T) \setminus \{e\}},(Y_{1i},Y_{20}),  Y_{11},\dots,Y_{1(i-1)}, Y_{21},\dots,Y_{2(l+1)},Y_{1(i+1)},\dots,Y_{1(d-l)} \right) ,
\end{equation*}
where 
$$ Z_f := \begin{cases}
             Y_{1f} & \text{if } \ f \in E_{int}\left(T^e_1\right) \ , \\
             Y_{2f} & \text{if } \ f \in E_{int}\left(T^e_2\right) \ .
        \end{cases} $$
Here, we identify every $Y_{1j} \in \XX_\pm(M,k(T^e_1))$, $j \in \{0,1,\dots,d-l\}$, with the corresponding element of $\XX_\pm(M,k(T))$ which is independent of all parameters associated with internal edges which do not belong to $E_{int}(T^e_1)$. Analogously, we identify every 
$Y_{2i} \in \XX_\pm(M,k(T^e_2))$, $i \in \{0,1,\dots,l+1\}$, with the corresponding element of $\XX_\pm(M,k(T))$ which is independent of all parameters associated with internal edges which do not belong to $E_{int}(T^e_2)$. Similar identifications hold for the perturbations $Y_{1f}$ and $Y_{2g}$ for all $f \in E_{int}(T^e_1)$ and $g \in E_{int}(T^e_2)$. 

\begin{definition}
\begin{enumerate}[a)]
 \item A family $\fatY = \left( \fatY_d\right)_{d \geq 2}$,
 where $\fatY_d$ is a $d$-perturbation datum for every $d\geq 2$, is called a \emph{perturbation datum}. 
 
 $\fatY$ will be called \emph{regular} if $\fatY_d$ is regular for every $d \geq 2$. $\fatY$ will be called \emph{universal} if $\fatY_d$ is universal for every $d \geq 2$. 
 \item A perturbation datum $\fatY = \left( \fatY_d\right)_{d \geq 2}$ will be called \emph{consistent} if for all $d \geq 2$ and $T \in \RTree_d$ it holds that $\fatY_T \in \XX\left(T,\fatX^{\fatY_{< d}}\right)$, where $\fatY_{<d} := (\fatY_k)_{2\leq k<d}$ and $\fatX^{\fatY_{< d}} \in \XXb(T)$ is the background perturbation associated with $\fatY_{<d}$ as described above.
 \item A perturbation datum is called \emph{admissible} if it is regular, universal and consistent. \index{perturbation datum!admissible}
\end{enumerate}
\end{definition}

If $\fatY = \left(\fatY_d\right)_{d \geq 2}$ is a regular and universal perturbation datum we will write 
\begin{equation*}
 a^d_{\fatY}(x_0,x_1,\dots,x_d) := a^d_{\fatY_d}(x_0,x_1,\dots,x_d)
\end{equation*}
for the numbers from Definition \ref{AinftyMorseCoefficients} and $\mu_{d,\fatY} := \mu_{d,\fatY_d}$ for the maps from Definition \ref{AinftyMorseMaps}. \\

Our introduction of admissible perturbation data is justified by the following lemma. 

\begin{lemma}
\label{AdmissExistence} \index{perturbation datum!admissible}
 Admissible perturbation data exist.
\end{lemma}

\begin{proof}
 We will argue inductively over the number of leaves of the trees. For $d = 2$, there are no consistency conditions, so we can use an arbitrary regular and universal $2$-perturbation datum $\fatY_2$ to build an admissible perturbation datum. Assume now that for some $d \geq 2$ we have found a family $(\fatY_k)_{2 \leq k < d}$, where $\fatY_k$ is a $k$-perturbation datum for every $k$, such that the condition defining consistency is satisfied for every ribbon tree with at most $d-1$ leaves. 
 
 Let $\fatX^\fatY=(\fatX^{\fatY}_T)_{T \in \RTree_d}$ be the background $d$-perturbation datum constructed out of the family $(\fatY_k)_{2 \leq k < d}$ as described above. Since $\fatY_k$ is universal for every $k<d$, it follows that $\fatX^{\fatY}$ is universal. Thus, by Lemma \ref{ExistRegUniv} there exists a regular and universal $d$-perturbation datum $\fatY_d=(\fatY_T)_{T \in \RTree_d}$, such that $\fatY_T \in \XX(T,\fatX^{\fatY}_T)$ for every $T \in \RTree_d$. By definition of $\fatX^{\fatY}$, this $d$ perturbation datum is consistent, hence admissible. 
\end{proof}

\begin{theorem}
\label{BoundaryDescriptionT} 
 Let $\fatY$ be an admissible perturbation datum, $d \in \NN$ with $d \geq 2$, $T \in \RTree_d$ and $x_0,x_1,\dots,x_d \in \Crit f$. Then $\Acal^d_{\fatY}(x_0,x_1,\dots,x_d,T)$ can be compactified to a compact one-dimensional manifold $\bAcal^d_{\fatY}(x_0,x_1,\dots,x_d,T)$, whose boundary is given by
  \begin{align*}
  &\partial \bAcal^d_{\fatY}(x_0,x_1,\dots,x_d,T) =  \bigcup_{e \in E_{int}(T)} \Acal^d_{\fatY}(x_0,x_1,\dots,x_d,T/e) \\
  &\cup \bigcup_{\stackrel{y_0 \in \Crit f}{\mu(y_0) = \mu(x_0)-1}}\widehat{\MM}(x_0,y_0) \times\Acal^d_{\fatY}(y_0,x_1,\dots,x_d,T) \\ 
  &\cup \bigcup_{i=1}^d \bigcup_{\stackrel{y_i \in \Crit f}{\mu(y_i) = \mu(x_i)+1}} \Acal^d_{\fatY}(x_0,x_1,\dots,x_{i-1},y_i,x_{i+1},\dots,x_d,T) \times \widehat{\MM}(y_i,x_i) \\
  &\cup \bigcup_{e \in E_{int}(T)} \bigcup_{x_e \in \Crit f} \Acal^{d-l}_{\fatY}(x_0,x_1,\dots,x_{i-1},x_e,x_{i+l+1},\dots,x_d,T^e_1) \times \Acal^{l+1}_{\fatY}(x_e,x_i,\dots,x_{i+l},T^e_2) \ ,
 \end{align*}
 where in the last row the respective values of $i$ and $l$ are defined by $e$ being of type $(i,l)$.
\end{theorem}

\begin{remark}
In the last row of the boundary description in Theorem \ref{BoundaryDescriptionT}, the respective space is only non-empty if 
\begin{equation*}
 \mu(x_e) = \sum_{q=i}^{i+l} \mu(x_q)+1-l \ ,
\end{equation*}
since otherwise one of the factors, and thus the cartesian product, would be the empty set for dimensional reasons. 
\end{remark}

\begin{proof}[Proof of Theorem \ref{BoundaryDescriptionT}]
 In the light of Theorem \ref{PreCompactficationOneDim}, it suffices to show that the admissibility condition implies for all $e \in E_{int}(T)$ and $x_e \in \Crit f$ of the right index that
 \begin{align*}
  &\BB_{\fatY}\left((x_0,x_{e}),(x_1,\dots,x_d,x_{e}),T,\{e\} \right) \\
  &\qquad \qquad \stackrel{!}{=} \Acal^{d-l}_{\fatY}\left(x_0,\dots,x_{i-1},x_e,x_{i+l+1},\dots,x_d,T^e_1\right) \times \Acal^{l+1}_{\fatY}\left(x_e,x_i,\dots,x_{i+l},T^e_2\right) \ .
 \end{align*}
 As previously discussed, this is our original motivation for the notion of admissibile perturbation data. Since (\ref{refactorDeltaT1}) and (\ref{refactorDeltaT2}) hold, we only need to show that
 \begin{align*}
  &\left(\gamma_0,\left(l_f,\gamma_f \right)_{f \in E_{int}(T^e_1)},\gamma_1,\dots,\gamma_{i-1},\gamma_+,\gamma_{i+l+1},\dots,\gamma_d\right) \\
  &\qquad \qquad \qquad \qquad \qquad \qquad \in \MM^d_{\fatY}(x_0,x_1,\dots,x_{i-1},x_e,x_{i+l+1},\dots,x_d,T^e_1)  \ ,\\
 &\left(\gamma_-,\left(l_f,\gamma_f\right)_{f \in E_{int}(T^e_2)},\gamma_i,\dots,\gamma_{i+l} \right) \in \MM^d_{\fatY}(x_e,x_i,\dots,x_{i+l},T^e_2) \ .
 \end{align*}
 This condition is equivalent to the following componentwise description:
 \begin{align*}
 &\gamma_0 \in W^-\left(x_0,Y_{10}\left((l_f)_{f \in E_{int}(T) \setminus \{e\}} \right) \right) \ , \\ 
 &(l_f,\gamma_f) \in \MM\left(Y_{if}\left((l_g)_{g \in E_{int}(T^e_i)\setminus f } \right) \right) \qquad  \forall f \in E_{int}(T^e_i) , \ i \in \{1,2\} \ , \\
 &\gamma_j \in W^+\left(x_j,Y_{1j}\left((l_f)_{f \in E_{int}(T^e_1)} \right) \right) \qquad \forall j \in \{1,2,\dots,i-1\} \ , \\
 &\gamma_j \in W^+\left(x_j,Y_{2(j-i+1)}\left((l_f)_{f \in E_{int}(T^e_2)} \right) \right) \qquad  \forall j \in \{i,i+1,\dots,i+l\} \\
 &\gamma_j \in W^+\left(x_j,Y_{1(j-l)}\left((l_f)_{f \in E_{int}(T^e_1)} \right) \right)  \qquad \forall j \in \{i+l+1,i+l+2,\dots,d\} , \\
 &\gamma_+ \in W^+\left(x_e,Y_{1i}\left((l_f)_{f \in E_{int}(T^e_1)} \right) \right) , \qquad \gamma_- \in W^-\left(x_e,Y_{20}\left((l_f)_{f \in E_{int}(T^e_2)} \right) \right) \ , 
\end{align*}
 where we put
 \begin{align*}
  &\fatY_{T^e_1} =: \left(Y_{10},\left(Y_{1f}\right)_{f \in E_{int}(T^e_1)},Y_{11},\dots,Y_{1(d-l)}\right) \ , \\
  &\fatY_{T^e_2} =: \left(Y_{20},\left(Y_{2f}\right)_{f \in E_{int}(T^e_2)},Y_{21},\dots,Y_{2(l+1)}\right) \ .
 \end{align*}
 But these conditions are a direct consequence of the admissibility of the perturbation datum $\fatY$. They are obtained by inserting the definition of the background perturbations $\fatX^{\fatY}$ into conditions (\ref{BackPertConseq1}) to (\ref{BackPertConseq5}).
 \end{proof}

Our final aim is to consider the  union of all moduli spaces of perturbed Morse ribbon trees modelled on all possible choices of ribbon trees at once. We will construct a certain quotient space of the union of their compactifications which will finally enable us to prove that the higher order multiplications constructed in this section satisfy the defining equations of an $A_\infty$-algebra. 

The starting point is the following lemma from graph theory which can be proven by elementary methods. We omit the details. 

\begin{lemma}
\label{TwoQuotientTrees} 
 Let $d \geq 2$ and $T \in \RTree_d$ with $k(T) = d-3$. Then there are precisely two binary trees $T_1,T_2 \in \RTree_d$ as well as unique edges $e_i \in E_{int}(T_i)$ for $i \in \{1,2\}$ such that
 \begin{equation*}
  T_1/e_1 = T = T_2/e_2 \ .
 \end{equation*}
\end{lemma}

\begin{proof}[Sketch of proof]
 For $d=3$, the claim is obvious, see Figure \ref{FigTwoQuotients}. For arbitrary $d \geq 2$, one checks that a tree $T \in \RTree_d$ with $k(T)=d-3$ contains a unique $4$-valent internal vertex while every other internal vertex is trivalent. This $4$-valent vertex $v$ can be ``resolved`` in precisely two different ways. 
 
 If we add another edge $e_{new}$ to the tree whose incoming vertex is identified with $v$, then there are precisely two ways to connect the other edges connected to $v$ with $e_{new}$. Let $e_0$ be the unique edge with $\vout(e_0)=v$ and let $e_1,e_2,e_3 \in E(T)$ denote the edges with $\vin(e_i)=v$ for $i \in \{1,2,3\}$, ordered according to the ordering induced by the given one of the leaves of $T$. The two different ribbon trees including $e_{new}$ are obtained as follows:
 \begin{itemize}
  \item either we define a tree by putting
  \begin{equation*}
  \vin(e_1)=\vin(e_2)=\vout(e_{new}) \quad \text{and} \quad \vin(e_3)=\vin(e_{new})=v \ ,  
  \end{equation*}
  \item or we construct a tree by putting
  \begin{equation*}
  \vin(e_1)=\vin(e_{new})=v \quad \text{and} \quad \vin(e_2)=\vin(e_3)=\vout(e_{new}) \ .
  \end{equation*}
 \end{itemize}
The two resulting trees are obviously different from each other and one checks that these two ways are indeed the only ways of constructing binary trees which reduce to $T$ after collapsing a single edge.
\end{proof}

\begin{figure}[h]
 \centering
 \includegraphics[scale=0.9]{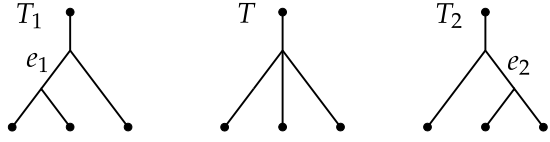}
 \caption{An example of the situation of Lemma \ref{TwoQuotientTrees} for $d=3$.}
 \label{FigTwoQuotients}
\end{figure}

See Figure \ref{FigTwoQuotients} for an illustration of Lemma \ref{TwoQuotientTrees}. Theorem \ref{PreCompactficationOneDim} and Lemma \ref{TwoQuotientTrees} together yield:

\begin{cor}
\label{TwiceAsBoundary}
 Let $\fatY$ be a regular and universal perturbation datum. Let $d \geq 2$ and $T \in \RTree_d$ with $k(T) = d-3$. Then there are precisely two binary trees $T_1, T_2 \in \RTree_d$ with 
 \begin{equation*}
  \Acal^d_{\fatY}(x_0,x_1,\dots,x_d,T) \subset \partial \bAcal^d_{\fatY}(x_0,x_1,\dots,x_d,T_i) \qquad \forall i \in \{1,2\} \ . 
 \end{equation*}
\end{cor}

Let $\BinTree_d \subset \RTree_d$ denote the set of all binary trees with $d$ leaves. For our choice of $\fatY$, $x_0,x_1,\dots,x_d \in \Crit f$ and $T \in \BinTree_d$, the space $\bAcal^d_{\fatY}(x_0,x_1,\dots,x_d,T)$ is a compact one-dimensional manifold with boundary, hence the space $\bigsqcup_{T \in \BinTree_d} \bAcal^d_{\fatY}(x_0,x_1,\dots,x_d,T)$ is a compact one-dimensional manifold with boundary as well. By Corollary \ref{TwiceAsBoundary}, every element of 
\begin{equation*}
 \bigsqcup_{\stackrel{T \in \RTree_d}{k(T)=d-3}} \Acal^d_{\fatY}(x_0,x_1,\dots,x_d,T)
\end{equation*}
appears precisely twice as a boundary element of $\bigsqcup_{T \in \BinTree_d} \bAcal^d_{\fatY}(x_0,x_1,\dots,x_d,T)$. \bigskip

Note that any compact one-dimensional manifold with boundary is diffeomorphic to a disjoint union of finitely many compact intervals and circles, so we can easily glue components of one-dimensional manifolds along their boundaries. We use this gluing procedure to define a single moduli space including elements of $\Acal^d_{\fatY}(x_0,x_1,\dots,x_d,T)$ for different choices of $T \in \RTree_d$. 

\begin{definition}
 Let $x_0,x_1,\dots,x_d \in \Crit f$ with 
 $$\mu(x_0)= \sum_{i=1}^d \mu(x_i)+3-d \ $$
 and let $\fatY$ be a regular and universal perturbation datum. We introduce an equivalence relation on $\bigsqcup_{T \in \BinTree_d} \bAcal^d_{\fatY}(x_0,x_1,\dots,x_d,T)$ by identifying both copies of the respective elements of $\bigsqcup_{k(T)=d-3} \Acal^d_{\fatY}(x_0,x_1,\dots,x_d,T)$, which exist by Corollary \ref{TwiceAsBoundary}. With respect to this equivalence relation, we define 
 \begin{equation*}
  \bAcal^d_{\fatY}(x_0,x_1,\dots,x_d) :=  \bigsqcup_{T \in \BinTree_d} \bAcal^d_{\fatY}(x_0,x_1,\dots,x_d,T) \Big/ \sim \ , 
 \end{equation*}
 Moreover, we put
 \begin{equation*}
  \Acal^d_{\fatY}(x_0,x_1,\dots,x_d) := \bigsqcup_{T \in \RTree_d} \Acal^d_{\fatY}(x_0,x_1,\dots,x_d,T) \ .
 \end{equation*}
\end{definition}

It is clear that the space $\bAcal^d_{\fatY}(x_0,x_1,\dots,x_d)$ can again be given the structure of a compact one-dimensional manifold with boundary, and that 
\begin{equation*}
 \Acal^d_{\fatY}(x_0,x_1,\dots,x_d) \subset \mathrm{int}\left(\bAcal^d_{\fatY}(x_0,x_1,\dots,x_d) \right) \ .
\end{equation*}
Furthermore, if we additionally choose $\fatY$ to be consistent, its boundary has a useful description as we will see in the next theorem.

\begin{theorem} \index{boundary spaces of moduli spaces!of Morse ribbon trees} \index{compactness of moduli spaces!of Morse ribbon trees!one-dimensional}
\label{boundaryaftergluingT}
 Let $d \geq 2$, $x_0,x_1,\dots,x_d \in \Crit f$ with 
 \begin{equation*}
  \mu(x_0) = \sum_{i=1}^d \mu(x_i) +3-d \ ,
 \end{equation*}
 and let $\fatY$ be an admissible perturbation datum.  The space $\bAcal^d_{\fatY}(x_0,x_1,\dots,x_d)$ can be equipped with the structure of a compact one-dimensional manifold with boundary, whose boundary is given by
$$ \partial \bAcal^d_{\fatY}(x_0,x_1,\dots,x_d)= \bigcup_{i=1}^d \bigcup_{l=0}^{d-i} \bigcup_y \Acal^{d-l}_{\fatY}(x_0,x_1,\dots,x_{i-1},y,x_{i+l+1},\dots,x_d) \times \Acal^{l+1}_{\fatY}(y,x_i,\dots,x_{i+l}) \ ,$$
 where we put $\Acal^1_{\fatY}(x,y) := \widehat{\MM}(x,y)$ for every $x,y \in \Crit f$ and where the union over $y$ is taken over all $y \in \Crit f$ satisfying
 \begin{equation*}
  \mu(y) = \sum_{q=i}^{i+l}\mu(x_q)+1-l \ .
 \end{equation*}
\end{theorem}

We need another graph-theoretic lemma to prove Theorem \ref{boundaryaftergluingT}.

\begin{lemma}
\label{TreeBreakingLemma} 
 For $d \geq 2$ let $\widetilde{\RTree}_d := \{(T,e) \ | \ T \in \RTree_d , \ e \in E_{int}(T) \}$. Then the following map is a bijection:
 \begin{align*}
  G: \widetilde{\RTree}_d &\to \bigcup_{l=1}^{d-2} \left(\RTree_{d-l} \times \RTree_{l+1} \times \{1,2,\dots,d-l\} \right) \ , \\
(T,e) &\mapsto \left(T^e_1,T^e_2,i\right) \ , \quad  \text{if} \ e \ \text{is of type} \ (i,l).
 \end{align*}
\end{lemma}

\begin{proof}
 The injectivity of $G$ is obvious, so we will only show its surjectivity. 
 
 Let $l \in \{1,2,\dots,d-2\}$ and $(T_1,T_2,i) \in \RTree_{d-l} \times \RTree_{l+1} \times \{1,2,\dots,d-l\}$. With respect to these choices, let $T \in \RTree_d$ the unique ribbon tree with
 \begin{align*}
  V(T) &= \left(V(T_1)\setminus\{v_i(T_1)\}\right) \cup \left(V(T_2) \setminus \{v_0(T_2)\}\right) \ , \\
  E(T) &= \left(E(T_1)\setminus\{e_i(T_1)\}\right) \cup \left(E(T_2) \setminus \{e_0(T_2)\}\right) \cup \{e_{\text{new}}\} \ ,
 \end{align*}
 where $e_{\text{new}}$ satisfies $\vin(e_{\text{new}}) = \vin(e_i(T_1))$ and $\vout(e_{\text{new}}) = \vout(e_0(T_2))$. Moreover, let $V(T)$ be ordered as $(v_1(T_1),\dots,v_{i-1}(T_1),v_1(T_2),\dots,v_{l+1}(T_2),v_{i+1}(T_1),\dots,v_{d-l}(T_1))$. Then it is clear from its definition that $e_{\text{new}}$ is of type $(i,l)$ in $T$, and that $G(T,e_{\text{new}}) = (T_1,T_2,i)$.
\end{proof}

\begin{proof}[Proof of Theorem \ref{boundaryaftergluingT}]
 We have discussed the manifold property before, so it only remains to describe the boundary of $\bAcal^d_{\fatY}(x_0,x_1,\dots,x_d)$. From the definition of $\bAcal^d_{\fatY}(x_0,x_1,\dots,x_d)$ as a quotient of $\bigsqcup_{T \in \RTree_d} \Acal^d_{\fatY}(x_0,x_1,\dots,x_d,T)$, we derive
 \begin{equation*}
  \partial\bAcal^d_{\fatY}(x_0,x_1,\dots,x_d) \subset \bigsqcup_{T \in \RTree_d} \partial \bAcal^d_{\fatY}(x_0,x_1,\dots,x_d,T) \ .
 \end{equation*}
 The boundary spaces $\partial \bAcal^d_{\fatY}(x_0,x_1,\dots,x_d,T)$ were described in Theorem \ref{BoundaryDescriptionT}. Note that the boundary parts of the form $\Acal^d_{\fatY}(x_0,x_1,\dots,x_d,T/e)$ are identified in the gluing procedure used to define $\bAcal^d_{\fatY}(x_0,x_1,\dots,x_d)$. All other boundary curves of any of the spaces $\bAcal^d_{\fatY}(x_0,x_1,\dots,x_d,T)$ are also boundary curves of $\bAcal^d_{\fatY}(x_0,x_1,\dots,x_d)$, so we obtain
 \label{boundarydescriptionalleT}
 \begin{align*}
  &\partial\bAcal^d_{\fatY}(x_0,x_1,\dots,x_d) =\bigcup_{T \in \RTree_d} \bigcup_{\stackrel{y_0 \in \Crit f}{\mu(y_0)=\mu(x_0)-1}} \widehat{\MM}(x_0,y_0) \times \Acal^d_{\fatY}(y_0,x_1,\dots,x_d,T) \\
  &\cup \bigcup_{T \in \RTree_d} \bigcup_{e \in E_{int}(T)} \bigcup_{x_e \in \Crit f} \Acal^{d-l}_{\fatY}(x_0,x_1,\dots,x_{i-1},x_e,x_{i+l+1},\dots,x_d,T^e_1) \times \Acal^{l+1}_{\fatY}(x_e,x_i,\dots,x_{i+l},T^e_2) \\
  &\cup \bigcup_{T \in \RTree_d} \bigcup_{i=1}^d \bigcup_{\stackrel{y_i \in \Crit f}{\mu(y_i)=\mu(x_i)+1}} \Acal^d_{\fatY}(x_0,x_1,\dots,x_{i-1},y_i,x_{i+1},\dots,x_d,T) \times \widehat{\MM}(y_i,x_i) \\
  &=\bigcup_{\stackrel{y_0 \in \Crit f}{\mu(y_0)=\mu(x_0)-1}} \Acal^1_{\fatY}(x_0,y_0) \times \Acal^d_{\fatY}(y_0,x_1,\dots,x_d) \\
  &\cup \bigcup_{T \in \RTree_d} \bigcup_{e \in E_{int}(T)} \bigcup_{x_e \in \Crit f} \Acal^{d-l}_{\fatY}(x_0,x_1,\dots,x_{i-1},x_e,x_{i+l+1},\dots,x_d,T^e_1)   \times \Acal^{l+1}_{\fatY}(x_e,x_i,\dots,x_{i+l},T^e_2) \\
  &\cup \bigcup_{i=1}^d \bigcup_{\stackrel{y_i \in \Crit f}{\mu(y_i)=\mu(x_i)+1}} \Acal^d_{\fatY}(x_0,x_1,\dots,x_{i-1},y_i,x_{i+1},\dots,x_d)\times \Acal^1_{\fatY}(y_i,x_i) \ ,
 \end{align*}
 where the unions over $x_e$ are taken over all $x_e \in \Crit f$ satisfying $\mu(x_e)= \sum_{q=i}^{i+l} \mu(x_q)+1-l$. Using Lemma \ref{TreeBreakingLemma}, we can identify
 \begin{align*} 
  &\bigcup_{T \in \RTree_d} \bigcup_{e \in E_{int}(T)} \bigcup_{x_e \in \Crit f} \Acal^{d-l}_{\fatY}(x_0,x_1,\dots,x_{i-1},x_e,x_{i+l+1},\dots,x_d,T^e_1)   \times \Acal^{l+1}_{\fatY}(x_e,x_i,\dots,x_{i+l},T^e_2) \\
&= \bigcup_{l=1}^{d-2} \bigcup_{i=1}^{d-1} \bigcup_{y \in \Crit f} \Acal^{d-l}_{\fatY}(x_0,x_1,\dots,x_{i-1},y,x_{i+l+1},\dots,x_d) \times \Acal^{l+1}_{\fatY}(y,x_i,\dots,x_{i+l}) \ ,
 \end{align*}
 where the unions over $y$ are taken over all $y \in \Crit f$ satisfying 
  \begin{equation}
  \label{yindexcond}
 \mu(y)= \sum_{q=i}^{i+l} \mu(x_q)+1-l \ .
 \end{equation}
Inserting this into our description of $\partial\bAcal^d_{\fatY}(x_0,x_1,\dots,x_d)$ yields:
\begin{align*}
 &\partial\bAcal^d_{\fatY}(x_0,x_1,\dots,x_d) =\bigcup_{\stackrel{y_0 \in \Crit f}{\mu(y_0)=\mu(x_0)-1}} \Acal^1_{\fatY}(x_0,y_0) \times \Acal^d_{\fatY}(y_0,x_1,\dots,x_d) \\
  &\cup  \bigcup_{l=1}^{d-2} \bigcup_{i=1}^{d-l} \bigcup_{y \in \Crit f} \Acal^{d-l}_{\fatY_1}(x_0,x_1,\dots,x_{i-1},y,x_{i+l+1},\dots,x_d) \times \Acal^{l+1}_{\fatY_2}(y,x_i,\dots,x_{i+l}) \\
  &\cup \bigcup_{i=1}^d \bigcup_{\stackrel{y_i \in \Crit f}{\mu(y_i)=\mu(x_i)+1}} \Acal^d_{\fatY}(x_0,x_1,\dots,x_{i-1},y_i,x_{i+1},\dots,x_d)\times \Acal^1_{\fatY}(y_i,x_i)  \\
  &= \bigcup_{i=1}^{d} \bigcup_{l=0}^{d-i}  \bigcup_{y \in \Crit f} \Acal^{d-l}_{\fatY_1}(x_0,x_1,\dots,x_{i-1},y,x_{i+l+1},\dots,x_d) \times \Acal^{l+1}_{\fatY_2}(y,x_i,\dots,x_{i+l}) \ ,  
  \end{align*}
where the unions over $y$ are again taken over all $y \in \Crit f$ satisfying \eqref{yindexcond}. 
\end{proof}

Theorem \ref{boundaryaftergluingT} is the decisive Theorem to derive relations on the twisted intersection numbers. We will combine it with the following basic fact from differential topology: \bigskip

\emph{Every compact one-dimensional manifold with boundary has an even number of boundary points.} \bigskip

This follows immediately from the fact that a compact one-dimensional manifold with boundary is diffeomorphic to a finite union of compact intervals and circles. See \cite[Appendix 2]{GuilleminPollack} for details. 

In addition to the twisted intersection numbers $a^d_{\fatY}(x_0,x_1,\dots,x_d)$ for $d \geq 2$ from Definition \ref{AinftyMorseCoefficients}, we further define for any $x_0,x_1 \in \Crit f$ and any admissible perturbation datum $\fatY$:
\begin{equation*}
 a^1_{\fatY}(x_0,x_1) := (-1)^{n+1} n(x_0,x_1) := (-1)^{n+1} \algint \widehat{\MM}(x_0,x_1) \ ,
\end{equation*}
where $n(x_0,x_1)$ is the corresponding coefficient of the Morse codifferential as explained in the introduction of this article. See Appendix \ref{AppendixOrientMorseTraj} for details on these coefficients. 

\begin{cor} \index{coefficient relation!of $\left\{\mu_{d,\fatY}\right\}_{d\in\NN}$}
\label{CoeffRelMod2}
Let $d \geq 2$, $x_0,x_1,\dots,x_d \in \Crit f$ with 
 \begin{equation*}
  \mu(x_0) = \sum_{i=1}^d \mu(x_i) +3-d
 \end{equation*}
 and let $\fatY= \left(\fatY_T\right)_{T \in \bigcup_{d \geq 2} \RTree_d}$ be an admissible perturbation datum. Then the following congruence modulo two is true:
 \begin{equation*}
  \sum_{i=1}^d \sum_{l=0}^{d-i} \sum_{y \in \Crit f} a^{d-l}_{\fatY}(x_0,x_1,\dots,x_{i-1},y,x_{i+l+1},\dots,x_d)\cdot a^{l+1}_{\fatY}(y,x_i,\dots,x_{i+l}) \equiv 0 \ ,
 \end{equation*}
 where the sum over $y$ is taken over all $y \in \Crit f$ satisfying $\mu(y)= \sum_{q=i}^{i+l} \mu(x_q)+1-l$.
\end{cor}

\begin{proof}
 By definition of the coefficients as twisted oriented intersection numbers, the following holds as a congruence modulo two for all $i \in \{1,2,\dots,d\}$, $l \in \{0,1,\dots,d-i\}$ and $y \in \Crit f$ with $\mu(y) = \sum_{q=i}^{i+l}\mu(x_q)+1-l$:
 \begin{align*}
  &\sum_{i=1}^d \sum_{l=0}^{d-i} \sum_{y \in \Crit f}a^{d-l}_{\fatY}(x_0,x_1,\dots,x_{i-1},y,x_{i+l+1},\dots,x_d)\cdot a^{l+1}_{\fatY}(y,x_i,\dots,x_{i+l}) \\
  &\equiv \sum_{i=1}^d \sum_{l=0}^{d-i} \sum_{y \in \Crit f} \left|\Acal^{d-l}_{\fatY}(x_0,x_1,\dots,x_{i-1},y,x_{i+l+1},\dots,x_d) \times \Acal^{l+1}_{\fatY}(y,x_i,\dots,x_{i+l}) \right|  \\
&\equiv \left| \partial \bAcal^d_{\fatY}(x_0,x_1,\dots,x_d)\right| \equiv 0 \quad \text{mod} \ 2 \ ,
\end{align*}
where the last congruence is a consequence of Theorem \ref{boundaryaftergluingT}.
\end{proof}

We want to derive a relation of the twisted intersection numbers from the one in Corollary \ref{CoeffRelMod2}, which is an equality of integers and not a congruence modulo two. To do so, one needs to take a look at the orientations of the moduli spaces $\Acal^d_{\fatY}(x_0,x_1,\dots,x_d,T)$ and of the compatibility of these orientations with respect to the gluing procedure that we used to define $\bAcal_{\fatY}^d(x_0,x_1,\dots,x_d)$. 

Afterwards, one needs to perform a thorough sign investigation of the sum of coefficients appearing in Corollary \ref{CoeffRelMod2}. Since the line of argument is very technical, the proof of Theorem \ref{CoeffRelAinfty} is postponed to Appendix \ref{AppendixProofTheoremCoeffRel}.

\begin{theorem} \index{coefficient relation!of $\left\{\mu_{d,\fatY}\right\}_{d\in\NN}$}
\label{CoeffRelAinfty}
Let $d \geq 2$, $x_0,x_1,\dots,x_d \in \Crit f$ with 
 \begin{equation*}
  \mu(x_0) = \sum_{i=1}^d \mu(x_i) +3-d 
 \end{equation*}all and let $\fatY$ be an admissible perturbation datum. Then
 \begin{equation*}
  \sum_{i=1}^d \sum_{l=0}^{d-i} \sum_{y \in \Crit f} (-1)^{\maltese_1^{i-1}} a^{d-l}_{\fatY}(x_0,x_1,\dots,x_{i-1},y,x_{i+l+1},\dots,x_d)\cdot a^{l+1}_{\fatY}(y,x_i,\dots,x_{i+l})=0 \ ,
 \end{equation*}
 where\footnote{The use of the Maltese cross ($\maltese$) for the coefficients defining the signs in this thesis happens in accordance with the notation of the works of M. Abouzaid and P. Seidel, e.g. \cite{SeidelBook} or \cite{AbouzaidPlumbings}. In particular, the author distances himself from any political meaning or implication of this symbol.} 
 \begin{equation*}
  \maltese_1^{i-1} := \maltese_1^{i-1}(x_1,\dots,x_d) = \sum_{q=1}^{i-1}\|x_q\| = \sum_{q=1}^{i-1}\mu(q) - i \ ,
 \end{equation*}
 and where the sum over $y$ is taken over all $y \in \Crit f$ satisfying $\displaystyle\mu(y) = \sum_{q=i}^{i+l} \mu(x_q)+1-l$. 
\end{theorem}

Remember that we have defined the maps $\left(\mu_{d,\fatY}\right)_{d \geq 2}$ in Definition \ref{AinftyMorseMaps} in terms of the numbers $a^d_{\fatY}(\cdots)$. Hence, using Theorem \ref{CoeffRelAinfty}, we are finally able to prove that these operations fulfil the defining equations of an $A_\infty$-algebra. 

\begin{theorem} \index{$A_\infty$-algebra!structure on $C^*(f)$} \index{higher order multiplications}
\label{TheoremMorseAinftyalgebra}
 For every admissible perturbation datum $\fatY$, the Morse cochain complex
 \begin{equation*}
 \left(C^*(f), \left(\mu_{d,\fatY}:C^*(f)^{\otimes d} \to C^*(f) \right)_{d \in \NN}\right)
 \end{equation*}
 is an $A_\infty$-algebra, where $\mu_{1,\fatY}:= (-1)^{n+1} \delta: C^*(f) \to C^*(f)$ denotes the twisted Morse codifferential.
\end{theorem}

\begin{proof}
 Note that the map $\mu_{1,\fatY}:C^*(f) \to C^*(f)$ is explicitly defined as the $\ZZ$-linear extension of  
 \begin{equation*}
 (x_0 \in \Crit f) \ \mapsto \sum_{\stackrel{x_1 \in \Crit f}{\mu(x_1)=\mu(x_0)+1}} a^1_{\fatY}(x_0,x_1)x_1 \ .
 \end{equation*}
 It suffices to show that the defining equations of an $A_\infty$-algebra hold on generators of $C^*(f)$, i.e. that for every $d \in \NN$ and every $x_1,\dots,x_d \in \Crit f$ the following holds:
 \begin{align*}
  &\sum_{\stackrel{d_1,d_2 \in \NN}{d_1+d_2=d+1}} \sum_{i=1}^{d+1-d_1}(-1)^{\maltese_1^{i-1}}\mu_{d_2,\fatY}(x_1,\dots,x_{i-1},\mu_{d_1,\fatY}(x_i,\dots,x_{i+d_1-1}),a_{i+d_1},\dots,x_d) \\
  &=\sum_{l=0}^{d-1} \sum_{i=1}^{d-l}(-1)^{\maltese_1^{i-1}}\mu_{d-l,\fatY}(x_1,\dots,x_{i-1},\mu_{l+1,\fatY}(x_i,\dots,x_{i+l}),a_{i+l+1},\dots,x_d) =0 \ .
 \end{align*}
 For all $l \in \{0,1,\dots,d-1\}$ and  $i \in \{1,2\dots,d-l\}$ we compute that
 \begin{align}
  &\mu_{d-l,\fatY}(x_1,\dots,x_{i-1},\mu_{l+1,\fatY}(x_i,\dots,x_{i+l}),x_{i+l+1},\dots,x_d) \notag \\
 &=\mu_{d-l,\fatY}(x_1,\dots,x_{i-1},\sum_{y}a^{l+1}_{\fatY}(y,x_i,\dots,x_{i+l})\cdot y,x_{i+l+1},\dots,x_d) \notag \\
 &=\sum_y a^{l+1}_{\fatY}(y,x_i,\dots,x_{i+l})\cdot \Bigl(\sum_{x_0} a^{d-l}_{\fatY}(x_0x_1,\dots,x_{i-1},y,x_{i+l+1},\dots,x_d)\cdot x_0 \Bigr) \ , \label{letzteformel}
 \end{align}
 where the sums over $x_0$ and $y$ are taken over all $x_0,y \in \Crit f$ satisfying
 \begin{equation*}
 \mu(y) = \sum_{q=i}^{i+l} \mu(x_q)+1-l \ , \quad \mu(x_0) = \sum_{q=1}^{i-1} \mu(x_q)+\mu(y)+\sum_{q=i+l+1}^d \mu(x_q) +2-d+l \ .
 \end{equation*}
 One checks that \eqref{letzteformel} is the same as the sum
 \begin{equation*}
 \sum_{x_0} \quad \sum_{y} a^{d-l}_{\fatY}(x_0x_1,\dots,x_{i-1},y,x_{i+l+1},\dots,x_d) \cdot a^{l+1}_{\fatY}(y,x_i,\dots,x_{i+l}) \cdot x_0 \ ,
 \end{equation*}
 where this time the sums are taken over all $x_0,y\in\Crit f$ satisfying 
 \begin{equation}
 \label{x0ycond}
  \mu(y) = \sum_{q=i}^{i+l} \mu(x_q)+1-l \ , \quad \mu(x_0) = \sum_{q=1}^{d} \mu(x_q)+1-d \ .
 \end{equation}
 Consequently, 
 \begin{align*}
 &\sum_{l=0}^{d-1} \sum_{i=1}^{d-l}(-1)^{\maltese_1^{i-1}}\mu_{d-l,\fatY}(x_1,\dots,x_{i-1},\mu_{l+1,\fatY}(x_i,\dots,x_{i+l}),a_{i+l+1},\dots,x_d) \\
 &=\sum_{x_0} \Bigl( \sum_{i=1}^{d} \sum_{l=0}^{d-i} \sum_{y} (-1)^{\maltese_1^{i-1}} a^{d-l}_{\fatY}(x_0x_1,\dots,x_{i-1},y,x_{i+l+1},\dots,x_d) \cdot a^{l+1}_{\fatY}(y,x_i,\dots,x_{i+l}) \Bigr) \cdot x_0 \ ,
 \end{align*}
 where the sums are taken over all $x_0,y \in \Crit f$ satisfying \eqref{x0ycond}. 
 
 But by Theorem \ref{AinftyMorseCoefficients}, the sum inside the brackes vanishes for every $x_0 \in \Crit f$ of the right index. Therefore, the whole sum vanishes and the claim follows.
\end{proof}

\appendix

\section{Orientations and sign computations for perturbed Morse ribbon trees}
\label{AppendixOrient}

We have mentioned the necessity of orientations on the Morse-theoretic moduli spaces involved. More precisely, we have defined the maps $\mu_{d,\fatY}$ in terms of the twisted oriented intersection numbers of the form $a^d_{\fatY}(x_0,x_1,\dots,x_d)$, which requires a choice of orientations on the spaces $\Acal^d_{\fatY}(x_0,x_1,\dots,x_d)$. 

For clarity's sake, we refrained from explaining the details in the main part of this article. Nevertheless, we will provide the constructions and explanations in this appendix.
It contains several tedious, but indispensable computations related to orientations on Morse-theoretic moduli spaces and the coefficients $a^d_{\fatY}(x_0,x_1,\dots,x_d)$.

\bigskip

In Section \ref{AppendixOrientMorseTraj} we define orientations on spaces of semi-infinite and finite-length Morse trajectories by using identifications of these moduli spaces with subsets of the target manifold. Afterwards, we investigate the compatibilities of the chosen orientations with Morse-theoretic gluing maps. 

Section \ref{SectionOrientPerturbedMRT} extends the constructions of the aforegoing section to spaces of perturbed Morse trajectories. The orientations on the perturbed trajectories are then used to define orientations on the moduli spaces of type $\Acal^d_{\fatY}(x_0,x_1,\dots,x_d,T)$, filling the gaps in the definition of the coefficients $a^d_{\fatY}(x_0,x_1,\dots,x_d)$. A similar discussion is given by Abouzaid in \cite[Section 8]{AbouzaidPlumbings}.

Eventually, Section \ref{AppendixProofTheoremCoeffRel} provides all missing computations for the proof of Theorem \ref{CoeffRelAinfty}. In other words it concludes the proof of the maps $\mu_{d,\fatY}$ fulfilling the defining equations of an $A_\infty$-algebra by a number of long, but straightforward computations. This section is strongly oriented on \cite[Appendix C]{AbouzaidMorseTrop}, which discusses the unperturbed analogue of Theorem \ref{CoeffRelAinfty}.

\bigskip

In contrast to the more general approach of coherent orientations as in \cite{Schwarz} and \cite{FloerHoferCoherent}, we will pursue the classic finite-dimensional approach as presented in \cite[Appendix B]{Schwarz}, \cite[Section 7.1]{BanyagaHurtubise}, \cite[Section 6.6]{JostRiemGeom} or \cite[Section 6.6]{Schaetz}.  

\bigskip

\textit{Let always '$\equiv$' denote the congruence modulo two of two integers. }

\subsection{Orientations on Morse trajectory spaces}
\label{AppendixOrientMorseTraj}

By the Stable/Unstable Manifold Theorem from Morse theory, see \cite[Theorem 4.2]{BanyagaHurtubise}, both the unstable and the stable manifolds of a negative gradient flow of a Morse function on a complete manifold without boundary are diffeomorphic to open balls, hence orientable. 

We orient the unstable and stable manifolds with respect to $f \in C^\infty(M)$ as follows: \index{orientations!of unstable and stable manifolds}

\begin{itemize}
 \item We choose an arbitrary orientation on $W^u(x)$ for every $x \in \Crit f$. 
 \item We equip every $W^s(x)$ with the unique orientation such that the orientation on $T_xM$ which is induced by the splitting
 \begin{equation*}
  T_xM = T_xW^s(x) \oplus T_x W^u(x)
 \end{equation*}
 coincides with the given orientation on $M$ for every $x \in \Crit f$. In other words,  a positive basis of $T_x W^s(x)$ followed by a positive basis of $T_x W^u(x)$ is a positive basis of $T_xM$.
\end{itemize}

{\it Throughout this article, we assume such orientations on the unstable and stable manifolds to be chosen and fixed.We further equip the spaces $\WW^u(x)$ and $\WW^s(x)$ for every $x \in \Crit f$ with the unique orientations which make the endpoint evaluations
$\WW^u(x)\stackrel{\cong}{\to} W^u(x)$ and $\WW^s(x) \stackrel{\cong}{\to} W^s(x)$, both given by $\gamma \mapsto \gamma(0)$, orientation-preserving.}

\bigskip

The spaces of finite-length trajectories of the negative gradient flow of $f$ are even easier to orient. By construction, there is a diffeomorphism 
\begin{equation}
\label{MMRMdiffeo}
\MM(f,g) \stackrel{\cong}{\to} [0,+\infty) \times M \ ,  \quad (l,\gamma) \mapsto (l,\gamma(0)) \ .
\end{equation}
We equip $[0,+\infty)$ with the canonical and $[0,+\infty)\times M$ with the product orientation. 

\bigskip

\textit{Throughout this article, we equip the space $\MM(f,g)$ with the unique orientation, such that the diffeomorphism in (\ref{MMRMdiffeo}) is orientation-preserving if and only if $n$ is odd. In other words, such that the sign of the diffeorphism in (\ref{MMRMdiffeo}) is given by $(-1)^{n+1}$.} \index{orientations!of finite-length trajectory spaces}

\bigskip
This choice of orientation for the space of finite-length trajectories might seem awkward at first sight. For technical reasons, this choice of orientation will be the right one to establish the $A_\infty$-equations for the Morse-theoretic operations on $C^*(f)$.

\bigskip

Remember from the introduction that the Morse codifferential $\delta: C^*(f) \to C^*(f)$ is given as the $\ZZ$-linear extension of
\begin{equation*}
 \delta(x) = \sum_{\stackrel{z \in \Crit f}{\mu(z) = \mu(x)+1}} n(z,x) \cdot z 
\end{equation*}
for $x \in \Crit f$, where $n(z,x) := \algint \widehat{\MM}(z,x)$, i.e. the oriented intersection number of the zero-dimensional manifold $\widehat{\MM}(z,x)$. (We will discuss oriented intersection numbers in Section \ref{SectionOrientPerturbedMRT}.)

\bigskip

The last types of Morse trajectory spaces that we use in this article are moduli spaces of pa\-ra\-me\-trized and unparametrized Morse trajectories starting and ending in fixed critical points, i.e. moduli spaces of type $\MM(x,y):= \MM(x,y,g)$ and $\widehat{\MM}(x,y):=\widehat{\MM}(x,y,g)$ for $x,y \in \Crit f$ which we defined in the introduction. We will proceed in strict analogy with \cite[Appendix B]{Schwarz} and define orientations on these moduli spaces. 

\bigskip

For any $x,y \in \Crit f$, there is an inclusion
\begin{equation*}
 \iota: \MM(x,y) \to \WW^u(x) \times \WW^s(y) \ , \quad  \gamma \mapsto \left(\gamma|_{(-\infty,0]},\gamma|_{[0,+\infty)} \right) \ .
\end{equation*}
One checks without difficulties that $\iota$ is a smooth embedding. Let $\gamma \in \MM(x,y)$ and put $(\gamma_1,\gamma_2) := \iota(\gamma)$. The following sequence is a short exact sequence of vector spaces:
\begin{equation}
\label{MMxySES}
 0 \longrightarrow T_\gamma \MM(x,y) \stackrel{i}{\longrightarrow} T_{\gamma_1}\WW^u(x)\oplus T_{\gamma_2} \WW^s(y) \stackrel{p}{\longrightarrow} T_{\gamma(0)}M \longrightarrow 0 \ , 
\end{equation}
where $\gamma_1 := \gamma|_{(-\infty,0]}$, $\gamma_2 := \gamma|_{[0,+\infty)}$, 
\begin{equation}
\label{Mapsiandp}
i(\xi) = \left(\xi|_{(-\infty,0]},-\xi|_{[0,+\infty)} \right) \ , \qquad p(\xi_1,\xi_2) = \xi_1(0) + \xi_2(0) \ . 
\end{equation}
The injectivity of $i$ is clear while the surjectivity of $p$ follows from the Morse-Smale property, i.e. because all unstable and stable manifolds intersect transversely. 

Given a short exact sequence of vector spaces 
\begin{equation*}
0 \to V_0 \stackrel{i}{\to} V_1 \stackrel{p}{\to} V_2 \to 0 
\end{equation*}
and orientations on $V_1$ and $V_2$, these orientations induce a well-defined orientation on $V_0$ by using the following convention: A basis of $V_0$ is positive if and only if the image of this basis under $i$ followed by the preimage of a positive basis of $V_2$ is a positive basis on $V_1$. 

\bigskip

\textit{For all $x,y \in \Crit f$ with $\mu(x) > \mu(y)$, we equip $\MM(x,y)$ with the orientation induced by the orientation on $M$, the chosen orientations on $\WW^u(x)$ and $\WW^s(y)$ and the short exact sequence (\ref{MMxySES}).} \index{orientations!of Morse trajectory spaces!parametrized}

\bigskip

Let $x,y \in \Crit f$ with $\MM(x,y)\neq \emptyset$ and $x\neq y$. Since $f$ is decreasing along the trajectories of its negative gradient flow, this implies: $f(x) > f(y)$. Let $a \in \RR$ be a regular value of $f$ with 
\begin{equation*}
 f(x)>a>f(y) \ . 
\end{equation*}
Then the following map is well-defined and smooth:
\begin{equation}
\label{iotaa}
 \iota_a: \widehat{\MM}(x,y) \to \MM(x,y) \ , \quad  \hat{\gamma} \mapsto \gamma_a \ , 
\end{equation}
where $\gamma_a$ denotes the unique element of the equivalence class $\hat{\gamma}$ satisfying $f(\gamma_a(0))=a$. For $\hat{\gamma} \in \widehat{\MM}(x,y)$ let $(v_1,\dots,v_N)$ be a basis of $T_{\hat{\gamma}}\widehat{\MM}(x,y)$. Then 
\begin{equation*}
 \left(-\nabla f\circ \gamma,(D\iota_a)_{\hat{\gamma}}(v_1),\dots,(D\iota_a)_{\hat{\gamma}}(v_N) \right)
\end{equation*}
is a basis of $T_{\gamma_a}\MM(x,y)$, see \cite[Appendix B]{Schwarz} or \cite[Section 2.5]{Schaetz}. We call $(v_1,\dots,v_N)$ a \emph{positive} basis of $T_{\hat{\gamma}}\widehat{\MM}(x,y)$ if and only if the thus-induced basis of $T_{\gamma_a}\MM(x,y)$ is positive with respect to the orientation defined above.  

\bigskip

\textit{For every $x,y \in \Crit f$ with $f(x)> f(y)$, this definition of positive bases yields a well-defined orientation of $\widehat{\MM}(x,y)$. This orientation is independent of the choice of regular value $a$ and we will always consider $\widehat{\MM}(x,y)$ as equipped with this orientation.}  \index{orientations!of spaces of Morse trajectories!unparameterized}

\bigskip

We want to investigate how the Morse-theoretic gluing maps that we used in the introduction behave with respect to the orientations on Morse trajectory spaces. Before we do so, we discuss the relation between Morse trajectory spaces with respect to the functions $f$ and $-f$ that will be useful in the discussing of gluing maps and orientations. 

We will derive the result for positive semi-infinite Morse trajectories from the result for negative semi-infinite Morse trajectories by identifying positive half-trajectories of the negative gradient flow of $f$ with negative half-trajectories of its positive gradient flow. 

\bigskip 

We make the simple observation that the critical points of the Morse functions $f$ and $-f$ coincide. In the following, we will define all negative and positive gradient flow lines with respect to the same given Riemannian metric on $M$. 

The following map is a smooth diffeomorphism for every $x \in \Crit f$: 
\begin{equation*}
 \varphi_s: \WW^s(x,f) \to \WW^u(x,-f) \ , \qquad
  \left(\varphi_s(\gamma)\right)(t) := \gamma(-t) \quad \forall t \in (-\infty,0] \ . 
\end{equation*}
We equip every $\WW^u(x,-f)$ with the unique orientation that makes the map $\varphi_s$ orientation-preserving. 

Moreover, we equip every $\WW^s(x,-f)$ with the complementary orientation, i.e. the unique orientation induced by the following splitting:
\begin{equation*}
 T_x M = T_xW^s(x,-f)\oplus T_x W^u(x,-f) \cong T_{\gamma_x} \WW^s(x,-f) \oplus T_{\gamma_x} \WW^u(x,-f) \ , 
\end{equation*}
where $\gamma_x$ denotes both the constant positive and negative semi-infinite Morse trajectory, given by $\gamma_x(t)=x$ for every $t$. 

In addition to the map $\varphi_s$, we consider the following smooth diffeomorphism:
\begin{equation*}
 \varphi_u: \WW^u(x,f) \to \WW^s(x,-f) \ , \qquad \left(\varphi_u(\gamma)\right)(t) = \gamma(-t) \quad \forall t \in [0,+\infty) \ .
\end{equation*}

\index{orientations!of unstable and stable manifolds}
\begin{lemma}
\label{OrOnWsminusf}
 With respect to the chosen orientations, the map $\varphi_u$ is orientation-preserving if and only if $(n+1)\mu(x)$ is even. 
\end{lemma}

\begin{proof}
 We will write $\varphi_u(\WW^u(x,f))$ for $\WW^s(x,-f)$ whenever we want to consider it as equipped the orientation induced by $\varphi_u$. 
 
 By definition, $\varphi_u$ is orientation-preserving if and only if the following splitting induces a positive orientation on $T_xM$:
 \begin{equation}
 \label{SplittingVarphiu}
  T_xM \cong T_{\gamma_x}\varphi_u(\WW^u(x,f)) \oplus T_{\gamma_x} \WW^u(x,-f) \ , 
 \end{equation}
 where $\gamma_x$ again denotes the respective constant trajectory.  The following map is an orientation-preserving diffeomorphism:
 \begin{equation*}
  \varphi_u^{-1} \times \varphi_s^{-1}: \varphi_u(\WW^u(x,f)) \times \WW^u(x,-f) \to \WW^u(x,f) \times \WW^s(x,f) \ .
 \end{equation*}
 Thus, there is an orientation-preserving isomorphism 
 \begin{equation*}
 T_{\gamma_x}\varphi_u(\WW^u(x,f)) \oplus T_{\gamma_x} \WW^u(x,-f) \cong T_{\gamma_x}\WW^u(x,f) \oplus T_{\gamma_x} \WW^s(x,f) \ .
 \end{equation*}
 Moreover, there are isomorphisms 
 \begin{equation*}
  T_{\gamma_x} \WW^u(x,f) \oplus T_{\gamma_x}\WW^s(x,f) \cong  T_{\gamma_x}\WW^s(x,f) \oplus T_{\gamma_x} \WW^u(x,f) \cong T_xM \ , 
 \end{equation*}
 where the latter one is by definition orientation-preserving and the former one is a permutation of factors. It is orientation-preserving if and only if the following number is even:
 \begin{equation*}
  \dim \WW^s(x,f) \cdot \dim \WW^u(x,f) = (n-\mu(x))\mu(x) \equiv (n+1)\mu(x) \ .
 \end{equation*}
 Consequently, $\varphi_u$ induces the complementary orientation in (\ref{SplittingVarphiu}) iff $(n+1)\mu(x)$ is even.
\end{proof}

Another way of stating the content of Lemma \ref{OrOnWsminusf} is the following: If $\WW^s(x,-f)$ is equipped with the orientation complementary to the one of $\WW^u(x,-f)$, then the sign of the diffeomorphism $\varphi_u: \WW^u(x,f) \to \WW^s(x,-f)$ will be given by 
\begin{equation*}
 \sign \varphi_u = (-1)^{(n+1)\mu(x)} \ . 
\end{equation*}
We can also identify spaces of the form $\MM(x,y)$ with trajectories of the positive gradient flow of $f$. For clarity's sake, we will denote $\MM(x,y)$ by $\MM(x,y,f)$ for all $x,y \in \Crit f$ and put
\begin{equation*}
 \MM(y,x,-f):= \left\{\gamma \in C^\infty(\RR,M) \ \left| \ \dot{\gamma}-\nabla^g f \circ \gamma = 0, \ \lim_{s \to -\infty}\gamma(s)=y, \ \lim_{s \to+\infty}\gamma(s)=x \right. \right\} \ .
\end{equation*}
Since $(-f,g)$ is obviously a Morse-Smale pair if $(f,g)$ is and since 
\begin{equation*}
\morseind(x,-f)=n-\mu(x)
\end{equation*}
for every $x \in \Crit f$, the space $\MM(y,x,-f)$ is a smooth manifold with 
\begin{equation*}
\dim \MM(y,x,-f)=(n-\mu(y))-(n-\mu(x)) = \mu(x)-\mu(y)\ .
\end{equation*}

For every $x,y \in \Crit f$ the following map is a smooth diffeomorphism: 
\begin{equation*}
 \psi: \MM(x,y,f) \to \MM(y,x,-f) \ , \qquad (\psi(\gamma))(t) := \gamma(-t) \quad \forall t \in \RR \ . 
\end{equation*}

\begin{lemma} \index{orientations!of Morse trajectory spaces!parametrized}
\label{PsiMMOr}
 The map $\psi: \MM(x,y,f) \to \MM(y,x,-f)$ is orientation-preserving if and only if $(\mu(x)+1)\mu(y)$ is even. 
\end{lemma}

\begin{proof}
 The rows of the following diagram are short exact sequences and one checks without difficulties that the diagram commutes: 
 \begin{equation*}
 \minCDarrowwidth22pt
 \begin{CD}
  0 @>>> T_\gamma \MM(x,y,f) @>{i}>> T_{\gamma_1}\WW^u(x,f)\oplus T_{\gamma_2}\WW^s(y,f) @>{p}>> T_{\gamma(0)}M @>>> 0 \\
  @. @V{\cong}V{(D\psi)_{\gamma}}V @V{\cong}V{V}V @V{\cong}V{-\id_{T_{\gamma(0)}M}}V @. \\
  0 @>>> T_{\psi(\gamma)} \MM(y,x,-f) @>{i}>> T_{\bar{\gamma}_2}\WW^u(y,-f)\oplus T_{\bar{\gamma}_1}\WW^s(x,-f) @>{p}>> T_{\gamma(0)}M @>>> 0 , 
  \end{CD}
 \end{equation*}
 where $\gamma_1 :=\gamma|_{(-\infty,0]}$, $\gamma_2:= \gamma|_{[0,+\infty)}$, $\bar{\gamma}_1 := \varphi_u(\gamma_1)$ and $\bar{\gamma}_2:= \varphi_s(\gamma_2)$. Here, $i$ and $p$ are defined as in \eqref{Mapsiandp} and the middle vertical map $V$ is given by 
 \begin{equation*}
  V(\xi_1,\xi_2) = \left(-(D\varphi_s)_{\gamma_2}(\xi_2),-(D\varphi_u)_{\gamma_1}(\xi_1) \right) \ .
 \end{equation*}
 Since the diagram commutes, it follows that 
 \begin{equation}
 \label{EqSignPsi}
  \sign \psi = \sign (D\psi)_{\gamma} = (\sign V)\left(\sign \left(-\id_{T_{\gamma(0)}M}\right)\right) = (-1)^n \cdot \sign V \ . 
 \end{equation}
 To determine the sign of $V$, we note that $V$ factorizes as
 \begin{align*}
  T_{\gamma_1}\WW^u(x,f)\oplus &T_{\gamma_2}\WW^s(y,f) \stackrel{V_1}{\longrightarrow} T_{\gamma_2}\WW^s(y,f)\oplus T_{\gamma_1}\WW^u(x,f) \stackrel{V_2}{\longrightarrow} \\ 
  &T_{\bar{\gamma}_2}\WW^u(y,-f)\oplus T_{\bar{\gamma}_1}\WW^s(x,-f) \stackrel{V_3}{\longrightarrow} T_{\bar{\gamma}_2}\WW^u(y,-f)\oplus T_{\bar{\gamma}_1}\WW^s(x,-f) \ ,
 \end{align*}
 where $V_1$ is the map which transposes the two factors, $V_2$ is given by applying $(D\varphi_s)_{\gamma_2}$ and $(D\varphi_u)_{\gamma_1}$ and $V_3 = - \id$. Since $V = V_3 \circ V_2 \circ V_1$, it follows that
 \begin{equation*}
  \sign V = \sign V_1 \cdot \sign V_2 \cdot \sign V_3 \ . 
 \end{equation*}
The sign of $V_1$ is given by the parity of 
\begin{equation*}
 \left(\dim T_{\gamma_1}\WW^u(x,f)\right) \cdot \left(\dim T_{\gamma_2}\WW^s(y,f)\right) = \mu(x) (n-\mu(y)) \ . 
\end{equation*}
The sign of $V_2$ is given by 
\begin{equation*}
 \sign V_2 = \sign (D\varphi_s)_{\gamma_2} \cdot \sign (D\varphi_u)_{\gamma_1} = \sign (D\varphi_u)_{\gamma_1} = (-1)^{(n+1)\mu(x)} \ ,
\end{equation*}
where we used Lemma \ref{OrOnWsminusf} and the fact that $\varphi_s$ is orientation-preserving. The sign of $V_3$ is obviously given by the parity of 
\begin{equation*}
 \dim T_{\bar{\gamma}_2} \WW^u(y,-f) + \dim T_{\bar{\gamma}_1}\WW^s(x,-f) = n-\mu(y)+\mu(x) \ . 
\end{equation*}
We conclude that the sign of $V$ is given by the parity of 
$$ \mu(x) (n-\mu(y)) + (n+1)\mu(x) + n-\mu(y)+\mu(x) \equiv (\mu(x)+1)\mu(y)+n \ .$$
Consequently, we derive from (\ref{EqSignPsi}) that $\sign \psi = (-1)^n (-1)^{(\mu(x)+1)\mu(y)+n}= (-1)^{(\mu(x)+1)\mu(y)}$.

\end{proof}

One observes that the map $\psi: \MM(x,y,f) \to \MM(y,x,-f)$ induces a well-defined map
\begin{equation*}
 \widehat{\psi}:\widehat{\MM}(x,y,f) \to \widehat{\MM}(y,x,-f) \ ,
\end{equation*}
which assigns to each $\hat{\gamma}$ the equivalence class of $\psi(\gamma)$, where $\gamma$ is a representative of $\hat{\gamma}$. Since $\psi$ is a smooth diffeomorphism, one derives that $\widehat{\psi}$ is a smooth diffeomorphism as well. 

\begin{prop} \index{orientations!of Morse trajectory spaces!unparametrized}
 The map $\widehat{\psi}:\widehat{\MM}(x,y,f) \to \widehat{\MM}(y,x,-f)$ is orientation-preserving if and only if $(\mu(x)+1)\mu(y)$ is odd.
\end{prop}

\begin{proof}
For every regular value $a$ of $f$ with $f(x)>a>f(y)$, the following diagram obviously commutes, where $\iota_a$ is defined as in (\ref{iotaa}):
\begin{equation}
\label{psihatCD}
 \begin{CD}
  \widehat{\MM}(x,y,f) @>{\widehat{\psi}}>> \widehat{\MM}(y,x-f) \\
  @VV{\iota_a}V @VV{\iota_a}V \\
  \MM(x,y,f) @>{\psi}>> \MM(y,x,-f) 
 \end{CD}
\end{equation}
Let $\hat\gamma \in \widehat{\MM}(x,y,f)$ and let $(v_1,\dots,v_N)$ be a positive basis of $T_{\hat\gamma}\widehat{\MM}(x,y,f)$. We need to determine the sign of the basis 
\begin{equation*}
 \left(\left(D\widehat\psi\right)_{\hat\gamma}(v_1),\dots,\left(D\widehat\psi\right)_{\hat\gamma}(v_N)\right)
\end{equation*}
of $T_{\widehat\psi(\hat{\gamma})}\widehat{\MM}(y,x,-f)$. By definition of the orientation on $\widehat{\MM}(y,x,-f)$, this basis is positive if and only if the following is a positive basis of $T_{\left(\iota_a\circ\widehat{\psi}\right)\left(\hat{\gamma}\right)}\MM(y,x,-f)$:
\begin{align}
 &\left(\nabla f \circ \left(\iota_a\circ \widehat{\psi}\right)(\hat\gamma),\left(D\left(\iota_a \circ \widehat{\psi}\right)\right)_{\hat\gamma}(v_1),\dots,\left(D\left(\iota_a \circ \widehat{\psi}\right)\right)_{\hat\gamma}(v_N)\right) \label{firstbasis} \\
 =&\left(\nabla f \circ \psi(\gamma_a),\left(D\left(\psi \circ \iota_a\right)\right)_{\hat\gamma}(v_1),\dots,\left(D\left(\psi \circ \iota_a\right)\right)_{\hat\gamma}(v_N)\right) \ , \notag
\end{align}
where the equality follows from the commutativity of diagram (\ref{psihatCD}).

 Since $\psi$ inverts the flow direction, the following obviously holds for every $\gamma \in \MM(x,y,f)$:
 \begin{equation*}
  (D\psi)_{\gamma}(-\nabla f \circ \gamma) = \nabla f \circ \psi(\gamma) \ .
 \end{equation*}
 Hence the basis in (\ref{firstbasis}) is given by 
 \begin{align*}
 \left((D\psi)_{\gamma_a}(-\nabla f \circ \gamma_a),(D\psi)_{\gamma_a}\left(\left(D\iota_a\right)_{\hat\gamma}(v_1)\right),\dots,(D\psi)_{\gamma_a}\left(\left(D\iota_a\right)_{\hat\gamma}(v_N)\right)\right) \ .
 \end{align*}
 By Lemma \ref{PsiMMOr}, the sign of $(D\psi)_{\gamma_a}$ is given by the parity of $(\mu(x)+1)\mu(y)$. Since $(D\psi)_{\gamma_a}$ maps the first vector of the basis to its opposite, the effect on the last $N$ vectors is given by the parity of $(\mu(x)+1)\mu(y)+1$. By definition of the orientations involved, this implies that the basis in (\ref{firstbasis}) is positive if and only if $(\mu(x)+1)\mu(y)+1$ is even, which shows the claim.
 \end{proof}

 \begin{remark}
 \label{SignPsiHatZeroDim}
 In particular, if $\mu(x) = \mu(y)+1$, i.e. if $\widehat{\MM}(x,y,f)$ is zero-dimensional, then $\psi$ is orientation preserving if and only if the following number is even:
 \begin{equation*}
  (\mu(x)+1)\mu(y)+1 = (\mu(y)+2)\mu(y)+1 \equiv \mu(y)^2+1 \equiv \mu(y)+1 \equiv \mu(x) \ .
 \end{equation*}
 In terms of oriented intersection numbers, this implies that
 \begin{equation*}
  \algint \widehat{\MM}(x,y,f) = (-1)^{\mu(y)+1} \  \algint \widehat{\MM}(y,x,-f) =(-1)^{\mu(x)} \  \algint \widehat{\MM}(y,x,-f) \ .
 \end{equation*} 
 \end{remark}

We continue with the abovementioned investigation of the behaviour of Morse-theoretic gluing maps with respect to the chosen orientations on the trajectory spaces, starting with proper definitions of gluing maps. 

\begin{prop}[{\cite[Lemma 4.1]{SchwarzEqui}}] \index{orientations!of unstable and stable manifolds}
\label{PropGluingUnstable}
 Let $x,y \in \Crit f$ with $\mu(x) =\mu(y)+1$ and let $V \subset \WW^u(y)$ be an open subset. There exist $\rho_V>0$ and a map 
 $$G: \widehat{\MM}(x,y) \times V \times [\rho_V,+\infty) \to \WW^u(x)  $$
 with the following properties:
 \begin{enumerate}[(i)]
  \item $G$ is a smooth orientation-preserving embedding,
  \item if $\{r_n\}_{n \in \NN}$ is a sequence in $[\rho_V,+\infty)$ that diverges against $+\infty$, then $\left\{G\left(\hat\gamma,\gamma_0,r_n\right) \right\}_{n \in \NN}$ will converge geometrically against $\left(\hat{\gamma},\gamma_0\right)$ for all $\hat\gamma \in \widehat\MM(x,y)$ and $\gamma_0 \in V$. 
 \end{enumerate} 
 We will call such a map a \emph{gluing map for negative half-trajectories}.
\end{prop}

See also \cite[Proposition 2.34]{Schaetz} for a detailed proof of this existence result.

\begin{prop} \index{orientations!of unstable and stable manifolds}
\label{PropGluingStable}
 Let $x_0,x_1 \in \Crit f$ with $\mu(x_0) =\mu(x_1)+1$ and let $V \subset \WW^s(x_0)$ be open. There exist $\rho_V>0$ and a map $G: V \times  \widehat{\MM}(x_0,x_1)\times [\rho_V,+\infty) \to \WW^s(x_1)$ with the following properties:
 \begin{enumerate}[(i)]
  \item $G$ is a smooth embedding which is orientation-preserving if and only if $\mu(x_0)$ is even,
  \item if a sequence $\{r_n\}_{n \in \NN}$ in $[\rho_V,+\infty)$ diverges against $+\infty$, then $\left\{G\left(\gamma_0,\hat\gamma,r_n\right) \right\}_{n \in \NN}$ will converge geometrically against $\left(\gamma_0,\hat{\gamma}\right)$ for all $\gamma_0 \in V$ and $\hat\gamma \in \widehat\MM(x_0,x_1)$. 
 \end{enumerate} 
 We will call such a map a \emph{gluing map for positive half-trajectories}.
\end{prop}
\begin{proof}
Assume w.l.o.g. that $V= \WW^s(x_0)$ and put $\rho_0 := \rho_{\WW^s(x_0)}$. We will show the claim by identifying the stable manifolds of $f$ with the unstable manifolds of $-f$ as we did earlier in this section. Afterwards, we apply part 1 of the proposition. For clarity's sake, we will write the domain of $G_0$ as $\WW^s(x_0,f)\times \widehat{\MM}(x_0,x_1,f) \times [\rho_0,+\infty)$. 
 
 Let $G': \MM(x_1,x_0,-f)\times \WW^u(x_0,-f) \times [\rho_0,+\infty) \to \WW^u(x_1,-f)$ be a gluing map for negative half-trajectories with respect to the Morse function $-f$. Define $G$ as the map that makes the following diagram commputative: 
 \begin{equation*}
  \begin{CD}
   \WW^s(x_0,f)\times \widehat{\MM}(x_0,x_1,f) \times [\rho_0,+\infty) @>{G}>> \WW^s(x_1,f) \\
   @V{\cong}V{\sigma}V @V{\cong}V{\varphi_s}V \\
   \widehat\MM(x_1,x_0,-f)\times \WW^u(x_0,-f) \times [\rho_0,+\infty) @>{G'}>> \WW^u(x_1,-f) 
  \end{CD}
 \end{equation*}
 where $\sigma$ is the diffeomorphism defined as the composition of the transposition of the first two factors and the product of the maps $\varphi_s$ and $\widehat\psi$ from above. 
 Since $G'$ is a smooth embedding and the vertical maps are diffeomorphisms, $G$ is a smooth embedding as well. Moreover, one checks without difficulties that geometric convergence is preserved under the vertical diffeomorphisms, such that property (ii) of $G'$ implies property (ii) for $G$. It remains to discuss the behaviour of $G$ with respect to orientations.
 
 By construction, the map $\varphi_s: \WW^s(x_1,f) \to \WW^u(x_1,-f)$ is orientation-preserving. Moreover, the bottom gluing map is orientation-preserving by Proposition \ref{PropGluingUnstable}. The commutativity of the diagram thus implies that the sign of $G$ is given by the sign of $\sigma$, which we compute as follows. 
 
 We first transpose $\widehat{\MM}(x_0,x_1,f)$ and $\WW^s(x_0,f)$. Since $\widehat{\MM}(x_0,x_1,f)$ is zero-dimensional, this transposition is orientation-preserving. Afterwards, we apply $\varphi_s$ to the factor $\WW^s(x_0,f)$ and $\widehat\psi$ to $\widehat{\MM}(x_0,x_1,f)$. As previously mentioned, $\varphi_s$ is orientation-preserving and $\widehat{\psi}$ preserves orientations if and only if $\mu(x_0)$ is even, see Remark \ref{SignPsiHatZeroDim}. Hence, $\sign \sigma = (-1)^{\mu(x_0)}$, which shows that $G$ has the desired properties. 
\end{proof}

The existence of gluing maps for positive half-trajectories follows immediately from the existence of gluing maps for negative half-trajectories by applying Proposition \ref{PropGluingUnstable} to the Morse function $-f$. 

\begin{prop} \index{orientations!of finite-length trajectory spaces}
\label{PropGluingFinite}
 Let $x \in \Crit f$ and $V_0 \subset \WW^s(x)$ and $V_1 \subset \WW^u(x)$ be open subsets. There exists $\rho_{V_0,V_1}>0$ and a map $G: [\rho_{V_0,V_1},+\infty) \times V_0 \times V_1 \to \MM(f,g)$ with the following properties:
 \begin{enumerate}[(i)]
  \item $G$ is a smooth embedding,
  \item if $\{r_n\}_{n \in \NN}$ is a sequence in $[\rho_V,+\infty)$ that diverges against $+\infty$, then $\left\{G\left(r_n,\gamma_0,\gamma_1\right) \right\}_{n \in \NN}$ will converge geometrically against $\left(\gamma_0,\gamma_1\right)$ for all $\gamma_0 \in V_0$ and $\gamma_1 \in V_1$, 
  \item for all $\gamma_0 \in V_0$ and $\gamma_1\in V_1$ the map $[\rho_{V_0,V_1},+\infty) \to [0,+\infty)$, $r \mapsto \ell(r,\gamma_0,\gamma_1)$, is strictly increasing, where $\ell(r,\gamma_0,\gamma_1)$ denotes the interval length of the finite-length trajectory $G(r,\gamma_0,\gamma_1)$. 
 \end{enumerate} 
 We will call such a map a \emph{gluing map for finite-length trajectories}.
\end{prop}

 The existence of gluing maps for finite-length trajectories is discussed in \cite[Chapter 18]{KronheimerMrowka}. We want to investigate their behaviour with respect to orientations. 
 
 \begin{prop}
 \label{PropGluingFiniteOrient}
  Let $x \in \Crit f$, $V_0 \subset \WW^s(x)$ and $V_1 \subset \WW^u(x)$ be open subsets and let 
  $$G:[\rho_{V_0,V_1},+\infty) \times V_0 \times V_1 \to \MM(f,g)$$
  be a gluing map for finite-length trajectories. Then $G$ is orientation-preserving if and only if $n+1$ is even. 
 \end{prop}

\begin{proof}
Assume w.l.o.g. that $V_0=\WW^s(x)$ and $V_1 = \WW^u(x)$ and put $\rho_0 := \rho_{\WW^s(x),\WW^u(x)}$. It suffices to show the claim at a fixed point $(\rho,\gamma_1,\gamma_2) \in [\rho_0,+\infty) \times \WW^s(x) \times \WW^u(x)$. Let $\gamma_1$ and $\gamma_2$ be the constant trajectories given by 
 \begin{equation*}
  \gamma_1(s)=x \quad \forall s \in [0,+\infty) \ , \quad \gamma_2(s)=x \quad \forall s \in (-\infty,0] \ ,
 \end{equation*}
 and put $\ell: [\rho_0,+\infty) \to [0,+\infty)$, $\ell(r):= \ell(r,\gamma_1,\gamma_2)$, given as in Proposition \ref{PropGluingFinite}. By definition of a gluing map for finite-length trajectories, $\ell$ is smooth and strictly increasing, i.e. orientation-preserving. From this observation and the properties of gluing maps, it follows for the constant trajectories $\gamma_1$ and $\gamma_2$ and for all $\rho \in [\rho_0,+\infty)$ that $G(\rho,\gamma_1,\gamma_2)$ is the constant trajectory $[0,\ell(\rho)]\to M$, $t \mapsto x$. 
 
One observes that the differential of $G$ at $(\rho,\gamma_1,\gamma_2)$ is a map 
\begin{equation*}
DG_{(\rho,\gamma_1,\gamma_2)}: T_\rho [\rho_2,+\infty) \oplus T_{\gamma_1}\WW^s(x) \oplus T_{\gamma_2}\WW^u(x) \oplus \to T_{G(\rho,\gamma_1,\gamma_2)} \MM(f,g) 
\end{equation*}
that makes the following diagram commutative:
 \begin{equation*}
  \begin{CD}
   T_\rho [\rho_0,+\infty) \oplus T_{\gamma_1}\WW^s(x) \oplus T_{\gamma_2}\WW^u(x) @>{DG_{(\rho,\gamma_1,\gamma_2)}}>> T_{G(\rho,\gamma_1,\gamma_2)} \MM(f,g) \\
   @VV{\cong}V @VV{\cong}V \\
   T_\rho [\rho_0,+\infty) \oplus T_xW^s(x) \oplus T_xW^u(x) @>{\cong}>> T_{\ell(\rho)} [0,+\infty) \oplus T_x M \ , \\
  \end{CD}
 \end{equation*}
 Here, the bottom map is induced by $D\ell_\rho: T_{\rho}[\rho_2,+\infty) \to T_{l(\rho)} [0,+\infty)$ and the identification $T_xW^s(x) \oplus T_xW^u(x)\cong  T_x M$, which are both orientation-preserving by definition of the orientation on $W^s(x)$ and the fact that $\ell$ is orientation-preserving. The left-hand vertical map is given by the differentials of the evaluation maps and is therefore orientation-preserving. The right-hand vertical map is given by the differential of the diffeomorphism $\MM(f,g) \stackrel{\cong}{\to} [0,+\infty) \times M$, $(l,\gamma)\mapsto(l,\gamma(0))$, which has sign $(-1)^{n+1}$ by definition of the orientation on $\MM(f,g)$. 
 
 Consequently, the differential of the gluing map has sign $(-1)^{n+1}$ at $(\rho,\gamma_1,\gamma_2)$ and thus at every point since its domain is connected.  
\end{proof}

\begin{remark}
 In \cite[Chapter 3]{Schwarz}, the notion of coherent orientations is introduced to define orientations and algebraic counts of moduli spaces of unparametrized Morse trajectories. This approach has the big advantage that an orientation on the ambient manifold $M$ is no longer required, which makes it possible to define Morse homology with integer coefficients on non-orientable manifolds. 
 In \cite[Appendix B]{Schwarz} it is shown that on an oriented closed manifold, both approaches to orienting trajectory spaces are indeed equivalent. We will not discuss coherent orientations in this article any further.
\end{remark}

We conclude this section by a general statement on orientations of transverse intersections that is not directly connected to the above, but will be of great use in later sections. 

For two finite-dimensional oriented vector spaces $V_1$ and $V_2$ we let \emph{the direct sum orientation} on $V_1\oplus V_2$ be given by demanding that a positive basis of $V_1$ followed by a positive basis of $V_2$ is a positive basis of $V_1\oplus V_2$. \bigskip 

The following statement is shown by standard methods from elementary differential topology. We leave the details to the reader.

\begin{theorem} 
\label{TheoremOrientDiffeoTransverseint}
 Let $M_1$, $M_2$, $P_1$ and $P_2$ be smooth oriented manifolds and $N_1 \subset P_1$ and $N_2 \subset P_2$ be oriented submanifolds. Let $f_1: M_1 \to P_1$ and $f_2:M_2 \to P_2$ be smooth maps with $f_1 \tv N_1$ and $f_2 \tv N_2$.
 Let $S_1 := f_1^{-1}(N_1)$ and $S_2 := f_2^{-1}(N_2)$ be equipped with the orientations induced by the transverse intersections. 
 
 If there are smooth diffeomorphisms $\varphi: M_1 \to M_2$ and $\psi: P_1 \to P_2$, such that 
 \begin{itemize}
  \item $\psi$ restricts to a diffeomorphism $\psi|_{N_1}: N_1 \stackrel{\cong}{\to} N_2$,
  \item the following diagram commutes: 
     $\qquad \begin{CD}
    M_1 @>{\varphi}>{\cong}> M_2 \\
    @V{f_1}VV @V{f_2}VV \\
    P_1 @>{\psi}>{\cong}> P_2 \ , 
   \end{CD}$
 \end{itemize}
then $\varphi$ restricts to a diffeomorphism $\varphi|_{S_1}:S_1 \stackrel{\cong}{\to} S_2$ with 
$$\sign (\varphi|_{S_1})= \sign \varphi \cdot \sign \psi \cdot \sign(\psi|_{N_1}) \ . $$
\end{theorem}

\subsection{Orientations of perturbed Morse trajectory spaces}
\label{AppendixOrientPertubedTraj}

We will use the similarity between spaces of perturbed and unperturbed trajectories to equip spaces of perturbed Morse trajectories with orientations as well using the orientations on Morse trajectory spaces that we constructed in the previous section. 

\bigskip

For $Y \in \XX_-(M)$ and $x \in \Crit f$ we define a map
$$ \varphi_Y: W^-(x,Y) \to \WW^u(x,f,g) \ ,  \qquad 
 \gamma \mapsto \left( s \mapsto \begin{cases}
                  \gamma(s) & \text{if } \ s \leq -1 \ , \\
                   \phi_{s+1}(\gamma(-1)) & \text{if } \ s \in (-1,0] \ ,
                 \end{cases} \right) $$
where $\phi$ denotes the negative gradient flow of $f$ with respect to $g$. 

\begin{prop} \index{perturbed negative half-trajectory}
\label{NegPertUnpert}
 For every $Y \in \XX_-(M)$, the map $\varphi_Y: W^-(x,Y) \to \WW^u(x,f,g)$ is a diffeomorphism of class $C^{n+1}$.
\end{prop}

\begin{proof}
 Using the unique existence of solutions of ordinary differential equations, one checks that the map $\varphi_Y$ is indeed well-defined and bijective for every $Y \in \XX_-(T)$. In the following, we will identify the map with a composition of $(n+1)$-times differentiable maps to show its differentiability properties. Consider the Banach manifold
 \begin{equation*}
  \widetilde{\PP}_-(x) := \Bigl\{ \gamma \in H^1\left(\bRR_{\leq -1},M\right) \ \Big| \ \lim_{s \to -\infty} \gamma(s) = x \Bigr\} \ ,
 \end{equation*}
 where $\bRR_{\leq -1} := \{-\infty\} \cup (-\infty,-1]$, equipped with a smooth structure induced by the one on $\bRR_{\leq 0}$. The restriction map $r: \PP_-(x) \to \widetilde{\PP}_-(x)$, $\gamma \mapsto \gamma|_{\bRR_{\leq -1}}$, is obviously a smooth map of Banach manifolds. Moreover, again by the unique existence of solutions of ordinary differential equations, both of the restrictions
 $$ r|_{W^-(x,Y)}: W^-(x,Y) \to  \widetilde{\PP}_-(x) \ , \qquad r|_{\WW^u(x,f,g)}:\WW^u(x,f,g) \to \widetilde{\PP}_-(x) \ ,$$
 are injective. Since by definition of $\XX_-(M)$, the time-dependent vector field $Y$ vanishes in $(s,x)$ if $s\leq -1$, the restriction $\gamma|_{(-\infty,-1]}$ satisfies the negative gradient flow equation of $f$ with respect to $g$. It immediately follows that 
$ \im r|_{W^-(x,Y)} = \im r|_{\WW^u(x,f,g)}$. In terms of the restriction maps, we can therefore write $\varphi_Y$ as the following well-defined composition:
 \begin{equation*}
  \varphi_Y = \left(r|_{\WW^u(x,f,g)} \right)^{-1} \circ r|_{W^-(x,Y)} \ .
 \end{equation*}
 Thus, $\varphi_Y$ is a composition of a smooth map and a map of class $C^{n+1}$ and is therefore itself of class $C^{n+1}$. Moreover, its inverse is given by $\varphi_Y^{-1} = \left(r|_{W⁻(x,Y)} \right)^{-1} \circ r|_{\WW^u(x,f,g)}$, which is a map of class $C^{n+1}$ by the very same argument. Therefore, $\varphi_Y$ is a diffeomorphism of class $C^{n+1}$.
\end{proof}

\bigskip

{\it For all $x \in \Crit f$ and $Y \in \XX_-(M)$, we equip $W^-(x,Y)$ with the unique orientation for which $\varphi_Y: W^-(x,Y) \stackrel{\cong}{\to} \WW^u(x)$ is orientation-preserving.} \index{orientations!of $W^-(x,Y)$} \\

We define a map for positive half-trajectories along the same lines. Let $x \in \Crit f$, $Y \in \XX_+(M)$ and define
$$ \varphi_Y: W^+(x,Y) \to \WW^s(x,f,g) \ , \qquad  \gamma \mapsto \left( s \mapsto \begin{cases}
                     \phi_{s-1}(\gamma(1)) & \text{if } \ s \in [0,1) \ , \\
                  \gamma(s) & \text{if } \ s \geq 1 \ .
                 \end{cases} \right) $$
In the same way as in Proposition \ref{NegPertUnpert}, we show that in the positive case, the map $\varphi_Y$ is a diffeomorphism of class $C^{n+1}$ for any $Y \in \XX_+(M)$. \index{perturbed positive half-trajectory}

\bigskip

{\it For all $x \in \Crit f$ and  $Y \in \XX_+(M)$, we equip $W^+(x,Y)$ with the unique orientation for which $\varphi_Y:W^+(x,Y) \stackrel{\cong}{\to} \WW^s(x)$ is orientation-preserving.} \index{orientations!of $W^+(x,Y)$}

\bigskip

We recall from Proposition \ref{FiniteLengthManifold} that for $Y \in \XX_0(M)$ there is a diffeomorphism
\begin{equation*}
 \psi_Y: \MM(Y) \to \MM(f,g) \ ,
 \end{equation*} 
which is constructed in the spirit of the maps $\varphi_Y$. 

\bigskip
 
 {\it For every $Y \in \XX_0(M)$, we equip $\MM(Y)$ with the unique orientation for which the diffeomorphism $\psi_Y: \MM(Y) \stackrel{\cong}{\to} \MM(f,g)$ is orientation-preserving.} \index{orientations!of $\MM(Y)$}

\subsection{Ordering the internal edges of ribbon trees}
\label{AppendixOrderingEdges}

Let $T \in \BinTree_d$ for some $d \geq 3$, such that $E_{int}(T)\neq \emptyset$. Given $e \in E_{int}(T)\cup \{e_0(T)\}$, the set $\left\{ f \in E(T) \ | \ \vin(f) = \vout(e) \right\}$ has precisely two elements since $T$ is a binary tree. We denote them by $f_1(e)$ and $f_2(e)$, which are uniquely defined by demanding that there exist $i,j \in \{1,2,\dots,d\}$ with $i <j$, such that 
$$f_1(e) \in E(P_i(T)) \quad \text{and} \quad f_2(e) \in E(P_j(T)) \setminus E(P_i(T)) \ , $$
and consider the maps 
$$f_1,f_2:E_{int}(T) \cup \{e_0(T)\} \to E(T) \ , \qquad e \mapsto f_i(e) \ \ \text{for  } i \in \{1,2\} \ . $$
Intuitively, drawing the trees as before, $f_1(e)$ denotes the left-hand edge and $f_2(e)$ the right-hand edge emanating from $\vout(e)$. For each $j \in \NN$ for which it is well-defined we let $f^j_1(e) := (f_1 \circ f_1 \circ \dots \circ f_1)(e)$ denote the $j$-fold iterate of $f_1$.

In addition we define a map 
$$h: E(T) \setminus \{e_0(T)\} \to E(T) \ , $$
where $h(e)$ is the unique edge with $\vout(h(e))=\vin(e)$ for every $ e \in E(T)\setminus \{e_0(T)\}$. For every $j \in \NN$ for which it is well-defined we denote the $j$-fold iterate of $h$ by $h^j$.

We want to define a total ordering on $E_{int}(T)$. For this purpose, we will inductively label the edges as $\{g_1,g_2,\dots,g_{d-2}\}=E_{int}(T)$ and define the ordering by demanding that $g_i$ defines the $i$-th element of $E_{int}(T)$. We first give a formal definition of the ordering procedure and afterwards describe it in a more informal and intuitive way. \\

Let $g_0 \in E(T)$ be the unique edge with $\vout(g_0)=\vin(e_1(T))$. We put 
$$g_1 := \begin{cases}
          g_0 & \text{if} \ \ g_0 \in E_{int}(T) \ , \\
          f_2(g_0) & \text{if} \ \ g_0 \notin E_{int}(T) \ . 
         \end{cases} $$
Since $d \geq 3$, it holds that $g_1 \in E_{int}(T)$.  \bigskip 

Assume that we have already labelled edges $\{g_1,g_2,\dots,g_{i-1}\}$ for some $i \in \{2,3,\dots,d-2\}$. 
\begin{itemize}
 \item If $f_1(g_{i-1}) \in E_{int}(T)\setminus \{g_1,g_2,\dots,g_{i-1}\}$, then we put $$g_i := f_1^k(g_{i-1})\ ,$$
 where $k = \max \{j \in \NN \ | \ f^j_1(g_{i-1}) \in E_{int}(T)\}$. 
 \item Assume that $f_1(g_{i-1}) \in E_{ext}(T) \cup \{g_1,\dots,g_{i-1}\}$.
 \begin{itemize}
  \item If $f_2(g_{i-1}) \in E_{int}(T)$, then put $g_i := f_2(g_{i-1})$. 
  \item If $f_2(g_{i-1}) \in E_{ext}(T)$, then put $g_i:=h^k(g_{i-1})$, where $$k = \min \{j \in \NN \ | \ h^j(g_{i-1}) \in E_{int}(T)\setminus \{g_1,\dots,g_{i-1}\}\} \ . $$  
 \end{itemize}
\end{itemize}

Less formally speaking, the above ordering can be described in simpler words. We first consider the edge for which $e_1(T)$ is the (left-hand) edge attached to its outgoing vertex. If this edge is internal, we define it as the first edge of the ordering. 

If this edge is external, it is necessarily given by $e_0(T)$. Since $T$ is binary and has at least three leaves, it follows that $f_2(e_0(T))$ is internal and we define $f_2(e_0(T))$ as the first edge. Intuitively, if the left-hand subtree emanating from $\vout(e_0(T))$ consists of a single external edge, then we define the edge of the right-hand subtree that is closest to the root of $T$ as the first one. 

Assuming that we have labelled the first $i-1$ edges, we want to define the $i$-th one, for which we have to distinguish several cases. 

If the left-hand edge emanating from the $(i-1)$-st edge, call it $e$, is internal and has not been labelled yet, then we check if the left-hand edge emanating from $e$ has the very same property. We iterate this procedure and follow the unlabelled left-hand edges until we have reached the bottom of the tree. The last internal edge that we obtain is this way is then defined as the $i$-th edge of the ordering. 

If the left-hand edge emanating from the $(i-1)$-st edge is external or already labelled, then we consider the right-hand edge emanating from the $(i-1)$-st edge. If that edge is internal, then we define it as $g_i$. If not, then we ``backtrack'' along the tree. We check if the edge ``on top'' of $g_{i-1}$, i.e. the edge $g$ whose outgoing vertex is attached to the incoming vertex of $g_{i-1}$, is labelled or not. If $g$ is unlabelled we let it be the $i$-th edge. If $g$ is labelled, we continue with the edge whose outgoing vertex coincides with $\vin(g)$ and check if it is labelled.  We iterate this procedure until we arrive at an unlabelled edge. This edge is then defined to be the $i$-th one. 

\begin{definition} 
\label{DefCanonicalOrdering}
 Let $T \in \BinTree_d$ for some $d \geq 3$. We call the ordering of $E_{int}(T)$ which is obtained by the above procedure \emph{the canonical ordering of $E_{int}(T)$}.
\end{definition}

\begin{figure}[h]
 \centering
 \includegraphics[scale=0.8]{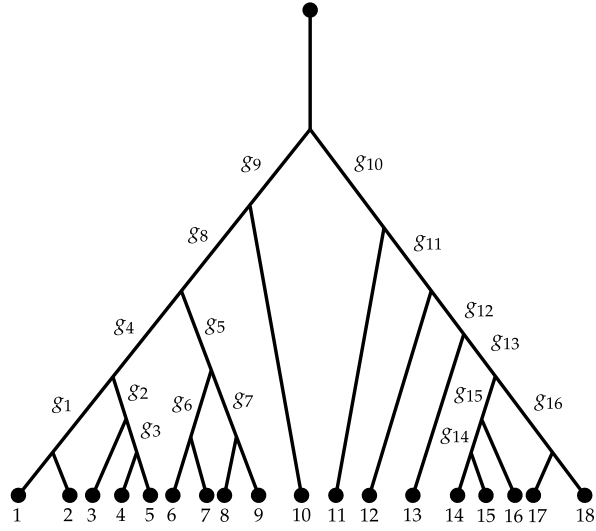}
 \caption{An example of the canonical ordering of the internal edges of a binary tree with $d=18$ leaves.}
 \label{FigureLabelledTree}
\end{figure}

See Figure \ref{FigureLabelledTree} for an example of the canonical ordering of a binary tree. The canonical ordering is described as an implementation of a depth-first search algorithm for binary trees. Thus we refrain from a proof of the following statement. 

\begin{prop}
 The canonical ordering is a well-defined total order of $E_{int}(T)$ for every $T \in \BinTree_d$ with $d \geq 3$. 
\end{prop}

We will next investigate how the canonical orientations behave with respect to the decomposition operation of binary trees and the splitting operation along an internal edge. For this purpose, we introduce the following notion.

\begin{definition}\index{left- and right-handed eges}
 Let $d \in \NN$ with $d \geq 2$, $T \in \BinTree_d$ and $e \in E(T)\setminus \{e_0(T)\}$. The edge $e$ is called \emph{left-handed} if there exists $g \in E(T)$ with $e = f_1(g)$ and \emph{right-handed} if there exists $h \in E(T)$ with $e=f_2(h)$. Note that every element of $E(T) \setminus \{e_0(T)\}$ is either left-handed or right-handed. 
 \end{definition}

\begin{prop}
\label{PropEdgeOrderingBreaking}
 Let $d \in \NN$ with $d \geq 3$, $T \in \BinTree_d$ and $E_{int}(T)=\{g_1,g_2,\dots,g_{d-2}\}$, where the edges are labelled according to the canonical ordering. Let $i \in \{1,2,\dots,d\}$ and $l \in \{1,2,\dots,d-i\}$. If $e \in E_{int}(T)$ is of type $(i,l)$, then 
 $$E_{int}(T^e_2) = \left\{g_i,g_{i+1},\dots,g_{i+l-2} \right\} , \qquad e = \begin{cases}
                                                                               g_{i+l-1} & \text{if $e$ is left-handed,} \\
                                                                               g_{i-1} & \text{if $e$ is right-handed,}
                                                                              \end{cases}$$
 and the canonical ordering of $E_{int}(T^e_2)$ coincides with the ordering which is induced by the ordering of $E_{int}(T)$. 
\end{prop}
\begin{proof}
 This is derived from the definition of the canonical ordering by an elementary line of argument. Thus, we omit the details. 
\end{proof}

\begin{theorem}
\label{TheoremSigntaue}
 Let $d \in \NN$ with $d\geq 3$, $T \in \RTree_d$, $i \in \{1,2,\dots,d\}$, $l \in \{1,2,\dots,d-i\}$ and $e \in E_{int}(T)$ be of type $(i,l)$. Let $E_{int}(T)$, $E_{int}(T^e_1)$, $E_{int}(T^e_2)$ be equipped with the canonical orderings. Then the sign of the permutation 
 $$\tau_e: \{e\} \times E_{int}(T^e_1) \times E_{int}(T^e_2) \to E_{int}(T)$$
 is given by 
 $$\sign \tau_e = \begin{cases}
 (-1)^{(d-i)l+d-1} & \text{if $e$ is left-handed,} \\
 (-1)^{(d-i)l+d} & \text{if $e$ is right-handed.} 
 \end{cases}$$
\end{theorem}
\begin{proof}
 If $e$ is of type $(i,l)$, then by Proposition \ref{PropEdgeOrderingBreaking} it will hold that
 \begin{align*}
 E_{int}(T^e_2)&= \{g_i,g_{i+1},\dots,g_{i+l-2}\} \ , \ \
 E_{int}(T^e_1)= \begin{cases}
 \{g_1,\dots,g_{i-2},g_{i+l-1},\dots,g_{d-2}\} & \text{if} \ \ e = g_{i-1} \ , \\
 \{g_1,\dots,g_{i-1},g_{i+l},\dots,g_{d-2}\} & \text{if} \ \ e = g_{i+l-1} \ .
 \end{cases} 
 \end{align*}
 If $e=g_{i+l-1}$, i.e. if $e$ is left-handed, then $\tau_e= \tau_2 \circ \tau_1$, where $\tau_1$ permutes $e$ past $(g_1,\dots,g_{i-1})$ and $\tau_2$ permutes $(g_i,\dots,g_{i+l-2})$ past $(g_{i+l-1},\dots,g_{d-2})$ with respect to the orderings. Hence, if $e=g_{i+l-1}$, the sign of $\tau_e$ is given by the parity of 
$$ i-1 + (l-1)(d-2-(i+l-2)) \equiv d-1+l(d-i) \ . $$
 Analogously, if $e=g_{i-1}$, i.e. if $e$ is right-handed, then $\tau_e$ permutes $e$ past $(g_1,\dots,g_{i-2})$ and  $(g_i,\dots,g_{i+l-2})$ past $(g_{i+l-1},\dots,g_{d-2})$ with respect to the orderings, such that $\sign \tau_e$ is given by the parity of 
 $$i-2 + (l-1)(d-2-(i+l-2)) \equiv (d-i)l+d \ . $$
\end{proof}

So far, we have defined orderings of internal edges of \emph{binary} trees. We will generalize canonical orderings to edges of arbitrary ribbon trees starting from the binary case. One checks that for every $T \in \RTree_d$ there exist $\tilde{T} \in \BinTree_d$ and $F \subset E_{int}\left(\tilde{T}\right)$, such that 
$$T \cong \tilde{T}/F \ . $$

\begin{definition}
 Let $d \geq 3$ and $T \in \RTree_d$. Let $\tilde{T} \in \BinTree_d$ and $F \subset E_{int}\left(\tilde{T}\right)$ with $T \cong \tilde{T}/F$. Then $E_{int}(T) = E_{int}\left(\tilde{T}\right)\setminus F$ and we define \emph{the canonical ordering of $E_{int}(T)$} as the ordering which is induced by the canonical ordering of $E_{int}\left(\tilde{T}\right)$.
\end{definition}

By elementary arguments from graph theory, one shows that the canonical ordering of internal edges of ribbon trees is well-defined, i.e. that for every $T \in \RTree_d$ there exists a pair $(\tilde{T},F)$ as in the previous definition and that the ordering is independent of the choice of $(\tilde{T},F)$. \bigskip 

The following quantity will be used in Section \ref{SectionOrientPerturbedMRT}. 

\begin{definition} 
\label{DefDexterity} 
 Let $d \in \NN$ with $d\geq 2$ and $T \in \BinTree_d$. The \emph{dexterity of $T$} is the number $r(T) \in \NN_0$ defined by 
 $$r(T) = \left|\left\{g \in E_{int}(T) \ | \ g \ \text{is right-handed} \right\} \right| \  . $$
\end{definition}

\begin{prop}
\label{PropDexterities}
 Let $d \in \NN$ with $d \geq 3$. 
 \begin{enumerate}[a)]
  \item Let $T_1,T_2 \in \BinTree_d$, $e_1 \in E_{int}(T_1)$ and $e_2 \in E_{int}(T_2)$ be given such that $T_1/e_1 \cong T_2/e_2$.
 Then $$|r(T_1)-r(T_2)|=1 \ . $$
 \item Let $T \in \RTree_d$ and $e \in E_{int}(T)$. Then 
 \begin{equation*}
  r(T^e_1)+r(T^e_2) = \begin{cases}
                       r(T) & \text{if $e$ is left-handed,} \\
                       r(T)-1 & \text{if $e$ is right-handed.}
                      \end{cases}
 \end{equation*}
 \end{enumerate}
\end{prop}
\begin{proof}
\begin{enumerate}[a)]
 \item Apparently, precisely one of the two edges $e_1$ and $e_2$ is right-handed. Consequently, one of the two trees $T_1$ and $T_2$ has an additional right-handed edge. 
 \item This is obvious.
 \end{enumerate}
\end{proof}

\subsection{Orientations on moduli spaces of perturbed Morse ribbon trees}
 \label{SectionOrientPerturbedMRT}
  
{\it Throughout this section, we assume that $E_{int}(T)$ is equipped with the canonical ordering for all $T \in \RTree_d$ and $d \geq 3$. Assume that such a $T$ is chosen and that we consider a product of the form $\prod_{e\in E_{int}(T)} Q_e$, where $Q_e$ is an oriented manifold for every $e \in E_{int}(T)$. Throughout this section, we assume that such a product is oriented with the product orientation, given by 
$$\prod_{e\in E_{int}(T)}Q_e = Q_{g_1}\times Q_{g_2} \times \dots \times Q_{g_{k(T)}} \ , $$
where the elements of $E_{int}(T)=\{g_1,g_2,\dots,g_{k(T)}\}$ are labelled according to the canonical ordering.} 
 
 \bigskip
 
 Let $d \geq 2$, $T \in \RTree_d$ and $\fatY \in \XX(T)$. As a first step towards orientations on the spaces $\Acal^d_{\fatY}(x_0,x_1,\dots,x_d,T)$, we want to equip the spaces  
 \begin{equation*}
 \MM_{\fatY}^d(x_0,x_1,\dots,x_d,T) \ , \quad \text{for} \ x_0,x_1,\dots,x_d \in \Crit f \ , 
 \end{equation*}
 from Section \ref{CompactificationsOneDimAinfty} with orientations. Recall that in shorthand notation
 \begin{align*}
  &\MM^d_{\fatY}(x_0,x_1,\dots,x_d,T)= \left\{ \left. \left(\gamma_0, (l_e,\gamma_e)_{e \in E_{int}(T)},\gamma_1,\dots,\gamma_d\right) \ \right| \ \gamma_0 \in W^-(x_0,Y_0((l_e)_e)), \right. \\
  &\qquad \qquad \qquad \qquad \quad l_e>0, \ \ (l_e,\gamma_e) \in \MM(Y_e((l_f)_{f \neq e}))  \ \forall e , \ \  \gamma_i \in W^+(x_i,Y_i((l_e)_{e})) \ \forall i  \Big\} \ ,
 \end{align*}
 with $\fatY=(Y_0,(Y_e)_{e \in É_{int}(T)},Y_1,\dots,Y_d)$.
 
As we discussed in Section \ref{CompactificationsOneDimAinfty}, the space $\MM^d_{\fatY}(x_0,x_1,\dots,x_d,T)$ is a manifold of class $C^{n+1}$ and one checks without difficulties that the following map is a diffeomorphism of manifolds of class $C^{n+1}$:
\begin{align*}
 &F_T:\MM^d_{\fatY}(x_0,x_1,\dots,x_d,T) \stackrel{\cong}{\to} \WW^u(x_0) \times \prod_{e \in E_{int}(T)} \MM(f,g) \times \WW^s(x_1) \times \dots \times \WW^s(x_d) \ , \\
 &\ \ \ \left(\gamma_0,(l_e,\gamma_e)_e,\gamma_1,\dots,\gamma_d \right) \mapsto \\ 
 &\qquad \qquad \left(\varphi_{Y_0((l_e)_e)}(\gamma_0), \left(\psi_{Y_e((l_f)_{f\neq e})}(l_e,\gamma_e) \right)_{e \in E_{int}(T)}, \varphi_{Y_1((l_e)_e)}(\gamma_1),\dots, \varphi_{Y_d((l_e)_e)}(\gamma_d) \right) \ .
\end{align*}

 \textit{For all $d \geq 2$, $T \in \RTree_d$ and $x_0,x_1,\dots,x_d \in \Crit f$, we equip $\MM^d_{\fatY}(x_0,x_1,\dots,x_d,T)$ with the unique orientation with which $F_T$ becomes an orientation-preserving diffeomorphism.}  \index{orientations!of $\MM^d_{\fatY}(x_0,x_1,\dots,x_d,T)$}
 
 \bigskip

 As a second step towards the construction of orientations on moduli spaces of type $\Acal^d_{\fatY}(x_0,x_1,\dots,x_d,T)$, we investigate the orientability of the $T$-diagonals. 
 
 \begin{prop}
 \label{DeltaTDiffeo}
  Let $d \geq 2$. For every $T \in \RTree_d$ the canonical ordering of $E_{int}(T)$ induces a diffeomorphism  
  \begin{equation*}
   M^{1+k(T)} {\stackrel{\cong}{\longrightarrow}} \Delta_T \ .
  \end{equation*}  
 \end{prop}

 \begin{proof}
  We identify $M^{k(T)}$ with $M^{E_{int}(T)}$ and $M^{2k(T)}$ with $\left(M^2\right)^{E_{int}(T)}$ according to the canonical ordering and define a map 
  \begin{align*}
  M^{1+k(T)} &\to M^{1+2k(T)+d} \ , \quad
   \left(x_0,(x_e)_{e \in E_{int}(T)}\right) \mapsto \left(x_0,\left(q^e_{in},x_e\right)_{e \in E_{int}(T)},q_1,q_2,\dots,q_d\right) \  ,
  \end{align*}
  where
  \begin{align*}
   q^e_{in} &:= \begin{cases} 
		      x_0 &  \text{if} \ \ \vin(e)=\vout(e_0(T)) \ , \\  
		      x_g & \text{if} \ \ \vin(e)=\vout(g) \ \ \text{for some} \ \ g \in E_{int}(T) \ , 
                \end{cases} \\
   q_i &:=  \begin{cases}
   x_0 & \text{if} \ \ \vin(e_i(T))=\vout(e_0(T)) \ , \\
   x_g  & \text{if} \ \ \vin(e_i(T))=\vout(g) \ \ \text{for some} \ \ g \in E_{int}(T) \ ,
   \end{cases}
  \end{align*}
  for all $e \in E_{int}(T)$ and $i \in \{1,2,\dots,d\}$. One checks without difficulties that this map is well-defined and a diffeomorphism onto its image. Moreover, it follows from the description of $\Delta_T$ in Remark \ref{DeltaTexplicitly} that the image of this map is $\Delta_T$. 
 \end{proof}

 \textit{For all $d \geq 2$ and $T \in \RTree_d$, we equip $\Delta_T$ with the unique orientation which makes the diffeomorphism $M^{1+k(T)} \cong \Delta_T$ from Proposition \ref{DeltaTDiffeo} orientation-preserving.} \index{orientations!of $T$-diagonals}
 
 \bigskip
 
 Next we will use the last results to construct the desired orientations on moduli spaces of perturbed Morse ribbon trees. 
 
 Let $T \in \RTree_d$ and $x_0,x_1,\dots,x_d \in \Crit f$. As we have described in Section \ref{CompactificationsOneDimAinfty}, the space $\Acal^d_{\fatY}(x_0,x_1,\dots,x_d,T)$ is a transverse intersection of manifolds of class $C^{n+1}$, namely 
 \begin{equation*}
  \Acal^d_{\fatY}(x_0,x_1,\dots,x_d,T) = \Eunder_{\fatY}^{-1}\left(\Delta_T\right) \ ,
 \end{equation*} 
 where $\Eunder_{\fatY}: \MM^d_{\fatY}(x_0,x_1,\dots,x_d,T) \to M^{1+2k(T)+d}$ is the endpoint evaluation map defined in (\ref{DefofEunderY}). Hence, if $\fatY$ is a regular perturbation datum, the space $\Acal^d_{\fatY}(x_0,x_1,\dots,x_d,T)$ is the transverse intersection of an oriented manifold with an oriented submanifold of an oriented manifold. 
 
By standard results from differential topology the given orientations induce an orientation on $\Acal^d_{\fatY}(x_0,x_1,\dots,x_d,T)$ in this case, see \cite[Chapter 3]{GuilleminPollack} or \cite[Section 5.6]{BanyagaHurtubise}. 
 
 \bigskip
 
 \textit{For all $x_0,x_1,\dots,x_d \in \Crit f$ and $T \in \RTree_d$ let $\Acal^d_{\fatY}(x_0,x_1,\dots,x_d,T)$ be equipped with the orientation induced by its description as a transverse intersection of oriented manifolds.} \bigskip 
 
 
 In the following, we will focus on the zero-dimensional case and use the notion of oriented intersection numbers, following the approach of \cite[Chapter 3]{GuilleminPollack}. \bigskip 

 We make this explicit for zero-dimensional spaces $\Acal^d_{\fatY}(x_0,x_1,\dots,x_d,T)$. 
 Let $\fatY$ be a regular perturbation datum, $d \geq 2$, $T \in \RTree_d$ and $x_0,x_1,\dots,x_d \in \Crit f$ such that $\Acal^d_{\fatY}(x_0,x_1,\dots,x_d,T)$ is a zero-dimensional manifold and put 
 $$\epsilon_T := \epsilon_{\Acal^d_{\fatY}(x_0,x_1,\dots,x_d,T)}: \Acal^d_{\fatY}(x_0,x_1,\dots,x_d,T) \to \{-1,1\} \ .$$
 Let $\gammaunder \in \Acal^{d}_{\fatY}(x_0,x_1,\dots,x_d,T)$. For suitable numbers $N_1,N_2 \in \NN_0$, we let $(b_1,b_2,\dots,b_{N_1})$ be a positive basis of $T_{\gammaunder}\MM^d_{\fatY}(x_0,x_1,\dots,x_d,T)$ and $(\beta_1,\beta_2,\dots,\beta_{N_2})$ be a positive basis of $T_{\Eunder_{\fatY}(\gammaunder)}\Delta_T$. Then $\epsilon_T\left(\gammaunder\right) = +1$ if and only if  
 \begin{equation}
 \label{inducedbasis}
  \left(\left(D\Eunder_{\fatY}\right)_{\gammaunder}[b_1],\left(D\Eunder_{\fatY}\right)_{\gammaunder}[b_2],\dots,\left(D\Eunder_{\fatY}\right)_{\gammaunder}[b_{N_1}],\beta_1,\beta_2,\dots,\beta_{N_2}\right) 
 \end{equation}
  is a \emph{positive} basis of $T_{\Eunder_{\fatY}\left(\gammaunder\right)}M^{1+2k(T)+d}$.  \bigskip 
  
  Finally, we are in a position to define the coefficients of the higher order multiplications in terms of oriented intersection numbers. 

  \begin{definition}
  Let $d \in \NN$ with $d \geq 2$ and $x_0,x_1,\dots,x_d \in \Crit f$ satisfy
  \begin{equation}
  \label{EqAcalTzerodim}
  \mu(x_0)=\sum_{i=1}^d \mu(x_i)+2-d \ . 
  \end{equation}
  We define $a^d_{\fatY}(x_0,x_1,\dots,x_d) \in \ZZ$ by 
  $$a^d_{\fatY}(x_0,x_1,\dots,x_d) = \sum_{T \in \RTree_d} (-1)^{\sigma(x_0,x_1,\dots,x_d)+r(T)} \orint \Acal^d_{\fatY}(x_0,x_1,\dots,x_d,T) \ , $$ 
  where $r(T)$ denotes the dexterity of $T$ for every $T \in \RTree_d$. 
  \end{definition}
  
 Note that this definition of the coefficients will coincide with the one from Definition \ref{AinftyMorseCoefficients} if we put 
 $$\algint \Acal^d_{\fatY}(x_0,x_1,\dots,x_d,T) := (-1)^{r(T)} \orint \Acal^d_{\fatY}(x_0,x_1,\dots,x_d,T)$$
 for every $T \in \RTree_d$. For critical points satisfying \eqref{EqAcalTzerodim}, it holds that $\Acal^d_{\fatY}(x_0,x_1,\dots,x_d,T)=\emptyset$ whenever $T \notin \BinTree_d$, so it follows that
 $$a^d_{\fatY}(x_0,x_1,\dots,x_d) = \sum_{T \in \BinTree_d} (-1)^{\sigma(x_0,x_1,\dots,x_d)+r(T)} \orint \Acal^d_{\fatY}(x_0,x_1,\dots,x_d,T) \ .$$
 In analogy with the unperturbed case, it is possible to define gluing maps for spaces of perturbed Morse trajectories and more generally for the spaces discussed in Section \ref{NonlocalGeneralizations}, in particular for spaces of perturbed Morse ribbon trees. 
 Since the technical effort behind carrying out the details of the constructions is disproportionate, we will not carry out the analysis of gluing maps for spaces of perturbed Morse trajectories. We note that, as in the unperturbed case, one reduces the question of existence of gluing maps to an application of the Banach Fixed-Point Theorem. Instead of carrying out the details, we will work with the following slightly vague definition that we already give for later purposes.
 
 \begin{definition} \index{geometric gluing map}
 \label{DefGeomGluingMap}
  Let $\Acal$ and $\BB$ be differentiable manifolds whose elements are families of perturbed Morse trajectories of all three types (negative semi-infinite, positive semi-infinite and finite-length).  We call a map $G: [\rho_0,+\infty) \times \BB \to \Acal$ a \emph{geometric gluing map} if it has the following properties: 
  \begin{itemize}
   \item $G$ is a differentiable embedding,
   \item if $\{r_n\}_{n\in \NN}$ is a sequence in $[\rho_0,+\infty)$ diverging to $+\infty$, then the sequence $\left\{G\left(r_n,\gammaunder\right) \right\}_{n \in \NN}$ in $\Acal$ will converge geometrically against $\gammaunder$ for every $\gammaunder \in \BB$. 
  \end{itemize}
  Here, we define geometric convergence by demanding that all component sequences of perturbed Morse trajectories converge or converge geometrically. 
 \end{definition}

\subsection{Proof of Theorem \ref{CoeffRelAinfty}}
\label{AppendixProofTheoremCoeffRel}
 
 Throughout this section, we let $\fatY$ be an admissible perturbation datum. We further let $d \geq 2$ and $x_0,x_1,\dots,x_d \in \Crit f$ be chosen such that 
 \begin{equation}
 \label{AppEqCritOneDim}
  \mu(x_0)=\sum_{q=1}^d \mu(x_q)+3-d \ , 
 \end{equation}
 i.e. such that $\Acal^d_{\fatY}(x_0,x_1,\dots,x_d,T)$ is a one-dimensional manifold for every $T \in \BinTree_d$. We will further make frequent use of the following helpful observation whose proof can e.g. be found in \cite[p. 101]{GuilleminPollack}.
 
 \begin{prop}
 \label{PropOrientTransverseBoundary}
  Let $P$ be a smooth oriented manifold with boundary, $Q$ be a smooth oriented manifold and $N\subset Q$ be a smooth oriented submanifold, both without boundary. Let $f: P \to Q$ be of class $C^1$ and assume that both $f$ and $f|_{\partial M}$ are transverse to $N$, such that $S:= f^{-1}(N)$ is a manifold with boundary. Equip $S$ with the orientation induced by the transverse intersection. 
  
  The boundary orientation on $\partial S$ coincides with the orientation induced by the transverse intersection of $f|_{\partial M}$ with $N$ if and only if $\codim_Q N$ is even. 
  \end{prop}

  In the previous section, we have defined orientations on the spaces $\Acal^d_{\fatY}(x_0,x_1,\dots,x_d,T)$ for fixed $T \in \RTree_d$. If the critical points are chosen to satisfy (\ref{AppEqCritOneDim}), these spaces are at most one-dimensional. In the one-dimensional case their orientations induce orientations of their compactifications $\bAcal^d_{\fatY}(x_0,x_1,\dots,x_d,T)$. We will consider the space $\bAcal^d_{\fatY}(x_0,x_1,\dots,x_d,T)$ as equipped with this extended orientation for every $T \in \RTree_d$. 
 
 \bigskip
 
 The space $\bAcal^d_{\fatY}(x_0,x_1,\dots,x_d)$ is defined in Section \ref{CompactificationsOneDimAinfty} as a quotient space of the disjoint union of the $\bAcal^d_{\fatY}(x_0,x_1,\dots,x_d,T)$ over all binary trees $T$.
 
 To prove Theorem \ref{CoeffRelAinfty}, we need to define an orientation on the space $\bAcal^d_{\fatY}(x_0,x_1,\dots,x_d)$. It is a natural question if the orientations on the spaces $\bAcal^d_{\fatY}(x_0,x_1,\dots,x_d,T)$ induce an orientation on $\bAcal^d_{\fatY}(x_0,x_1,\dots,x_d)$ and it is precisely at this point that we will finally use the properties of the canonical orderings of sets of internal edges. 
  
 \begin{lemma}
 \label{LemmaBoundaryOrientCollapse}
  Let $T_1,T_2 \in \BinTree_d$, $T \in \RTree_d$, $e_1 \in E_{int}(T_1)$ and $e_2 \in E_{int}(T_2)$, such that 
  $$T_1/e_1 = T = T_2/e_2 \ . $$
  The two boundary orientations on $\Acal^d_{\fatY}(x_0,x_1,\dots,x_d,T)$, which are obtained by identifying it with a subset of the boundary of 
  $$\text{both} \quad \bAcal^d_{\fatY}(x_0,x_1,\dots,x_d,T_1) \quad \text{and} \quad \bAcal^d_{\fatY}(x_0,x_1,\dots,x_d,T_2) \ , $$
  coincide. 
 \end{lemma}

 Because of its technicalities, we will only outline the proof of Lemma \ref{LemmaBoundaryOrientCollapse}. 
 
\begin{proof}[Sketch of proof]
Apparently, precisely one of the two edges $e_1$ and $e_2$ is left-handed and we assume w.l.o.g. that $e_1$ is the left-handed one. 
In the following, we identify $$E_{int}(T)=E_{int}(T_1)\setminus \{e_1\}=E_{int}(T_2)\setminus\{e_2\} \ .$$
Let $f_1,f_2,f_0,g \in E_{int}(T)$ be given by demanding that $f_1$ is left-handed and that in $T_1$, it holds that
$$\vin(e_1) = \vout(f_0), \quad \vout(e_1)=\vin(f_1)=\vin(f_2), \quad \vin(e_1)=\vin(g) \ , $$
such that $g$ is right-handed in $T_1$. One checks from the definitions of the trees that $g$ is then left-handed in $T_2$ and that in $T_2$, it holds that
$$\vin(e_2)=\vout(f_0), \quad \vout(e_2)=\vin(g)=\vin(f_2), \quad \vin(e_2)=\vin(f_1) \ . $$
Moreover, all other vertex identifications of $T_1$ coincide with identifications of $T_2$ and vice versa. Thus, one checks from these equalities that one gets a diffeomorphism 
$$\varphi: M^{1+2k(T_1)+d} \to M^{1+2(T_2)+d}$$
with $\varphi(\Delta_{T_1}) = \Delta_{T_2}$ as follows: writing elements of $M^{1+2k(T_2)+d}$ as $(q_0,(q_e^1,q_e^2)_{e\in E_{int}(T_1)},q_1,\dots,q_d)$, we need to interchange the $q_g^2$-component and the $q_{f_1}^2$-component and afterwards permute the pairs $(q_e^1,q_e^2)$ according to the orderings of the edges, where $(q_{e_1}^1,q_{e_1}^2)$ takes the role of the component associated with $e_2$. Since the latter permutation only permutes the order of even-dimensional manifolds, it is orientation-preserving. Moreover, one checks that transposing the two components is orientation-preserving if and only if $n$ is even, independent of the orderings of the edges. 

Taking a closer look at orientations of $T$-diagonals one checks that the sign of $\varphi|_{\Delta_{T_1}}$ coincides with the sign of the permutation $M^{1+k(T_1)} \to M^{1+k(T_2)}$ that interchanges the factors associated with $f_1$ and $g$ and leaves all other factors invariant. Consequently, 
$$\sign \varphi = \sign (\varphi|_{\Delta_{T_1}}) = (-1)^n \ . $$
Let $j \in \{1,2,\dots,d-2\}$ be given such that $e_1$ is the $j$-th internal edge of $T_1$ with respect to the canonical ordering. Then it follows from the definition of the canonical orderings that $e_2$ is the $j$-th internal edge of $T_2$ as well and that the map 
$$E_{int}(T_1) \to E_{int}(T_2), \qquad e \mapsto \begin{cases}
                                                  e & \text{if $e \neq e_1$,} \\
                                                  e_2 & \text{if $e=e_1$,}
                                                 \end{cases} $$
preserves the canonical orderings. By definition of the moduli spaces, one derives that this map induces a diffeomorphism 
$$\MM^d_{\fatY}(x_0,x_1,\dots,x_d,T_1) \to \MM^d_{\fatY}(x_0,x_1,\dots,x_d,T_2)$$
of class $C^{n+1}$ which is orientation-preserving. Since $\Acal^d_{\fatY}(x_0,x_1,\dots,x_d,T_i)$ is for $i \in \{1,2\}$ given as the transverse intersection of $\MM_{\fatY}(x_0,x_1,\dots,x_d,T_i)$ with $\Delta_{T_i}$, it follows from Theorem \ref{TheoremOrientDiffeoTransverseint} that the identity induces an diffeomorphism 
$$\Acal^d_{\fatY}(x_0,x_1,\dots,x_d,T_1) \to \Acal^d_{\fatY}(x_0,x_1,\dots,x_d,T_2)$$
whose sign is given by 
$$\sign \varphi \cdot \sign (\varphi|_{\Delta_{T_1}}) = (-1)^n\cdot(-1)^n= 1 \ , $$
i.e. the diffeomorphism is orientation-preserving. The claim then follows from passing to the boundary orientations in both cases.
\end{proof}

 For every $T \in \BinTree_d$ we let $\widetilde{\Acal}^d_{\fatY}(x_0,x_1,\dots,x_d,T)$ denote the oriented manifold diffeomorphic to $\bAcal^d_{\fatY}(x_0,x_1,\dots,x_d,T)$, but whose orientation coincides with the one of $\bAcal^d_{\fatY}(x_0,x_1,\dots,x_d,T)$ if and only if $r(T)$ is even. 
   F
 \begin{prop}
 \label{PropTwoCopiesOrient}
  Let $T_0 \in \RTree_d$ have a unique four-valent internal vertex while all other internal vertices are three-valent. Then the space $\Acal^d_{\fatY}(x_0,x_1,\dots,x_d,T_0)$ appears twice in the boundary of the oriented manifold
  $$\bigsqcup_{T \in \BinTree_d} \widetilde{\Acal}^d_{\fatY}(x_0,x_1,\dots,x_d,T)$$
  and the boundary orientations of the two copies are opposite to each other. 
 \end{prop}
\begin{proof}
Consider the union 
$$\bigsqcup_{T \in \BinTree_d} \bAcal^d_{\fatY}(x_0,x_1,\dots,x_d,T)  \ . $$
There are two copies of $\Acal^d_{\fatY}(x_0,x_1,\dots,x_d,T_0)$ in the boundary of this union. By Lemma \ref{LemmaBoundaryOrientCollapse}, the boundary orientations of both of these copies have the same orientation. Introducing orientation twists by $(-1)^{r(T)}$, Proposition \ref{PropDexterities} implies that both copies of $\Acal^d_{\fatY}(x_0,x_1,\dots,x_d,T_0)$ will have opposite sign in $\bigsqcup_{T \in \BinTree_d} \widetilde{\Acal}^d_{\fatY}(x_0,x_1,\dots,x_d,T)$.  
\end{proof}

By Proposition \ref{PropTwoCopiesOrient}, the orientation of $\bigsqcup_{T \in \BinTree_d} \widetilde{\Acal}^d_{\fatY}(x_0,x_1,\dots,x_d,T)$
 induces an orientation of 
 $$\bAcal^d_{\fatY}(x_0,x_1,\dots,x_d) = \bigsqcup_{T \in \BinTree_d} \widetilde{\Acal}^d_{\fatY}(x_0,x_1,\dots,x_d,T)  \Big/ \sim \ , $$
 since the pairs of boundary curves that we are gluing have opposite orientations. \bigskip 
 
 {\it Throughout the rest of this section, let $\bAcal^d_{\fatY}(x_0,x_1,\dots,x_d)$ be equipped with the orientation that we have just described.} 
 
 \bigskip 
 
As we discussed in the proof of Corollary \ref{CoeffRelMod2}, the coefficients appearing in Theorem \ref{CoeffRelAinfty} are obtained by counting elements of moduli spaces which are components of the boundary of $\bAcal^d_{\fatY}(x_0,x_1,\dots,x_d)$. We have proven Corollary \ref{CoeffRelMod2} by using the simple fact that the compact one-dimensional manifold $\bAcal^d_{\fatY}(x_0,x_1,\dots,x_d)$ has an even number of boundary points.

In an \emph{oriented} compact one-dimensional manifold with boundary, every component with non-empty boundary is diffeomorphic to a compact interval. Hence, it has precisely two boundary points of which precisely one is positively oriented. If we denote the boundary orientation on $\partial \bAcal^d_{\fatY}(x_0,x_1,\dots,x_d)$ by 
\begin{equation*}
 \epsilon_\partial: \partial \bAcal^d_{\fatY}(x_0,x_1,\dots,x_d) \to \{-1,1\} \ , 
\end{equation*} 
this implies that 
\begin{equation}
\label{BoundarySumZero}
 \sum_{\left(\gammaunder_1,\gammaunder_2\right) \in \partial \bAcal^d_{\fatY}(x_0,x_1,\dots,x_d)} \epsilon_\partial\left(\gammaunder_1,\gammaunder_2\right) = 0 \ . 
\end{equation}
Equation (\ref{BoundarySumZero}) will be the decisive tool for the proof of Theorem \ref{CoeffRelAinfty}. Using this equation, the proof reduces to a comparison of the boundary orientations and the oriented intersection numbers which define the coefficients appearing in the statement.

We will proceed along the following line of argument: 
 \begin{itemize}
  \item We will compare the different orientations on the corresponding boundary components of the spaces $\MM^d_{\fatY}(x_0,x_1,\dots,x_d,T)$, which are the domains of the endpoint evaluations used to define the intersection numbers. The three cases of convergence phenomena along the rooted edge, an internal edge and a leafed edge will be treated seperately, see Lemma \ref{SignGluingPermute} and Theorem \ref{SignGluingRootedLeafed}.
  \item While these considerations suffice for external edges,  we will compare the orientations of the different $T$-diagonals involved in the situation for internal edges, see Lemma \ref{LemmaDeltaTOrientPermute}.
  \item Moreover, we compute the contribution of the twisting signs $(-1)^{\sigma(x_0,x_1,\dots,x_d)}$ in the definition of the coefficients in Proposition \ref{SigmaComputations}.
  \item Finally, we will conclude the proof by putting all the results together to compute the differences between the coefficients under investigation and the sums of the boundary orientations and use equation (\ref{BoundarySumZero}) to derive the statement. 
 \end{itemize} 
For brevity's sake, we will denote the space of finite-length trajectories of positive length by 
\begin{equation*}
\MM := \{(l,\gamma) \in \MM(f,g) \ | \ l > 0 \} \ .
\end{equation*}
We first deal with those components of $\partial \bAcal^d_{\fatY}(x_0,x_1,\dots,x_d)$ coming from the geometric convergence along an external edge. 

\begin{theorem} \index{orientations!and gluing of perturbed Morse ribbon trees}
 \label{SignGluingRootedLeafed}
 Let $T \in \BinTree_d$. 
\begin{enumerate}[a)]
 \item Let $y_0 \in \Crit f$ satisfy $\mu(y_0)=\mu(x_0)-1$. Then the product orientation on 
 \begin{equation*}
  \widehat{\MM}(x_0,y_0) \times \Acal^d_{\fatY}(y_0,x_1,\dots,x_d,T)
 \end{equation*}
 coincides with the boundary orientation with respect to $\bAcal^d_{\fatY}(x_0,x_1,\dots,x_d,T)$ if and only if $\mu(x_0)$ is odd.
 \item For $i \in \{1,2,\dots,d\}$ let $y_i \in \Crit f$ with $\mu(y_i)=\mu(x_i)+1$. Then the product orientation on 
  \begin{equation*}
   \Acal^d_{\fatY}(x_0,x_1,\dots,x_{i-1},y_i,x_{i+1},\dots,x_d,T) \times \widehat{\MM}(y_i,x_i)
  \end{equation*}
 coincides with the induced boundary orientation with respect to $\bAcal^d_{\fatY}(x_0,x_1,\dots,x_d,T)$ if and only if the following number is even:
 \begin{equation*}
  \mu(x_0)+1+(d-i)(n+1)+ \maltese_1^{i-1} \ .
 \end{equation*} 
\end{enumerate}
\end{theorem}

\begin{proof}
The main tool in proving the two parts of the theorem will be the construction of geometric gluing maps 
\begin{align}
G_0: &[\rho_0,+\infty) \times \widehat{\MM}(x_0,y_0) \times \Acal^d_{\fatY}(y_0,x_1,\dots,x_d,T) \to \Acal^d_{\fatY}(x_0,x_1,\dots,x_d,T)  \ , \label{EqDefG0} \\
G_i: &[\rho_i,+\infty) \times \Acal^d_{\fatY}(x_0,x_1,\dots,x_{i-1},y_i,\dots,x_d,T) \times \widehat{\MM}(y_i,x_i)\to \Acal^d_{\fatY}(x_0,x_1,\dots,x_d,T) \ , \notag 
\end{align}
for all $i \in \{1,2,\dots,d\}$. If we equip their domains with the product orientations, the definition of a geometric gluing map will imply that the $G_i$, where $i \in \{0,1,\dots,d\}$, will be orientation-preserving if and only if the product orientation of 
$$\begin{cases} 
   \widehat{\MM}(x_0,y_0) \times \Acal^d_{\fatY}(y_0,x_1,\dots,x_d,T) & \text{for $i=0$,} \\
   \Acal^d_{\fatY}(x_0,x_1,\dots,x_{i-1},y_i,x_{i+1},\dots,x_d,T) \times \widehat{\MM}(y_i,x_i) & \text{for $i \in \{1,2,\dots,d\}$,}
  \end{cases}$$
 coincides with its boundary orientation as a boundary of $\bAcal^d_{\fatY}(x_0,x_1,\dots,x_d,T)$.
\begin{enumerate}[a)]
  \item We first want to consider a geometric gluing map of the form 
 \begin{equation*}
 G'_0: [\rho_0,+\infty) \times \widehat{\MM}(x_0,y_0) \times \MM^d_{\fatY}(y_0,x_1,\dots,x_d,T) \to \MM^d_{\fatY}(x_0,x_1,\dots,x_d,T) \ .
 \end{equation*}
 One checks that one obtains such a geometric gluing map by 
 $$G'_0 := F_T^{-1} \circ G_0'' \circ (\id_{[\rho_0,+\infty)\times \widehat\MM(x_0,y_0)}\times F'_T) \ , $$ 
 where $F_T$ and $F_T'$ are the diffeomorphisms 
 \begin{align*}
  F_T:  \Acal^d_{\fatY}(x_0,x_1,\dots,x_d,T) &\to  \WW^u(x_0) \times \prod_{e \in E_{int}(T)} \MM \times \prod_{q=1}^d \WW^s(x_q) \ , \\
  F_T':  \Acal^d_{\fatY}(y_0,x_1,\dots,x_d,T) &\to  \WW^u(y_0) \times \prod_{e \in E_{int}(T)} \MM \times \prod_{q=1}^d \WW^s(x_q) \ , 
 \end{align*}
described in Section \ref{SectionOrientPerturbedMRT}, which are by definition orientation-preserving, and where
 \begin{align*}
  G_0'': [\rho_0,+\infty) \times \widehat{\MM}(x_0,y_0) \times \WW^u(y_0) \times &\prod_{e \in E_{int}(T)} \MM \times \prod_{q=1}^d \WW^s(x_q) \to \\
  &\WW^u(x_0)\times \prod_{e \in E_{int}(T)} \MM \times \prod_{q=1}^d \WW^s(x_q) \ ,
 \end{align*}
 where $\rho_0>0$ is sufficiently big, is defined as a composition of permutations of the factors and a gluing map for negative half-trajectories. We explicitly compute the sign of this map. We permute $[\rho_0,+\infty)$ along $\widehat{\MM}(x_0,y_0) \times \WW^u(y_0)$. The sign of this permutation is given by the parity of
 \begin{equation*}
  \dim [\rho_0,+\infty)\cdot \left(\dim \widehat{\MM}(x_0,y_0)+\dim \WW^u(y_0)\right) = \mu(y_0)\equiv \mu(x_0)+1 \ .
 \end{equation*}
 Afterwards, we apply a gluing map for negative half-trajectories
 \begin{equation*}
 \widehat{\MM}(x_0,y_0) \times \WW^u(y_0) \times [\rho_0,+\infty) \to \WW^u(x_0) \ ,
 \end{equation*}
 which is by definition orientation-preserving. Hence, the sign of $G_0'$ is given by the parity of $\mu(x_0)+1$. By standard gluing analysis methods of Morse theory, one shows that one can choose the gluing map in the construction of $G''_0$ in such a way that $G'_0$ restricts to a map $G_0$ given as in \eqref{EqDefG0} and the properties of $G''_0$ imply that $G_0$ is indeed a geometric gluing map. The domain and target of $G_0$ in \eqref{EqDefG0} are obtained as the transverse intersections of the domain and target of $G'_0$ with $\Delta_T$. By Proposition \ref{PropOrientTransverseBoundary}, the orientation of the space we obtain from intersecting $\widehat{\MM}(x_0,y_0) \times \MM^d_{\fatY}(y_0,x_1,\dots,x_d,T)$ with the boundary orientation with $\Delta_T$ differs from the boundary orientation of $\bAcal^d_{\fatY}(x_0,x_1,\dots,x_d,T)$ by the parity of 
 \begin{align*}
  \codim \Delta_T &= ((1+2k(T)+d)-(k(T)+1))n=(k(T)+d)n = (2d-2)n \equiv 0 \ . 
 \end{align*}
 Thus, we have shown that the orientation of $\Acal^d_{\fatY}(y_0,x_1,\dots,x_d,T)$ coincides with the boundary orientation of $\bAcal^d_{\fatY}(x_0,x_1,\dots,x_d,T)$ if and only if $\mu(x_0)+1$ is even. 
 \item Fix $i \in \{1,2,\dots,d\}$. Similar to part a), one constructs a geometric gluing map 
 \begin{align*}
   G'_i: [\rho_0,+\infty) \times \MM^d_{\fatY}(y_0,x_1,\dots,x_{i-1},y_i,x_{i+1},\dots,x_d,T) &\times \widehat\MM(y_i,x_i) \\
   &\to \MM^d_{\fatY}(x_0,x_1,\dots,x_d,T) \ .
 \end{align*}
 for sufficiently big $\rho_0>0$ that restricts to a gluing map of the form $G_i$, such that $G'_i$ is orientation-preserving if and only if the following embedding is orientation-preserving, which is defined as a composition of permutations and a gluing map for positive half-trajectories:
 \begin{align*}
  [\rho_0,+\infty)\times \WW^u(x_0) \times &\prod_{e \in E_{int}(T)} \MM \times \prod_{q=0}^{i-1}\WW^s(x_q)\times \WW^s(y_i)\times \prod_{q=i+1}^d \WW^s(x_q) \times \widehat{\MM}(y_i,x_i) \\
  &\qquad \qquad \qquad \qquad \qquad \qquad \to \WW^u(x_0) \times \prod_{e \in E_{int}(T)} \MM \times \prod_{q=1}^d \WW^s(x_q) \ .
 \end{align*}
 We compute the sign of this map explicitly by decomposing it into three simpler steps. 
 
 Firstly, we permute the factor $\widehat{\MM}(y_i,x_i)$ with $\prod_{q=i+1}^d \WW^s(x_q)$. Since $\widehat{\MM}(y_i,x_i)$ is zero-dimensional, this does not affect the sign. Secondly, we move $[\rho_0,+\infty)$ along $\WW^u(x_0) \times \prod_{e \in E_{int}(T)} \MM \times \prod_{q=1}^{i-1}\WW^s(x_q)\times \WW^s(y_i) \times \widehat{\MM}(y_i,x_i)$. The sign of this permutation is given by the parity of
 \begin{align*}
  &\dim \WW^u(x_0) + \sum_{e \in E_{int}(T)} \dim \MM + \sum_{q=1}^{i-1} \dim \WW^s(x_q)+ \dim \WW^s(y_i) + \dim \widehat{\MM}(y_i,x_i) \\
  &\equiv \mu(x_0)+ (d-i)(n+1)+ \maltese_1^{i-1} +\mu(y_i)+1 \ .
 \end{align*}
 Thirdly, we apply a gluing map $\WW^s(y_i)\times \widehat{\MM}(y_i,x_i) \times [\rho_0,+\infty)\to \WW^s(x_i)$ for positive half-trajectories, whose sign is by Proposition \ref{PropGluingStable} given by the parity of $\mu(y_i)$. Hence, the total sign of the above diffeomorphism is given by the parity of 
 \begin{equation*}
  \mu(x_0)+1+ (d-i)(n+1)+ \maltese_1^{i-1} \ .
 \end{equation*}
 Continuing the line of argument as in part 1 shows the claim.
 \end{enumerate}
\end{proof}

We turn our attention to boundary spaces of $\bAcal^d_{\fatY}(x_0,x_1,\dots,x_d,T)$ that are obtained by geometric convergence along internal edges of $T$. The proof of the following lemma is given in strict analogy with the line of argument in \cite[Appendix C]{AbouzaidMorseTrop}.

\begin{lemma} \index{orientations!and gluing of perturbed Morse ribbon trees}
\label{SignGluingPermute}
 Let $T \in \BinTree_d$, $e \in E_{int}(T)$ be of type $(i,l)$, $y \in \Crit f$ satisfy
 \begin{equation}
 \label{parityindexeqy}
  \mu(y) = \sum_{q=i}^{i+l} \mu(x_q)+1-l 
 \end{equation}
and let $(T^e_1,T^e_2)$ denote the splitting of $T$ along $e$. For sufficiently big $\rho_0>0$ there exists a geometric gluing map 
\begin{align*}
 G_{T,e}: [\rho_0,+\infty)\times\MM^{d-l}_{\fatY}(x_0,x_1,\dots,x_{i-1},y,x_{i+l+1},\dots,x_d,T^e_1) \times &\MM^{l+1}_{\fatY}(y,x_i,\dots,x_{i+l},T^e_2) \\
    &\to \MM^d_{\fatY}(x_0,x_1,\dots,x_d,T) \ , 
\end{align*}
which is orientation-preserving if and only if the following number is even:
 \begin{equation*}
 \begin{cases}
 \mu(x_0) + \maltese_1^{i-1} + (n+1)\left( l\cdot \maltese_1^{i-1} +dl+i+l+d+1\right) & \text{if $e$ is left-handed,} \\
 \mu(x_0) + \maltese_1^{i-1} + (n+1)\left( l\cdot \maltese_1^{i-1} +dl+i+l+d\right) & \text{if $e$ is right-handed.} 
 \end{cases}
 \end{equation*}
\end{lemma}

\begin{proof}
  We consider a map 
 \begin{align*}
  &G_{T,e}':[\rho_0,\infty) \times \WW^u(x_0)\times \prod_{e' \in E_{int}(T^e_1)} \MM \times \prod_{q=1}^{i-1} \WW^s(x_q) \times \WW^s(y)\times \prod_{q=i+l+1}^d \WW^s(x_q) \\
  &\qquad \qquad \times \WW^u(y) \times \prod_{e' \in E_{int}(T^e_2)}\MM \times \prod_{q=i}^{i+l} \WW^s(x_q) \to \WW^u(x_0) \times \prod_{e' \in E_{int}(T)} \MM \times \prod_{q=1}^d \WW^s(x_q) 
 \end{align*}
  which is defined in the obvious way as a composition of permutations and a gluing map for finite-length trajectories. Via the maps $F_{T^e_1}$, $F_{T^e_2}$ and $F_T$ from Section \ref{SectionOrientPerturbedMRT}, domain and target of $G_{T,e}'$ are identified with the domain and target of the map $G_{T,e}$ that we want to show to exist. Indeed, one obtains a geometric gluing map of the form $G_{T,e}$ by composing $G'_{T,e}$ with $F_{T^e_1} \times F_{T^e_2}$ and $F_T^{-1}$ in analogy with the line of argument in the proof of Theorem \ref{SignGluingRootedLeafed}. Moreover, $G_{T,e}$ is orientation-preserving if and only if $G'_{T,e}$ is and we will compute the sign of $G'_{T,e}$ explicitly. We start by computing the sign of the permutation which maps the domain of $G'_{T,e}$ onto  
 \begin{equation}
 \label{paritytarget}
  \WW^u(x_0)\times [\rho_0,\infty) \times \WW^s(y)\times \WW^u(y) \times \prod_{e' \in E_{int}(T^e_1)} \MM \times \prod_{e' \in E_{int}(T^e_2)}\MM  \times \prod_{q=1}^d\WW^s(x_q) \ .
 \end{equation}
 We view this permutation as a composition of simpler permutations in the following way:
 \begin{enumerate}
  \item We move $\WW^s(y)$ past the factors $\WW^s(x_1)\times \dots \times \WW^s(x_{i-1})$. The sign of this permutation is given by the parity of 
  \begin{align}
   \dim \WW^s(y) \cdot \sum_{q=1}^{i-1} \dim \WW^s(x_q) \equiv n\left( (i-1)\mu(y) + \maltese_1^{i-1}\right) + \mu(y) \sum_{q=1}^{i-1} \mu(x_q) \ . \label{parity0}
  \end{align}
  \item In the modified product, we move $\WW^u(y)$ past 
  \begin{equation*}
  \WW^s(x_1)\times \dots \times \WW^s(x_{i-1})\times \WW^s(x_{i+l+1}) \times \dots \times \WW^s(x_d) \ .
  \end{equation*}
  The sign of this permutation is given by the parity of 
  \begin{align}
   \mu(y) \Bigl( (d-l-1)n - \sum_{q=1}^{i-1}\mu(x_q)-\sum_{q=i+l+1}^d \mu(x_q) \Bigr) \ . \label{parity1}
  \end{align}
  \item Afterwards, we move $\prod_{e' \in E_{int}(T^e_2)} \MM$ past 
  \begin{equation*}
  \WW^s(y) \times \WW^u(y) \times \WW^s(x_1)\times \dots \times \WW^s(x_{i-1})\times \WW^s(x_{i+l+1}) \times \dots \times \WW^s(x_d) \ .
  \end{equation*}
  This affects the sign by the parity of 
  \begin{equation}
   (l-1)(n+1) \Bigl(\sum_{q=1}^{i-1} \mu(x_q) + \sum_{q=i+l+1}^d \mu(x_q) \Bigr) \ . \label{parity2}
  \end{equation}
  \item We permute $\WW^s(x_i)\times \dots \times \WW^s(x_{i+l})$ past $\WW^s(x_{i+l+1}) \times \dots \times \WW^s(x_d)$ such that all the $\WW^s(x_q)$ appear in the right order. The sign of this permutation is given by the parity of 
  \begin{align}
   &\sum_{q=i}^{i+l} \dim \WW^s(x_q) \cdot \sum_{q=i+l+1}^d \dim \WW^s(x_q)= \sum_{q=i}^{i+l}(n-\mu(x_q)) \cdot \sum_{q=i+l+1}^d (n-\mu(x_q)) \notag \\
     &= n \Bigl( (d-i-l)\Bigl(\sum_{q=i}^{i+l}\mu(x_q)+l+1\Bigr) + (l+1)\sum_{q=i+l+1}^d \mu(x_q)\Bigr)  + \sum_{q=i}^{i+l}\mu(x_q) \cdot \sum_{q=i+l+1}^d \mu(x_q)  \ . \label{parity3}
  \end{align}
  \item We permute $\WW^s(y) \times \WW^u(y)$ along $\prod_{e' \in E_{int}(T^e_1)} \MM \times \prod_{e' \in E_{int}(T^e_2)} \MM$. The sign of this permutation is given by 
  \begin{align*}
   &(\dim \WW^s(y) + \dim \WW^u(y)) \cdot \Bigl(\sum_{e' \in E_{int}(T^e_1)\cup E_{int}(T^e_2)} \dim \MM\Bigr) \equiv n \cdot (n+1) \cdot (d-1) \equiv 0  \ . 
  \end{align*}
  \item Finally, we permute $[\rho_0,\infty)$ past $\WW^u(x_0)$, which modifies the sign by the parity of   
  \begin{equation}
  \label{parity4}
   \dim \WW^u(x_0) =\mu(x_0) \ .
  \end{equation}
 \end{enumerate}

 The sign of the permutation under investigation is given by the parity of the sum of (\ref{parity0}), (\ref{parity1}), (\ref{parity2}), (\ref{parity3}) and (\ref{parity4}). A straightforward computation shows that this results in the parity of 
$$\mu(x_0) + \maltese_1^{i-1} + (n+1)\left( l\cdot \maltese_1^{i-1}  +il+i+l+1\right) \ . $$
 This is the parity of the sign of the permutation which permutes the domain of our map onto (\ref{paritytarget}). Since the gluing map $[\rho_0,\infty)\times\WW^s(y) \times \WW^u(y) \to \MM(f,g)$ has sign $(-1)^{n+1}$ by Proposition \ref{PropGluingFiniteOrient}, we need to add $n+1$ to the above, which yields: 
 \begin{equation}
 \label{parityzwischen}
  \mu(x_0) + \maltese_1^{i-1} + (n+1)\left( l\cdot \maltese_1^{i-1}  +il+i+l\right) \ .
 \end{equation}
 The final step is to take into account the sign of the permutation
 \begin{equation}
 \label{permintedges}
  \MM \times \prod_{e' \in T^e_1} \MM \times \prod_{e' \in T^e_2} \MM \to \prod_{e' \in E_{int}(T)} \MM 
 \end{equation}
 according to the ordering of the edges of the trees. This sign is given by $(-1)^{n+1} \cdot \sign \tau_e$, where $\tau_e$ again denotes the order-preserving permutation 
\begin{equation*}
 \tau_e: \{e\} \times E_{int}(T^e_1) \times E_{int}(T^e_2) \to E_{int}(T) \ .
\end{equation*}
Since $e$ is of type $(i,l)$, it follows from Theorem \ref{TheoremSigntaue} that the sign of (\ref{permintedges}) is given by the parity of 
\begin{equation*}
\begin{cases}
 (n+1)((d-i)l+d-1) & \text{if $e$ is left-handed,} \\
 (n+1)((d-i)l+d) & \text{if $e$ is right-handed.} 
 \end{cases}
\end{equation*}
Adding this to (\ref{parityzwischen}) shows the claim.
\end{proof}

The definition of the oriented intersection numbers $\orint \Acal^d_{\fatY}(x_0,x_1,\dots,x_d,T)$  relies on the orientations of the $T$-diagonals. Therefore, we need to establish the following lemma to compare the boundary orientations under consideration:

\begin{lemma}
\label{LemmaDeltaTOrientPermute}
 Let $T \in \RTree_d$ and let $e \in E_{int}(T)$ be of type $(i,l)$. The diffeomorphism 
$$ s_{T,e}:\Delta_{T^e_1} \times \Delta_{T^e_2} \to \Delta_T \ , $$
 given as in Lemma \ref{PermutProductDiag}, is orientation-preserving if and only if the following number is even: 
 $$\begin{cases} 
n((d-i-1)l+1) & \text{if $e$ is left-handed,} \\    
n(d-i-1)l & \text{if $e$ is right-handed.}
   \end{cases}$$
 It extends to a diffeomorphism $\sigma_{T,e}: M^{1 + 2 k(T^e_1) + d-l} \times M^{1+2k(T^e_2)+l+1} \to M^{1+2k(T)+d} $
 which is orientation-preserving if and only if $n(d-i-1)l$ is even. 
\end{lemma}
\begin{proof}
 One checks in the proof of Lemma \ref{PermutProductDiag} that $s_{T,e}$ is given as the restriction of a diffeomorphism
 $$\sigma_{T,e}: M^{1 + 2 k(T^e_1) + d-l} \times M^{1+2k(T^e_2)+l+1} \to M^{1+2k(T)+d} $$
 given by permuting copies of $M$, where we put $k_1 := k(T^e_1)$ and $k_2 := k(T^e_2)$. We first compute the sign of $\sigma_{T,e}$ by decomposing it into a sequence of simpler permutations and compute the respective signs:
\begin{enumerate}
 \item We permute the $i$-th of the last $d-l$ factors of $M^{1+2k_1+d-l}$ along $M^{2k_1+i-1}$. The sign of this map is given by the parity of $(2k_1+i-1)n \equiv (i-1)n$.
 \item We move the first factor of $M^{1+2k_2+l+1}$ along $M^{2k_1+d-l-1}$. This affects the sign by the parity of $(2k_1+d-l-1)n \equiv (d-l-1)n$.
 \item We interchange $M^{2k_2}$ with $M^{d-l-1}$. This permutation is orientation-preserving, since $M^{2k_2}$ is always even-dimensional. 
 \item We move $M^{l+1}$ along $M^{d-i-l}$. The sign of this map is given by the parity of $n(d-i-l)(l+1) \equiv n(d-i)(l+1)$. 
 \item Finally, we need to permute $M \times M^2 \times M^{2k_1} \times M^{2k_2} \times M^d$ onto $M \times M^{2k} \times M^d$ according to the orderings of $E_{int}(T^e_1)$, $E_{int}(T^e_2)$ and $E_{int}(T)$. This permutation is always orientation-preserving, because we are only moving even-dimensional manifolds. 
\end{enumerate}
Therefore, the total sign of the permutation $M^{1+2k_1+d-l} \times M^{1+2k_2+l+1} \stackrel{\cong}{\longrightarrow} M^{1+2k+d}$ is given by 
\begin{equation*}
 n( (d-i)(l+1)+d-i-l) \equiv n((d-i)l-l) \equiv n(d-i-1)l \ .
\end{equation*}
We next compute the sign of the map $s_{T,e}:= \sigma_{T,e}|_{\Delta_{T^e_1}\times \Delta_{T^e_2}}:\Delta_{T^e_1} \times \Delta_{T^e_2} \to \Delta_T$. We have defined the orientations on $\Delta_{T^e_1}$, $\Delta_{T^e_2}$ and $\Delta_T$ such that there are orientation-preserving diffeomorphisms $M^{1+k_1}\stackrel{\cong}{\to}\Delta_{T^e_1}$, $M^{1+k_2}\stackrel{\cong}{\to}\Delta_{T^e_2}$ and $M^{1+k} \stackrel{\cong}{\to} \Delta_T$. 

Consequently, the sign of $s_{T,e}$ is given by the sign of the diffeomorphism 
$$M^{1+k_1} \times M^{1+k_2} \to M^{1+k}\ , $$
obtained via moving the first copy of $M$ in $M^{1+k_2}$ along the last $k_1$ copies of $M$ in the product $M^{1+k_1}$ and afterwards permute the last $1+k_1+k_2$ factors of $M^{1+1+k_1+k_2}$ according to the ordering of $E_{int}(T)$. Since $T^e_1$ is binary, it holds that $k_1=d-l-2$. Consequently, the sign is given by the parity of
 \begin{align*}
 n(d-l-2+\sign \tau_e) &\equiv n(d-l+\sign \tau_e) \equiv \begin{cases}
                                                 n((d-i-1)l+1) & \text{if $e$ is left-handed,} \\
                                                 n(d-i-1)l & \text{if $e$ is right-handed,} 
                                                \end{cases}
 \end{align*}
 where we have used Theorem \ref{TheoremSigntaue}. This completes the proof. 
\end{proof}

The following lemma will turn out to be useful in several of the upcoming sign computations. 

\begin{lemma}
\label{LemmaBinaryTreeMMevendimensional}
 Let $r \in \NN$, $T \in \RTree_r$ and let $y_0,y_1,\dots,y_r \in \Crit f$ satisfy 
 $$\mu(y_0)=\sum_{q=1}^r \mu(y_q) +2-r \ . $$
 If $T$ is a binary tree, then $\MM^r_{\fatY}(y_0,y_1,\dots,y_r,T)$ will be even-dimensional. 
\end{lemma}
\begin{proof}
This follows from an elementary computation.
\end{proof}

\begin{theorem}
\label{ThmGluingIntEdge}
  Let $T \in \BinTree_d$, $e \in E_{int}(T)$ be of type $(i,l)$, $y \in \Crit f$ satisfy
 \begin{equation*}
  \mu(y) = \sum_{q=i}^{i+l} \mu(x_q)+1-l \ .
 \end{equation*}
 and let $(T^e_1,T^e_2)$ be the splitting of $T$ along $e$. The product orientation of 
$$\Acal^{d-l}_{\fatY}(x_0,x_1,\dots,x_{i-1},y,x_{i+l+1},\dots,x_d,T^e_1) \times \Acal^{l+1}_{\fatY}(y,x_i,\dots,x_{i+l},T^e_2)$$ 
coincides with its boundary orientation with respect to $\bAcal_{\fatY}^d(x_0,x_1,\dots,x_d,T)$ if and only if the following number is even:
 \begin{equation*}
 r(T^e_1)+r(T^e_2) + r(T) + \mu(x_0) + 1+\maltese_1^{i-1} + (n+1)\left( l\cdot \maltese_1^{i-1} +dl+i+l+d\right)  \ .
 \end{equation*}
\end{theorem}
\begin{proof}
By standard methods of gluing analysis in Morse theory, one shows that there exists a geometric gluing map of the form $G_{T,e}$ as in Lemma \ref{SignGluingPermute}, which restricts to a map 
\begin{align*}
 G: [\rho_0,+\infty)\times\Acal^{d-l}_{\fatY}(x_0,x_1,\dots,x_{i-1},y,x_{i+l+1},\dots,x_d,T^e_1) \times &\Acal^{l+1}_{\fatY}(y,x_i,\dots,x_{i+l},T^e_2) \\
    &\to \Acal^d_{\fatY}(x_0,x_1,\dots,x_d,T) \ .
\end{align*}
In analogy with the line of argument in the proof of Theorem \ref{SignGluingRootedLeafed} the two orientations under consideration coincide if and only if $G$ is orientation-preserving. In the following, we will compute the sign of $G$ explicitly. The orientation of the domain of $G$ is given by the product of
\begin{itemize}
 \item the standard orientation of $[\rho_0,+\infty)$, 
 \item the orientation induced by the transverse intersection of $$\MM^{d-l}_{\fatY}(x_0,x_1,\dots,x_{i-1},y,x_{i+l+1},\dots,x_d,T^e_1)$$ with $\Delta_{T^e_1}$ under $\Eunder_{\fatY}$, 
 \item the orientation induced by the transverse intersection of $\MM^{l+1}_{\fatY}(y,x_i,\dots,x_{i+l},T^e_2)$ with $\Delta_{T^e_2}$ under $\Eunder_{\fatY}$.
\end{itemize}
We want to compare this orientation to the one induced by the transverse intersection of 
\begin{equation}
\label{EqDingsbumsGluingDomain}
[\rho_0,+\infty)\times\MM^{d-l}_{\fatY}(x_0,x_1,\dots,x_{i-1},y,\dots,x_d,T^e_1) \times \MM^{l+1}_{\fatY}(y,x_i,\dots,x_{i+l},T^e_2) 
\end{equation}
with $\Delta_{T^e_1} \times \Delta_{T^e_2}$ under the product of the evaluation maps. In the zero-dimensional case which we are considering, these two orientations coincide if and only if the following number is even: 
$$\dim \MM^{l+1}_{\fatY}(y,x_i,\dots,x_{i+l},T^e_2) \cdot \dim \Delta_{T^e_1} \equiv 0 \ , $$
where we have applied Lemma \ref{LemmaBinaryTreeMMevendimensional}, using that $T^e_2$ is a binary tree. Thus, the two orientations coincide such that it suffices to compare the orientation induced by the transverse intersection of \eqref{EqDingsbumsGluingDomain} with $\Delta_{T^e_1} \times \Delta_{T^e_2}$ with the orientation of $\Acal^d_{\fatY}(x_0,x_1,\dots,x_d)$. This situation is a special case of Theorem \ref{TheoremOrientDiffeoTransverseint}. More precisely, that theorem applies with
\begin{align*}
&M_1 = [\rho_0,+\infty)\times\MM^{d-l}_{\fatY}(x_0,x_1,\dots,x_{i-1},y,x_{i+l+1},\dots,x_d,T^e_1) \times \MM^{l+1}_{\fatY}(y,x_i,\dots,x_{i+l},T^e_2) \ , \\
&S_1 = [\rho_0,+\infty)\times\Acal^{d-l}_{\fatY}(x_0,x_1,\dots,x_{i-1},y,x_{i+l+1},\dots,x_d,T^e_1) \times \Acal^{l+1}_{\fatY}(y,x_i,\dots,x_{i+l},T^e_2) \ , \\
&M_2 = \MM^d_{\fatY}(x_0,x_1,\dots,x_d,T) \ , \quad S_2 = \Acal^d_{\fatY}(x_0,x_1,\dots,x_d,T) \ , \quad N_1 = \Delta_{T^e_1} \times \Delta_{T^e_2} \ , \quad N_2 = \Delta_T \ , \\
&\varphi = G_{T,e} \ , \quad \psi = \sigma_{T,e} \ , \quad f_1 = \Eunder_{\fatY} \times \Eunder_{\fatY} \ , \quad f_2 = \Eunder_{\fatY} \ ,
\end{align*}
 where $G_{T,e}$ is suitably chosen as explained above and where $\sigma_{T,e}$ is the diffeomorphism from Lemma \ref{LemmaDeltaTOrientPermute}. If $e$ is left-handed, then combining Theorem \ref{TheoremOrientDiffeoTransverseint} with Lemmas \ref{SignGluingPermute} and \ref{LemmaDeltaTOrientPermute} will yield that $G_{T,e}|_{S_1}$ is orientation-preserving if and only if the following number is even: 
\begin{align*}
&\mu(x_0) + \maltese_1^{i-1} + (n+1)\left( l\cdot \maltese_1^{i-1} +dl+i+l+d+1\right)  +n((d-i-1)l-1)+n(d-i-1)l\\
&\equiv \mu(x_0) +1+ \maltese_1^{i-1} + (n+1)\left( l\cdot \maltese_1^{i-1} +dl+i+l+d\right) \ .
\end{align*}
If $e$ is right-handed, then one shows along the same lines that the sign of $G_{T,e}|_{S_1}$ is given by the parity of 
$$\mu(x_0) + \maltese_1^{i-1} + (n+1)\left( l\cdot \maltese_1^{i-1} +dl+i+l+d\right) \ .$$
The claim follows from applying part b) of Proposition \ref{PropDexterities}.
\end{proof}

The last ingredient for the proof of Theorem \ref{CoeffRelAinfty} is the computation of the contribution of the signs $(-1)^{\sigma(x_0,x_1,\dots,x_d)}$. The occuring cases are subsumed in Proposition \ref{SigmaComputations}. We will eventually see in the proof of Theorem \ref{CoeffRelAinfty} that the twisting numbers $\sigma(x_0,x_1,\dots,x_d)$ are defined such that their contributions ensure the validity of the desired equations.

\begin{prop}
\label{SigmaComputations}
Let $x_0,x_1,\dots,x_d \in \Crit f$ satisfy $\mu(x_0)= \sum_{q=1}^d \mu(x_q) +3-d$.
\begin{enumerate}
\item Let $i \in \{1,2,\dots,d\}$, $l \in \{0,1,\dots,d-i\}$ and $y \in \Crit f$ satisfying  
 \begin{equation}
 \label{yindexorient}
  \mu(y) = \sum_{q=i}^{i+l} \mu(x_q)+1-l \ .
 \end{equation}
 Then the following congruence holds modulo two:
 \begin{align*}
   &\sigma(x_0,x_1,\dots,x_d)+\sigma(x_0,x_1,\dots,x_{i-1},y,x_{i+l+1},\dots,x_d) + \sigma(y,x_i,\dots,x_{i+l}) \\
   &\equiv (n+1)\left( l\cdot \maltese_1^{i-1} +dl+d+l+i\right) \ .
 \end{align*}
 \item Let $y_0 \in \Crit f$ with $\mu(y_0)=\mu(x_0)-1$. The following congruence holds modulo two:
  \begin{equation*}
  \sigma(x_0,y_0)+\sigma(y_0,x_1,\dots,x_d) \equiv \sigma(x_0,x_1,\dots,x_d) \ .
 \end{equation*}
 \item For $i \in \{1,2,\dots,d\}$ and $y_i \in \Crit f$ with $\mu(y_i)=\mu(x_i)+1$ the following congruence holds modulo two:
 \begin{equation*}
  \sigma(x_0,x_1,\dots,x_{i-1},y_i,x_{i+1},\dots,x_d)+\sigma(y_i,x_i)+\sigma(x_0,x_1,\dots,x_d)\equiv (n+1)(d-i) \ .
 \end{equation*}
 \end{enumerate}
\end{prop}

\begin{proof} 
In the situation of (1), one explicitly computes that
 \begin{align*}
&\sigma(x_0,x_0,x_1,\dots,x_{i-1},y,x_{i+l+1}, \dots,x_d) \\ 
 &=(n+1)\Bigl(\mu(x_0) + \sum_{j=1}^{i-1} (d-l+1-j)\mu(x_j) + (d-l+1-i)\mu(y)  +\sum_{k=i+l+1}^{d} (d+1-k)\mu(x_{k})\Bigr) \ , \\
 &\sigma(y,x_i,\dots,x_{i+l}) =(n+1)\Bigl(\mu(y) + \sum_{k=i}^{i+l}(l+1-k-i)\mu(x_k)\Bigr) \ .  
\end{align*}
Combining this with the definition of $\sigma(x_0,x_1,\dots,x_d)$, one derives
\begin{align*}
&\sigma(x_0,x_1,\dots,x_d)+\sigma(x_0,x_0,x_1,\dots,x_{i-1},y,x_{i+l+1}, \dots,x_d) + \sigma(y,x_i,\dots,x_{i+l}) \\
&\equiv (n+1) \Bigl(l \cdot \sum_{j=1}^{i-1} \mu(x_j) + (d-i-l) \Bigl(\mu(y)-\sum_{k=i}^{i+l} \mu(x_k)\Bigr) \Bigr) \\
&\stackrel{\eqref{yindexorient}}{\equiv}(n+1) \Bigl(l \cdot \sum_{j=1}^{i-1} \mu(x_j) + (d-i)(l-1)\Bigr)\equiv (n+1) \Bigl(l \cdot \maltese_1^{i-1} +dl+d+l+i \Bigr) \  .
\end{align*}
Parts (2) and (3) are shown by simple computations.
\end{proof}

We have collected all ingredients required to prove Theorem \ref{CoeffRelAinfty}. 

\begin{proof}[Proof of Theorem \ref{CoeffRelAinfty}]
 By definition of the coefficients, it holds for all $i \in \{1,\dots,d-1\}$, $l \in \{1,2,\dots,d-1-i\}$ and $y \in \Crit f$ of the right index that 
 \begin{align*}
  &a^{d-l}_{\fatY}(x_0,x_1,\dots,x_{i-1},y,x_{i+l+1},\dots,x_d)\cdot a^{l+1}_{\fatY}(y,x_i,\dots,x_{i+l}) \\
  &=\sum_{(T_1,T_2)} (-1)^{\sigma(x_0,x_1,\dots,x_{i-1},y,x_{i+l+1},\dots,x_d)+r(T_1)+\sigma(y,x_i,\dots,x_{i+l})+r(T_2)} \\
  &\qquad \qquad \orint\Acal^{d-l}_{\fatY}(x_0,x_1,\dots,x_{i-1},y,x_{i+l+1},\dots,x_d,T_1)\cdot \orint \Acal^{l+1}_{\fatY}(y,x_i,\dots,x_{i+l},T_2) \\
  &=\sum_{T \in \BinTree_d} \sum_{e}\sum_{\left(\gammaunder_1,\gammaunder_2\right)} (-1)^{\sigma(x_0,x_1,\dots,x_{i-1},y,\dots,x_d)+\sigma(y,x_i,\dots,x_{i+l})+r(T^e_1)+r(T^e_2)} \epsilon_{T^e_1}\left(\gammaunder_1\right)\cdot \epsilon_{T^e_2}\left(\gammaunder_2\right) \ , \notag
 \end{align*}
 where we used Lemma \ref{TreeBreakingLemma} and  where the sum over $\left(\gammaunder_1,\gammaunder_2\right)$ in the last line is taken over all 
 \begin{equation*}
  \left(\gammaunder_1,\gammaunder_2\right) \in \Acal^{d-l}_{\fatY}\left(x_0,x_1,\dots,x_{i-1},y,x_{i+l+1},\dots,x_d,T^e_1\right)\times \Acal^{l+1}_{\fatY}\left(y,x_i,\dots,x_{i+l},T^e_2\right)=:\Acal_e \ .
 \end{equation*}
 As a consequence of Theorem \ref{ThmGluingIntEdge},
 \begin{align*}
 &(-1)^{r(T^e_1)+r(T^e_2)}\epsilon_{T^e_1}\left(\gammaunder_1\right)\cdot\epsilon_{T^e_2}\left(\gammaunder_2\right) \\
 &= (-1)^{r(T)+\mu(x_0) +1+ \maltese_1^{i-1} + (n+1)\left( l\cdot \maltese_1^{i-1} +dl+i+l+d\right)} \epsilon_{\partial,T} \left(\gammaunder_1,\gammaunder_2\right)\qquad \forall \left(\gammaunder_1,\gammaunder_2\right) \in \Acal_e \ ,
 \end{align*}
  where $\epsilon_{\partial,T}$ denotes the boundary orientation of $\bAcal^d_{\fatY}(x_0,x_1,\dots,x_d,T)$. Thus,
 $$(-1)^{r(T^e_1)+r(T^e_2)}\epsilon_{T^e_1}\left(\gammaunder_1\right)\cdot\epsilon_{T^e_2}\left(\gammaunder_2\right) = (-1)^{\mu(x_0) +1+ \maltese_1^{i-1} + (n+1)\left( l\cdot \maltese_1^{i-1} +dl+i+l+d\right)} \epsilon_{\partial} \left(\gammaunder_1,\gammaunder_2\right) $$ 
 and part 1 of Proposition \ref{SigmaComputations} implies
  \begin{align*}
  &(-1)^{\sigma(x_0,x_1,\dots,x_{i-1},y,\dots,x_d)+r(T^e_1)+\sigma(y,x_i,\dots,x_{i+l})+r(T^e_2)}\epsilon_{T^e_1}\left(\gammaunder_1\right)\cdot\epsilon_{T^e_2}\left(\gammaunder_2\right) \\
  &\qquad \qquad \qquad \qquad \qquad = (-1)^{\mu(x_0)+1+\maltese_1^{i-1} +\sigma(x_0,x_1,\dots,x_d)} \epsilon_\partial \left(\gammaunder_1,\gammaunder_2\right) \qquad \forall \left(\gammaunder_1,\gammaunder_2\right) \in \Acal_e \ . 
  \end{align*}
 Inserting this into the above yields
  \begin{align*}
  (-1)^{\maltese_1^{i-1}}&a^{d-l}_{\fatY}(x_0,x_1,\dots,x_{i-1},y,x_{i+l+1},\dots,x_d)\cdot a^{l+1}_{\fatY}(y,x_i,\dots,x_{i+l}) \\
  &\qquad \qquad \qquad =(-1)^{\sigma(x_0,x_1,\dots,x_d)+\mu(x_0)+1}\sum_{T \in \BinTree_d} \sum_{e}\sum_{\left(\gammaunder_1,\gammaunder_2\right)} \epsilon_\partial\left(\gammaunder_1,\gammaunder_2\right) \ .
 \end{align*}
 where the sums are given as above. By taking the sum of the last equation over all $i$ and $l$, i.e. over all types of internal edges, we derive the following result:
 \begin{equation}
 \label{EqIntSign}
 \begin{aligned}
   &\sum_{i=1}^{d-1} \sum_{l=1}^{d-1-i}\sum_y (-1)^{\maltese_1^{i-1}}a^{d-l}_{\fatY}(x_0,x_1,\dots,x_{i-1},y,x_{i+l+1},\dots,x_d)\cdot a^{l+1}_{\fatY}(y,x_i,\dots,x_{i+l}) \\
   &\qquad \qquad =(-1)^{\sigma(x_0,x_1,\dots,x_d)+\mu(x_0)+1}\sum_{T \in \RTree_d} \sum_{e \in E_{int}(T)}\sum_{\left(\gammaunder_1,\gammaunder_2\right)} \epsilon_\partial\left(\gammaunder_1,\gammaunder_2\right) \ .
 \end{aligned}
\end{equation}
We keep this equation in mind and continue by considering the other two types of boundary curves of $\Acal^d_{\fatY}(x_0,x_1,\dots,x_d)$, namely elements of spaces of type $\widehat{\MM}(x_0,y_0) \times \Acal^d_{\fatY}(y_0,x_1,\dots,x_d,T)$ and $\Acal^d_{\fatY}(x_0,x_1,\dots,x_{i-1},y_i,x_{i+1},\dots,x_d,T)\times \widehat{\MM}(y_i,x_i)$ for $i \in \{1,2,\dots,d\}$ and $T \in \BinTree_d$. We start with the former type. Let $y_0 \in \Crit f$ with $\mu(y_0)=\mu(x_0)-1$. By definition, it holds that
 \begin{align*}
  &a^1_{\fatY}(x_0,y_0) \cdot a^d_{\fatY}(y_0,x_1,\dots,x_d) \\
  &= \sum_{T \in \BinTree_d} \sum_{\left(\hat\gamma,\gammaunder\right)\in\widehat{\MM}(x_0,y_0) \times \Acal^d_{\fatY}(y_0,x_1,\dots,x_d,T)} (-1)^{\sigma(x_0,x_1\dots,x_d)+r(T)} \epsilon\left(\hat\gamma\right) \cdot \epsilon_T\left(\gammaunder\right) \ ,
 \end{align*}
 where we have used part 2 of Proposition \ref{SigmaComputations}. Part a) of Theorem \ref{SignGluingRootedLeafed} implies 
 \begin{equation*}
  \epsilon\left(\hat\gamma\right)\cdot \epsilon_T\left(\gammaunder\right) = (-1)^{\mu(x_0)+1} \epsilon_{\partial,T}\left(\hat\gamma,\gammaunder\right)=(-1)^{\mu(x_0)+1+r(T)} \epsilon_{\partial}\left(\hat\gamma,\gammaunder\right)
 \end{equation*}
 for all $\left(\hat\gamma,\gammaunder\right)\in\widehat{\MM}(x_0,y_0) \times \Acal^d_{\fatY}(y_0,x_1,\dots,x_d,T)$ and $T \in \BinTree_d$. Inserting this into the above computation yields:
 \begin{equation}
 \label{EqRootedSign}
 \begin{aligned}
  &a^1_{\fatY}(x_0,y_0) \cdot a^d_{\fatY}(y_0,x_1,\dots,x_d) \\
  &= (-1)^{\sigma(x_0,x_1\dots,x_d)+\mu(x_0)+1} \sum_{T \in \BinTree_d} \sum_{\left(\hat\gamma,\gammaunder\right)\in\widehat{\MM}(x_0,y_0) \times \Acal^d_{\fatY}(y_0,x_1,\dots,x_d,T)} \epsilon_\partial\left(\hat\gamma,\gammaunder\right) \ .
 \end{aligned}
\end{equation}
 Let $i \in \{1,2,\dots,d\}$ and $y_i \in \Crit f$ with $\mu(y_i)=\mu(x_i)+1$. Then
 \begin{align*}
  &a^d_{\fatY}(x_0,x_1,\dots,x_{i-1},y_i,x_{i+1},\dots,x_d) \cdot a^1_{\fatY}(y_i,x_i) \\
  &=\sum_{T \in \BinTree_d} \sum_{\left(\gammaunder,\hat{\gamma} \right) \in\Acal^d_{\fatY}(x_0,x_1,\dots,y_i,\dots,x_d,T) \times \widehat\MM(y_i,x_i) }  (-1)^{\sigma(x_0,x_1,\dots,x_{i-1},y_i,x_{i+1},\dots,x_d)+r(T)+\sigma(y_i,x_i)}  \epsilon\left(\gammaunder\right) \cdot \epsilon\left(\hat\gamma\right).
 \end{align*}
 We derive from part b) of Theorem \ref{SignGluingRootedLeafed} that
 \begin{align*}
  \epsilon_T\left(\gammaunder\right)\cdot\epsilon\left(\hat\gamma\right)=(-1)^{\mu(x_0)+ 1+(d-i)(n+1)+ \maltese_1^{i-1}+r(T)} \epsilon_{\partial} \left(\gammaunder,\hat\gamma\right) 
 \end{align*}
 for all $\left(\gammaunder,\hat{\gamma} \right) \in\Acal^d_{\fatY}(x_0,x_1,\dots,x_{i-1},y_i,x_{i+1},\dots,x_d,T) \times \widehat\MM(y_i,x_i)$ and $T \in \BinTree_d$. Inserting this into the above computation yields
 \begin{align*}
  &(-1)^{\maltese_1^{i-1}}a^d_{\fatY}(x_0,x_1,\dots,x_{i-1},y_i,x_{i+1},\dots,x_d) \cdot a^1_{\fatY}(y_i,x_i) \\
  &= (-1)^{\mu(x_0)+1+\sigma(x_0,x_1,\dots,x_{i-1},y_i,x_{i+1},\dots,x_d)+\sigma(y_i,x_i)+(d-i)(n+1)+\maltese_1^{i-1}} \\ &\qquad \qquad \sum_{T \in \BinTree_d} \sum_{\left(\gammaunder,\hat{\gamma} \right) \in\Acal^d_{\fatY}(x_0,x_1,\dots,x_{i-1},y_i,x_{i+1},\dots,x_d,T) \times \widehat\MM(y_i,x_i) } \epsilon_\partial\left(\gammaunder,\hat\gamma\right) \\
  &= (-1)^{\sigma(x_0,x_1,\dots,x_d)+\mu(x_0)+1}  \sum_{T \in \BinTree_d} \sum_{\left(\gammaunder,\hat{\gamma} \right) \in\Acal^d_{\fatY}(x_0,x_1,\dots,x_{i-1},y_i,x_{i+1},\dots,x_d,T) \times \widehat\MM(y_i,x_i) } \epsilon_\partial\left(\gammaunder,\hat\gamma\right) \ ,
  \end{align*}
  where we have applied part 3 of Proposition \ref{SigmaComputations}. Finally, one combines this last result with equations (\ref{EqIntSign}), (\ref{EqRootedSign}) and (\ref{BoundarySumZero}) as well as the boundary description from page \pageref{boundarydescriptionalleT} to prove that 
 \begin{align*}
  &\sum_{i=1}^d \sum_{l=0}^{d-i} \sum_{y} (-1)^{\maltese_1^{i-1}} a^{d-l}_{\fatY}(x_0,x_1,\dots,x_{i-1},y,x_{i+l+1},\dots,x_d)\cdot a^{l+1}_{\fatY}(y,x_i,\dots,x_{i+l}) \\
  &= (-1)^{\sigma(x_0,x_1\dots,x_d)+\mu(x_0)+1} \sum_{\left(\gammaunder_1,\gammaunder_2\right) \in \partial \bAcal^d_{\fatY}(x_0,x_1,\dots,x_d)}\epsilon_\partial\left(\gammaunder_1,\gammaunder_2\right) \quad = 0  \ .
 \end{align*} 
\end{proof}

 \bibliography{diss}
 \bibliographystyle{amsalpha}

 \end{document}